\documentclass[CJK,10pt]{article}
\usepackage[T1]{fontenc}
\usepackage{amsfonts}
\usepackage{amssymb}
\usepackage{amsmath}
\usepackage[utf8]{inputenc}
\usepackage[english]{babel}
\usepackage{amsthm}
\usepackage{esint}
\usepackage{dsfont}
\usepackage{geometry}
\usepackage{fancyhdr}
\usepackage{graphicx}
\usepackage{xcolor}
\usepackage{comment}
\usepackage{stmaryrd}
\usepackage[toc,page]{appendix}
\usepackage[nottoc, notlof, notlot]{tocbibind}
\numberwithin{equation}{section}

\usepackage[colorlinks,linkcolor=blue]{hyperref}
\newcommand{\R}{\mathbb{R}}
\newcommand{\Pp}{\mathbb{P}}
\newcommand{\N}{\mathbb{N}}
\newcommand{\Z}{\mathbb{Z}}

\newcommand{\calK}{\mathcal{K}}
\newcommand{\calF}{\mathcal{F}}

\newcommand{\calC}{\mathcal{C}}
\newcommand{\calA}{\mathcal{A}}
\newcommand{\calT}{\mathcal{T}}

\newcommand{\e}{\varepsilon}
\renewcommand{\aa}{a_\circ }

\newcommand{\loc}{\mathrm{loc}}
\newcommand{\dist}{\mathrm{dist}}
\newcommand{\sym}{\mathrm{sym}}
\newcommand{\Md}{\mathbb{M}_d}

\newcommand{\expec}[1]{\mathbb{E}\Big[ #1 \Big]}

\newcommand{\expecL}[1]{\mathbb{E}_L\Big[ #1 \Big]}

\newcommand{\step}[1]{\noindent \textit{Step} #1.}
\newcommand{\substep}[1]{\noindent \textit{Substep} #1.}

\newcommand{\E}{\mathbb{E}}

\newcommand{\dd}{\mathrm{d}}
\newcommand{\id}{\mathrm{Id}}
\newcommand{\rNL}{r_{\star,\xi,L}}
\newcommand{\urNL}{{\underline r_{\star,\xi,L}}}
\renewcommand{\r}{r_{\star}}
\newcommand{\rb}{{\underline r_{\star}}}
\newcommand{\Br}{B_{\star}}
\newcommand{\Brr}{B_{2\star}}
\newcommand{\Bt}{B_{5\star}}
\newcommand{\Dr}{D_{\star}}

\newcommand{\aL}{a_{\xi}}
\newcommand{\TaL}{a^*_{\xi}}

\newcommand{\rL}{\tilde{r}_{\star,\xi,e,L}}
\newcommand{\rr}{\tilde{r}_{\star}}

\newcommand{\corL}{{\tilde \phi_{\xi,e}}}
\newcommand{\corNL}{\phi_{\xi}}
\newcommand{\phiL}{\tilde \phi}

\newcommand{\sigL}{\tilde \sigma}

\newcommand{\per}{{\mathrm{per}}}
\newcommand{\us}{\bar{u}^{2s}_{\varepsilon,\delta}}

\newtheorem{theorem}{Theorem}[section]

\newtheorem{hypo}{Hypothesis}[section]
\newtheorem{corollary}[theorem]{Corollary}
\newtheorem{lemma}[theorem]{Lemma}

\newtheorem{proposition}[theorem]{Proposition}
\newtheorem*{proposition*}{Proposition}
\newtheorem{definition}{Definition}[section]
\newtheorem{example}{Example}[section]
\newtheorem*{theorem*}{Theorem}
\newtheorem*{lemma*}{Lemma}
\newtheorem*{corollary*}{Corollary}
\newtheoremstyle{TheoremNum}
      {\topsep}{\topsep}              		
      {\itshape}                      		
      {}                              		
      {\bfseries}                     		
      {.}                             		
      { }                             		
      {\thmname{#1}\thmnote{ \bfseries #3}}	
\theoremstyle{TheoremNum}

\theoremstyle{remark}
\newtheorem{remark}{Remark}

\title{Quantitative nonlinear homogenization: control of oscillations} 
\author{Nicolas Clozeau\thanks{Sorbonne Universit\'e, CNRS, Universit\'e de Paris, Laboratoire Jacques-Louis Lions, 75005~Paris, France} \  and  Antoine Gloria\thanks{Sorbonne Universit\'e, CNRS, Universit\'e de Paris, Laboratoire Jacques-Louis Lions, 75005~Paris, France \& Institut Universitaire de France \& Universit\'e Libre de Bruxelles, D\'epartement de Math\'ematique, 1050~Brussels, Belgium}}

\begin{document}

\maketitle

\begin{abstract}
Quantitative stochastic homogenization of linear elliptic operators is by now well-understood. In this contribution we move forward to the nonlinear setting of monotone operators with $p$-growth. This work is dedicated to a quantitative two-scale expansion result. 
By treating the range of exponents $2\le p <\infty$ in dimensions $d\le 3$, we are
able to consider genuinely nonlinear elliptic equations and systems such as $-\nabla \cdot A(x)(1+|\nabla u|^{p-2})\nabla u=f$ (with $A$ random, non-necessarily symmetric) for the first time.
When going from $p=2$ to $p>2$, the main difficulty is to analyze the associated linearized operator, whose coefficients are degenerate, unbounded, and depend on the random input $A$ via the solution of a nonlinear equation.
One of our main achievements is the 
control of this intricate nonlinear dependence, leading to annealed Meyers' estimates 
for the linearized operator, which are key to the optimal quantitative two-scale expansion result we derive (this is also new in the periodic setting).

\medskip

\end{abstract}


\tableofcontents

\section{Introduction}

\subsection{Nonlinear elliptic equations and homogenization}

Monotone operators are natural nonlinear extensions of linear operators in divergence form.
They model physical phenomena such as (nonlinear) conductivity in $\R^d$ ($d\ge 1$)
in form of
$$
-\nabla \cdot a (\nabla u(x)) = g(x),
$$
for some map $a:\R^d \to \R^d$ and function $g$ (which we shall take in conservative form later on). In the language of conductivity, such an equation is obtained by combining
\begin{itemize}
\item a conservation law: $\nabla \cdot q+g=0$, where $-q$ is the heat flux\footnote{although $-q$ is the physical heat flux,  we shall rather call $q$ the flux in this contribution.} and $g$ some exterior forcing,
\item with a constitutive law: $q=a(\nabla u)$, where $u$ is the temperature field.
\end{itemize}
For $a(\xi)=\xi$, we recover the Laplace equation (linear heat conduction), and for $a(\xi)=|\xi|^{p-2}\xi$ with $p\ge 2$ the $p$-Laplace equation (nonlinear heat conduction). 

Homogenization is the mathematical theory of composite materials. 
In the above picture, this means that the constitutive law depends on the space variable: the conductive medium is made of different materials with different conductivities.
This leads us to the more general conductivity problem 
$$
-\nabla \cdot a (x,\nabla u(x)) = g(x),
$$
with $a:\R^d \times \R^d, (x,\xi) \mapsto a(x,\xi)$.
To speak of composite materials we need two additional properties: the way the different materials are mixed should display some pattern with averaging properties (such as periodicity or stationarity and ergodicity) and there should be a scale separation between the size of the heterogeneities and the forcing term $g$. In more mathematical terms, we call $0<\e_0\ll 1$ the scale separation parameter of the ``actual'' model, $a$ the map when heterogeneities vary at the unit scale, and embed the problem at physical scale $0<\e_0\ll 1$ into the family of equations with 
arbitrary parameter $\e>0$
\begin{equation}\label{*eps}
-\nabla \cdot a (\tfrac x\e,\nabla u_\e(x)) = g(x).
\end{equation}
Homogenization aims at characterizing the asymptotic behavior of the temperature gradient  $\nabla u_\e$ and of the heat flux $q_\e=a(\frac x\e,\nabla u_\e)$ in the regime $0<\e \ll 1$.
These fields obviously have strong oscillations at scale $\e$, and, if any, convergence can only take place in weak norms (or, in physical terms, after local averaging) when $\e \downarrow 0$. 
The typical output is the existence of map $\bar a:\R^d\to \R^d$ such that $\nabla u_\e$ converges to $\nabla \bar u$, $q_\e$ converges to $\bar q = \bar a(\nabla \bar u)$, 
and $\bar u$ solves
\begin{equation}\label{*hom}
-\nabla \cdot \bar a (\nabla \bar u (x)) = g(x).
\end{equation}
In this case, $\bar a$ is the effective (or homogenized) conductivity of the composite material.
Homogenization can be summarized by the commutative diagram of Figure~\ref{fig:diag}, 
which is nothing but a particular instance of the $H$-convergence by Murat and Tartar.
\begin{figure*}[!h]
$$
\begin{array}{ccc|c}
  \text{Gradient field} & \text{Constitutive law} & \text{Minus flux} & \text{Conservation law}\\
  \nabla u_\e & \stackrel{\text{apply }a_\e} \longrightarrow & q_\e=a(\tfrac \cdot \e,\nabla u_\e) & -\nabla \cdot q_\e =   g
\\
  \downarrow & & \downarrow & \downarrow 
\\
  \nabla \bar u & \stackrel{\text{apply }\bar a} \longrightarrow & \bar  q=\bar a (\nabla \bar u)&-\nabla \cdot \bar q =   g
\end{array}
$$
\caption{Homogenization as a commutative diagram}\label{fig:diag}
\end{figure*}

The main motivation of the homogenization theory is to replace $\nabla u_{\e_0}$
and $q_{\e_0}$ by some effective quantities. The above answer
amounts to taking the weak limit of $\nabla u_\e$ and $q_\e$ as $\e \downarrow 0$
and therefore smooth out the oscillations in a consistent way. The natural following question is whether one can quantify the error made by replacing $(\nabla u_{\e_0},q_{\e_0})$ by $(\nabla \bar u, \bar q)$, and whether we can reconstruct a posteriori the oscillations of $(\nabla u_{\e_0},q_{\e_0})$ given $(\nabla \bar u, \bar q)$ and some intrinsic objects. 
This is the first aim of quantitative homogenization, which is by now well-developed for linear equations. In this contribution, we address genuinely nonlinear operators (such as regularized $p$-Laplacians with $p>2$) for the first time, and prove a quantitative two-scale expansion -- which characterize the \emph{spatial oscillations} of $(\nabla u_{\e_0},q_{\e_0})$ at scale $\e_0$, see Theorem~\ref{th:2s}.
The second aim of quantitative stochastic homogenization is to characterize the \emph{random fluctuations} of observables of $(\nabla u_{\e_0},q_{\e_0})$, see~Section~\ref{sec:towards} for a short discussion.

\medskip

In the rest of this section, we introduce the class of monotone operators we shall consider (and make precise assumptions on $a$), recall the associated qualitative homogenization results, and then turn to quantitative aspects.

\subsection{Qualitative homogenization of monotone operators}

The stochastic homogenization of monotone operators was first addressed by Dal Maso and Modica \cite{DalMaso86,DalMaso-Modica-86b} in their seminal papers ``Nonlinear stochastic homogenization (and ergodic theory)'', based on $\Gamma$-convergence (\cite{DeGiorgi-Franzoni-75,Marcellini-78}) and the subadditive ergodic theorem. We refer to the reference textbooks \cite{DalMaso-93} by Dal Maso, \cite{Braides98} by Braides and Defranceschi,  \cite{JKO94} by Jikov, Kozlov, and Ole{\u\i}nik, and   \cite{Pankov} by Pankov for the finest qualitative results available -- see also \cite{Marcellini-78,DalMaso86,DalMaso-Modica-86b,DalMaso-Gconv,DalMaso-corr,FMT,DG-16b}. 
 
Let us precisely define the class of maps $\hat a :\R^d \to \R^d$ we shall consider in this contribution.
We assume that $\hat a$ satisfies the following
three properties: $\hat a(0)=0$ and there exist $p\ge2$ and $C>0$ such that 
for all $\xi_1,\xi_2 \in \R^d$,
\begin{eqnarray}
|\hat a(\xi_1)-\hat a(\xi_2)|&\le& C  (1+|\xi_1|^{p-2}+|\xi_2|^{p-2})|\xi_1-\xi_2|,
\label{*cont}
\\
(\hat a(\xi_1)-\hat a(\xi_2))\cdot (\xi_1-\xi_2)& \ge& \frac1C  |\xi_1-\xi_2|^p.
\label{*coer}
\end{eqnarray}
Estimate~\eqref{*cont} is a continuity or boundedness property, whereas estimate~\eqref{*coer} is a monotonicity or coercivity property.
With a map $\hat a$ we associate the monotone differential operator $C^2(\R^d) \to C^0(\R^d), v \mapsto \nabla \cdot \hat a(\nabla v)$.
In view of regularity theory (see below) we also consider two strengthenings of \eqref{*coer}, which read for all $\xi_1,\xi_2\in \R^d$
\begin{eqnarray}
\label{*coer+-}
(\hat a(\xi_1)-\hat a(\xi_2))\cdot (\xi_1-\xi_2) &\ge& \frac1C ( |\xi_1-\xi_2|^{2}+ |\xi_1-\xi_2|^{p})=:\frac1C |\xi_1-\xi_2|^{2\& p}, \\
\label{*coer+}
(\hat a(\xi_1)-\hat a(\xi_2))\cdot (\xi_1-\xi_2) &\ge& \frac1C (|\xi_1|^{p-2}+|\xi_2|^{p-2}) |\xi_1-\xi_2|^{2}.
\end{eqnarray}
Let us give two examples.
The $p$-Laplacian, that is, $\hat a:\xi \mapsto |\xi|^{p-2}\xi$, satisfies \eqref{*cont}, \eqref{*coer},
and \eqref{*coer+}. 
The $p$-Laplacian regularized at 0, that is, $\hat a:\xi \mapsto (1+|\xi|^{p-2})\xi$, satisfies \eqref{*cont}, \eqref{*coer}, \eqref{*coer+-}, and \eqref{*coer+}. 

\medskip

In this contribution, we shall say that  an operator $\hat a : \R^d\to \R^d$ with $\hat a(0)=0$ is
\begin{itemize}
\item[(i)] \emph{monotone} with growth $p\ge 2$ if it satisfies \eqref{*cont} and~\eqref{*coer};
\item[(ii)] \emph{strongly monotone} with growth $p\ge 2$ if it  satisfies \eqref{*cont} and~\eqref{*coer+};
\item[(iii)]  monotone and \emph{non-degenerate} with growth $p\ge 2$ if it  satisfies \eqref{*cont} and~\eqref{*coer+-}.
\end{itemize}
In regularity theory, provided $\hat a$ is $C^1$, one usually state these assumptions in form of: for some $s \ge 0$ and for all $\xi,h\in \R^d$, we have
\begin{eqnarray*}
|\hat a(\xi)| \,\le\,C(s^2+|\xi|^2)^\frac{p-1}2, \quad |D_\xi \hat a(\xi)|\,\le\, C(s^2+|\xi|^2)^\frac{p-2}2, \quad D_\xi \hat a(\xi) h \cdot h \ge c(s^2+|\xi|^2)^\frac{p-2}2 |h|^2.
\end{eqnarray*}
If this holds for $s=0$, then we have \eqref{*cont}, \eqref{*coer},
and \eqref{*coer+}. If it holds for $s=1$, then we additionally have \eqref{*coer+}.
This is however in the form  \eqref{*cont} and \eqref{*coer+} that these assumptions will be used in terms of estimates in this contribution.

\medskip

The informal discussion of the previous paragraph can be made rigorous for monotone operators with growth $p\ge 2$ (more general results exist).
Consider a random Carath\'eodory\footnote{that is, measurable with respect to $x$ and continuous with respect to $\xi$.} map $a:\R^d \times \R^d \to \R^d, (x,\xi)\mapsto a(x,\xi)$,
which is stationary and ergodic in the space variable (see Section~\ref{sec:detail-rand} for details),
and is such that $a(x,\cdot)$ satisfies  \eqref{*cont} and~\eqref{*coer}
for all $x\in \R^d$ almost surely (with uniform exponent $p\ge 2$ and constant $C$).
We consider the family of solutions $\{u_\e\}_{\e>0}$ of \eqref{*eps}
for a suitable forcing term of the form  $g=\nabla \cdot f$ with $f \in L^p(\R^d)^d$, in
which case the natural solution space is the homogeneous Sobolev space 
$$\dot W^{1,p}(\R^d):=\{ v \in W^{1,p}_\loc(\R^d)\,|\,\nabla v \in L^p(\R^d)\}/\R.$$
Then, $\nabla u_\e$ and $q_\e=a(\frac \cdot \e,\nabla u_\e)$ converge weakly
in $L^p(\R^d)^d$ almost surely to $\nabla \bar u$ and $\bar q=\bar a(\nabla \bar u)$, respectively, where $\bar u \in \dot W^{1,p}(\R^d)$ solves the homogenized monotone equation \eqref{*hom} on $\R^d$, where $\bar a:\R^d \to \R^d$ is a monotone operator (no spatial dependence) which only depends on $a$ (and in particular, not on $f$ -- see Theorem~\ref{th:qual-hom} for the precise statement) and satisfies   \eqref{*cont} and~\eqref{*coer} with another constant $C'$ (depending only on $C$ and $d$).

\medskip

A natural question is whether one can infer more properties on $\bar a$ if we make more assumptions on $a$.
The answer is more subtle than one would expect. In general, if for all $x\in \R^d$,
$a(x,\cdot)$ 
\begin{itemize}
\item is quadratic, then $\bar a$ is quadratic (linear equations homogenize to linear equations);
\item is a $p$-Laplacian, that is, if $a(x,\xi)$ is homogeneous of degree $p-1$ in $\xi$ and satisfies \eqref{*cont} and \eqref{*coer}, then~\eqref{*coer+} also holds, and, by scaling,  $\bar a$ is also homogeneous of degree $p-1$ and satisfies \eqref{*cont} and~\eqref{*coer+} (for some constant $C$);
\item satisfies~\eqref{*cont} and~\eqref{*coer+-}, then $\bar a$ also satisfies~\eqref{*cont} and~\eqref{*coer+-} (for some constant $C$);
\item satisfies~\eqref{*cont}, \eqref{*coer+-} and~\eqref{*coer+},
then $\bar a$ satisfies~\eqref{*cont} and~\eqref{*coer+-}, but (most presumably) not necessarily~\eqref{*coer+}.
\end{itemize}
In particular, homogenization of a \emph{monotone} map yields a \emph{monotone} map, homogenization of a \emph{non-degenerate} monotone map yields a \emph{non-degenerate} monotone map, but homogenization of a \emph{strongly monotone} map \emph{might not} yield a \emph{strongly monotone} map (we are not aware of an explicit counter-example though).
This suggests that the homogenized operator $\nabla \cdot \bar a(\nabla)$ might not possess $C^{1,\alpha}$-regularity even if  the monotone operator $\nabla \cdot a(\cdot,\nabla)$  does.
Since elliptic regularity plays an important role in quantitative homogenization, this raises interesting questions and will impose restrictions on the operators we can consider.

\subsection{Classical regularity theory for monotone operators}\label{sec:class-reg}

In this section we recall what regularity theory one can expect for 
elliptic operators of the form $\nabla \cdot \hat a(\nabla)$ depending on properties of $\hat a$, which we then apply both to the random and the homogenized operators.
There are essentially two classes of results:
\begin{itemize}
\item \emph{Standard} growth conditions: If $\xi \mapsto \hat a(\xi)$ is smooth, and $\hat a$ is strongly monotone (that is, it satisfies \eqref{*cont} and \eqref{*coer+}), then $\nabla \cdot \hat a(\nabla)$ possesses $C^{k,\alpha}$-regularity and nonlinear Calder\'on-Zygmund theory, cf.~\cite{kuusi2014guide,DarkSide}; 
\item \emph{Non-standard} growth conditions: If  $\xi \mapsto \hat a(\xi)$ is smooth, and $\hat a$ is monotone and non-degenerate (that is,   it satisfies \eqref{*cont} and \eqref{*coer+-}), then $\nabla \cdot \hat a(\nabla)$ possesses $C^{k,\alpha}$-regularity and nonlinear Calder\'on-Zygmund theory provided $2\le p < \frac{2(d-1)}{d-3}$ (only active for $d\ge 4$), cf.~\cite{Bella_2020}\footnote{The new feature in~\cite{Bella_2020}, which establishes $C^{k,\alpha}$-regularity for local solutions of $\nabla \cdot a(\nabla u)=0$, is the largest range of exponents  $2\le p < \frac{2(d-1)}{d-3}$ compared to previous contributions -- more general results with right-hand sides and nonlinear Calder\'on-Zygmund theory  hold true as well.}.
\end{itemize}
On the one hand, strongly monotone operators are nicer since they possess regularity without restriction on the exponent $p$,  but strong monotonicity is not stable by homogenization (in the sense that the homogenized operator might not be strongly monotone even if the random operator is). On the other hand, although non-degenerate operators do only possess regularity if $p$ is close enough to $2$ in high dimensions (there is no restriction for $d\le 3$),  this property is stable by homogenization.

\medskip

As we shall see below, to establish quantitative homogenization results in the scalar setting:%
\begin{itemize}
\item The non-degeneracy condition \eqref{*coer+-} is needed for reasons that are independent of the regularity theory;
\item We need local regularity theory for the random operator, which (provided $x \mapsto a(x,\xi)$ is smooth enough) is automatic for $2\le p < \frac{2(d-1)}{d-3}$ in view of the assumption \eqref{*coer+-}, and follows from any $p\ge 2$ if we further assume \eqref{*coer+};
\item In the random setting, we also need the condition $2\le p < \frac{2(d-1)}{d-3}$ for the analysis of linearized operators. 
\end{itemize}
Most of our analysis also applies to monotone systems, provided we assume in addition that the operator has the Uhlenbeck structure. 

\medskip

To conclude, let us discuss the regularity properties of the homogenized operator in view of these assumptions.
For systems, since the Uhlenbeck structure is not stable by homogenization, our assumptions do not imply that the homogenized operator possesses $C^{1,\alpha}$ regularity.
The same holds in the scalar periodic setting in the range of exponents $p \ge \frac{2(d-1)}{d-3}$ under assumption~\eqref{*coer+}. In the (periodic or stochastic) setting in the range of exponents $p < \frac{2(d-1)}{d-3}$, the homogenized operator does possess $C^{1,\alpha}$ regularity.
We refer the reader to Section~\ref{sec:towards} for a further discussion of this observation.

\subsection{Quantitative homogenization for $p=2$}\label{sec:p=2}

The prototypical example for $p=2$ is the linear equation and its homogenized limit
\begin{equation}\label{e.intro0}
-\nabla \cdot A(\tfrac x\e) \nabla u_\e \,=\, \nabla \cdot f, \qquad -\nabla \cdot \bar A \nabla \bar u \,=\, \nabla \cdot f,
\end{equation}
where $A$ is a stationary and ergodic field of (say, symmetric for simplicity) matrices, which satisfies the uniform bound $\frac1C\id \le A\le C \id$ for some deterministic $C>0$.
The first quantitative estimates for this problem started with the early contributions \cite{Yurinskii-86} by Yurinski\u{\i} and \cite{NS} by Naddaf and Spencer, followed by  \cite{GO1,GO2,Gloria-Otto-10b,GNO1,MaO,GNO2} by Otto, Marahrens, Neukamm, and the second author.
The crucial ingredients in these works are Meyers' estimates, functional calculus in probability (in the form pushed forward in \cite{DG1,DG2} by Duerinckx and the second author for applications to mechanics, see Appendix~\ref{sec:FC} for the Gaussian version), and the central limit theorem scaling (a typical sign of integrable correlations).
Functional calculus is a powerful tool that allows to linearize the dependence of nonlinear random fields with respect to the underlying randomness (here the coefficients of the equation) and provide concentration of measures (in form of the control of high moments).
This strategy was recently revisited and very efficiently streamlined in  \cite{Otto-Tlse,josien2020annealed}.

Later, Armstrong and Smart adapted in \cite{AS} the strategy of Avellaneda and Lin \cite{Avellaneda-Lin-87,Avellaneda-Lin-91} from the periodic to the random setting and proved that the solution $u_\e$ of the heterogeneous equation enjoys the same regularity properties as the solution $\bar u$ of the homogenized equation, albeit on large scales (with respect to the heterogeneities). Since the homogenized equation has constant coefficients $\bar A$, one can derive large-scale $C^{k,\alpha}$-regularity for all $k\ge 0,0\le \alpha\le 1$ -- which constitutes non-perturbative large-scale regularity that goes way beyond Meyers' estimates.
On the one hand, this allowed them to distinguish the (more robust) large-scale regularity properties from the (finer) rates of convergence.  
On the other hand, this opened the way to go beyond the CLT scaling (and typically consider Gaussian coefficient fields with fat tails in \cite{GNO-reg,GNO-quant}), to by-pass functional calculus (and typically treat coefficient fields with a finite rate of dependence, which culminated in  \cite{AKM1,GO4,AKM2,AKM-book}), and to give a complete characterization of fluctuations \cite{DGO1}.
We refer the reader to  \cite{GNO-reg,GO4} for a thorough account of the literature emphasizing the different assumptions and approaches.

\medskip

The nonlinear version of \eqref{e.intro0} for $p=2$ reads
with $a(x,\xi)=A(x,\xi)\xi$ and $\bar a(\xi)=\bar A(\xi)\xi$,
$$
-\nabla \cdot a(\tfrac x\e,\nabla u_\e)  \,=\, \nabla \cdot f, \qquad -\nabla \cdot \bar a(\nabla \bar u)   \,=\, \nabla \cdot f,
$$
with the boundedness property $\frac1C\id   \le A(x,\xi)\le C \id$ for all $x,\xi \in \R^d$,
and the assumption that $a(\cdot,\xi) =D_\xi W(\cdot,\xi)$ for some convex function $\xi \mapsto W(\cdot,\xi)$ (which, in the linear case, takes the form $\frac12 \xi \cdot A \xi$).
A crucial feature of this model is that $\bar A$ satisfies the very same properties as $A$, which makes it the simplest nonlinear model possible. It was first successfully analyzed in \cite{AS,AFK-+,AFK-20} by Armstrong, Smart, Ferguson, and Kuusi (where mostly large-scale regularity is proved), and then in \cite{fischer2019optimal} by Fischer and Neukamm (where rates are obtained).
The new insight compared to \eqref{e.intro0} is that one needs to study the linearized corrector equation (see~\eqref{e.Lcorr} below) on top of the nonlinear corrector equation (see~\eqref{e.cor-eq} below), and the main merit of these works is to show that a certain version of the theory for the linear case extends ``verbatim'' to this nonlinear case.

\medskip

In the following paragraph we detail the new difficulties that appear when going from $p=2$ to a genuinely nonlinear operator with $p>2$. 
 
\subsection{New difficulties for $p>2$}

Whereas for $p=2$, the homogenized operator $\bar a$ has a good regularity theory provided it 
is smooth (which is proved in \cite{AS,AFK-+,AFK-20}), this is less clear as soon as $p>2$, cf.~Section~\ref{sec:class-reg} .

\medskip

When linearizing the nonlinear problem, one obtains a linear operator with coefficients that are heterogeneous and depend on the solution itself. 
Can we prove perturbative regularity (typically in form of Meyers' estimates) for this equation?
In the specific case $p=2$ treated in \cite{AFK-+,AFK-20,fischer2019optimal}, the coefficients of the linearized operator are bounded from above and below, so that Meyers' estimates are standard (which allowed Fischer and Neukamm to essentially follow the linear proof in \cite{fischer2019optimal}). This is not the case for $p>2$,
and this linear equation is hard to handle for two reasons:
\begin{itemize}
\item The coefficients may be degenerate (for the $p$-Laplacian e.g.), and 
despite recent progress on degenerate models, this degeneracy is currently out of reach.
Large-scale regularity has been established for the Laplacian on the percolation cluster by Armstrong and Dario in \cite{AD-18} (and optimal convergence rates by Dario in \cite{Dario}) and for linear elliptic systems with degenerate and unbounded coefficients under  moment bounds assumptions by Bella, Fehrmann and Otto in \cite{bella2018liouville}. The main new difficulty we face here is that, unlike in \cite{AD-18,bella2018liouville}, the degeneracy 
is \textbf{not prescribed a priori}: it is given by the critical set of harmonic coordinates  (that is the set of $x \in \R^d$ such that $\nabla \phi_\xi(x)+\xi=0$, cf.~\eqref{e.cor-eq} below). 
Precise information on this critical set is however currently unavailable for $d>2$ (even  
for the $p$-Laplacian, the unique continuation principle is not known to hold for $d>2$ and $p>2$, e.g.~\cite{Lindvist-06});
\item  The coefficients are unbounded. Since they depend on the solution of a nonlinear PDE,  our sole a priori control is given by the growth exponent $\frac{p}{p-2}$ (the larger $p$, the weaker the integrability). 
\end{itemize}

\medskip

The rest of the article is organized as follows.
In the upcoming section, we state our main results. 
Before we turn to the proofs, we describe our strategy and explain how we deal with the above difficulties. We also provide a thorough discussion of extensions and limitations of our approach,  and finally turn to the core of the proofs.

\section{Main results}

In a nutshell this article extends the results of \cite{Gloria-Otto-10b} (covering scalar linear equations in divergence form) to the setting of genuinely nonlinear monotone equations and systems. Although this article is mostly self-contained, \cite{Gloria-Otto-10b} and \cite{josien2020annealed} may serve as a gentle introduction to the subject and tools used in the present contribution.

\subsection{Qualitative assumptions and qualitative homogenization}\label{sec:qual}

We start with the well-known qualitative homogenization result, essentially due to Dal Maso and Modica~\cite{DalMaso-Modica-86b} -- see also \cite[Chapter~15]{JKO94} and \cite[Chapter~3]{Pankov}. 
In particular we assume that $a$ satisfies  \eqref{*cont} and \eqref{*coer+-}, conditions which are also satisfied by the homogenized operator.
We ask the unfamiliar reader not to worry too much about stochastic assumptions at this stage. Explicit assumptions and standard terminology (such as stationarity and ergodicity) will be given in Hypotheses~\ref{hypo0} and~\ref{hypo} below.
\begin{theorem}[Qualitative homogenization]\label{th:qual-hom}
If $(x,\xi) \mapsto a(x,\xi)$ is a stationary and ergodic random monotone operator of growth $p\ge 2$ such that $a(\cdot,0) \equiv 0$ and $\xi \mapsto a(\cdot,\xi)$ satisfies   \eqref{*cont} and \eqref{*coer+-} for a deterministic constant $C>0$, then there exists a monotone operator $\bar a$ of growth $p$ satisfying $\bar a(0)=0$, \eqref{*cont}, and \eqref{*coer+-},
such that for all $f\in L^p(\R^d)^d$, the unique weak solution $u_\e \in \dot W^{1,p}(\R^d)$ of 
\begin{equation}\label{e.eps-eq}
-\nabla \cdot a(\tfrac x\e,\nabla u_\e(x))=\nabla \cdot f(x)
\end{equation}
weakly converges   almost surely to the unique weak solution $\bar u  \in \dot W^{1,p}(\R^d)$ of
\begin{equation}\label{e.hom-eq}
-\nabla \cdot \bar a( \nabla \bar u(x))=\nabla \cdot f(x)
\end{equation}
(that is, $\nabla u_\e \rightharpoonup \nabla \bar u$ weakly in $L^p(\R^d)^d$),
where the operator $\bar a$ is characterized in direction $\xi \in \R^d$ by
$
\bar a(\xi)=\mathbb E[a(0,\nabla \phi_\xi(0)+\xi)]  
$
(where $\mathbb E[\cdot]$ denotes the expectation),
and $\phi_\xi$ is the corrector, defined as the unique almost sure distributional solution in $W^{1,p}_\loc(\R^d)$ of 
\begin{equation}
-\nabla \cdot a(x,\nabla \phi_\xi(x)+\xi)=0,
\label{e.cor-eq}
\end{equation}
anchored at the origin via $\int_B \phi_\xi=0$, and whose gradient  $\nabla \phi_\xi$ is stationary, has vanishing
expectation $\mathbb E[{\nabla \phi_\xi}]=0$, and satisfies 
\begin{equation}
\expec{|\nabla \phi_\xi|^{2\& p}} \lesssim |\xi|^{2\& p}:=|\xi|^2+|\xi|^p.
\label{e.cormoment}
\end{equation}
\end{theorem}

\subsection{Quantitative assumptions and quantitative two-scale expansion}\label{sec:detail-rand}

In view of the discussion above, to fix ideas and keep results and proofs readable, we consider the explicit class of $p$-Laplacians regularized at zero (see~Section~\ref{sec:extension} for more general conditions).
\begin{hypo}\label{hypo0}
Let $p\ge 2$, and consider the strongly monotone and non-degenerate operator
\begin{equation}\label{e.def-a}
a(x,\xi) \,:=\, A(x)(1+|\xi|^{p-2})\xi,
\end{equation}
where $A$ is a uniformly elliptic (non-necessarily symmetric) stationary ergodic matrix field.
More precisely, we assume that $A$ is smooth (uniformly wrt to the randomness)
and satisfies the ellipticity conditions for some $0<\lambda\le 1$
$$
\forall x,\xi \in \R^d: \quad \xi \cdot A(x)\xi \ge \lambda|\xi|^2 \text{ and }|A(x)\xi|\le |\xi|.
$$
\end{hypo}
Under Hypothesis~\ref{hypo0}, the monotone map $a$
almost surely satisfies  $a(\cdot,0)\equiv 0$, \eqref{*cont}, and \eqref{*coer+-} (and, incidently, also  \eqref{*coer+}) for  some $C$ depending only on $p$ and $\lambda$, so that the qualitative homogenization result of Theorem~\ref{th:qual-hom} applies.

\medskip

Let us now be more precise on the stochastic setting.
It is convenient to define the probability space via $\Omega=\{A:\R^d \to \Md(\lambda)\}$, endowed with some probability measure $\mathbb P$ (we denote by $\mathbb E[\cdot]$ the expectation).
In this setting, a random variable $Y$ can be seen as a (measurable) function of the form $A\mapsto Y(A)$. We say that the measure  $\mathbb P$ is ergodic if we have the implication: $Y(A(\cdot+z))=Y(A)$ for all $z\in \R^d$ $\implies$ $Y=\mathbb E[{Y}]$ almost surely.
We say that a random field $X:\R^d \times \Omega\to \R^k$ (for $k\in \N$) is stationary if for all $z\in \R^d$
and almost all $x\in \R^d$ we have $X(x+z,A)=X(x,A(\cdot+z))$, where $A(\cdot+z):x \mapsto A(x+z)$.
(Note that the expectation $\mathbb E[X(x)]$ of a stationary random field $X$ does not depend on $x\in \R^d$ and we simply write $\mathbb E[ X]$.) We use the notation $L^q(d\mathbb P)$ for the space of $q$-integrable random variables.

\medskip

In order to prove quantitative results, we need to quantify the ergodicity assumption, which we do by assuming Gaussianity of $\mathbb P$ and the integrability of the 
covariance function in the following sense. 
\begin{hypo}\label{hypo}
On top Hypothesis~\ref{hypo0}, assume that 
\begin{equation}\label{e.defAGauss}
A(y)= \chi * B(G(y)),
\end{equation}
where $B:\R \to \Md$ is a Lipschitz map, 
 $G$ is a stationary centered random Gaussian field on $\R^d$ (that is, $\mathbb E[{G}]=0$)
characterized by its covariance function $\mathcal C:\R^d \to \R, x\mapsto \mathcal C(x):=\mathbb E[{G(x)G(0)}]$, which we assume to be integrable on $\R^d$ (in the precise form of \eqref{ModelC} and \eqref{DecayModel} in Appendix~\ref{append:per}), and $\chi : \R^d \to [0,1]$ is a smooth compactly supported convolution kernel (the convolution is taken componentwise). 
In particular, $A$ is smooth (uniformly wrt to the randomness).
We further require $2\le p < \frac{2(d-1)}{d-3}$
in dimensions $d\ge 4$.\footnote{This condition, which comes from an argument of \cite{bella2019local}
and is used in the main part of this paper dedicated to large-scale Meyers estimates, is also crucially used in \cite{Bella_2020}.}
\end{hypo}
Our main achievement is an optimal quantitative corrector result, which extends the results of \cite{fischer2019optimal} to the genuinely nonlinear setting of $p>2$.
Following Dal Maso and Defranceschi \cite{DalMaso-corr}, we start with the suitable definition of
the two-scale expansion.
To this aim, we introduce a scale $\delta>0$ (which we should think of as being $\e$ in the upcoming result), set $\calK_\delta:=\delta\mathbb{Z}^d$ and for all $k\in \calK_\delta$, 
we define the cube ${Q}_{\delta}(k)=k+[-\delta,\delta)^d$ centered at $k$ and of sidelength $2\delta$. 
We also consider  a partition $(\eta_k)_{k\in \calK_\delta}$  of unity on $\mathbb{R}^d$ with the following properties: 
$0\leq \eta_k\leq 1$, $\eta_k\equiv 1$ on $\mathcal{Q}_{\frac{\delta}{2d}}(k)$, $\eta_k\geq c$ on $\mathcal{Q}_{(1-\frac{1}{3d})\delta}$, $\mathrm{supp}\, \eta_k \subset Q_{2\delta}(k)$, and $\vert\nabla\eta_k\vert\leq C\delta^{-1}$ (for some suitable $c,C>0$ independent of $\delta$). 
Given the solution~$\bar u$ of \eqref{e.hom-eq}, we introduce local averages associated with the partition of unity in form 
for all $k \in \calK_\delta$ of 
$$
(\nabla \bar u)_{k,\delta}\,:=\, \frac{\int_{\mathbb{R}^d}\eta_k\nabla\bar{u}}{\int_{\mathbb{R}^d}\eta_k},
$$
and define the two-scale expansion $\bar u^{2s}_{\e,\delta}$ associated with $\bar u$ via
\begin{equation}\label{e.2s}
\bar u^{2s}_{\e,\delta}:= \bar{u}+\e \sum_{k\in \calK_\delta}\eta_k\phi_{(\nabla \bar u)_{k,\delta}}(\tfrac \cdot\e),
\end{equation}
where $\phi_\xi$ denotes the corrector in direction $\xi \in \R^d$ (cf.~Theorem~\ref{th:qual-hom}).
This constitutes a convenient variant (introduced in  \cite{DalMaso-corr} to deal with monotone operators)
of the classical two-scale expansion $x\mapsto \bar u(x) + \e \phi_{\nabla \bar u(x)}(\tfrac x\e)$, which may raise 
measurability issues. Based on this two-scale expansion, we have the following optimal convergence result.
\begin{theorem}\label{th:2s}
Assume Hypothesis~\ref{hypo} and let $f \in L^p(\R^d)^d$. 
Let the weight $\mu_d: \R^k \to \R_+$ (for $k=1$ and $d$) be given by 
\begin{equation}\label{e.def-mud}
\mu_d(z)\,=\,\left\{
\begin{array}{rcl}
d=1&:& 1+\sqrt{|z|}, \\
d=2&:&  \log (2+|z|)^\frac12, \\
d>2&:& 1.
\end{array}
\right.
\end{equation}
For all $\e>0$ we denote by $u_\e \in \dot W^{1,p}(\R^d)$
the unique weak solution of \eqref{e.eps-eq}, by $\bar u  \in \dot W^{1,p}(\R^d)$ the unique weak solution of the 
homogenized equation~\eqref{e.hom-eq}, and by $\bar u^{2s}_{\e}$ the two-scale expansion \eqref{e.2s} for the choice $\delta=\e$. 
If the homogenized solution $\bar u$ satisfies
$\nabla \bar u \in L^\infty(\R^d)^d$ and $\mu_d \nabla^2 \bar u \in L^2(\R^d)^{d\times d}$, then we have
\begin{equation}
\|\nabla u_\e - \nabla \bar u^{2s}_{\e}\|_{L^2(\R^d)} \,\le \, C_{\e,\bar u} \, \e \mu_d(\tfrac1\e),
\label{th:2sIneg}
\end{equation}
where $C_{\e,\bar u}$ denotes a random variable that satisfies
\begin{equation}
\expec{\exp(c C_{\e,\bar u}^\alpha)}\le 2,
\label{th:2smoment}
\end{equation}
for some exponent $\alpha>0$ depending on $d$, $p$, $\lambda$, and $\|\nabla \bar u\|_{L^\infty(\R^d)}$, and some constant $c$ further depending on $\|\mu_d \nabla^2 \bar u\|_{L^2(\R^d)}$, but not on $\e$.
\end{theorem}
Some comments are in order:
\begin{itemize}
\item This result also holds for nonlinear systems with Uhlenbeck structure under the same assumptions on~$p$.
\item In the periodic setting, Theorem~\ref{th:2s} holds without restrictions on $p\ge 2$
and with $\mu_d \equiv 1$ in any dimension. This result is sharper than \cite{CS-04}, which contains the first quantitative two-scale expansion estimate for monotone periodic operators with $p>2$ (there, one needs to know that $\bar a$ satisfies \eqref{*coer+}   to construct a second-order two-scale expansion, which gives \eqref{th:2sIneg} after truncation and with a dependence of the constant  on stronger norms of $\bar u$).
\item The choice to work on the whole space with a right-hand side in divergence form allows one to avoid boundary layers (and therefore to truly focus on the homogenization error, in line with \cite{GNO-quant}) and to treat all dimensions at once. In particular one could state and prove a similar result on a bounded domain with Dirichlet boundary conditions, in which case the bound would be of the order of the square root of that in \eqref{th:2sIneg}.
\item This result takes the same form (with the same optimal rates) as for the linear case \cite{GNO-quant,GO4,AKM2} and for 
the nonlinear case \cite{fischer2019optimal} with $p=2$. As opposed to the latter, the stretched exponential exponent $\alpha$ in Theorem~\ref{th:2s} depends on $\|\nabla \bar u\|_{L^\infty(\R^d)}$ itself. This intricate dependence could be made explicit (provided we make the exponent and constants explicit in Gehring's lemma) and is reminiscent of the way we treat the non-degeneracy of the linearized equation (that is, perturbatively).
\item This result makes the a priori assumption that  $\nabla \bar u \in L^\infty(\R^d)^d$ and $\mu_d \nabla^2 \bar u \in L^2(\R^d)^{d\times d}$:
\begin{itemize}
\item In the scalar setting, under Hypothesis~\ref{hypo}, since $\bar a \in C^{1,1}_\loc$ (cf.~Corollary~\ref{coro:corr-diff} below), the conditions   $\nabla \bar u \in L^\infty(\R^d)^d$ and $\mu_d \nabla^2 \bar u \in L^2(\R^d)^{d\times d}$ are not restrictive and hold under suitable assumptions on the right-hand side $f$, capitalizing on the results \cite{Bella_2020} by Bella and Sch\"affner.
\item Since the above result is local in nature, this estimate holds on domains of $\R^d$ on which $\bar u$ has the required regularity. In any case, if $\nabla \bar u$ develops some singularity somewhere, one does not expect the two-scale expansion to be accurate in that region. This remark applies in particular to systems, and to the periodic setting in the regime of exponents $p \ge \frac{2(d-1)}{d-3}$ (which does not imply regularity of solutions of the homogenized problem).
\end{itemize}
\item The restriction $2\le p<\frac{2(d-1)}{d-3}$  on the exponent $p$ in Hypothesis~\ref{hypo} (which is only active in high dimensions $d\ge 4$) is related to the perturbative regularity theory in the large that we develop for the linearized operator in Section~\ref{mainresultNLunifL}. Indeed, the coefficient $a_\xi:=Da(\cdot,\xi+\nabla \phi_\xi)$ of the operator $a$ linearized at $\xi+\nabla \phi_\xi$
scales like $1+|\xi+\nabla \phi_\xi|^{p-2}$ and therefore only satisfies $\mathbb E[{|a_\xi|^\frac{p}{p-2}}]<\infty$ a priori: as $p$ increases, the stochastic integrability of the coefficients decreases. At some threshold (depending on dimension), this poor stochastic integrability cannot be compensated any longer by the Sobolev embedding --- whence our restriction (even in dimension 3, the argument to get all the exponents $2\le p<\infty$ 
is not straightforward -- see Sections~\ref{sec:strat} and~\ref{mainresultNLunifL}).
\end{itemize}

\subsection{Remarks on the strong monotonicity of $\bar a$}

Since the quantitative two-scale expansion of Theorem~\ref{th:2s} is conditional to the regularity of $\bar u$, it is worth further investigating under which additional assumptions one could prove that $\bar a$ possesses good regularity properties.
In the range of exponents $2\le p<\frac{2(d-1)}{d-3}$, this is fine as already emphasized.

\medskip

Let us now discuss the case $p\ge \frac{2(d-1)}{d-3}$ and assume that $a$ satisfies \eqref{*coer+} next to \eqref{*coer} and \eqref{*coer+-}.
If $\bar a$ also satisfied \eqref{*coer+}, then it would possess the desired regularity theory.
Let us first explain why this question is subtle, and assume for simplicity that $A$ is symmetric. In what follows we denote by $D_i$ the derivative with respect to the $i$-th entry of the vector $\xi \in \R^d$ 
(so that $D_i a(x,\xi):=\nabla_{\xi_i} a(x,\xi)$).
Informal computations (that are made rigorous in this paper) suggest that 
$
D \bar a(\xi) = \bar a_\xi,
$
where $\bar a_\xi$ is the homogenized matrix associated with the random coefficient field
\begin{equation}\label{e.axi-details}
a_\xi:= D a(\cdot,\xi+\nabla \phi_\xi)\,=\, (1+|\xi+\nabla \phi_\xi|^{p-2})A+(p-2)A\frac{(\xi+\nabla \phi_\xi)\otimes (\xi+\nabla \phi_\xi)}{|\xi+\nabla\phi_\xi|^2} |\xi+\nabla \phi_\xi|^{p-2}.
\end{equation}
As a first attempt to control $\bar a_\xi$ from below, we appeal to the Voigt-Reiss bounds (see \cite[Section~1.6]{JKO94}), which yields after neglecting the second contribution to $a_\xi$
$$
\bar a_\xi \ge \expec{ a_\xi^{-1}}^{-1} \ge \lambda \expec{(1+|\xi+\nabla \phi_\xi|^{p-2})^{-1}}^{-1},
$$
and amounts to controlling the harmonic average of $1+|\xi+\nabla \phi_\xi|^{p-2}$ from below (and therefore to have information on the critical set of the harmonic coordinate $x \mapsto \xi \cdot x+\phi_\xi(x)$). 
A second attempt is to consider the specific direction $\xi \cdot \bar a_\xi \xi$.
 Starting point is the (informal) minimization problem
\begin{eqnarray*}
\xi \cdot \bar a_\xi \xi& =& \inf_{\text{suitable }\nabla \psi} \expec{(\xi+\nabla \psi) \cdot a_\xi (\xi+\nabla \psi)}
\\
&\ge & \inf_{\text{suitable }\nabla \psi} \expec{(\xi+\nabla \psi) \cdot A (1+|\xi+\nabla \phi_\xi|^{p-2})(\xi+\nabla \psi)}.
\end{eqnarray*}
Call $\nabla \psi^*$ a minimizer. By minimality of $\nabla \psi^*$, the definition \eqref{e.def-a} of the monotone map $a$, and the corrector equation~\eqref{e.cor-eq}, we then have
\begin{eqnarray*}
\xi \cdot \bar a_\xi \xi&\ge & \expec{(\xi+\nabla \psi^*) \cdot A(1+|\xi+\nabla \phi_\xi|^{p-2})(\xi+\nabla \psi^*)}
\\
&=& \expec{(\xi+\nabla \phi_\xi) \cdot A(1+|\xi+\nabla \phi_\xi|^{p-2})(\xi+\nabla \psi^*)}
\\
&=&\expec{(\xi+\nabla \psi^*) \cdot a(\xi+\nabla \phi_\xi) }
\\
&\stackrel{\eqref{e.cor-eq}}=&\expec{(\xi+\nabla \phi_\xi) \cdot A(1+|\xi+\nabla \phi_\xi|^{p-2})(\xi+\nabla \phi_\xi) } \,\gtrsim \, |\xi|^{2\& p}.
\end{eqnarray*}
In dimension $d=1$, one directly has $D \bar a (\xi)\gtrsim 1+|\xi|^{p-2}$, which yields   \eqref{*coer+}. For $d>1$ this is different since from the 
a priori estimate $\xi \cdot D \bar a (\xi)\xi \gtrsim |\xi|^2(1+|\xi|^{p-2})$ we cannot deduce
$e \cdot D \bar a (\xi) e \gtrsim |e|^2(1+|\xi|^{p-2})$ for general $e\in \R^d$, unless combined with some isotropy arguments (which would ensure that controlling one direction is enough to control all directions).

\medskip

The upcoming results follow both paths. 
First we show that $\bar a$ satisfies \eqref{*coer+} provided we have a quantitative version of unique continuation, at least for periodic coefficients (as essentially noticed by 
Cherednichenko and  Smyshlyaev in \cite{CS-04}).
\begin{theorem}\label{th:isotropic-per}
Let $A$ be a $Q$-periodic Lipschitz matrix field.
For all $\xi \in \R^d$, denote by $\psi_\xi \in W^{1,p}_\per(Q)$ the unique weak solution of 
$$
-\nabla \cdot A(x)|\nabla \psi_\xi+\xi|^{p-2}(\nabla \psi_\xi+\xi)=0.
$$
Assume that for all $\xi \in \R^d$, there exists $r>0$ such that the $r$-tubular neighborhood 
$\calT_r(\xi)=\{x+B_r \,|\,x \in \calC(\xi)\}$ of the critical set $\calC(\xi)=\{x \in \R^d\,|\, \xi+\nabla \psi_\xi(x)=0\}$ is such that $\R^d \setminus \calT_r(\xi)$ is a connected set.
Then there exists $C>0$ such that $\bar a$ satisfies  \eqref{*coer+-} and \eqref{*coer+}.
\end{theorem}
\begin{remark}
The assumptions of Theorem~\ref{th:isotropic-per} are quite strong. They are satisfied in dimension $d=2$ by \cite{AS-01} (which shows that $\calC(\xi)\cap Q$ is indeed a finite union of points) -- but, in this setting, regularity for the homogenized operator also holds because $\bar a$ is non-degenerate, cf.~\cite{Marcellini-91}. For $d>2$ this is a widely open problem. For linear equations, this follows from~\cite{CNV-15}.
\end{remark}
In the random setting, we have a positive result assuming 
the statistical isotropy of $A$, which is new and holds in any dimension.
\begin{theorem}\label{th:isotropic}
On top of Hypothesis~\ref{hypo0}, assume that $A(x)=b(x) \id$ for some scalar-valued function $b$ and that for all $R \in SO(d)$, $b(R \cdot)$ and $b$ have the same (joint) distribution (in which case $A$ is statistically isotropic). Then there exists $C>0$ such that $\bar a$ satisfies \eqref{*coer+-} and \eqref{*coer+}.
\end{theorem}
These results are proved in~Appendices~\ref{app:isotropic-per} and~\ref{app:closed}.
We suspect that Theorems~\ref{th:isotropic-per} and~\ref{th:isotropic} hold under weaker assumptions but we are currently unable to establish this (even using the quantitative estimates proved in this paper).

\section{Strategy of the proof, extensions, and limitations}\label{sec:strat}

Throughout the paper, we use the short-hand notation $\lesssim$, $\gtrsim$ and $\sim$ for $\le C\times$, $\ge C\times$, and $\frac1C \times \le \cdot \le C\times$ for some universal constant $C$ depending on $\lambda,d,p$ (and possibly on further quantities displayed as subscripts). We use  $\ll$ and $\gg$ for $\le \frac1C \times$ and $\ge C \times$ in the case when $C$ needs to be chosen large enough. 
Recall that for all $t\ge 0$ and $p\ge 2$, we set $t^{2\& p}=t^2+t^p$.

\subsection{General strategy and auxiliary results}

In this section, we motivate our general strategy by comparison to the linear setting.
To be more precise, we also consider the linear homogenization problem on $\R^d$
\begin{equation}\label{e.lin-discu}
-\nabla \cdot A(\tfrac x\e) \nabla v_\e(x)=\nabla \cdot f(x),
\end{equation}
with the same assumptions on $A$ as in Hypothesis~\ref{hypo}.

\medskip

We start by defining the notion of (nonlinear) flux corrector.
\begin{definition}\label{defsigmaNL}
For all $\xi\in \R^d$, there exists a unique skew-symmetric random matrix field $(\sigma_{\xi,ij})_{1\le i,j\le d}$, which 
solves almost surely in the distributional sense in $\R^d$ the flux corrector equation
\begin{equation}
-\triangle \sigma_{\xi,ij} \,=\, \partial_i (a(\cdot,\xi+\nabla\phi_{\xi})\cdot e_j)-\partial_j(a(\cdot,\xi+\nabla\phi_{\xi})\cdot e_i),
\label{e.Laplace-sig}
\end{equation}
which is anchored at the origin via $\int_B\sigma_{\xi}=0$, and whose gradient $\nabla \sigma_\xi$ is stationary, has 
vanishing expectation $\mathbb E[{\nabla \sigma_\xi}]=0$, and is bounded
in the sense
$$
\expec{|\nabla \sigma_\xi|^\frac{p}{p-1}}\lesssim |\xi|^{2\& p}.
$$
In addition we have
\begin{equation}\label{e.div-sig}
\nabla\cdot \sigma_{\xi}=a(\cdot,\xi+\nabla\phi_{\xi})-\bar a(\xi),
\end{equation}
where the divergence of a matrix field $\sigma$ is understood as $(\nabla \cdot \sigma_\xi)_i=\sum_{j=1}^d \partial_j \sigma_{\xi,ij}$.
\end{definition}
The proof of existence and uniqueness of $\sigma_\xi$ is essentially the same as in the linear setting in \cite{GNO-reg} when $p=2$, provided the adaptations of  \cite[Lemma~1]{bella2018liouville} for $p>2$. 

\medskip

Note that in the linear case \eqref{e.lin-discu}, we have the scaling relation $(\phi_{t\xi},\sigma_{t\xi})=t(\phi_\xi,\sigma_\xi)$ (in the nonlinear setting, the homogeneities of $\phi_\xi$ and of $\sigma_\xi$ with respect to $\xi$ are different, and not explicit since $a$ has no homogeneity due to the regularization at $0$, cf.~\eqref{e.def-a}).

The interest of the flux corrector is that it allows to put the remainder in the equation 
satisfied by the two-scale expansion error in conservative form (which is convenient to use energy estimates or montonicity). More precisely, in the linear setting, the two-scale expansion 
of $v_\e$ takes the simpler form
\begin{equation}\label{e.2sv}
\bar v^{2s}_{\e}:= \bar{v}+\e \phi_i (\tfrac \cdot\e)\partial_i \bar v,
\end{equation}
where $\bar v$ solves the homogenized equation $-\nabla \cdot \bar A \nabla \bar v=\nabla \cdot f$,  $\phi_i$ denotes the corrector in the canonical direction $e_i$, and where we use implicit summation on the repeated index $i$. Since $\phi_i (\tfrac x\e)\partial_i \bar v(x)=\phi_{\nabla \bar v(x)}(\tfrac x\e)$, \eqref{e.2sv} corresponds to \eqref{e.2s} in the limit $\delta \downarrow 0$.
In this case, one can prove (making a crucial use of the skew-symmetry of $\sigma$ -- see the proof of Theorem~\ref{th:2s} in our nonlinear setting) that $v_\e- \bar v^{2s}_{\e}$ satisfies the equation
\begin{equation}\label{e.2sc-errLin}
-\nabla \cdot A(\tfrac \cdot\e) \nabla (v_\e- \bar v^{2s}_{\e})
\,=\,\e \nabla \cdot ((A\phi_i-\sigma_i)(\tfrac \cdot \e) \nabla \partial_i \bar v)
\end{equation}
(the factor $\e$ comes by scaling since there is one gradient more in the right-hand side). 
This yields the bound \eqref{th:2sIneg} by an energy estimate provided
we control the growth of $(\phi,\sigma)$. 
In the nonlinear setting, we rather expect a term of the form $a_\xi \phi_\xi-\sigma_\xi$ 
in the right-hand side of \eqref{e.2sc-errLin} (see the proof of Theorem~\ref{th:2s} for the precise statement). Note that  $a_\xi \phi_\xi$ (with $a_\xi=Da(\cdot,\xi+\nabla \phi_\xi)$, cf.~\eqref{e.axi-details}) and $\sigma_\xi$  both scale like $(|\xi|^{2\& p})^\frac{p-1}{p}$ (which has to be compared to the difference of scalings of $\phi_\xi$ and $\sigma_\xi$ themselves, cf.~Definition~\ref{defsigmaNL}).

The main result on the extended nonlinear corrector $(\phi_\xi,\sigma_\xi)$ is as follows.
(Despite the above discussion, we do not make a difference in the scalings of $\phi_\xi$
and $\sigma_\xi$ wrt $\xi$. This dependence is indeed not explicit as one could have expected, due to the nonlinear nature of the problem.)
\begin{theorem}\label{th:corrNL}
Under Hypothesis~\ref{hypo},  there is an exponent $\alpha>0$ depending on $\lambda$, $p$,  and $d$ such that for all $K\ge 1$ and $\xi\in \R^d$ with $|\xi|\le K$ the stationary extended corrector gradient 
$\nabla (\phi_\xi, \sigma_\xi)$
 satisfies for some constant $c_{K}>0$ (additionally depending on $K$)
\begin{equation}\label{e.bdd-grad-corrNL}
\expec{\exp( c_{K}|\nabla (\phi_\xi,\sigma_\xi)|^\alpha
)}\le 2.
\end{equation}
For all $g \in L^2(\R^d)$, averages of $(\nabla \phi_\xi,\nabla \sigma_{\xi,ij})$ display the CLT scaling\footnote{Indeed, for $g=|B_R|^{-1}\mathds{1}_{B_R}$, the right-gand side of \eqref{e:corr-NL-CLT} scales like $R^{-\frac d2}$.} in the form  
\begin{equation}\label{e:corr-NL-CLT}
 \Big| \int g (\nabla \phi_\xi ,\nabla \sigma_{\xi})\Big|\,\le \,  C_{\xi,g} \Big(\int |g|^2\Big)^\frac12,
\end{equation}
where $C_{\xi,g}$ is a random variable with finite stretched exponential moment
$$
\expec{ \exp(c_{K} C_{\xi,g}^\alpha)}\le 2,
$$
for some exponent $\alpha>0$ depending on $p$, $\lambda$, and $d$, and some constant $c_{K}>0$ further depending on $K$ (but all independent of $g$).
This directly implies the following bounds on the growth of $(\phi_\xi,\sigma_\xi)$:
For all $x\in \R^d$,
\begin{equation}\label{e.growth-nlc}
 |(\phi_\xi,\sigma_\xi)(x)| \,\le\,  C_{x,\xi} \mu_d(x),
\end{equation}
where $\mu_d$ is defined in \eqref{e.def-mud} and $C_{x,\xi}$ is a random variable with the same stochastic integrability as $C_{\xi,g}$. 
\end{theorem}
\begin{remark}
Under Hypothesis~\ref{hypo0}, for $Q$-periodic matrix fields $A$, the nonlinear correctors are bounded in $C^{1,\alpha}(Q)$ (no restriction on $p\ge 2$).
\end{remark}
As pointed out in \cite{AFK-+,AFK-20} and also used in \cite{fischer2019optimal} for $p=2$,  controlling the growth of correctors is not enough in the nonlinear setting.
This should not come as a surprise when comparing \eqref{e.2sv} to \eqref{e.2s}.
The additional gradient on $\nabla \bar v$ in the right-hand side of \eqref{e.2sc-errLin} (at the origin of the factor $\e$)  indeed comes for local differences 
$$
\phi_{\nabla \bar v(x_1)} - \phi_{\nabla \bar v(x_2)}=|x_1-x_2| \phi_{ \frac{\nabla\bar v(x_1)-\nabla\bar v(x_2)}{|x_1-x_2|}},
$$
when reformulated by taking advantage of the linearity of the corrector.
In the nonlinear setting, this identity does not hold any longer for $\phi_\xi$ (and even less for $\sigma_\xi$ by homogeneity). It is however replaced by the following Lipschitz-continuity results.
\begin{corollary}[Control of nonlinear corrector differences]\label{coro:corr-diff}
Under Hypothesis~\ref{hypo}, for all $K\ge 1$ and for all $\xi_1,\xi_2 \in \R^d$ with
$|\xi_1|,|\xi_2|\le K$, 
we have for all $x\in \R^d$,
$$
|\nabla (\phi_{\xi_1}-\phi_{\xi_2}, \sigma_{\xi_1}-\sigma_{\xi_2})(x)|\,\le\, C_{x,\xi_1,\xi_2}|\xi_1-\xi_2|, \quad |(\phi_{\xi_1}-\phi_{\xi_2}, \sigma_{\xi_1}-\sigma_{\xi_2})(x)|\,\le\, C_{x,\xi_1,\xi_2} |\xi_1-\xi_2| \mu_d(x),
$$
where $C_{x,\xi_1,\xi_2}$ satisfies for some exponent $\alpha_K>0$ and some constant $c_{K}>0$ depending on $\lambda$, $p$, $d$, and $K$,
\begin{equation*}
\expec{\exp(c_K C_{x,\xi_1,\xi_2}^{\alpha_K})}\le 2.
\end{equation*}
In particular, $\xi \mapsto \bar a(\xi)$ is locally $C^{1,1}_\loc$.
\end{corollary}
\begin{remark}
Under Hypothesis~\ref{hypo0}, for $Q$-periodic matrix fields $A$,  corrector differences
are controlled by $|\xi_1-\xi_2|$ in $C^{1,\alpha}(Q)$ (no restriction on $p\ge 2$), and $\bar a$ is $C^{1,1}_\loc$ as well.
\end{remark}
To prove such a result, we have to analyze the  dependence of correctors with respect to the direction $\xi$, which leads us to the notion of linearized correctors.
The control of (nonlinear) corrector differences is obtained as a corollary of bounds
on these linearized correctors.

\medskip

The following lemma (which is only used in an approximation argument) defines these linearized correctors. It is a consequence of \cite[Section~4]{CD-16} and \cite[Lemma~1]{bella2018liouville} (which is devoted to the existence and uniqueness for linear corrector equations with unbounded and degenerate coefficients that are prescribed in advance). 
In the actual proofs, we shall only consider linearized correctors on bounded domains (see the discussion on periodization below), and the definition and control of 
the whole-space linearized corrector is given for completeness. 
\begin{lemma}\label{lem:def-lincorr}
Under Hypothesis~\ref{hypo}, for all $\xi \in \R^d$ set $a_\xi:=D a(\cdot,\xi+\nabla \phi_\xi)$ (cf.~\eqref{e.axi-details}).
For all $e \in \R^d$ there exists a unique 
random field $\tilde \phi_{\xi,e}$ that solves almost surely in the distributional sense in $\R^d$ the linearized corrector equation
\begin{equation}
-\nabla \cdot a_\xi (e+\nabla \tilde \phi_{\xi,e}) \,=\, 0,
\label{e.Lcorr} 
\end{equation}
anchored at the origin via $\int_B \tilde \phi_{\xi,e}=0$, and whose gradient $\nabla \tilde \phi_{\xi,e}$ is stationary, has
vanishing expectation $\mathbb E[{\nabla \tilde \phi_{\xi,e}}]=0$, and is bounded
in the sense of
\begin{equation*}
\expec{|\nabla \tilde \phi_{\xi,e}|^2(1+|\xi+\nabla \phi_\xi|)^{p-2}}\lesssim (1+|\xi|^{p-2}) |e|^2.
\label{e.Lboundcor}
\end{equation*}
In addition, there exists a skew-symmetric random matrix field $(\sigma_{\xi,e,ij})_{1\le i,j\le d}$,
which solves almost surely in the distributional sense in $\R^d$ the linearized flux corrector equation
\begin{equation}\label{e:eq-sigmaL}
-\triangle \tilde \sigma_{\xi,e,ij} \,=\, \partial_i (a_\xi(e+\nabla \tilde \phi_{\xi,e})\cdot e_j)-\partial_j(a_\xi(e+\nabla \tilde \phi_{\xi,e})\cdot e_i),
\end{equation}
which is anchored via $\int_B \tilde \sigma_{\xi,e}=0$ almost surely, whose gradient $\nabla \tilde \sigma_{\xi,e}$ is stationary and is bounded
in the sense
$$
\expec{|\nabla \tilde \sigma_{\xi,e}|^\frac{p}{p-1}}\lesssim (1+|\xi|^{p-2})^\frac{p}{p-1} |e|^\frac{p}{p-1},
$$
and which satisfies the property
$$
\nabla\cdot \tilde\sigma_{\xi,e}= a_\xi(e+\nabla \tilde \phi_{\xi,e})-\bar a_\xi e,
$$
where $\bar a_\xi e=\expec{a_\xi(e+\nabla \tilde \phi_{\xi,e})}$.
\end{lemma}
The upcoming theorem gives further information on the linearized correctors, in line with Theorem~\ref{th:corrNL}
for the nonlinear correctors.
\begin{theorem}\label{th:corrL}
Under Hypothesis~\ref{hypo}, for all 
$K \ge 1$, $\xi,e \in \R^d$ with $|e|=1$ and $|\xi|\le K$, the stationary extended linearized corrector gradient $\nabla (\tilde \phi_{\xi,e}, \tilde\sigma_{\xi,e})$ satisfies
for some exponent $\alpha_K>0$ and some constant $c_{K}>0$ depending on $\lambda$, $p$, $d$, and $K$,
\begin{equation}\label{e.bdd-grad-corrL}
\expec{\exp(c_K |\nabla (\tilde\phi_{\xi,e},\tilde\sigma_{\xi,e})|^{\alpha_K})}\le 2.
\end{equation}
For all $g \in L^2(\R^d)$, averages of $(\nabla \tilde\phi_{\xi,e},\nabla \tilde\sigma_{\xi,e})$ display the CLT scaling in the form 
\begin{equation}\label{e:corr-L-CLT}
 \Big| \int g (\nabla \tilde \phi_{\xi,e} ,\nabla \tilde \sigma_{\xi,e})\Big|\,\le \,  C_{\xi,g} \Big(\int |g|^2\Big)^\frac12,
\end{equation}
where  $C_{\xi,g}$ is a random variable with stretched exponential moments 
$
\expec{\exp( c_{K} C_{\xi,g}^{\alpha_K})}\le 2,
$
for some exponent  $\alpha_K>0$ and some constant $c_{K}>0$ depending
on $p$, $\lambda$,  $d$, and  $K$.
This directly implies that for all $x\in \R^d$, we have $|(\tilde\phi_{\xi,e},\tilde\sigma_{\xi,e})(x)| \,\le\,   C_{x,\xi}\mu_d(x)$,
where $C_{x,\xi}$ is a random variable with the same moment bounds as $C_{\xi,g}$. 
\end{theorem}
\begin{remark}
Under Hypothesis~\ref{hypo0}, for $Q$-periodic matrix fields $A$, the linearized correctors exist and are bounded in $C^{1,\alpha}(Q)$ (no restriction on $p\ge 2$).
\end{remark}

\medskip

The general strategy we described above is essentially the same as in \cite{fischer2019optimal} for $p=2$, itself very close to the strategy in the linear setting \cite{GNO-quant}.
In line with \cite{GO1,GO2,Gloria-Otto-10b,GNO1,Otto-Tlse,josien2020annealed}, we are after Meyers' type estimates for the linearized operator $-\nabla \cdot a_\xi \nabla$ (defined in Lemma~\ref{lem:def-lincorr}) .
The main difference between the present work and 
\cite{AFK-+,AFK-20,fischer2019optimal} is the way we obtain 
these estimates -- that is, in the annealed version of Theorem~\ref{th:annealedmeyers} below.
For $p=2$, annealed Meyers estimates (even without loss of stochastic integrability) 
follow rather directly from the boundedness of $a_\xi$ from above and below. 
In our genuinely nonlinear setting, these estimates are difficult to establish since  $a_\xi$ may be degenerate and unbounded (recall that  the original proof of Meyers' estimates argues by perturbation and requires that $\frac{\inf a_\xi}{\sup a_\xi}>0$, whereas we have  $\frac{\inf a_\xi}{\sup a_\xi}=0$ almost surely). 
There are two reasons why $\frac{\inf a_\xi}{\sup a_\xi}=0$: the degeneracy of $a_\xi$ and the unboundedness of $a_\xi$. 
The upcoming technical discussion points out the difficulties in the analysis, give hints on how to treat them, and explains how it leads to the above results.

\smallskip

First, in view of the difficulty to control the critical set of harmonic coordinates, we have imposed a non-degeneracy condition from the very beginning and assumed that $a(x,\cdot)$ satisfies \eqref{*coer+-} (which rules out the $p$-Laplacian, but not the $p$-Laplacian regularized at 0). Let us emphasize that this only yields the \emph{non-degeneracy} of the linearized operator in a \emph{perturbative} way (it disappears in the regime when the solution has a large gradient). Doing so, the main remaining (and most important) difficulty is the unboundedness of the coefficients of the linearized operator.

\smallskip

Meyers estimates are the object of Section~\ref{mainresultNLunifL}.
We start with discussing Subsections~\ref{sec:MR}~\&~\ref{perturbativeregsection} and the quenched Meyers estimates in the large.
The starting point is the energy estimate~\eqref{e.cormoment} in form of $\expec{|a_\xi|^\frac{p}{p-2}} \lesssim 1+|\xi|^p$ (cf.~\eqref{e.axi-details}), which 
yields two challenges: it is stochastically-averaged and gives a poor integrability for large $p$.
We begin with the stochastically-averaged part.
The idea is to relax the condition $\frac{\inf a_\xi}{\sup a_\xi}>0$ into the milder requirement $\frac{\inf \fint_{B_r(x)} a_\xi}{\sup \fint_{B_r(x)}a_\xi}>0$ for some $r>0$, which, in turn, would yield the weaker (yet sufficient) Meyers' estimate at  scale $r>0$ in form of: For $u,g$ related via $-\nabla \cdot a_\xi \nabla u=\nabla \cdot g$, we have
for some Meyers' exponent $m>2$%
\begin{equation}\label{e.LSMey}
\int_{\R^d}\Big(\fint_{B_{r}(x)}\vert\nabla u \vert^2  \Big)^{\frac{m}{2}}\dd x\lesssim \Big(\int_{\R^d}\Big(\fint_{B_{r}(x)}\vert\nabla u\vert^2 \Big)\,\dd x\Big)^{\frac{m}{2}}
+\int_{\R^d}\Big(\fint_{B_{r}(x)}\vert g \vert^2 \Big)^{\frac{m}{2}}\dd x.
\end{equation}
Such a weakening of Meyers' estimates still does not hold in our setting because estimate~\eqref{e.cormoment} cannot be turned into an almost sure bound (with a uniform choice of $r$). A further weakening consists in letting the radius of the ball $B_r(x)$ in \eqref{e.LSMey} be random and depend on the point $x\in \R^d$.
We thus introduce in  Definition~\ref{minimalscaleNL} the Meyers minimal radius $\r$, a random field on $\R^d$ which essentially ensures that, with the notation $B_\star(x):=B_{\r(x)}(x)$,  $\frac{\inf_x \fint_{B_{\star}(x)} a_\xi}{\sup_x \fint_{B_{\star}(x)}a_\xi} >0$ (with a deterministic positive lower bound).
Such a form of the Meyers' estimates with a random scale $\r$ was first used by Armstrong and Dario in \cite{AD-18} to deal with homogenization in percolation.
To obtain an estimate in the spirit of \eqref{e.LSMey} (with $B_r(x)$ replaced by $B_{\star}(x)$), we rely on the standard proof of Meyers' estimates going through a reverse H\"older inequality and Gehring's lemma.
For uniformly elliptic equations, the reverse H\"older inequality is a consequence of Caccioppoli's inequality and of the Sobolev embedding. 
Caccioppoli's inequality for $a_\xi$-harmonic functions $u$ (say, $-\nabla \cdot a_\xi \nabla u=0$ in the ball $B_{2R}$ centered at 0) typically takes the form 
$$
\int_{B_R} \nabla u \cdot a_\xi \nabla u \, \lesssim \, \frac1{R^2} \int_{B_{2R}} |u|^2 |a_\xi|.
$$
Assuming that $\int_{B_{2R}} u=0$, the next step is to appeal to the Poincar\'e-Sobolev inequality $\fint_{B_{2R}} |u|^2 \lesssim R^2 \Big(\fint_{B_{2R}} |\nabla u|^\frac{2d}{d+2}\Big)^\frac{d+2}{d}$. The random coefficient $|a_\xi| \sim 1+|\nabla \phi_\xi+\xi|^{p-2}$ is however unbounded (and not of the Muckenhoupt class), and the estimate~\eqref{e.cormoment}
only yields $\fint_{B_{2R}} |a_\xi|^\frac{p}{p-2}\sim \fint_{B_{2R}} 1+|\nabla \phi_\xi+\xi|^p \lesssim 1+|\xi|^p$ provided $2R \ge \r(0)$ (as a consequence of the definition of $\r$). We thus need to first appeal to H\"older's inequality with exponents $(\frac{p}{p-2},\frac{p}2)$ to upgrade
 the Caccioppoli inequality into 
$$
\fint_{B_R} \nabla u \cdot a_\xi \nabla u \, \lesssim \, \frac1{R^2}  \Big(\fint_{B_{2R}} |u|^p\Big)^\frac2p.
$$
Then, assuming that $p  < \frac{2d}{d-2}$, the Poincar\'e-Sobolev inequality yields
the desired reverse H\"older's inequality 
$
\fint_{B_R} \nabla u \cdot a_\xi \nabla u \, \lesssim \,  \Big(\fint_{B_{2R}} |\nabla u|^{p_*}\Big)^\frac2{p_*}
$
with exponent $p_*=\frac{d+p}{dp}<2$.
In dimension $d=3$, the condition on $p$ reads $p<6$.
To reach the larger range of exponents of Theorem~\ref{th:2s}, the main observation is that 
one can improve Caccioppoli's inequality by choosing wisely the cut-off function -- following an idea by Bella and Sch\"affner  \cite{bella2019local} (see Lemma~\ref{lemmabella} below). Doing so, we are able to use the Sobolev embedding in dimension $d-1$ rather than $d$, and therefore treat all exponents $p\ge 2$ in the physically-relevant  dimension $d=3$. This yields the large scale Meyers' estimates of Theorem~\ref{unweightmeyers} (with a condition on $p$ in dimensions $d\ge 4$).

\smallskip

In the form of Theorem~\ref{unweightmeyers}, the Meyers estimates are only useful if we have a good control of the Meyers minimal radius $\r$ (the larger $\r$, the weaker the estimate), which we obtain in Subsection~\ref{sec:MR-control}. Since $\r$ is a stationary random field, by control we mean moment bounds in probability. This is where our contribution further differs from other contributions of the literature on degenerate or unbounded coefficients: the statistics of $a_\xi$ (which drives the moments of $\r$) are not given a priori (as opposed to the percolation cluster in \cite{AD-18}, or to the moment bounds on $a_\xi$ in \cite{bella2018liouville}) but part of the problem -- estimate~\eqref{e.bdd-grad-corrNL}, which is essentially equivalent to the control of $\r$
in Theorem~\ref{boundrNLprop}, is the very output of the analysis.
Indeed, the coefficients $a_\xi$ are a function of $\nabla \phi_\xi$, which depends itself on the random input $A$ as the solution of the nonlinear corrector equation \eqref{e.cor-eq} (and therefore far from explicit).

\smallskip

Here comes the second ingredient to our approach: sensitivity calculus and concentration of measures (see Appendix~\ref{sec:FC}).
On the one hand, using the Meyers estimate in the large (in its improved weighted form of Theorem~\ref{largescalereg} based on the hole-filling estimate) and sensitivity calculus, we control the stochastic moments of averages of $\nabla \phi_\xi$ by the CLT scaling and moments of $\r$ -- see Proposition~\ref{weakNL}.
On the other hand, by Caccioppoli's inequality for the nonlinear corrector equation, we control super level sets of $\r$ by moments of averages of $\nabla \phi_\xi$. 
The desired control of the moments of $\r$ of Theorem~\ref{boundrNLprop} follows
by combining these two nonlinear estimates and taking advantage of the CLT scaling and the small room given by the hole-filling to buckle, single out, and control $\r$.

\smallskip

Once we have good control of $\r$, the quenched large scale Meyers estimates of Theorem~\ref{unweightmeyers} can finally be upgraded to the annealed Meyers estimates
of Theorem~\ref{th:annealedmeyers}  (as introduced by Duerinckx and Otto in \cite{DO-20}, using Shen's lemma \cite{Shen-07}), cf.~Subsection~\ref{sec:Annealed}.

\smallskip

Having these estimates at hand, Theorem~\ref{th:corrNL} follows from another application of sensitivity calculus, cf.~Section~\ref{sec:NL}.
The proof of Theorem~\ref{th:corrL} is similar, cf.~Section~\ref{sec:corr-diff}. Although we have Theorem~\ref{th:annealedmeyers}, we cannot use the elegant and efficient buckling argument of \cite{Otto-Tlse} for the linearized corrector either (due to unboundedness), and we have to pass again via the super level sets of another minimal radius. 
We then conclude with the routine proof of Theorem~\ref{th:2s} in Section~\ref{sec:2s}.
 
\bigskip

In order to establish these estimates on nonlinear and linearized correctors, we first use an approximation argument which allows us
to discard the long-range correlations induced by the elliptic character of the equation, and actually define the associated approximations of the nonlinear flux corrector and of the linearized correctors by elementary deterministic arguments.
In this contribution, we proceed by periodization in law, which has the advantage
to keep differential relations neat in the approximation (in particular the identity \eqref{e.div-sig}). For all $L>0$, we introduce in Definition~\ref{defi:PL} (see Appendix~\ref{append:per})  a probability measure $\mathbb P_L$ supported on $Q_L=[-\frac L2,\frac L2)^d$-periodic functions. The associated maps $x\mapsto a(x,\xi)$ are therefore $Q_L$-periodic $\mathbb P_L$-almost surely, and the corrector equations are posed on the bounded domain $Q_L$. 
The coupling between $\mathbb P$ and $\mathbb P_L$ given in Lemma~\ref{approxcoef} then allows  us to infer results on $\mathbb P$
from corresponding results on $\mathbb P_L$, see in particular Proposition~\ref{convergenceofperiodiccorrectors}.
The choice of periodization in law is convenient but it is not essential.
In the linear setting one often adds a massive term to the equation (which yields an exponential cut-off
for long-range interactions) \cite{GO1,GO2,Gloria-Otto-10b,GNO-reg,GNO-quant} or disintegrates scales via a semi-group approach \cite{GNO-quant,GO4,Clozeau-20}. 
All our estimates are proved for fixed periodization and the above results follow by letting the 
periodization parameter go to infinity.

\subsection{Towards large-scale regularity and a nonlinear theory of fluctuations}\label{sec:towards}

As discussed in Section~\ref{sec:p=2}, there are three main types of results in (quantitative) stochastic homogenization of the linear elliptic equation~\eqref{e.lin-discu}:
\begin{itemize}
\item Control of oscillations of the solution via a quantitative two-scale expansion (here in form of Theorem~\ref{th:2s} in the nonlinear setting);
\item Large-scale regularity for the operator $-\nabla \cdot A(\tfrac x\e) \nabla$ (both for solutions and for differences of solutions);
\item Control of the fluctuations of observables of the form $\int g \cdot \nabla u_\e$
and $\int g \cdot A(\tfrac \cdot\e) \nabla u_\e$ (using the so-called homogenization commutator).
\end{itemize}
Let us start with large-scale regularity.
The general principle \cite{Avellaneda-Lin-87,Avellaneda-Lin-91,AS,GNO-reg} is that the heterogeneous equation should possess the same regularity properties as the homogenized equation at the scale at which homogenization kicks in (characterized by the size of the corrector).  
In our range $2\le p < \frac{2(d-1)}{d-3}$ of exponents (for which we already control the growth of correctors) and if we restrict ourselves to scalar equations,  the results \cite{Bella_2020} by Bella and Sch\"affner (combined with
the $C^{1,1}_\loc$ regularity of $\xi \mapsto \bar a(\xi)$)  ensures that $-\nabla \cdot \bar a(\nabla)$ does possess nice regularity theory (both in terms $C^{1,\alpha}$ regularity and nonlinear Calder\'on-Zygmund theory), so that there is no obstruction to large-scale regularity in this setting. 
Next to the large-scale regularity for solutions, it is also necessary to have large-scale regularity for differences of solutions, which is more subtle.
Such large-scale regularity for differences would typically allow to upgrade the quantitative two-scale expansion of \eqref{th:2sIneg} stated in $L^2(\R^d$) to any $L^q(\R^d)$ with $1<q<\infty$ (with the same convergence rate in $\e$), and in particular cover the natural  $L^p(\R^d)$-norm.

\medskip

Next to the large-scale regularity for the nonlinear operator $-\nabla \cdot  a(\cdot,\nabla)$, one may wish to establish large-scale regularity for the linearized operator $-\nabla \cdot a_\xi \nabla$. In Section~\ref{mainresultNLunifL} below, large-scale Meyers estimates are proved (in their convenient annealed form of Theorem~\ref{th:annealedmeyers}). Although this perturbative result is enough to prove the quantitative two-scale expansion of Theorem~\ref{th:2s}, Theorem~\ref{th:annealedmeyers} should hold for all exponents $1<q<\infty$ by adapting the arguments of \cite{GNO-reg} to mildly unbounded coefficients ($a_\xi$ indeed has finite stretched exponential moments as a consequence of Lemma~\ref{lem:supnablaphi} and Theorem~\ref{boundrNLprop}) and using the bounds of Theorem~\ref{th:corrL} on the linearized correctors (to control the scale at which homogenization kicks in). 

\medskip

We now turn to fluctuations, the third main topic of stochastic homogenization.
In the linear setting, the theory of fluctuations relies on the quantity $\Xi_\e:=(A(\frac \cdot \e)-\bar A)\nabla u_\e$, called the homogenization commutator in \cite{DGO1} (this object first appeared in a  different form in \cite{AS} in the context of large-scale regularity). 
The homogenization commutator is a natural object to consider since observables of $\Xi_\e$ can be post-processed into observables of the field and flux (the two corner-stones of homogenization), that is, $\int g \cdot \nabla u_\e$ and $\int g \cdot A(\frac \cdot \e) \nabla u_\e$ -- see below for the argument in the nonlinear setting.
The theory splits into two parts.
On the one hand, fluctuations of the homogenization commutator $\Xi_\e$ can be accurately described by the fluctuations of its two-scale expansion based on the standard commutator $\Xi:=(A-\bar A)(\nabla \phi+\id)$, cf.~\cite{DGO1,DGO2,DO-20}. On the other hand, the standard commutator $\Xi$ behaves (in law) on large scales as a Gaussian random field, cf.~\cite{DGO1,DO-20,duerinckx2019scaling}.
The proofs of these results in the linear setting make heavy use of  large-scale regularity.
On top of large-scale regularity in the nonlinear setting, one  needs a suitable notion of nonlinear homogenization commutator. 
A na\"ive guess would be to define the nonlinear commutator as $a_\e(\nabla u_\e)-\bar a(\nabla u_\e)$. This quantity however does not weakly converge to zero since $\bar a$ is nonlinear. We have to devise a quantity that is compatible with weak convergence and encapsulates the diagram of Figure~\ref{fig:diag}.
To this aim we reformulate the constitutive law by linearizing (assuming that $\bar a$ is differentiable): Since $a_\e(0)=\bar a (0)=0$,
we have for all $\xi \in \R^d$
$$
a(\tfrac \cdot \e,\xi) = a(\tfrac \cdot \e,\xi)-a(\tfrac \cdot \e,0) = \Big(\int_0^1 Da(\tfrac \cdot \e,t\xi)dt\Big) \xi, \quad \bar a(\xi) = \bar a(\xi)-\bar a(0) = \Big(\int_0^1 D\bar a(t\xi)dt\Big) \xi,
$$
so that the diagram of Figure~\ref{fig:diag} takes the equivalent form of Figure~\ref{fig:diag2}.
\begin{figure*}[h]
$$
\begin{array}{ccc|c}
  \text{Gradient field} & \text{Constitutive law} & \text{Minus flux} & \text{Conservation law}\\
  \nabla u_\e & \stackrel{\text{multiply by }\int_0^1 Da(\frac \cdot \e,t\nabla u_\e)dt} \longrightarrow & q_\e= (\int_0^1 Da(\tfrac \cdot \e,t\nabla u_\e)dt) \nabla u_\e & -\nabla \cdot q_\e =   \nabla \cdot f
\\
  \downarrow & & \downarrow & \downarrow 
\\
  \nabla \bar u & \stackrel{\text{multiply by }\int_0^1 D\bar a(t\nabla \bar u)dt} \longrightarrow & \bar  q=(\int_0^1 D\bar a(t\nabla \bar u)dt)  \nabla \bar u&-\nabla \cdot \bar q =    \nabla \cdot f
\end{array}
$$
\caption{Reformulation of the commutative diagram}\label{fig:diag2}
\end{figure*}

In this way, homogenization 
can be concisely reduced to the single condition $(\int_0^1 Da(\tfrac \cdot \e,t\nabla u_\e)dt) \nabla u_\e- (\int_0^1 D\bar a(t\nabla \bar u)dt) \nabla u_\e   \rightharpoonup 0$,
and we define the nonlinear homogenization commutator of $u_\e$ as
\begin{equation}\label{e.def:HC}
\Xi_\e(f)\,:=\,\Big(\int_0^1 Da(\tfrac \cdot \e,t\nabla u_\e)dt\Big) \nabla u_\e- \Big(\int_0^1 D\bar a(t\nabla \bar u)dt\Big) \nabla u_\e .
\end{equation}
Let us argue that, as in the linear setting, the homogenization commutator contains both the fluctuations of the field $\nabla u_\e$ and of the flux $a(\frac \cdot \e, \nabla u_\e)$.
We start with $\nabla u_\e$ and consider fluctuations of the observable $\int g \cdot \nabla u_\e$. We introduce the auxiliary map $\bar v$ solution of the linear equation
\begin{equation}\label{e.def-barv}
\nabla \cdot  \Big(\int_0^1 D\bar a(t\nabla \bar u)dt\Big)^* \nabla \bar v = \nabla \cdot g,
\end{equation}
(which we assume to be well-posed in this discussion -- the notation $^*$ is used for the transposition).
Then we have
\begin{eqnarray*}
\int g \cdot \nabla u_\e &\stackrel{\eqref{e.def-barv}}=&- \int  \Big(\int_0^1 D\bar a(t\nabla \bar u)dt\Big)^* \nabla \bar v \cdot \nabla u_\e
\\
&\stackrel{\eqref{e.def:HC}}=& -\int    \nabla \bar v \cdot (\Xi_\e(f)-a(\tfrac \cdot \e,\nabla u_\e))
\\
&\stackrel{\eqref{e.eps-eq}}=& -\int    \nabla \bar v \cdot  \Xi_\e(f) -\int \nabla \bar v \cdot f,
\end{eqnarray*}
so that the fluctuations of $\int g \cdot \nabla u_\e$ are given by those of $-\int    \nabla \bar v \cdot  \Xi_\e(f)$ (since the additional term $ -\int \nabla \bar v \cdot f$ is deterministic).
Likewise, for the flux we introduce the solution $\bar w$ of
\begin{equation}\label{e.def-barw}
\nabla \cdot  \Big(\int_0^1 D\bar a(t\nabla \bar u)dt\Big)^* \nabla \bar w = \nabla \cdot \Big(\int_0^1 D\bar a(t\nabla \bar u)dt\Big)^* g,
\end{equation}
and obtain
\begin{eqnarray*}
{\int g \cdot a(\tfrac \cdot \e,\nabla u_\e) }
&\stackrel{\eqref{e.eps-eq}}=&
\int (g-\nabla \bar w) \cdot a(\tfrac \cdot \e,\nabla u_\e)-\int \nabla \bar w \cdot f
\\
&\stackrel{\eqref{e.def:HC}}=&\int (g-\nabla \bar w) \cdot \Xi_\e(f) + \int (g-\nabla \bar w) \cdot \Big(\int_0^1 D\bar a(t\nabla \bar u)dt\Big) \nabla u_\e -\int \nabla \bar w \cdot f
\\
&\stackrel{\eqref{e.def-barw}}=&\int (g-\nabla \bar w) \cdot\Xi_\e(f)   -\int \nabla \bar w \cdot f,
\end{eqnarray*}
so that the fluctuations of $\int g \cdot a(\tfrac \cdot \e,\nabla u_\e)$ are given by those of $-\int (g-\nabla \bar w) \cdot\Xi_\e(f) $ (since the additional term $ -\int \nabla \bar w \cdot f$ is deterministic).

\medskip

Next to the homogenization commutator of the solution, we introduce the standard homogenization commutator, associated with the corrector. For all $\xi \in \R^d$, we define
\begin{equation}\label{e.def:SHC}
\Xi_\xi\,:=\, \Big(\int_0^1 Da(\cdot,t(\xi+\nabla \phi_\xi))dt\Big) (\xi+\nabla \phi_\xi)- \Big(\int_0^1 D\bar a(t\xi)dt\Big) (\xi+\nabla \phi_\xi).
\end{equation}
In the linear setting, the pathwise structure of fluctuations implies that the homogenization commutator of the solution can be replaced at leading order in the fluctuation scaling by its two-scale expansion using the standard homogenization commutator (for fluctuations, the standard homogenization commutator plays the same role of the corrector for oscillations \eqref{e.2s} -- see \cite{DGO1,DGO2,DO-20} in the linear setting), and the scaling limit of the standard commutator is Gaussian (see \cite{DGO1,DO-20,duerinckx2019scaling} in the linear setting). We expect most of these results to extend to the nonlinear setting. 

\subsection{Extensions and limitations}\label{sec:extension}

Hypothesis~\ref{hypo} makes several assumptions on the monotone operator and the randomness:
\begin{itemize}
\item The underlying probability law is Gaussian with integrable correlations;
\item The monotone map $a(x,\xi)$ is a multiple of $(1+|\xi|^{p-2})\xi$, the randomness is multiplicative (in form of a random matrix field), and the admissible range of $p$ depends on $d$;
\item The spatial dependence $x \mapsto a(x,\xi)$ is smooth on a deterministic level;
\item If it admits a variational form, the operator is associated with a convex energy functional.
\end{itemize}
Several of these assumptions can be slightly relaxed, while others are crucial.
They are discussed in the following paragraphs.

\subsubsection{Probability laws}  

Consider the multiplicative model of the form~\eqref{e.def-a}, that is, $(x,\xi)\mapsto a(x,\xi)\,=\,A(x)(1+|\xi|^{p-2})\xi$.
Our approach is based on a sensitivity calculus which allows us to linearize 
quantities with respect to the randomness (say, wrt $A$) and on functional inequalities
which allow us to control variances using this sensitivity calculus. 
In Hypothesis~\ref{hypo} we consider a Gaussian random field with integrable covariance function, and one might wonder to what extent Gaussianity and the integrability of the covariance function are necessary. 
Our argument strongly relies on the CLT scaling $r^{-\frac d2}$ of spatial averages $\fint_{B_r} \nabla\phi_\xi$ of the corrector gradient, which essentially 
follows from the same property for $a(x,\xi)-\mathbb{E}[a(x,\xi)]$.
On the one hand, sensitivity calculus, functional inequalities, and CLT scaling are not limited to Gaussian fields: they can be developed
as soon as the stationary field $A$ is constructed via a ``hidden'' product structure. In particular, the random checkerboard
and various Poisson-based processes also enjoy such tools, and we refer the reader to \cite{DG1,DG2} for a systematic study 
of sensitivity calculus and (multiscale) functional inequalities for random fields commonly used in the mechanics of composite materials \cite{Torquato-02}.
Such models could be considered here as well.
On the other hand, the CLT scaling indeed requires the integrability of the covariance function. (Since there is some little room in the argument, one could consider a covariance function such that $\int_{\R^d} |c(x)|(1+|x|)^{-\beta}dx<\infty$ provided $0<\beta\ll 1$, but this is detail.)
In order to address Gaussian coefficients with heavier tail, one would first need to establish (nonlinear and linear) large-scale regularity for the random operator (and its linearized version), as in \cite{GNO-reg,GNO-quant,Clozeau-20}.
By \cite{Bella_2020} 	and \eqref{*coer+-}, there is no obstruction to this approach in the range $2 \le p < \frac{2(d-1)}{d-3}$ for scalar equations (but not for systems or for periodic operators in the range of exponents $p \ge \frac{2(d-1)}{d-3}$).

\medskip

Another setting would ensure that spatial averages of  $a(x,\xi)-\mathbb{E}[a(x,\xi)]$ decay at the CLT scaling: if $A$ has finite range of dependence  -- as addressed in \cite{AS,AFK-+,AFK-20} for $p=2$.
The quenched Meyers estimates of Theorem~\ref{unweightmeyers} (and its weighted version of Theorem~\ref{largescalereg}) proved below do hold for general stationary ergodic coefficients -- and therefore in the setting of finite range of dependence. They are however of little use without a good control of the Meyers minimal radius (provided by Theorem~\ref{boundrNLprop} for Gaussian coefficients with integrable covariance).
In  \cite{AS,AFK-+,AFK-20} for $p=2$,  estimates of moments of the corrector gradient (which would control the Meyers minimal radius) are obtained by combining 
a rate of convergence (any would do) for the Dirichlet problem with a Campanato argument
based on $C^{1,\alpha}$-regularity for the homogenized operator.
In particular, one would have to adapt the duality arguments of \cite{AS,AFK-+,AFK-20} to $p>2$
to prove convergence rates. Since the natural object considered in \cite{AS,AFK-+,AFK-20}
is close to the homogenization commutator for $p=2$, the nonlinear commutators we introduced in \eqref{e.def:HC} and \eqref{e.def:SHC} might be good objects to start with for $p>2$. 

\subsubsection{Form of the monotone map}

There are three different assumptions when considering a monotone map of the form~\eqref{e.def-a}: coercivity conditions, regularity with respect to $\xi$, and multiplicative character of the randomness.
To start with, we must assume that $\xi \mapsto a(x,\xi)$ is twice-differentiable (for all $x$) in order to apply sensitivity calculus to the linearized corrector.

\medskip

\textbf{Multiplicative models.}
The form of $a$ is such that one can easily differentiate $a$ with respect to the randomness.
This is not strictly necessary but quite convenient.
Any model having such a property would do, and we can consider 
coefficients of the form $a(x,\xi)=\rho(A(x),\xi)\xi$ provided $M\mapsto \rho(M,\cdot)$ satisfies 
$\vert D_M\rho(M,\xi)\vert\lesssim 1+\vert\xi\vert^{p-1}$ and $\vert D_M \partial_\xi\rho(M,\xi)\vert\lesssim 1+\vert\xi\vert^{p-2}$.
This holds for instance for 
\begin{equation}\label{e.2phase}
a(x,\xi)=\chi(x) a_1(\xi)+(1-\chi(x)) a_2(\xi),
\end{equation}
where $\chi:\R^d \to [0,1]$ is a smooth random field (with a sensitivity calculus and a suitable functional inequality) and $a_1$ and $a_2$ are two given (suitable) monotone maps.
This model is more in line with composite materials.

\medskip

\textbf{Coercivity conditions.}
What is crucial is the non-degeneracy \eqref{*coer+-} (strong monotonicity \eqref{*coer+} is not needed for local regularity for $p<\frac{2(d-1)}{p-3}$). This is forced upon us to rule out the degeneracy of the linearized operator (cf. the critical set of $\nabla \phi_\xi$). In particular, this condition does not hold for the $p$-Laplacian, to which our results do not apply. To our opinion, relaxing this condition constitutes a very challenging problem.

\medskip

\textbf{Restriction on $p$.}
In dimensions $d \ge 4$, we further impose the condition $2\le p < \frac{2(d-1)}{d-3}$.
This condition  does not only allow to prove the large-scale Meyers estimates of the present contribution for the linearized operator, but it also implies regularity theory for the homogenized operator using~\cite{Bella_2020}. Counter-examples to regularity do exist in high dimensions if $p$ is not close to $2$, and it is not clear to us whether there is a similar phenomenon for quantitative estimates in homogenization.

\subsubsection{Local regularity}

It is quite tempting to assert that quantitative homogenization is a matter of large scales (or say, low frequencies), and that local regularity assumptions might be convenient but are not necessary. 
This is indeed quite relevant provided small scales do not interact with large scales. 
A convincing counterexample of that is the quasiperiodic (and almost periodic) setting, where small and large scales indeed interact via a weak Poincar\'e inequality in a high-dimensional torus, cf.~\cite{AGK-16}. In our nonlinear setting, local regularity is not so much needed for the nonlinear correctors, but it seems unavoidable for the linearization part.
This regularity requirement could be weakened in several directions:
\begin{itemize}
\item Only a local $C^\alpha$-control of the spatial dependence of $a$ is needed for some $\alpha>0$, and the control of this local norm can be random itself provided the latter has good moment bounds. In particular, with the same notation as in Hypothesis~\ref{hypo}, this is the case for coefficients of the form
$A(y)=B(G(y))$ provided the (non-negative) Fourier transform $\hat c$ of the covariance function satisfies $\hat c(k) \le (1+|k|)^{-d-2\alpha'}$ (for some $\alpha'>\alpha$). Then $x\mapsto \|A\|_{C^\alpha(B(x))}$ is stationary and has finite Gaussian moments 
(as a slight quantification of \cite[Appendix~A.3]{josien2020annealed} shows).
All our arguments can be adapted to this setting.
\item In the proofs we use local regularity to control pointwise values of the (nonlinear and linear) corrector gradient by its local averages, and therefore control a local supremum by a local $C^\alpha$-norm. Such a control would also follow from a local broken $C^\alpha$-norm, so that one could in principle be able to deal with some $A$ (or $\chi$ in \eqref{e.2phase}) that would be piecewise smooth (and a fortiori piecewise constant
with smooth boundaries, covering the case of smooth inclusions in a background material). This constitutes a question of classical regularity theory. For linear equations and systems, this is proved in \cite{LiVog-00,LiNiren-03} and for monotone operators and $p=2$ in \cite{NeukSch-19}. The case $p>2$ constitutes an interesting independent problem.
\item The state of the art of local regularity is as follows.
For scalar equations, the structure can be quite general, and only requires the H\"older continuity of the map $x\mapsto a(x,\cdot)$ in the sense (see \cite[Theorem~13]{kuusi2014guide})
\begin{equation}\label{darkside}
\sup_{r>0}\int_{0}^r \frac{(\omega(\rho))^{\frac{2}{p}}}{\rho^{\alpha}}\frac{\dd\rho}{\rho}<+\infty,
\end{equation}
where 
$$\omega : r\in (0,+\infty)\mapsto \Big(\sup_{\xi\in\mathbb{R}^d, B_r(x)\subset\mathbb{R}^d}\fint_{B_r(x)}\left(\frac{a(y,\xi)-(a)_{x,r}(\xi)}{(\vert \xi\vert+1)^{p-1}}\right)^2\dd y\Big)^{\frac{1}{2}},$$
for some $\alpha>0$ and $(a)_{x,r}(\xi):=\fint_{B_r(x)} a(y,\xi)\dd y$. For systems however, we are restricted to quasi-diagonal structures of the form $a(x,\xi)=\rho(x,\vert\xi\vert)\xi$, for some $\rho : \mathbb{R}^d\times\mathbb{R}^d\rightarrow \mathbb{R}$ (the so-called Uhlenbeck structure, see \cite{Uhlenbeck}).
\end{itemize} 

\subsubsection{Non-convex energy functionals?}

It would be natural to try to extend these results to the setting of nonlinear elasticity, for which a large part of the qualitative theory has been established (cf.~\cite{Muller-87,Braides-85,Messaoudi-Michaille}, and \cite{DG-16b} for
the most general results in this context). Besides the much more delicate regularity theory (cf.~\cite{DarkSide}), non-convexity essentially prevents us from using the corrector equation efficiently (cf.~the counter-examples to the cell formula in the periodic
setting by M\"uller \cite{Muller-87}, see also \cite{Barchiesi-Gloria-10}), and may cause loss of ellipticity upon linearization 
(see \cite{GMT-93} at the nonlinear level, and \cite{G99,BF-15,FG-16,GR-19} at the linear level) -- except in the vicinity of the identity (cf. \cite{MN-11,GN-11}, and the further use of rigidity \cite{FJM-02}
to establish quantitative results in this regime \cite{NS-18}).  
Hence, quantitative results in homogenization of nonlinear nonconvex models of elasticity remain widely out of reach today.
%


\section{Perturbative regularity theory for the linearized operator}\label{mainresultNLunifL}

In this section we consider periodized random operators $a_L$ distributed according to
the law $\mathbb P_L$ given in Definition~\ref{defi:PL}. In particular, for all $L\ge 1$, $a_L$ is almost surely $Q_L$-periodic in its space variable, and remains random and stationary (this owes to the fact that we use periodization in law rather than naive periodization, cf.~Appendix~\ref{append:per}). 
This implies that $\phi_\xi$ and $\sigma_\xi$ are necessarily $Q_L$-periodic fields almost surely, so that the equations \eqref{e.cor-eq} and \eqref{e.Laplace-sig}
can be posed on $Q_L$ rather than $\R^d$ -- and likewise for the linearized correctors.
For all $L\ge 1$ we use the notation $H^1_\per(Q_L)$ (resp. $W^{1,p}_\per(Q_L)$)
for $Q_L$-periodic fields of $H^1_\loc(\R^d)$ (resp. $W^{1,p}_\loc(\R^d)$) with vanishing average.
Our aim is to prove regularity statements and bounds that are uniform in the periodization parameter $L\ge 1$.

\subsection{The Meyers minimal radius}\label{sec:MR}
In this paragraph we introduce the notion of Meyers minimal radius, a stationary random field which quantifies the scale at which Meyers' estimates hold for the linearized operator.
We start with a definition.
\begin{definition}[Meyers minimal radius]\label{minimalscaleNL}Let $\xi\in\mathbb{R}^d$, $L\geq 1$ and $c>0$. If it exists, the ($Q_L$-periodic) minimal radius  $\rNL(\cdot,c)$ is defined for all $x\in\mathbb{R}^d$ via
\begin{equation}
\rNL(\cdot,c): x\in\mathbb{R}^d\mapsto \inf_{y\in\mathbb{R}^d}\Big( \urNL(y,c)+\ell  \vert x-y \vert\Big),
\label{defr*NL}
\end{equation}
where $\ell = \frac 1{9C \sqrt{d}} \wedge \frac1{16}$ (with $C$ defined in Lemma~\ref{addfint*})
and for all $y\in\mathbb{R}^d$
\begin{equation}
\urNL(y,c)\,:=\,\inf_{r=2^N,N\in\mathbb{N}}\bigg\{ \forall R\geq r, \fint_{B_R(y)}\vert \nabla \corNL\vert^{p} \leq c (1+\vert \xi\vert^p) \bigg\}.
\label{defr*NL2}
\end{equation}
\end{definition}
\begin{remark}
By Jensen's inequality on the left-hand side and the inequality $|\xi|^2\lesssim 1+|\xi|^p$ on the right-hand side, one can replace the condition in \eqref{defr*NL2} by
$$
\fint_{B_R(y)}\vert \nabla \corNL\vert^{2\&p} \le c (1+|\xi|^{2\&p}).
$$
However, we cannot drop the $1$ in the right-hand side. By doing so, we would also need to consider linearized correctors, which we shall only do in a second step.
\end{remark}
We now argue that $\rNL(\cdot,c)$ is a well-defined bounded random field if $c$ is chosen large enough.
\begin{lemma}[Well posedness of $\rNL$]\label{uniformboundr*NL}
Let $(x,\xi)\in\mathbb{R}^d\times\mathbb{R}^d$ and $L\geq 1$. There exist two constants $c_1,c_2>0$ depending on $p$ and $d$ such that, $\mathbb P_L$-almost surely, 
 $\rNL$ satisfies  
\begin{equation}
\urNL(x,c_2)\leq \rNL(x,c_1)\leq \urNL(x,c_1),
\label{encadrementrNL}
\end{equation}
and
\begin{equation}
\urNL(x,c_1)\leq L.
\label{unifboundr*NLeq}
\end{equation}
\end{lemma}
\begin{proof}
Without loss of generality, we may assume that $x=0$. We start with the proof of \eqref{unifboundr*NLeq}, and then turn to the proof of \eqref{encadrementrNL}.
We let $c$ denote a constant depending only on $d$, $\lambda$, and $p$, that may change from line to line.

\medskip

\step1 Proof of \eqref{unifboundr*NLeq}. 

\noindent From the defining equation \eqref{e.cor-eq} for $\phi_\xi$, we have
$$-\nabla\cdot (a(\cdot,\xi+\nabla\phi_{\xi})-a(\cdot,\xi))=\nabla\cdot a(\cdot,\xi)\text{ in $Q_L$},$$
so that by testing the equation with $\phi_\xi$ and using the monotonicity \eqref{*coer+-}
and boundedness \eqref{*cont}, 
we obtain for some constant $c$ depending on $\lambda$ and $d$ 
$$
\fint_{Q_L}\vert\nabla\corNL(x)\vert^2(1+|\xi|^{p-2}+\vert\xi+\nabla\corNL(x)\vert^{p-2})\dd x\leq c \fint_{Q_L}|\xi|(1+|\xi|^{p-2})|\nabla \phi_\xi|.
$$
By absorbing part of the right-hand side into the left-hand side, this yields
\begin{equation*}
\fint_{Q_L}\vert\nabla\corNL(x)\vert^2(1+|\xi|^{p-2}+\vert\xi+\nabla\corNL(x)\vert^{p-2})\dd x\leq c |\xi|^{2\& p}.
\end{equation*}
By the triangle inequality in form of $\vert\xi+\nabla\corNL(x)\vert^{p-2} \gtrsim \vert\nabla\corNL(x)\vert^{p-2}-\vert\xi \vert^{p-2}$, and using the above twice, 
we obtain
$$\fint_{Q_L}\vert\nabla\corNL(x)\vert^{2\& p} \dd x\leq c |\xi|^{2\& p}  .$$
Assume that $L$ is dyadic. 
Given now $R \ge L$, we cover $B_R$ by $N_{L,R}\le c_d  (\frac{R}L)^d$ translations of $Q_L$ (where $c_d$ only depends on dimension), which we denote by $Q_L^j$ for $1\le j \le N_{R,L}$. This yields 
\begin{eqnarray*}
{ \fint_{B_R(y)}\vert \nabla \corNL\vert^{2\&p}}
&\le &   \frac{L^d}{|B_R|} \sum_{j=1}^{N_{R,L}}  \fint_{Q_L^j}\vert \nabla \corNL\vert^{2\&p}
\\
&\le &c_d  \frac{R^d}{L^d}  \frac{L^d}{|B_R|}  c (1 +|\xi|^{p})
\,=\, c_1  (1+|\xi|^{p})
\end{eqnarray*}
for the choice $c_1:= c_d|B|^{-1}$, which only depends on $d$ and $\lambda$.
This yields \eqref{unifboundr*NLeq}.
If $L$ is not dyadic, we cover $B_R$ by cubes of sidelength $2^l$
with $l$ such that $2^l \le L < 2^{l+1}$, and obtain the result at the price of increasing~$c_1$.

\medskip

\step2 Proof of \eqref{encadrementrNL}. 

\noindent By definition~\eqref{defr*NL} of $\rNL$, we have $\rNL(0)\le \urNL(0)$
by testing the infimum problem with $y=0$.
Let us now prove that there exists $c_2$ such that for all $R \ge 1$ we
have the implication  $\rNL(0,c_1) \le R \implies   \urNL(0,c_2) \le R$, from which we deduce~\eqref{encadrementrNL}. By definition \eqref{defr*NL} of $\rNL$, if $\rNL(0,c_1)\le R$, there exists $y \in \R^d$ such that $|y| \le \frac R\ell$ and $\urNL(y,c_1)\le R$. This implies that $B_{R} \subset B_{\bar R}(y)$
with $\bar R:=(\frac 1\ell+1)R$ so that
\begin{equation*}
{\fint_{B_R} \vert\nabla \corNL\vert^{p}}
\,\le \, (\tfrac{\bar R}{R})^{d}\fint_{B_{\bar R}}\vert\nabla \corNL\vert^{p}\,\le \,(\tfrac 1\ell+1)^{d}c_1.
\end{equation*}
Hence, with $c_2:=(\tfrac 1\ell+1)^{d}c_1$, this yields $\urNL(y,c_2)\le R$,
and therefore~\eqref{encadrementrNL}. 
\end{proof}
In the rest of the paper, the notation $\rNL$ refers to the minimal scales $\rNL(\cdot,c_1)$ for which Lemma~\ref{uniformboundr*NL} holds. 
When no confusion occurs, we simply write $\r$ for $\rNL$, and use the short-hand notation $\Br(x)$ for $B_{\rNL(x)}(x)$.

\medskip

We conclude this paragraph by showing that the Meyers minimal radius controls local averages of the nonlinear corrector.
\begin{lemma}[Control of averages of the nonlinear correctors]\label{ctrlavNL}
There exists a nonlinear hole-filling exponent $0<\delta \le d$ depending on $d$, $p$, and $\lambda$ such that for all $(x,\xi)\in\mathbb{R}^d\times\mathbb{R}^d$, we have for all $r>0$
\begin{eqnarray}
\fint_{B_r(x)}\vert\xi+\nabla\corNL\vert^{2\& p}&\lesssim_{d,\lambda,p}&(1+\vert\xi\vert^p)\Big(\frac{\r(x) \vee r}{r}\Big)^{d-\delta}.
\label{controlunitball}
\end{eqnarray}
\end{lemma}
\begin{proof}
Without loss of generality, we may assume that $x=0$.
We use the short-hand notation $\rho:=\r \vee r \ge \r$. By the hole-filling estimate \eqref{HolefillingNL} applied to the defining equation~\eqref{e.cor-eq} for $\phi_\xi$,   there exists $\delta>0$ depending on $d$ and $\lambda$ such that
\begin{equation*}
{\fint_{B_r}\vert\xi+\nabla\corNL\vert^{2\&p}}
\,\lesssim \, \Big(\frac{\rho}{r}\Big)^{d-\delta}\fint_{B_{\rho}}\vert\xi+\nabla\corNL\vert^{2\&p} \,\lesssim\,\Big(\frac{\rho}{r}\Big)^{d-\delta}\fint_{B_{\rho}}|\xi|^{2\& p}+\vert\nabla\corNL\vert^{2\& p}.
\end{equation*}
Using then \eqref{encadrementrNL} in form of $\rho \ge \r(0)\ge \urNL(0,c_2)$, the definition \eqref{defr*NL2}, and Jensen's inequality, this yields the reformulation of~\eqref{controlunitball} 
\begin{equation*}
{\fint_{B_r}\vert\xi+\nabla\corNL\vert^{2\& p} }
\,\lesssim \, \Big(\frac{\rho}{r}\Big)^{d-\delta}(c_2+1)(1+|\xi|^p).
\end{equation*}
\end{proof}

\subsection{Quenched perturbative regularity in the large}\label{perturbativeregsection}

\subsubsection{Quenched Meyers' estimate in the large}

Recall that $a_\xi:=D a(\cdot,\xi+\nabla \phi_\xi)$.
The elliptic operator $-\nabla\cdot\aL\nabla$ has unbounded coefficients,
whose growth depends on the nonlinear corrector $\nabla\corNL$: There exists $(c,C)\in\mathbb{R}_+\times \mathbb{R}_+$, depending on $\lambda$ and $p$, such that for all $h\in\mathbb{R}^d$
\begin{equation}
c\vert h\vert^2\mu_{\xi}\leq h\cdot \aL h\leq C\vert h\vert^2\mu_{\xi},
\label{growthconditionlineaxi}
\end{equation}
where 
\begin{equation}\label{e.def-mu}
\mu_{\xi}:=1+\vert\xi+\nabla\corNL\vert^{p-2}.
\end{equation}
In addition, 
by \eqref{controlunitball} in Lemma~\ref{ctrlavNL} we have for all $r\geq \r$
\begin{equation}
\|\mu_{\xi}\|_{L^{\frac{p}{p-2}}(B_r)}^{\frac{p}{p-2}} \lesssim_{d,\lambda,p}r^{d }(1+\vert\xi\vert^p).
\label{boundnormweight}
\end{equation}
The main result of this section is the following quenched Meyers estimate in the large (which is in the spirit of \cite[Proposition~3.8]{AD-18} for the Laplacian on the percolation cluster).
\begin{theorem}[Quenched Meyers' estimate in the large]\label{unweightmeyers}
Under Hypothesis~\ref{hypo}, for all $K\ge 1$, there exists $\bar m>2$ depending on $d,p,\lambda$, and $K$ such that for all $\xi \in \R^d$ with $|\xi|\le K$, all exponents $2 \le m \le \bar m$, and all $Q_L$-periodic functions $g$ and $u$ related via
\begin{equation}\label{LSMequationu}
-\nabla\cdot \aL\nabla u=\nabla \cdot (g\sqrt{\mu}_{\xi}),
\end{equation}
we have for  all $r>0$
\begin{equation}
\fint_{B_r}\Big(\fint_{B_{\star}(x)}\vert\nabla u \vert^2\mu_{\xi} \Big)^{\frac{m}{2}}\dd x\lesssim_{K}\Big(\fint_{B_{2r}}\Big(\fint_{B_{\star}(x)}\vert\nabla u\vert^2\mu_{\xi}\Big)\,\dd x\Big)^{\frac{m}{2}}
+\fint_{B_{2r}}\Big(\fint_{B_{\star}(x)}\vert g \vert^2 \Big)^{\frac{m}{2}}\dd x.
\label{unweightedmeyerslocal}
\end{equation}
In particular, 
\begin{equation}
\int_{Q_L}\Big(\fint_{B_{\star}(x)}\vert\nabla u\vert^2\mu_{\xi}\Big)^{\frac{m}{2}}\dd x\lesssim_{K} \int_{Q_L}\Big(\fint_{B_{\star}(x)}\vert g\vert^2\Big)^{\frac{m}{2}}\dd x.
\label{unweightedmeyers}
\end{equation}
The same result holds with $\aL$ replaced by its pointwise transpose field $\aL^*$.
\end{theorem}
We follow the standard strategy based on a reverse H\"older inequality and Gehring's lemma to prove this Meyers estimate. We start with the reverse H\"older inequality:
\begin{lemma}[Reverse H\"older inequality]\label{improvedholder}
Let Hypothesis~\ref{hypo} hold. Set $q=\frac{p}{p-2}$. For all $K\ge 1$, all $\xi\in \R^d$ with $|\xi|\le K$, all $x\in\mathbb{R}^d$, $r \geq \r(x)$, and all $g$ and $u$  related via
\begin{equation}
-\nabla\cdot \aL\nabla u=\nabla \cdot (g\sqrt{\mu_\xi}) \text{ in $B_{\frac{17}{12}r}(x)$},
\label{reverseholderequationu}
\end{equation}
we have
\begin{equation}
\Big(\fint_{B_{\frac{67}{48} r}(x)}\vert \nabla u\vert^2\mu_{\xi}\Big)^{\frac{1}{2}}\lesssim (1+ K^p)^{\frac{p-2}{2p}}\Big(\fint_{B_{\frac{17}{12}r}(x)}\vert\nabla u \vert^{q_*}\Big)^{\frac{1}{q_*}}+\Big(\fint_{B_{\frac{17}{12}r}(x)}\vert g\vert^2 \Big)^{\frac{1}{2}},
\label{reverseholderesti}
\end{equation}
with $1\le q_*<2$ given by
\begin{equation}
\frac{1}{q_*} = \left\{
    \begin{array}{ll}
		1 & \text{ for $d=2$},\\
        \frac{1}{2}-\frac{1}{2q}+\frac{1}{d-1} & \text{ for $d\geq 3$}.
    \end{array}
\right.
\label{exponentlemmabella}
\end{equation}
(The choice of $\frac{67}{48}$
and $\frac{17}{12}$ is convenient for the sequel, but obviously not essential.)
The same result holds with $\aL$ replaced by its pointwise transpose field $\aL^*$.
\end{lemma}
Not surprisingly, this estimate follows from the Caccioppoli and the Poincar\'e-Sobolev inequalities. As opposed to the case of uniformly bounded coefficients, the weight $\mu_\xi$ complexifies the matter (and cannot be treated as a Muckenhoupt weight, which it is not). In order to get the entire range of
exponents $2\le p <\infty$ in dimension $d=3$, we have to be careful in the Caccioppoli inequality. Inspired by  \cite[Lemma 1]{bella2019local}, we optimize with respect to the cut-off in Caccioppoli's inequality, which allows us to appeal to Poincar\'e-Sobolev in dimension $d-1$ rather than $d$ (and therefore improve the integrability).
\begin{lemma}\label{lemmabella}
Let $q\in [1,+\infty)$, assume that $q>\frac{d-1}{2}$ if $d\geq 3$, and let $q_*$ be given by \eqref{exponentlemmabella}. For $0<\rho<\sigma<+\infty$, $v\in W^{1,q_*}(B_{\sigma})$ and $\mu\in L^{q}_{\text{loc}}(\mathbb{R}^d)$, the quantity
\begin{equation}
\mathcal{J}(\rho,\sigma,\mu,v):=\inf\Big\{\int_{B_{\sigma}}\mu  v^2\vert\nabla\eta\vert^2 \Big\vert\eta\in C^{1}_c(B_{\sigma}),\, 0\leq \eta\leq 1,\, \eta\equiv 1 \text{ in } B_{\rho}\Big\}
\label{defJinfbella}
\end{equation}
satisfies 
\begin{equation}
\mathcal{J}(\rho,\sigma,\mu,v)\lesssim (\sigma-\rho)^{-\frac{2d}{d-1}}\|\mu\|_{L^q(B_{\sigma}\backslash B_{\rho})}\left(\|\nabla v\|^2_{L^{q_*}(B_{\sigma}\backslash B_{\rho})}+\rho^{-2}\|v\|^2_{L^{q_*}(B_{\sigma}\backslash B_{\rho})}\right).
\label{estilemmabella}
\end{equation}
\end{lemma}
The proof of Lemma~\ref{lemmabella}, which closely follows the proof of~\cite[Lemma 1]{bella2019local}, is postponed to Appendix~\ref{append:standard-ineq}.
We now prove Lemma~\ref{improvedholder}.
\begin{proof}[Proof of Lemma~\ref{improvedholder}]
Without loss of generality, we may assume  $x=0$ and $\int_{B_{\frac{17}{12}r}\backslash B_{\frac{67}{48} r}} u =0$. We first apply the Caccioppoli inequality \eqref{esticaccioppounbounded} with $\mu=\mu_{\xi}$ and $c_1=\frac{67}{48}<\frac{17}{12}=c_2$, and obtain with the notation \eqref{defJinfbella}
\begin{equation}
\int_{B_{\frac{67}{48} r}}\vert\nabla u \vert^2\mu_{\xi}\, \lesssim \, \mathcal{J}(\tfrac{67}{48} r,\tfrac{17}{12}r,\mu_{\xi},u)+\int_{B_{\frac{17}{12}r}}\vert g \vert^2 .
\label{lemmabellaapp}
\end{equation}
We then apply Lemma \ref{lemmabella} with exponent $q=\frac{p}{p-2}$ for $d\geq 3$ and $q=1$ for $d=2$, to the effect that
\begin{equation}
\mathcal{J}(\tfrac{67}{48} r,\tfrac{17}{12}r,\mu_{\xi},u)\lesssim r^{-\frac{2d}{d-1}}\|\mu_{\xi}\|_{L^q(B_{\frac{17}{12}r}\backslash{B_{\frac{67}{48} r}})}\left(\|\nabla u\|^2_{L^{q_*}(B_{\frac{17}{12}r}\backslash{B_{\frac{67}{48} r}})}+r^{-2}\|u\|^2_{L^{q_*}(B_{\frac{17}{12}r}\backslash{B_{\frac{67}{48} r}})}\right).
\label{lemmabellaapp2}
\end{equation}
Since $r\geq \r(0)$,  \eqref{boundnormweight} yields
$\|\mu_{\xi}\|_{L^q(B_{\frac{17}{12}r}\backslash{B_{\frac{67}{48} r}})}\lesssim (1+\vert \xi\vert^{p})^{\frac{p-2}{p}}r^{\frac{d}{q}}$, whereas
 Poincar\'e's inequality in $L^{q_*}(B_{2r}\backslash B_r)$ yields
$r^{-2}\|u\|^2_{L^{q_*}(B_{\frac{17}{12}r}\backslash{B_{\frac{67}{48} r}})}\lesssim \|\nabla u\|^2_{L^{q_*}(B_{\frac{17}{12}r}\backslash{B_{\frac{67}{48} r}})}$. Hence, \eqref{lemmabellaapp2} turns into
$$
\mathcal{J}(\tfrac{67}{48} r,\tfrac{17}{12}r,\mu_{\xi},u)\lesssim r^{-\frac{2d}{d-1}+\frac{d}{q}}(1+\vert\xi\vert^p)^{\frac{p-2}{p}}\|\nabla u\|^2_{L^{q_*}(B_{\frac{17}{12}r}\backslash{B_{\frac{67}{48} r}})}.
$$
Combined with  \eqref{lemmabellaapp}, this entails
$$
\fint_{B_{\frac{67}{48} r}}\vert\nabla u \vert^2\mu_{\xi}\,\lesssim \,r^{-\frac{2d}{d-1}+\frac{d}{q}-d+\frac{2d}{q_*}}(1+ \vert\xi\vert^p)^{\frac{p-2}{p}}\Big(\fint_{B_{\frac{17}{12}r}}\vert\nabla u\vert^{q_*} \Big)^{\frac{2}{q_*}}+\fint_{B_{\frac{17}{12}r}}\vert g \vert^2,
$$
which concludes the proof since, by definition \eqref{exponentlemmabella} of $q_*$, $-\frac{2d}{d-1}+\frac{d}{q}-d+\frac{2d}{q_*}=0$.
\end{proof}
Theorem~\ref{unweightmeyers} relies on the combination of Lemma~\ref{improvedholder} with Gehring's inequality in form of (see for instance \cite[Theorem 6.38]{giaquinta2013introduction})
\begin{lemma}[Gehring's lemma]\label{gehring}
Let $s>1$, and  let $f$ and $h$ be two non-negative measurable functions in $L^q_\loc(\mathbb{R}^d)$ such that there exists $C>0$ for which for all $r>0$ and $x\in\mathbb{R}^d$
$$
\Big(\fint_{B_r(x)} f^s\Big)^{\frac{1}{s}}\leq C\Big(\fint_{B_{2r}(x)}f+\Big(\fint_{B_{2r}(x)}h^s\Big)^{\frac{1}{s}}\Big).
$$
Then, there exists $\bar{s}> s$ depending on $q$ and $C$ such that for all $r>0$ and $x\in\mathbb{R}^d$, we have 
$$
\Big(\fint_{B_r(x)}f^{\bar{s}}\Big)^{\frac{1}{\bar{s}}}\lesssim \fint_{B_{2r}(x)}f+\Big(\fint_{B_{2r}(x)}h^{\bar{s}}\Big)^{\frac{1}{\bar{s}}}.
$$
\end{lemma}
We are now in the position to prove Theorem \ref{unweightmeyers}. 
\begin{proof}[Proof of Theorem \ref{unweightmeyers}]
Let $1\le q_* <2$ be given by \eqref{exponentlemmabella}.
We first prove that for all $r>0$
\begin{equation}
\fint_{B_r}\Big(\fint_{\Br(x)}\vert\nabla u \vert^2\mu_{\xi}\Big)\dd x\lesssim \Big(\fint_{B_{2r}}\Big(\fint_{\Br(x)}\vert\nabla u\vert^2\mu_{\xi}\Big)^\frac{q_*}2 \dd x\Big)^{\frac{2}{q_*}}
+\fint_{B_{2r}}\Big(\fint_{\Br(x)}\vert g \vert^2\Big)\dd x .
\label{reverseholder}
\end{equation}
If $r\le 3 \r(0)$ this estimate follows from Lemma~\ref{reverse}, and it remains to treat the case $r\geq 3\r(0)$. 
We first use \eqref{fintout} with $f=\vert\nabla u\vert^2\mu_{\xi}$ to the effect of
\begin{equation}
\int_{B_r}\Big(\fint_{\Br(x)}\vert\nabla u\vert^2\mu_{\xi} \Big) \dd x\lesssim \fint_{B_{\frac{67}{48}r}}\vert\nabla u\vert^2\mu_{\xi}.
\label{propunweightmeyers1}
\end{equation}
Then, by the reverse H\"older inequality \eqref{reverseholderesti} followed by \eqref{fintint}, we obtain 
\begin{eqnarray}
\fint_{B_{\frac{67}{48}r}}\vert\nabla u \vert^2\mu_{\xi} &\stackrel{\eqref{reverseholderesti}}{\lesssim}& (1+ \vert\xi\vert^p)^{\frac{p}{p-2}}\Big(\fint_{B_{\frac{17}{12} r}}\vert\nabla u \vert^{q_*} \Big)^{\frac{2}{q_*}}+\fint_{B_{\frac{17}{12}r}}\vert g \vert^2 
\nonumber \\
&\stackrel{\eqref{fintint}}{\lesssim}&(1 +\vert\xi\vert^p)^{\frac{p}{p-2}}\Big(\fint_{B_{2r}}\Big(\fint_{\Br(x)}\vert\nabla u\vert^{q_*}\Big)\, \dd x\Big)^{\frac{2}{q_*}} +\fint_{B_{2r}}\Big(\fint_{\Br(x)}\vert g\vert^2\Big)\, \dd x\label{propunweightmeyers2}.
\end{eqnarray}
We then slightly reformulate the first right-hand side term using Jensen's inequality in the inner integral (since $q_*<2$) and the lower bound $\mu_{\xi}\geq 1$, so that
\begin{align}
\Big(\fint_{B_{2r}}\fint_{\Br(x)}\vert \nabla u\vert^{q_*}\, dx\Big)^{\frac{2}{q_*}}&\leq \Big(\fint_{B_{2r}}\Big(\fint_{\Br(x)}\vert \nabla u\vert^2 \Big)^{\frac{q_*}2} \dd x\Big)^{\frac{2}{q_*}}\nonumber\\
&\leq \Big(\fint_{B_{2r}}\Big(\fint_{\Br(x)}\vert \nabla u\vert^2 \mu_{\xi}\Big)^{\frac{q_*}2}  \dd x\Big)^{\frac{2}{q_*}}.
\label{propunweightmeyers3}
\end{align}
The combination of \eqref{propunweightmeyers1}, \eqref{propunweightmeyers2}, \eqref{propunweightmeyers3} yields the claimed estimate \eqref{reverseholder}.
To conclude, we apply Lemma~\ref{gehring} with 
$$f:x\mapsto \Big(\fint_{\Br(x)}\vert\nabla u\vert^{2}\mu_{\xi}\Big)^{\frac {q_*}{2}},
\quad h:x\mapsto \Big(\fint_{\Br(x)}\vert g \vert^2 \Big)^{\frac{q_*}2},
\quad s=\tfrac{2}{q_*}>1.
$$
This yields \eqref{unweightedmeyerslocal}, whereas
 \eqref{unweightedmeyers} follows by applying \eqref{unweightedmeyerslocal} for $B_r$
with $r=\frac{\sqrt{d}}2L$, and using Lemma~\ref{addfint*}, the bound \eqref{unifboundr*NLeq}, the periodicity of the quantities involved together with the plain energy estimate 
$
\int_{Q_L}\vert\nabla u \vert^2\mu_{\xi} \lesssim \int_{Q_L}\vert g \vert^2.
$
\end{proof}

\subsubsection{Quenched weighted Meyers' estimate in the large}

The main result of this paragraph is the following upgrade of Theorem~\ref{unweightmeyers}.
\begin{theorem}[Quenched weighted Meyers estimates in the large]\label{largescalereg}
Under Hypothesis~\ref{hypo}, for all $K\ge 1$, there exists $\beta >0$ depending only on $K$ and $d$ such that for all $\xi\in\mathbb{R}^d$ with $|\xi|\le K$ and $r>0$,  all $2\le m \le \bar m$  (cf.~Theorem~\ref{unweightmeyers}), $0\le 2\varepsilon \le \beta$ 
and all $Q_L$-periodic fields  $g$ and $u$ related via \eqref{LSMequationu},
we have
\begin{equation}
\int_{Q_L}\omega_{\e,r}(x)\Big(\fint_{\Br(x)}\vert\nabla u \vert^2\mu_{\xi} \Big)^{\frac{m}{2}}dx\lesssim_{K}\int_{Q_L}\omega_{2\varepsilon,r}(x)\Big(\fint_{\Br(x)}\vert g\vert^2 \Big)^{\frac{m}{2}}dx,
\label{weightedLp}
\end{equation}
where for all $x\in Q_L$
\begin{equation}
\omega_{\varepsilon,r}(x):=\Big(1+\frac{|x|+\r(0)}{r}\Big)^{\varepsilon}.
\label{weightmeyers}
\end{equation}
The same result holds with $\aL$ replaced by its pointwise transpose field $\aL^*$.
\end{theorem}
We proceed in two steps: From Theorem~\ref{unweightmeyers} we first prove a suitable linear hole-filling estimate which we use in turn to upgrade Theorem~\ref{unweightmeyers} into Theorem~\ref{largescalereg}.
\begin{corollary}[Linear hole-filling estimate in the large]\label{Lholefilling}
Under Hypothesis~\ref{hypo}, for all $K\ge 1$ there exist an exponent $\beta >0$, depending only on $d$, $p$ and $K$, 
and a constant $c_d\ge 1$ with the following properties.
Let $\xi \in \R^d$ with $|\xi|\le K$, and let $u$ be a $Q_L$-periodic function  which is $\aL$-harmonic in $Q_R(x)$ for some 
$x\in \R^d$ and $L\ge R \ge c_d \r(x)$, that is
$$-\nabla\cdot \aL\nabla u= 0 \text{ in $Q_R(x)$}.$$
Then for all $\r(x)\le r \le R$,
\begin{equation}
\int_{Q_r(x)}\vert \nabla u \vert^2\mu_{\xi} \lesssim_{K} (\tfrac{r}{R})^{\beta}\int_{Q_R(x)}\vert\nabla u \vert^2\mu_{\xi} .
\label{Lholefillingesti}
\end{equation}
The same result holds with $\aL$ replaced by its pointwise transpose field $\aL^*$.
\end{corollary}
\begin{proof}[Proof of Corollary \ref{Lholefilling} ] Without loss of generality, we may assume that $x=0$, $r \ge \r(0)$, and that $2c r \le \frac R8$ with $c= 3 \vee \frac{\sqrt d}{2}$.
By \eqref{fintout}, \eqref{fintint}, the H\"older inequality with exponents $(\bar{m} ,\frac{\bar m}{\bar m -1})$ (with $\bar m$ as in Theorem \ref{unweightmeyers}) and the unweighted Meyers estimate~\eqref{unweightedmeyerslocal}, we have with  $\beta :=d(1-\frac{1}{\bar m})$
\begin{eqnarray*}
\int_{Q_r}\vert \nabla u\vert^2\mu_{\xi} & \le & \int_{B_{cr}}\vert \nabla u\vert^2\mu_{\xi}
\\
&\stackrel{\eqref{fintint}}{\lesssim}&r^d\fint_{B_{2cr}}\Big(\fint_{\Br(x)}\vert\nabla u\vert^2\mu_{\xi} \Big)\,\dd x\\
&\leq &r^d\Big(\fint_{B_{2cr}}\Big(\fint_{\Br(x)}\vert\nabla u\vert^2\mu_{\xi}  \Big)^{\bar m}\dd x\Big)^{\frac{1}{\bar m}}\\
&\leq& r^d(\tfrac{R}{r})^{\frac{d}{\bar m}} \Big(\fint_{B_{\frac R8}}\Big(\fint_{\Br(x)}\vert\nabla u\vert^2\mu_{\xi}  \Big)^{\bar m}\dd x\Big)^{\frac{1}{\bar m}}\\
&\stackrel{\eqref{unweightedmeyerslocal}}{\lesssim}&r^d(\tfrac{R}{r})^{\frac{d}{\bar m}}\fint_{B_{\frac R4}}\Big(\fint_{\Br(x)}\vert\nabla u\vert^2\mu_{\xi}\Big)\,\dd x
\,\stackrel{\eqref{fintout}}{\lesssim} \,(\tfrac{r}{R})^{\beta }\int_{B_{\frac R2}}\vert\nabla u \vert^2\mu_{\xi} \le (\tfrac{r}{R})^{\beta }\int_{Q_R}\vert\nabla u \vert^2\mu_{\xi}.
\end{eqnarray*}
\end{proof}
We now prove Theorem \ref{largescalereg}. 
\begin{proof}[Proof of Theorem \ref{largescalereg}]
We split the proof into four steps. In the first step, we show that for  right-hand sides $g$ compactly supported in $Q_L$ the solution gradient decays algebraically fast away from the source in $L^2$, based on hole-filling. We then upgrade this $L^2$ estimate into an $L^m$ estimate for some $m>2$ using Meyers' estimate~\eqref{unweightedmeyerslocal}. In the third step, we remove the assumption that $g$ be compactly supported by using a dyadic decomposition of scales.
In the last step we exploit the algebraic decay  to add the desired weight.
Since the proof relies on a dyadic decomposition of the torus, it is convenient to work with cubes rather than balls when taking averages (which makes constants slightly cumbersome).
 
\medskip

\step1 $L^2$ algebraic decay rate.

\noindent We prove that for all $L\ge R \ge r\geq c_d\r(0)$ and all $g$ compactly supported in $Q_r$ we have
\begin{equation}
\int_{Q_L \backslash Q_R}\vert\nabla u \vert^2\mu_K \lesssim (\tfrac{r}{R})^{\beta}\int_{Q_r}\vert g \vert^2 ,
\label{holefil}
\end{equation}
where $\beta>0$ is defined in Corollary~\ref{Lholefilling}.\\
We proceed by duality, and write
\begin{equation}\label{e.dua1}
\int_{Q_L \backslash Q_R}\vert\nabla u\vert^2\mu_\xi \,=\,
\Big(\sup_{h }
\int_{Q_L \backslash Q_R}h \cdot \nabla u \sqrt{\mu_\xi }\Big)^2 ,
\end{equation}
where the supremum runs over functions $h \in  L^2(Q_L \backslash Q_R)^d$ with $\|h\|_{ L^2(Q_L \backslash Q_R)^d}=1$.
Consider such a test function $h$ (implicitly extended by zero on $Q_R$)
and denote by $v$  the unique weak solution in $H^1_\per(Q_L)$ of
\begin{equation}
-\nabla \cdot \aL^*\nabla v=\nabla\cdot (h\sqrt{\mu_{\xi}}),
\label{LSMequationdualv}
\end{equation}
which is well-posed since $\mu_{\xi}$ is bounded on $Q_L$ almost surely by Lemma~\ref{regestiNL}.
By testing \eqref{LSMequationdualv} with $u$ and \eqref{LSMequationu} with $v$, we obtain
by Cauchy-Schwarz' inequality and the support condition on $g$
\begin{equation}
\Big|\int_{Q_L}h\cdot \nabla u\sqrt{\mu_{\xi} }\Big|=\Big|\int_{Q_{L}}g \cdot \nabla v \sqrt{\mu_{\xi} } \Big|\leq \Big(\int_{Q_r}\vert g \vert^2 \Big)^{\frac{1}{2}}\Big(\int_{Q_r}\vert\nabla v \vert^2\mu_{\xi} \Big)^{\frac{1}{2}}.
\label{weightedmeyerstep11}
\end{equation}
Since $h$ vanishes on $Q_R$, $v$ is $\aL^*$-harmonic in $Q_R$, and 
the hole filling estimate \eqref{Lholefillingesti} with exponent $\beta$ yields in combination with the plain energy estimate $\int_{Q_{L}}\vert \nabla v \vert^2\mu_{\xi} \lesssim \int_{Q_{L}}\vert h \vert^2 $ and the assumption $\int_{Q_L}|h|^2=1$
\begin{equation}
\int_{Q_r}\vert\nabla v\vert^2\mu_{\xi} \lesssim (\tfrac{r}{R})^{\beta}\int_{Q_R}\vert\nabla v\vert^2\mu_{\xi} \lesssim (\tfrac{r}{R})^{\beta}\int_{Q_{L}}\vert h \vert^2 =(\tfrac{r}{R})^{\beta}.
\label{weightedmeyerstep12}
\end{equation}
The claim \eqref{holefil}  now follows from \eqref{e.dua1}, \eqref{weightedmeyerstep11} and 
\eqref{weightedmeyerstep12}.

\medskip

\step2 $L^m$ algebraic decay rate for $2 \le m \le \bar m$. 

\noindent In this step, we prove that, with $C_d=4C \vee c_d \vee 16$ (and $C\ge 1$ as 
in \eqref{fintoutC}), for all  $L> R \ge 2 r \geq 2C_d\r(0)$, and all $g$ compactly supported in $Q_r$, we may upgrade \eqref{holefil}  to 
\begin{equation}
\int_{Q_{L}\backslash Q_R}\Big(\fint_{\Br(x)}\vert\nabla u\vert^2 \mu_{\xi}\Big)^{\frac{m}{2}}\dd x\lesssim R^{d(1-\frac{m}{2})}(\tfrac{r}{R})^{\beta \frac{m}{2}}\Big(\int_{Q_r}\vert g\vert^2\Big)^{\frac{m}{2}}
\label{holefillingLp}
\end{equation}
for all $2\le m \le \bar m$,  where $\bar{m}$ is the Meyers exponent of Theorem~\ref{unweightmeyers}.

Let $J \in \N$ be such that $2^JR < L \le 2^{J+1}R$.
By writing $Q_L \setminus Q_R= (Q_L \setminus Q_{2^JR})\cup \cup_{j=1}^J (Q_{2^{j}R}\setminus Q_{2^{j-1}R})$ (with the convention that the second union is empty if $J=0$), it is enough to prove that for all  $1\le j\le J+1$, we have  
\begin{equation}
\int_{Q_{(2^{j}R) \wedge L}\backslash Q_{2^{j-1}R}}\Big(\fint_{\Br(x)}\vert \nabla u \vert^2 \mu_\xi \Big)^{\frac{m}{2}}\dd x\lesssim (2^jR)^{d(1-\frac{m}{2})}(\tfrac{r}{2^jR})^{\beta \frac{m}{2}}\Big(\int_{Q_r}\vert g\vert^2 \Big)^{\frac{m}{2}}.
\label{dyadicLp}
\end{equation}
Indeed, for all $m\ge 2$, $d(1-\frac{m}{2})-\beta  \frac m 2 \le -\beta$, so that the dyadic terms sum to \eqref{holefillingLp}.

We now prove~\eqref{dyadicLp}. To start with, reverting from balls to cubes, one may reformulate Theorem~\ref{unweightmeyers} with cubes instead of balls, and replace $B_r$ and $B_{2r}$
by $Q_r$ and $Q_{C_1 r}$, respectively (for some $C_1$ depending only on $d$).
Let $1\le j\le J$ be fixed (the case $j=J+1$ can be treated similarly). 
We partition $Q_{2^{j}R}\setminus Q_{2^{j-1}R}$ into the union of cubes $\{Q^k\}_{k=1,\dots,N}$ of side-length
$\frac1{C_2} 2^jR$ for some $C_2$ to be fixed later (the number $N$ of such cubes then depends on $d$ and $C_2$, but not on $j$ or $R$), to the effect that 
for all $m>2$ we have
\begin{equation}
\int_{Q_{2^{j}R}\backslash Q_{2^{j-1}R}}\Big(\fint_{\Br(x)}\vert \nabla u \vert^2 \mu_\xi\Big)^{\frac{m}{2}}\dd x= \sum_{k=1}^{N}\int_{Q^k}\Big(\fint_{\Br(x)}\vert \nabla u \vert^2 \mu_\xi  \Big)^{\frac{m}{2}}\dd x.
\label{decomposeint1}
\end{equation}
By Theorem \ref{unweightmeyers}, for all $2\le m \le \bar m$ and $1\le k \le N$,
\begin{equation}
\fint_{Q^k}\Big(\fint_{\Br(x)}\vert \nabla u \vert^2 \mu_\xi \Big)^{\frac{m}{2}}\dd x\lesssim \, \Big(\fint_{\bar Q^k}\Big(\fint_{\Br(x)}\vert \nabla u \vert^2 \mu_\xi\Big) \, \dd x\Big)^{\frac{m}{2}}+\fint_{\bar Q^k}\Big(\fint_{\Br(x)}\vert g \vert^2 \Big)^{\frac{m}{2}}\dd x,
\label{meyersesti1}
\end{equation}
where $\bar Q^k\supset Q^k$ denotes the cube of side-length $\frac{C_1}{C_2} 2^jR$ centered at the center $x_k \in Q_{2^jR}\setminus Q_{2^{j-1}R}$ of $Q^k$.
We now control the two right-hand side terms of \eqref{meyersesti1}.
On the one hand, by the $\ell$-Lipschitz property of $\r$ and the assumption $R \ge 2C_d\r(0)$, for all $x \in \bar Q^k$, we have 
$|x|\le |x_k|+\frac{\sqrt{d}}{2}\frac{C_1}{C_2} 2^j R \le \frac{\sqrt{d}}{2}(1+\frac{C_1}{C_2})2^j R$, and therefore
\begin{equation} \label{e.ugly}
\r(x)\le \r(0)+ \ell |x| \,\le \, R(\tfrac1{2C_d} 
+\ell \tfrac{\sqrt d}{2}2^{j} (1+\tfrac{C_1}{C_2})).
\end{equation}
Recall the constant $C$ in \eqref{fintoutC} and that $C_1$ only depends on dimension. We now choose $C_2:=8C_1$.
For our choice $C_d = 4C\vee c_d \vee 16$  and $R \ge 2C_d \r(0)$, and since $0<\ell = \frac 1{9C \sqrt{d}} \wedge \frac1{16}$, we have
$\frac{C_1}{C_2}2^j R= 2^{j-3}R$ and 
$$
C\Big(\frac1{2C_d}+\frac{\ell\sqrt d}{2}2^j (1+\frac{C_1}{C_2})\Big) \le C(\frac1{8C}+\frac1{18C}2^j(1+\frac18)) \le 2^{j-3},
$$
which, by \eqref{e.ugly}, entails $\frac{C_1}{C_2}2^j R \ge C\r(x)$,
condition under which \eqref{fintoutC}  yields
\begin{equation*}
\int_{\bar Q^k}\Big(\fint_{\Br(x)}\vert \nabla u \vert^2 \mu_\xi\Big)\, \dd x\lesssim \int_{\tilde Q^k}\vert\nabla u \vert^2 \mu_\xi ,
\end{equation*}
where $\tilde Q^k$ denotes the cube of side-length $\frac{C_1}{C_2} 2^{j+1}R=2^{j-2}R$ centered at $x_k$, so
that $\tilde Q^k \mod L\Z^d \subset Q_{2^{j+1}R \wedge L}\backslash Q_{2^{j-2}R}$. Hence, by \eqref{holefil}, 
\begin{equation}\label{e.lp1}
\int_{\bar Q^k}\Big(\fint_{\Br(x)}\vert \nabla u \vert^2 \mu_\xi\Big)\, \dd x\lesssim \int_{Q_{2^{j+1}R \wedge L}\backslash Q_{2^{j-2}R}}\vert\nabla u \vert^2 \mu_\xi  \, \lesssim\,  (\tfrac{r}{2^jR})^{\beta}\int_{Q_r}\vert g \vert^2 .
\end{equation}
On the other hand, the same argument implies
\begin{equation}\label{e.lp2}
\int_{\bar Q^k}\Big(\fint_{\Br(x)}\vert g \vert^2  \Big)\, \dd x\lesssim \int_{Q_{2^{j+1}R \wedge L}\backslash Q_{2^{j-2}R}}\vert g \vert^2   = 0,
\end{equation}
where we used that $g$ is supported in $Q_r$ and $r\le \frac R2$.
The claim \eqref{dyadicLp} then follows from \eqref{decomposeint1}, \eqref{meyersesti1}, \eqref{e.lp1}, and \eqref{e.lp2}, and the identity $|Q^k|=(2^{j-3}R)^d$. 

\medskip

\step3 Extension to general $g$.

\noindent In this step, we relax the support assumption on $g$ in \eqref{holefillingLp}, and claim that for all $L\ge R \ge 2C_d \r(0)$ and all $2\le m \le \bar m$,
\begin{multline}
\Big(\int_{Q_{L}\backslash Q_R}\Big(\fint_{\Br(x)}\vert\nabla u \vert^2\mu_{\xi} \Big)^{\frac{m}{2}}\dd x\Big)^{\frac{1}{m}}\,
\lesssim\, \Big(\int_{Q_{L}\backslash Q_{\frac R4}}\Big(\fint_{\Br(x)}\vert g\vert^2\Big)^{\frac{m}{2}}\dd x\Big)^{\frac{1}{m}}
\\
+\Big(\int_{Q_{R}}(\tfrac{|x|+\r(0)}R)^{\frac{\beta m}{4}}\Big(\fint_{\Br(x)}\vert g \vert^2 \Big)^{\frac{m}{2}}\dd x\Big)^{\frac{1}{m}}.
\label{holefilling2}
\end{multline}
Let $N \in \N$ be such that $2^{N}C_d\r(0) \le R < 2^{N+1}C_d\r(0)$ (note that $N\ge 1$ since 
$R \ge 2 C_d \r(0)$).
We decompose $g$ as $g=\sum_{i=0}^{N} g_i$ with 
$g_0:=g \mathds{1}_{Q_{C_d \r(0)}}$, 
$g_i:= g \mathds{1}_{Q_{2^i C_d \r(0)} \setminus Q_{2^{i-1} C_d \r(0)} }$
for all $1\le i \le N-1$, and $g_{N}:=g \mathds{1}_{Q_{L}\setminus Q_{2^{N-1}C_d \r(0)}}$.
By linearity (and uniqueness of the solution) of the equation, we have $u = \sum_{i=0}^{N} u_i$ where $u_i$ 
denotes the (unique) weak solution in $H^1_\per(Q_L)$ of  
$$
-\nabla\cdot \aL\nabla u_{i}=\nabla\cdot (g_{i}\sqrt{\mu_{\xi}}).
$$
By the triangle inequality, we then have for $2 \le m \le \bar m$,
\begin{equation}
\Big(\int_{\
Q_{L}\backslash Q_R}\Big(\fint_{\Br(x)}\vert\nabla u \vert^2\mu_{\xi} \Big)^{\frac{m}{2}}\dd x\Big)^{\frac{1}{m}}
\,
\le\,\sum_{i=0}^{N}\Big(\int_{Q_{L}\backslash Q_R}\Big(\fint_{\Br(x)}\vert\nabla u_{i}(y)\vert^2\mu_{\xi} \Big)^{\frac{m}{2}}\dd x\Big)^{\frac{1}{m}}.\label{decompdyadic}
\end{equation}
We start by estimating the term for $i=N$, for which we use the Meyers estimate~\eqref{unweightedmeyers} to the effect that
$$
\Big(\int_{Q_L} \Big(\fint_{\Br(x)} |\nabla u_{N} |^2 \mu_\xi \Big)^\frac m2 \dd x \Big)^\frac 1m \,\stackrel{\eqref{unweightedmeyers}}\lesssim \, \Big(\int_{Q_L} \Big(\fint_{\Br(x)} |g_{N}|^2  \Big)^\frac m2 \dd x\Big)^\frac1m.
$$
We then reformulate the right-hand side using the support condition on $g_{N}$. 
For $x \in Q_{2^{N-2}C_d \r(0)}$, since $\ell = \frac 1{9C \sqrt{d}} \wedge \frac1{16}$,
$C\ge 1$, $N\ge 0$, and $C_d \ge 16$, we have
$$
\r(x)\le \r(0)+\ell |x| \le \r(0)(1+\ell \tfrac{\sqrt{d}}{2}2^{N-2}C_d) \le 
2^{N-2} C_d\r(0)  (\tfrac14 + \tfrac19 ) \le 2^{N-3} C_d \r(0),
$$
so that we have the implication 
$$
y \in \Br(x) \,\implies \, y \in Q_{2^{N-2} C_d \r(0)}(x) \,\implies \, y\in Q_{2^{N-1}C_d \r(0)}(0)\,\implies \, g_N(y)=0.
$$
Since $R < 2^{N+1}C_d\r(0)=4\times 2^{N-1}C_d\r(0)$, $Q_{\frac R4} \subset Q_{2^{N-1}C_d \r(0)}$,  and the above implies
\begin{equation}\label{holefilling2-2}
\Big(\int_{Q_L} \Big(\fint_{\Br(x)} |\nabla u_{N} |^2 \mu_\xi \Big)^\frac m2 \dd x \Big)^\frac 1m \, \lesssim \, \Big(\int_{Q_L\setminus Q_{\frac R4}} \Big(\fint_{\Br(x)} |g|^2\Big)^\frac m2 \dd x\Big)^\frac1m.
\end{equation}
We then turn to the contributions for $0\le i\le N-1$, for which we appeal to \eqref{holefillingLp} with 
$r=2^i C_d \r(0)\ge C_d \r(0)$ and $R\ge 2^NC_d\r(0)\ge 2r$, and obtain
\begin{eqnarray*}
\int_{Q_{L}\backslash Q_R}\Big(\fint_{\Br(x)}\vert\nabla u_i\vert^2 \mu_{\xi}\Big)^{\frac{m}{2}}\dd x&\stackrel{\eqref{holefillingLp}}\lesssim &R^{d(1-\frac{m}{2})}(\tfrac{r}{R})^{\beta \frac{m}{2}}\Big(\int_{Q_r\setminus Q_{r/2}}\vert g\vert^2 \Big)^{\frac{m}{2}}
\\
&\lesssim &R^{d(1-\frac{m}{2})}(\tfrac{r}{R})^{\beta \frac{m}{4}}
\Big(\int_{Q_r\setminus Q_{r/2}}(\tfrac{|y|+\r(0)}{R})^\frac \beta 2\vert g(y)\vert^2\dd y\Big)^{\frac{m}{2}}.
\end{eqnarray*}
We then appeal to \eqref{fintintC} (which holds for $r$ since $r=2^i C_d \r(0) \ge 2^i C\r(0)$ by definition of $C_d$), to Jensen's inequality, and to the Lipschitz regularity of $\r$ in form of $\r(x)\le \r(0)+|x|$, and get 
\begin{eqnarray}
\lefteqn{\Big(\int_{Q_{L}\backslash Q_R}\Big(\fint_{\Br(x)}\vert\nabla u_i\vert^2 \mu_{\xi}\Big)^{\frac{m}{2}}\dd x\Big)^\frac1m}
\nonumber \\
&\stackrel{\eqref{fintintC} }\lesssim& R^{d \frac{2-m}{2m}} (\tfrac rR)^\frac \beta 4 \Big(\int_{Q_{2r}}\fint_{\Br(x)}(\tfrac{|y|+\r(0)}{R})^\frac \beta 2\vert g(y)\vert^2\dd y \dd x\Big)^{\frac{1}{2}}
\nonumber \\
& \le& R^{d \frac{2-m}{2m}} (\tfrac rR)^\frac \beta 4 
(2r)^{d(\frac12-\frac1m)}\Big(\int_{Q_{2r}}\Big(\fint_{\Br(x)}(\tfrac{|y|+\r(0)}{R})^\frac \beta 2\vert g(y)\vert^2\dd y\Big)^\frac m2 \dd x\Big)^{\frac{1}{m}}
\nonumber \\
&\lesssim&   (\tfrac rR)^{\frac \beta 4 +d\frac{m-2}{2m}} \Big(\int_{Q_{R}}(\tfrac{|x|+\r(0)}{R})^\frac {\beta m}4\Big(\fint_{\Br(x)}\vert g \vert^2 \Big)^\frac m2 \dd x\Big)^{\frac{1}{m}}
\nonumber\\
&\le & (2^{\frac \beta 4 +d\frac{m-2}{2m}})^{i-N} \Big(\int_{Q_{R}}(\tfrac{|x|+\r(0)}{R})^\frac {\beta m}4\Big(\fint_{\Br(x)}\vert g \vert^2 \Big)^\frac m2 \dd x\Big)^{\frac{1}{m}}.
\label{holefilling2-1}
\end{eqnarray}
The claimed estimate \eqref{holefilling2} then follows from \eqref{decompdyadic}, \eqref{holefilling2-2}, and \eqref{holefilling2-1}.

\medskip

\step4 Proof of  \eqref{weightedLp}. 

\noindent 
If $L \le 2C_d\r(0) \le 2C_d L$, then the weight is essentially constant: for all $x\in Q_L$, $\omega_{r,\e} (x)\simeq (1+\frac{L}{r})^\e$, and the conclusion~\eqref{weightedLp} is obviously satisfied. 
In the rest of this step we thus assume that $L>2C_d\r(0)$.
Let $2C_d\r(0)< r\le L$ (the case $0<r\le 2C_d\r(0)$ reduces to the case $r=2C_d\r(0)$ by homogeneity).  
Let $N \in \N$ be such that $2^N C_d\r(0) \le L < 2^{N+1} C_d \r(0)$ and let $N_0 \le N$ be such that $2^{N_0} C_d\r(0) \le r < 2^{N_0+1}C_d\r(0)$.
We then have
\begin{multline}
\int_{Q_L}\omega_{\frac{\varepsilon}{2},r}(x)\Big(\fint_{\Br(x)}\vert\nabla u \vert^2\mu_{\xi} \Big)^{\frac{m}{2}}\dd x
\,= \,\int_{Q_{2^{N_0}C_d\r(0)}}\omega_{\frac{\varepsilon}{2},r}(x)\Big(\fint_{\Br(x)}\vert\nabla u \vert^2\mu_{\xi} \Big)^{\frac{m}{2}}\dd x
\\
+\sum_{i=N_0}^{N-1} \int_{Q_{2^{i+1}C_d\r(0)}\setminus Q_{2^{i}C_d\r(0)}}\omega_{\frac{\varepsilon}{2},r}(x)\Big(\fint_{\Br(x)}\vert\nabla u \vert^2\mu_{\xi} \Big)^{\frac{m}{2}}\dd x
\\
+ \int_{Q_{L}\setminus Q_{2^{N}C_d\r(0)}}\omega_{\frac{\varepsilon}{2},r}(x)\Big(\fint_{\Br(x)}\vert\nabla u \vert^2\mu_{\xi} \Big)^{\frac{m}{2}}\dd x.
\label{e.weight-ant0}
\end{multline}
We then control each right-hand side term separately.
For the first right-hand side term of \eqref{e.weight-ant0}, we have by definition of~$N_0$
$$
\sup_{Q_{2^{N_0}C_d \r(0)}} \omega_{\frac \e 2, r} \lesssim \omega_{\frac \e 2,r}(0) \le \omega_{\frac\e 2,r}(x) \quad \forall \, x \in Q_L,
$$
so that by Theorem~\ref{unweightmeyers} 
\begin{eqnarray}
\int_{Q_{2^{N_0}C_d\r(0)}}\omega_{\frac{\varepsilon}{2},r}(x)\Big(\fint_{\Br(x)}\vert\nabla u \vert^2\mu_{\xi} \Big)^{\frac{m}{2}}\dd x
&\lesssim& \omega_{\frac \e 2,r}(0)
\int_{Q_L} \Big(\fint_{\Br(x)}\vert g \vert^2 \Big)^{\frac{m}{2}}\dd x
\nonumber
\\
&\lesssim& \int_{Q_L} \omega_{\frac \e 2,r}(x)\Big(\fint_{\Br(x)}\vert g \vert^2 \Big)^{\frac{m}{2}}\dd x.
\label{e.weight-ant1}
\end{eqnarray}
We now estimate the right-hand side sum of \eqref{e.weight-ant0}.
For all $N_0\le i\le N-1$, we combine
the bound $\omega_{\frac \e2, r}|_{Q_{2^{i+1}C_d \r(0)}\setminus Q_{2^{i}C_d \r(0)}} \, \simeq \,2^{\frac \e2(i-N_0)}$
with \eqref{holefilling2} to the effect that (using that $2\e \le \beta$)
\begin{eqnarray}
\lefteqn{\Big( \int_{Q_{2^{i+1}C_d\r(0)}\setminus Q_{2^{i}C_d\r(0)}}\omega_{\frac{\varepsilon}{2},r}(x)\Big(\fint_{\Br(x)}\vert\nabla u \vert^2\mu_{\xi} \Big)^{\frac{m}{2}}\dd x\Big)^\frac1m}
\nonumber\\
&\lesssim &2^{\frac{\e}{2m}(i-N_0)}
\Big(\int_{ Q_{L} \setminus Q_{2^{i}C_d\r(0)}} \Big(\fint_{\Br(x)}\vert\nabla u \vert^2\mu_{\xi} \Big)^{\frac{m}{2}}\dd x\Big)^\frac1m
\nonumber\\
&\stackrel{ \eqref{holefilling2},2\e \le \beta}\lesssim&2^{\frac{\e}{2m}(i-N_0)}\Big(\int_{Q_{L}\setminus Q_{2^{i-2}C_d\r(0)}} \Big(\fint_{\Br(x)}\vert g \vert^2  \Big)^{\frac{m}{2}}\dd x\Big)^\frac1m
\nonumber
\\
&&+2^{\frac{\e}{2m}(i-N_0)}
\Big(\int_{Q_{2^{i}C_d\r(0)}}(\tfrac{|x|+\r(0)}{2^{i}C_d\r(0)})^{\frac{\e m}{2}}\Big(\fint_{\Br(x)}\vert g \vert^2 \Big)^{\frac{m}{2}}\dd x\Big)^{\frac{1}{m}}.
\label{e.weight-ant2}
\end{eqnarray}
For the first right-hand side term of \eqref{e.weight-ant2}, we use that 
for all $x \in Q_L\setminus Q_{2^{i-2}C_d \r(0)}$ we have 
$2^{\frac{\e}{2}(i-N_0)} \lesssim 2^{-\frac{\e}{2}(i-N_0)}  \omega_{\e,r}(x)$, so that 
\begin{equation}\label{e.weight-ant3}
2^{\frac{\e}{2m}(i-N_0)}\Big(\int_{Q_{L}\setminus Q_{2^{i-2}C_d\r(0)}} \Big(\fint_{\Br(x)}\vert g\vert^2  \Big)^{\frac{m}{2}}\dd x\Big)^\frac1m\,\lesssim\, 2^{-\frac{\e}{2m}(i-N_0)}\Big(\int_{Q_{L}} \omega_{\e,r}\Big(\fint_{\Br(x)}\vert g\vert^2  \Big)^{\frac{m}{2}}\dd x\Big)^\frac1m.
\end{equation}
For the second right-hand side term of \eqref{e.weight-ant2}, we rather use that 
for all $x \in  Q_{2^{i}C_d \r(0)}$ we have 
by definition of $N_0$  and since $m\ge 2$
\begin{eqnarray*}
2^{\frac{\e}{2}(i-N_0)}(\tfrac{|x|+\r(0)}{2^{i}C_d\r(0)})^{\frac{\e m}{2}} &\lesssim &
2^{\frac{\e}{2}(i-N_0)}(\tfrac{|x|+\r(0)}{2^{i}C_d\r(0)})^{\e}
\\
&\lesssim & 2^{\frac{\e}{2}(i-N_0)} 2^{-\e (i-N_0)}  (\tfrac{|x|+\r(0)}{r})^{\e}
\lesssim\, 2^{-\frac{\e}{2}(i-N_0)} \omega_{\e,r}(x),
\end{eqnarray*}
so that 
\begin{equation}\label{e.weight-ant4}
2^{\frac{\e}{2m}(i-N_0)}
\Big(\int_{Q_{2^{i}C_d\r(0)}}(\tfrac{|x|+\r(0)}{2^{i}C_d\r(0)})^{\frac{\e m}{2}}\Big(\fint_{\Br(x)}\vert g\vert^2\Big)^{\frac{m}{2}}\dd x\Big)^{\frac{1}{m}}
\,
\lesssim\, 2^{-\frac{\e}{2m}(i-N_0)}\Big(\int_{Q_{L}} \omega_{\e,r}\Big(\fint_{\Br(x)}\vert g\vert^2 \Big)^{\frac{m}{2}}\dd x\Big)^\frac1m.
\end{equation}
Summing \eqref{e.weight-ant2}--\eqref{e.weight-ant4} over 
$i$ form $N_0$ to $N-1$ we then obtain
\begin{equation}\label{e.weight-ant5}
\sum_{i=N_0}^{N-1} \int_{Q_{2^{i+1}C_d\r(0)}\setminus Q_{2^{i}C_d\r(0)}}\omega_{\frac{\varepsilon}{2},r}(x)\Big(\fint_{\Br(x)}\vert\nabla u\vert^2\mu_{\xi}\Big)^{\frac{m}{2}}\dd x
\,
\lesssim \,  \int_{Q_{L}} \omega_{\e,r}\Big(\fint_{\Br(x)}\vert g \vert^2  \Big)^{\frac{m}{2}}\dd x.
\end{equation}
Controlling the last right-hand side term of \eqref{e.weight-ant0} the same way, \eqref{weightedLp} follows from \eqref{e.weight-ant1} and  \eqref{e.weight-ant5}.
\end{proof}

\subsection{Control of the Meyers minimal radius: sensitivity estimate and buckling}\label{sec:MR-control}

The main result of this section is the following control of the Meyers minimal radius.
\begin{theorem}\label{boundrNLprop}
Under Hypothesis~\ref{hypo},  there exists an exponent $\gamma>0$ depending on $d$, $\lambda$, and $p$, and  for all $K\ge 1$ there exists a constant $c_{K}>0$ depending additionally on $K$ (and all independent of $L\ge 1$) such that  for all $\xi \in \R^d$,
\begin{equation}
\expecL{ \exp( c_K \rNL^{\gamma}) } \,\le\,2.
\label{boundrNL}
\end{equation}
\end{theorem}
The proof of Theorem~\ref{boundrNLprop} relies on the combination of  the following sensitivity estimate (based on the quenched weighted Meyers estimate of Theorem~\ref{largescalereg}) with the Caccioppoli inequality via a buckling argument.
\begin{proposition}\label{weakNL}
Under Hypothesis~\ref{hypo}, for all $K\ge 1$, denote by $\bar m>2$, $\delta>0$, and $\beta>0$ the Meyers and nonlinear and linear hole-filling exponents, respectively (cf.~Theorem~\ref{unweightmeyers}, Lemma~\ref{ctrlavNL}, and Corollary~\ref{Lholefilling}). Then, for all $\xi \in \R^d$ with $|\xi|\le K$, and all $r\ge 1$
and $0<\tau <1$, the random variable 
$\calF_\xi:=\fint_{B_r}\nabla\corNL$
satisfies
\begin{equation}
\expecL{\vert\calF_\xi\vert^{2q}}^{\frac{1}{q}}\lesssim_{K}qr^{-d}\expecL{ \rNL^{\frac{d-\delta}{1-\tau}q}}^{\frac{1-\frac \tau2}{q}}
\label{sensimomentboundNL}
\end{equation}
for all $q\geq 1+\frac{d+1}{\varepsilon}$, where 
\begin{equation}\label{e-def-eps}
\e\,:=\, (\tfrac \beta 2) \wedge (\tfrac{(d+1)(\bar m-2)}{2}) \wedge (\tfrac{\tau(d-\delta)}{4(1-\tau)}).
\end{equation}
\end{proposition}
We start with the proof of Theorem~\ref{boundrNLprop}, and then turn to the proof of  
Proposition~\ref{weakNL}.
\begin{proof}[Proof of Theorem~\ref{boundrNLprop}]
We use the short-hand notation $\rb:= \urNL(0,c_1)$ (cf.~\eqref{defr*NL2} and
Lemma~\ref{uniformboundr*NL}).
We split the proof into two steps. In the first step, we control the probability of the level set $\{\rb=R\}$ for all dyadic $R\in [1,L]$ using averages of $\nabla \phi_\xi$, which we combine with Proposition~\ref{weakNL} and the bound $\r \le \rb$ to buckle on moments of $\rb$, 
and therefore on $\r$ in the second step.

\medskip

\step1 We claim that there exist $\theta \in (0,1)$ and $c>0$, depending on $p$, $d$ and $\lambda$ such that for all dyadic $R\in [1,L]$ and all exponents $q\ge 1$, 
\begin{equation}
\mathbb P_L[\rb =R]\, \le \, c^q \expecL{ \Big\vert\fint_{B_{\theta R}}\nabla\corNL\Big\vert^{pq}}.
\label{estiprobar*}
\end{equation}
Assume that $\rb =R$. 
By the definition~\eqref{defr*NL2} of $\rb$, we then have
\begin{eqnarray}
c_2 (1+|\xi|^p)& \ge& \fint_{B_{2R}} \vert \nabla \corNL\vert^{p},
\label{boundnabla2}
\\ 
\fint_{B_{\frac R2}} \vert \nabla \corNL\vert^{p} &>& c_2 (1+\vert \xi\vert^p).
\label{boundnabla1.0}
\end{eqnarray}
By the Caccioppoli inequality \eqref{cacciopoNLesti},~\eqref{boundnabla1.0} turns into
\begin{equation}
\inf_{\eta\in\R} \fint_{B_R} \tfrac{1}{R^2} \vert\corNL(x)-\eta\vert^2
+\tfrac1{R^p}\vert\corNL(x)-\eta\vert^{p} \,\gtrsim\, 1 +\vert \xi\vert^p,
\label{boundnabla1.1}
\end{equation}
which we shall use in the stronger form
\begin{equation}
\inf_{\eta\in\R} \fint_{B_R}\tfrac1{R^p}\vert\corNL(x)-\eta\vert^{p} \,\gtrsim\, 1 +\vert \xi\vert^p.
\label{boundnabla1}
\end{equation}
Indeed, by Jensen's inequality, with the short-hand notation $\alpha:=\inf_{\eta\in\R} \fint_{B_R} \tfrac{1}{R^p} \vert\corNL(x)-\eta\vert^p$, \eqref{boundnabla1.1} yields 
$\alpha^\frac2p+\alpha  \gtrsim 1+\vert \xi\vert^p$ so that $\alpha \gtrsim 1$, 
which implies $\alpha \gtrsim \alpha^\frac 2p$, whence the reformulation~\eqref{boundnabla1}.

\medskip 

Let $\theta\in (0,1)$ (the value of which we shall choose below), and set $c_R:=\fint_{B_R}\fint_{B_{\theta R}(x)}\corNL(y)\dd y\, \dd x$.
By the triangle inequality and Poincar\'e's inequality in $L^p(B_{R})$,
we obtain 
\begin{eqnarray}
 \inf_{\eta\in\R} \fint_{B_R} \tfrac{1}{R^p} \vert\corNL-\eta\vert^p
&\lesssim &\fint_{B_{R}}\tfrac{1}{R^p}\Big\vert\corNL(x)-\fint_{B_{\theta R}(x)}\corNL\Big\vert^p\dd x+\fint_{B_{R}}\tfrac{1}{R^p}\Big\vert\fint_{B_{\theta R}(x)}\corNL-c_R\Big\vert^p\dd x\nonumber\\
&\lesssim &\theta^p \fint_{B_{2R}}\vert\nabla\corNL(x)\vert^p\dd x+\fint_{B_{R}}\Big\vert\fint_{B_{\theta R}(x)}\nabla\corNL \Big\vert^p\dd x.\label{boundbanbla31}
\end{eqnarray} 
Combined with \eqref{boundbanbla31}, \eqref{boundnabla1} turns into
\begin{equation}
1+\vert \xi\vert^p \, \lesssim \, \theta^p \fint_{B_{2R}} \vert\nabla\corNL\vert^p+\fint_{B_{R}}   \Big\vert\fint_{B_{\theta R}(x)}\nabla\corNL\Big\vert^p\dd x.\label{boundbanbla3}
\end{equation} 
Using now \eqref{boundnabla2}, we may absorb the first right-hand side term into the left-hand side for $\theta$ small enough (independent of $R$), 
and therefore conclude that for some $c>0$ (depending only on $d$, $p$, $\lambda$)
$$\{\rb =R\} \subset \left\{\fint_{B_{R}}\Big\vert\fint_{B_{\theta R}(x)}\nabla\corNL\Big\vert^p \dd x \ge \frac 1c(1+\vert\xi\vert^p)\right\},$$
which yields \eqref{estiprobar*} by Markov's inequality and the stationarity of $\nabla\corNL$.

\medskip

\step2 Buckling argument. 

\noindent 
Fix $\tau := 1-\frac{d-\delta}{d-\frac{\delta}{2}}=\frac\delta{2d-\delta}>0$, to the effect that 
$$
\frac{d-\delta}{1-\tau}=d-\frac{\delta}{2}, \quad 1-\frac \tau 2= 1-\frac\delta{2(2d-\delta)},
\quad \e :=   (\tfrac{(d+1)(\bar m-2)}{2}) \wedge (\tfrac{\delta}8 ).
$$
For all dyadic $1\le R \le L$, by~\eqref{estiprobar*} and by Proposition~\ref{weakNL} with this choice of $\tau$ and $r=\theta R$, we obtain  for
all $q$ with $q \frac p2 \ge 1+\frac{d+1}\e$,
\begin{equation}
\mathbb P_L[\rb=R ]\,\le\, c_{K}^q   q^{\frac{p}{2}q}R^{-d\frac{p}{2}q}\expecL{\r^{(d-\frac{\delta}{2})\frac p2q}}^{ 1-\frac\delta{2(2d-\delta)}}
\,\stackrel{\eqref{encadrementrNL}}{\le} \,
c_{K}^q   q^{\frac{p}{2}q}R^{-d\frac{p}{2}q}\expecL{\rb^{(d-\frac{\delta}{2})\frac p2q}}^{ 1-\frac\delta{2(2d-\delta)}}.
\label{boundproba2}
\end{equation}
Therefore, using a dyadic decomposition (the sum is actually finite since $\r \le L$), we deduce that (up to changing the value of $c_{K}$)
\begin{eqnarray*}
\expecL{\rb^{(d-\frac{\delta}{2})\frac{p}{2}q}}&\leq &1+\sum_{n=1}^{+\infty} (2^{n})^{(d-\frac{\delta}{2})\frac{p}{2}q}\mathbb P[ \rb=2^n ]\\
&\stackrel{\eqref{boundproba2}}{\le}& 1 + c_{K}^qq^{q\frac{p}{2}}\expecL{\rb^{(d-\frac{\delta}{2})\frac p2q}}^{ 1-\frac\delta{2(2d-\delta)}}\sum_{n=1}^{+\infty}\underbrace{2^{(d-\frac{\delta}{2})q\frac{p}{2}n}2^{-dq\frac{p}{2}n}}_{\displaystyle = 2^{-\frac \delta 4 qpn}}\\
&\le& 1+ c_{K}^qq^{q\frac{p}{2}}\expecL{\rb^{(d-\frac{\delta}{2})\frac p2q}}^{ 1-\frac\delta{2(2d-\delta)}}.
\end{eqnarray*}
Since both terms of this inequality are finite, this gives by Young's inequality
provided $q \frac p2 \ge 1+\frac{d+1}\e$%
$$
\expecL{\rb^{(d-\frac{\delta}{2})\frac{p}{2}q}}^\frac1q \,\lesssim\, c_{K}q^{p \frac{2d-\delta}{\delta}},
$$
from which the stretched exponential moment bound~\eqref{boundrNL} follows
with $\gamma:=\frac\delta 8$ (cf.~Lemma~\ref{momentexp}), which is not expected to be sharp.
\end{proof}
We conclude this section with the proof of Proposition~\ref{weakNL}.
\begin{proof}[Proof of Proposition~\ref{weakNL}]
We split the proof into three steps. In the first step, we compute the functional derivative of $\mathcal{F}$, in the sense of \eqref{deffunctioderivper}, and apply the logarithmic-Sobolev inequality in the second step to control moments of $\calF$. In the third step, we then control these moments by suitable moments of $\r$ using the quenched weighted Meyers estimate in the large of Theorem~\ref{largescalereg}.

\medskip

\step1 Sensitivity calculus. 

\noindent 
In this step, we take a slightly more general version of $\calF_\xi$ (this will be further
used in the proof of Theorem~\ref{th:corrNL}), which we define, for some given $g \in L^2(Q_L)^d$ (extended by periodicity on $\R^d$), by
$$
\calF_\xi:=\int_{Q_L} \nabla \phi_\xi \cdot g.
$$
We then argue that for all $x\in Q_L$,
\begin{equation}
\partial_x\mathcal{F}_\xi = \int_{B(x)}\vert \aa(\xi+\nabla\phi_{\xi})\otimes \nabla u\vert,
\label{functionalderivesensiNLformula}
\end{equation}
with the short-hand notation $\aa(\zeta):=(1+|\zeta|^{p-2})\zeta$ and 
where $u$ is the unique weak $Q_L$-periodic solution (with zero average) of 
\begin{equation}
-\nabla\cdot \aL^*\nabla u=\nabla\cdot g
\label{equationdual}
\end{equation}
(recall that $\aL^*$ is bounded from above and below, since $a$ is assumed to be smooth, and satisfies \eqref{growthconditionlineaxi}, cf.~Lemma~\ref{regestiNL}).
Let denote by $h$ a sequence that goes to zero and by $\delta A$ a coefficient field supported in $B(x)$ (and extended by $Q_L$-periodicity) such that $\|\delta A\|_{L^{\infty}(\mathbb{R}^d)}\leq 1$. We let $h$ be small enough so that $A+h\delta A$ is uniformly elliptic, and define
\begin{eqnarray}
\delta^{h}\calF_\xi&:=&\frac{\calF_\xi(A+h\delta A)-\calF_\xi(A)}{h},\nonumber\\
\delta^{h}\phi_{\xi}&:=&\frac{\phi_{\xi}(A+h\delta A )-\phi_{\xi}(A)}{h},\label{functionderivNLnota1}\\
\aL^{h}&:=&\int_{0}^1 D a(\cdot,\xi+t\nabla\phi_{\xi}(A+h\delta A)+(1-t)\nabla\phi_{\xi}(A))\dd t.\label{functionderivNLnota2}
\end{eqnarray}
By the definition  of $\calF_\xi$, we have $\delta^{h}\calF_\xi=\int_{Q_L}\nabla\delta^{h}\corNL\cdot g$, and we need to characterize $\delta^{h}\corNL$. By the defining equation \eqref{e.cor-eq}, we obtain
\begin{equation}\label{e.NLgr-1}
-\nabla \cdot \Big(a(\cdot, \xi+\nabla \phi_\xi +h \nabla \delta^h \phi)-a(\cdot,\xi+\nabla \phi_\xi)\Big)
\,=\,h \nabla \cdot \delta A \aa(\xi+\nabla \phi_\xi(A+h\delta A)),
\end{equation}
which we rewrite, by the fundamental theorem of calculus and the definition of $\aL^{h}$, as
\begin{equation}
-\nabla\cdot \aL^h\nabla\delta^h\phi_{\xi}=\nabla\cdot \delta A \aa(\xi+\nabla\phi_{\xi}(A+h\delta A)).
\label{sensiNLequation1}
\end{equation}
Assume that $\delta^h\phi_{\xi}$ converges weakly in $H^1_\per(Q_L)$
to the solution $\delta\phi_{\xi} \in H^1_\per(Q_L)$ of
\begin{equation}
-\nabla\cdot \aL \nabla\delta\phi_{\xi}=\nabla\cdot \delta A \aa(\xi+\nabla\phi_{\xi}).
\label{sensiNLequation1+}
\end{equation}
Then $\lim_{h\downarrow 0} \delta^h \calF_\xi = \delta \calF_\xi=\int g \cdot  \nabla\delta\phi_{\xi}$,
which we now rewrite by duality.
Testing  \eqref{equationdual} with $\delta \phi_{\xi}$ and then \eqref{sensiNLequation1+} with $u$,
we obtain
$$
\delta \calF_\xi = \int \nabla u \cdot \delta A \aa(\xi+\nabla \phi_\xi),
$$
and the claim \eqref{functionalderivesensiNLformula} follows by taking the supremum over $\delta A$.
It remains to argue in favor of the convergence of $\delta^h\phi_{\xi}$ to $\delta\phi_{\xi}$, which actually holds in $C^{1,\alpha}(Q_L)$.
First, recall that $\{\phi_\xi(A+h\delta A)\}_{h}$ is a bounded set in $C^{1,\alpha}(Q_L)$ by Lemma~\ref{regestiNL}. 
By testing \eqref{e.NLgr-1} with $h\delta^h \phi_\xi$, we obtain by monotonicity 
$$
\int_{Q_L} |\nabla \phi_\xi(A+h\delta A)-\nabla \phi_\xi(A)|^2 + |\nabla \phi_\xi(A+h\delta A)-\nabla \phi_\xi(A)|^p \lesssim h^2,
$$
so that $\nabla \phi_\xi(A+h\delta A) \to \nabla \phi_\xi(A)$ in $L^p(Q_L)$, and therefore $\phi_\xi(A+h\delta A) \to  \phi_\xi(A)$ in $C^{1, \alpha}(Q_L)$ by Arzela-Ascoli's theorem as claimed.

\medskip

\step2  Application of the logarithmic-Sobolev inequality: For all $q\ge 1$, 
\begin{equation}\label{e.sensi-estimNL-ant}
\expecL{\vert\calF_\xi\vert^{2q}}^{\frac{1}{q}}\,\lesssim\,  q(1+K^p)\expecL{\Big(\int_{Q_L}\r(x)^{d-\delta}\Big(\int_{B(x)}\vert\nabla u\vert^2\mu_{\xi}\Big) \dd x\Big)^q}^{\frac{1}{q}}.
\end{equation}
Since $\mathbb E_L[{\nabla \phi_\xi}]=0$, by \eqref{SGinegp1} and \eqref{functionalderivesensiNLformula}, we have for all $q\ge 1$
\begin{equation}
\expecL{\calF_\xi^{2q}}^{\frac{1}{q}}\, \lesssim \,q  \expecL{\Big(\int_{Q_L}\Big(\int_{B(x)}\vert \aa(\xi+\nabla\corNL)|| \nabla u\vert \Big)^2\dd x \Big)^q}^{\frac{1}{q}}.
\label{SGsensitibityNL}
\end{equation}
By Cauchy-Schwarz' inequality, the definition~\eqref{e.def-mu}
 of $\mu_\xi$, and \eqref{controlunitball}, we have for all $x\in Q_L$
\begin{eqnarray*}
\lefteqn{\Big(\int_{B(x)}\vert \aa(\xi+\nabla\corNL)|| \nabla u\vert\Big)^2}\\
&\lesssim &\int_{B(x)}\vert\xi+\nabla\corNL\vert^2(1+\vert\xi+\nabla\corNL\vert^{p-2}) \int_{B(x)}\vert\nabla u\vert^2\mu_{\xi} \\
&\stackrel{\eqref{controlunitball}}{\lesssim}& (1+\vert\xi\vert^p)\r(x)^{d-\delta}\int_{B(x)}\vert\nabla u\vert^2\mu_{\xi}.
\end{eqnarray*}
The claim \eqref{e.sensi-estimNL-ant} then follows in combination with~\eqref{SGsensitibityNL}.

\medskip

\step3 Proof of \eqref{sensimomentboundNL}. 

\noindent 
For $0<\tau <1$ given, we define $\e$ as in \eqref{e-def-eps}, 
and set $m:=2+\frac{2\varepsilon}{d+1}$,
to the effect that $m\le \bar m$ and $\frac{2\varepsilon}{m-2}=d+1$.
Since $\r$ is $\frac{1}{16}$-Lipschitz and $\r \ge 1$, we have 
$$
\int_{Q_L}\r(x)^{d-\delta}\Big(\fint_{B(x)}\vert\nabla u\vert^2\mu_{\xi}\Big) \dd x \,\lesssim\, \int_{Q_L} \Big(\fint_{B(x)}\r^{d-\delta}\vert\nabla u\vert^2\mu_{\xi} \Big) \dd x,
$$
so that by \eqref{fintintC} combined with the estimate $\r \le L$, with periodicity, 
and using again the Lipschitz-continuity of $\r$, we obtain
$$
\int_{Q_L}\r(x)^{d-\delta}\Big(\fint_{B(x)}\vert\nabla u\vert^2\mu_{\xi} \Big) \dd x \,\lesssim\, \int_{Q_L}\r(x)^{d-\delta} \Big(\fint_{\Br(x)} \vert\nabla u\vert^2\mu_{\xi} \Big) \dd x.
$$
Inserting the weight $(1+\frac{|x|}{r})^\frac{2\e}{m}(1+\frac{|x|}{r})^{-\frac{2\e}{m}}$, and using H\"older's inequality in space with exponents $(\frac m2,\frac{m}{m-2})$ followed by H\"older's inequality in probability
with exponents $(\frac1{1-\tau},\frac1\tau)$, \eqref{e.sensi-estimNL-ant} turns into
\begin{align}
&{\frac{1}{q(1+|\xi|^p)} \expecL{\vert\calF_\xi\vert^{2q}}^{\frac{1}{q}}} \label{term0}
\\
\lesssim\,& 
\expecL{\Big(\int_{Q_L}(1+\tfrac{|x|}{r})^{-d-1}\r(x)^{\frac{m}{m-2}(d-\delta)}\, \dd x\Big)^{q\frac{m-2}{m}} \Big(\int_{Q_L}(1+\tfrac{|x|}{r})^{\varepsilon}\Big(\fint_{\Br(x)}|\nabla u|^2\mu_{\xi}\Big)^{\frac{m}{2}} \dd x\Big)^{q\frac2m}}^\frac1q
\nonumber\\
\le \,& \expecL{\Big(\int_{Q_L}(1+\tfrac{|x|}{r})^{-d-1}\r(x)^{\frac{m}{m-2}(d-\delta)}\, \dd x\Big)^{q\frac{m-2}{m(1-\tau)}}}^{\frac{1-\tau}{q}}\expecL{\Big(\int_{Q_L}(1+\tfrac{|x|}{r})^{\varepsilon}\Big(\fint_{\Br(x)}|\nabla u|^2\mu_{\xi} \Big)^{\frac{m}{2}} \dd x\Big)^{q\frac{2}{m\tau}}}^{\frac{\tau}{q}}.\nonumber
\end{align}
By the change of variables $\frac x r \leadsto x$, Jensen's inequality in space provided $q\geq \frac{m}{m-2} = 1+\frac{d+1}\e$, and the stationarity of $\r$, we control the first right-hand side term of \eqref{term0} by
\begin{equation}
\expecL{\Big(\int_{Q_L}(1+\tfrac{|x|}{r})^{-d-1} \r(x)^{\frac{m}{m-2}(d-\delta)}\, \dd x\Big)^{q\frac{m-2}{m(1-\tau)}}}^{\frac{1-\tau}{q}}\, \lesssim \, r^{d\frac{m-2}{m}}\expecL{ \r^{q\frac{d-\delta}{1-\tau}}}^{\frac{1-\tau}{q}}.
\label{firstterm}
\end{equation}
For the second right-hand side term of \eqref{term0}, we appeal to the quenched weighted Meyers estimate~\eqref{weightedLp}, which we may apply to equation~\eqref{equationdual} (rewriting the right-hand side as $\frac{1}{\sqrt{\mu_{\xi}}}g\sqrt{\mu_{\xi}}$) with weight $\omega_{\e,r}$ since $\e \le \frac \beta 2$:
\begin{eqnarray*}
\expecL{\Big(\int_{Q_L}(1+\tfrac{|x|}{r})^{\varepsilon}\Big(\fint_{\Br(x)}|\nabla u|^2\mu_{\xi}\Big)^{\frac{m}{2}} \dd x\Big)^{\frac{2}{m\tau}q}}^{\frac{\tau}{q}} &\lesssim_{K}&\expecL{\Big(\int_{Q_L}\omega_{2\varepsilon,r}(x)\Big(\fint_{\Br(x)}| g|^2 \tfrac{1}{\mu_{\xi}}\Big)^{\frac{m}{2}} \dd x\Big)^{\frac{2}{m\tau}q}}^{\frac{\tau}{q}}.
\end{eqnarray*}
Using that $\frac1{\mu_\xi} \le 1$, Jensen's inequality for the local integrals, the Lipschitz continuity of $\r$ in the form of $\sup_{\Br(x)} \omega_{2\varepsilon,r}\lesssim \inf_{\Br(x)} \omega_{2\varepsilon,r}$, and \eqref{eq-integrals}, we have
\begin{equation*}
{\int_{Q_L}\omega_{2\varepsilon,r}(x)\Big(\fint_{\Br(x)}| g|^2 \tfrac{1}{\mu_{\xi}}\Big)^{\frac{m}{2}} \dd x}
\,\lesssim\, \int_{Q_L}\fint_{\Br(x)}\omega_{2\varepsilon,r}| g|^m\dd x \,\simeq \, \int_{Q_L}\omega_{2\varepsilon,r}| g|^m
\,\lesssim \, r^{-d} \int_{B_r} \r^{2\e}.
\end{equation*} 
By stationarity of $\r$ and Jensen's inequality in probability, this entails
\begin{eqnarray}
\expecL{\Big(\int_{Q_L}(1+\tfrac{|x|}{r})^{\varepsilon}\Big(\fint_{\Br(x)}|\nabla u|^2\mu_{\xi}\Big)^{\frac{m}{2}} \dd x\Big)^{\frac{2}{m\tau}q}}^{\frac{\tau}{q}} &\lesssim_{K} 
& r^{-2d+\frac{2}{m}d}\expecL{ \r^{\frac{4\varepsilon}{m\tau}q}}^{\frac{\tau}{q}}\label{secondterm}.
\end{eqnarray}
By \eqref{e-def-eps} and our choice $m=2+\frac{2 \e}{d+1}$, $\frac{4\varepsilon}{m\tau}\leq \frac{d-\delta}{1-\tau}$,
we have  $\expecL{\r^{\frac{4\varepsilon}{m\tau}q}}^{\frac{\tau}{q}}\leq \expecL{\r^{q\frac{d-\delta}{1-\tau}}}^{\frac{4\e(1-\tau)}{q m(d-\delta)}}$
by H\"older's inequality. Using \eqref{e-def-eps} again, this time in form of $\frac{4\e(1-\tau)}{m(d-\delta)}\le \frac \tau 2$, and the lower bound $\r \ge 1$, \eqref{secondterm} finally turns into
\begin{equation}
\expecL{\Big(\int_{Q_L}(1+\tfrac{|x|}{r})^{\varepsilon}\Big(\fint_{\Br(x)}|\nabla u|^2\mu_{\xi}\Big)^{\frac{m}{2}} \dd x\Big)^{\frac{2}{m\tau}q}}^{\frac{\tau}{q}}\, \lesssim_{\vert\xi\vert}\, r^{-2d+\frac{2}{m}d}\expecL{\r^{q\frac{d-\delta}{1-\tau}}}^{\frac{\tau}{2q}}\label{secondterm+}.
\end{equation}
The claim \eqref{sensimomentboundNL} then follows from \eqref{term0}, \eqref{firstterm} and \eqref{secondterm+}.
\end{proof}

\subsection{Annealed Meyers' estimate}\label{sec:Annealed}

The annealed Meyers (or perturbative Calder\'on-Zygmund) estimates  introduced by Duerinckx and Otto in \cite{DO-20} (see also \cite{josien2020annealed}) constitute a very versatile upgrade of their quenched counterpart in stochastic homogenization.
In the present setting the annealed Meyers estimates take the following form.
\begin{theorem}[Annealed Meyers' estimate]\label{th:annealedmeyers}
Under Hypothesis~\ref{hypo}, for all $K\ge 1$, set  $\kappa := \frac{(\bar m -2)\wedge 1}{8}>0$ (where $\bar m$ is the Meyers exponent of Theorem~\ref{unweightmeyers} depends on $d,p,\lambda$, and $K$).  For all $\xi \in \R^d$ with $|\xi|\le K$, and for all $Q_L$-periodic random fields  $g$ and $u$ related via \eqref{LSMequationu},
we have for all exponents $2-\kappa \le q,m \le 2+\kappa$ and $0< \delta \le \frac12$,
\begin{equation}
\int_{Q_L}\expecL{\Big(\fint_{B(x)}\vert\nabla u\vert^2\mu_{\xi}\Big)^{\frac{q}{2}}}^\frac mq dx\lesssim_{K}\, \delta^{-\frac14} |\log \delta |^\frac12\int_{Q_L}\expecL{\Big(\fint_{B(x)}\vert g \vert^2\Big)^{\frac{q(1+\delta)}{2}}}^\frac m{q(1+\delta)} dx.
\label{annealedmeyers}
\end{equation}
The same result holds with $\aL$ replaced by its pointwise transpose field $\aL^*$.
\end{theorem}
The proof is based on the quenched Meyers estimate in the large of Theorem~\ref{unweightmeyers}, on the moment bounds of Theorem~\ref{boundrNLprop} on the Meyers minimal radius (which allows us to use duality at the price of a loss of stochastic integrability), real interpolation,
and the following refined dual version of the Calder\'on-Zygmund lemma due to Shen~\cite[Theorem~3.2]{Shen-07}, based on ideas by Caffarelli and Peral~\cite{CP-98}.
\begin{lemma}[\cite{CP-98,Shen-07}]\label{lem:Shen}
Given $1\le q<m\le\infty$, let $F,G\in L^{q}\cap L^{m}(Q_L)$ be nonnegative $Q_L$-periodic functions and let $C_0,C_1\ge 1$. Assume that for all balls $D$ (of radius $\lesssim L$) there exist measurable functions $F_{D,1}$ and $F_{D,2}$ such that $F\le F_{D,1}+F_{D,2}$ and $F_{D,2}\le F+F_{D,1}$ on $D$, and such that
\begin{equation*}
\Big(\fint_{D}F_{D,1}^{q}\Big)^\frac1{q}\,\le\,C_1\Big(\fint_{C_0D}G^{q}\Big)^\frac1{q},\quad
\Big(\fint_{\frac1{C_0}D}F_{D,2}^{m}\Big)^\frac1{m}\,\le\,C_1\Big(\fint_{D}F_{D,2}^{q}\Big)^\frac1{q}.
\end{equation*}
Then,
for all $q<s<m$,
\[\Big(\int_{Q_L}F^s\Big)^\frac1s\,\lesssim_{C_0,C_1,q,s,m}\,\Big(\int_{Q_L}G^s\Big)^\frac1s.\]
\end{lemma}
Before we prove Theorem~\ref{th:annealedmeyers}, let us note that Theorem~\ref{boundrNLprop} allows one to pass from averages on $\Br(x)$ to averages on $B(x)$ using \cite[Lemma~6.7]{DO-20} in the (slightly more general) form of 
\begin{lemma}\label{lem:postproc}
Let $\r$ be a stationary random field satisfying $\expecL{\exp(c \r^\alpha)}\le 2$ for some $\alpha>0$ and $c\simeq 1$. Set $\Br(x):=B_{\r(x)}(x)$ for all $x \in Q_L$.
For all $f\in C^\infty_\per(Q_L;L^\infty(d\mathbb{P}_L))$ and $1\le q_1\le q_2<\infty$, we have
\begin{itemize}
\item[(i)] for all $r>q_1$,
\begin{equation*}
\Big(\int_{Q_L}\E_L\Big[\Big(\fint_{\Br(x)}|f|^2\Big)^\frac{q_1}2\Big]^\frac{q_2}{q_1}dx\Big)^\frac1{q_2}
\,\lesssim\,(\tfrac1{q_1}-\tfrac1r)^{-(\frac{1}{q_1}-\frac{1}2)_+}\zeta(\tfrac1{q_1}-\tfrac1r)^{\frac{1}{q_1}-\frac1{q_2}}\Big(\int_{Q_L}\E_L\Big[\Big(\fint_{B(x)}|f|^2\Big)^\frac{r}2\Big]^\frac{q_2}{r}\Big)^\frac1{q_2};
\end{equation*}
\item[(ii)] for all $r<q_1$,
\begin{equation*}
\Big(\int_{Q_L}\E_L\Big[\Big(\fint_{\Br(x)}|f|^2\Big)^\frac{q_1}2\Big]^\frac{q_2}{q_1}dx\Big)^\frac1{q_2}
\,\gtrsim\,(\tfrac1r-\tfrac1{q_1})^{(\frac12-\frac1{q_2})_+}\zeta(\tfrac1r-\tfrac1{q_1})^{-(\frac1{q_1}-\frac1{q_2})}\Big(\int_{Q_L}\E_L\Big[\Big(\fint_{B(x)}|f|^2\Big)^\frac{r}2\Big]^\frac{q_2}{r}\Big)^\frac1{q_2};
\end{equation*}
\end{itemize}
where we have set $\zeta(t):=\log(2+\frac1t)$, and the multiplicative constants depend on 
$q_1,q_2,\alpha$.
\end{lemma}
The proof of this result is identical to that of \cite[Lemma~6.7]{DO-20}, noting that 
the assumption $\expecL{\exp(\frac1C \r^d)}\le 2$ can be weakened to $\expecL{\exp(\frac1C \r^\alpha)}\le 2$ for any $\alpha>0$ at the price of adding a dependence on $\alpha$ in the multiplicative factors in the estimates, and $\R^d$ can be replaced by $Q_L$.

\begin{proof}[Proof of Theorem~\ref{th:annealedmeyers}]
We split the proof into three steps.
In the first step, we upgrade Theorem~\ref{unweightmeyers} by adding expectations
using Lemma~\ref{lem:Shen} in a suitable way.
At the price of a loss of stochastic integrability we then remove the local averages at scale $\r$ in Step~2 by using Lemma~\ref{lem:postproc}.
The formulation with local averages at unit scale allows us to conclude using a standard duality argument, and real interpolation.

\medskip 

\step1 Proof that for all $2 \le q < m < \bar m$, we have
\begin{equation}
\int_{Q_L}\expecL{\Big(\fint_{\Br(x)}\vert\nabla u\vert^2\mu_{\xi}\Big)^{\frac{q}{2}}}^\frac mq dx\lesssim_{K}  \int_{Q_L}\expecL{\Big(\fint_{\Bt(x)}\vert g \vert^2\Big)^{\frac{q}{2}}}^\frac m{q} dx
\label{annealed-1.1}
\end{equation}
with the short hand notation $\Bt(x):=B_{5\r(x)}(x)$.

\noindent Let $2\le q_1 \le m_1 \le \bar m$.
Let $D$ be a ball centered at $x \in Q_L$ and of radius $0<r_D \lesssim L$, we define $\Dr:=B_{r_D \vee (2\r(x))}(x)$, and let $N$ be the smallest integer so that $\Dr \subset Q_{NL}$. We then decompose $u$ as $u=u_{D,1}+u_{D,2}$, where 
$u_{D,1}$ is the $Q_{NL}$-periodic solution of $-\nabla\cdot a_\xi \nabla u_{D,1}=\nabla \cdot g \sqrt{\mu_\xi} \mathds 1_{\Dr}$. Note that $u_{D,2}$ is $a_\xi$-harmonic on $\Dr$.
We start with the control of $u_{D,1}$ and claim that 
\begin{equation}\label{annealed-1.2}
\int_D \expecL{\Big(\fint_{\Br(y)}|\nabla u_{D,1}|^2 \mu_\xi\Big)^\frac{q_1}{2}}\dd y \,\lesssim\, \expecL{\int_{8D} \Big(\fint_{\Bt(y)} |g|^2 \Big)^\frac{q_1}{2}\dd y}.
\end{equation}
Assume first that $r_D \ge   2\r(x)$, so that $\Dr=D$.
By taking the expectation in Theorem~\ref{unweightmeyers}, we have
\begin{multline*}
\int_D \expecL{\Big(\fint_{\Br(y)}|\nabla u_{D,1}|^2 \mu_\xi\Big)^\frac{q_1}{2}}\dd y
\,\le\, \expecL{\int_{Q_{NL}} \Big(\fint_{\Br(y)}|\nabla u_{D,1}|^2 \mu_\xi\Big)^\frac{q_1}{2}\dd y}
\\
\stackrel{\eqref{unweightedmeyers}}\lesssim \,  \expecL{\int_{Q_{NL}} \Big(\fint_{\Br(y)} |g|^2 \mathds 1_{D}\Big)^\frac{q_1}{2}\dd y}.
\end{multline*}
By the $\frac1{16}$-Lipschitz property of $\r$, we have the implication for all $z \in \Br(y)$
\begin{align*}
|y-x|\ge 2r_D \,\implies \, |z-x|\ge |y-x|-\r(y) & \ge |y-x|-(\r(x)+\tfrac1{16}|y-x|) 
\\
&\ge \frac{15}{16}|y-x|-\r(x) \ge (\tfrac{15}{8}-\tfrac12)r_D \ge r_D
\,
\implies \, \mathds{1}_D(z)=0,
\end{align*}
so that \eqref{annealed-1.2} follows for $r_D \ge   2\r(x)$ in the stronger form
\begin{equation*}
\int_D \expecL{\Big(\fint_{\Br(y)}|\nabla u_{D,1}|^2 \mu_\xi\Big)^\frac{q_1}{2}}\dd y \,\lesssim\, \expecL{\int_{2D} \Big(\fint_{\Br(y)} |g|^2 \Big)^\frac{q_1}{2}\dd y}.
\end{equation*}
If $r_D \le 2 \r(x)$, then $\sup_{D} \r \lesssim \inf_D \r$, $\Dr=B_{2\r(x)}(x)=:\Brr(x)$, and a plain energy estimate yields
\begin{multline*}
\int_D \expecL{\Big(\fint_{\Br(y)}|\nabla u_{D,1}|^2 \mu_\xi\Big)^\frac{q_1}{2}}\dd y
\,\lesssim\, \expecL{|D| \r(x)^{-d\frac {q_1}2}\Big(\int_{Q_{NL}}  |\nabla u_{D,1}|^2 \mu_\xi\Big)^\frac{q_1}{2}}
\\
\lesssim \,  \expecL{|D| \r(x)^{-d\frac {q_1}2}\Big(\int_{\Dr} |g|^2 \Big)^\frac{q_1}{2}} \,\lesssim\,
 \expecL{|D| \Big(\fint_{\Brr(x)} |g|^2 \Big)^\frac{q_1}{2}} ,
\end{multline*}
and it remains to turn the right-hand side into an integral over $D$.
For all $y \in D$, we have $\r(y)\ge \r(x)-\frac{1}{16} r_D \ge \frac78 \r(x)$, and therefore 
for all $z\in \Brr(x)$, $|z-y|\le |z-x|+|x-y|\le 4\r(x) \le 5\r(y)$, to the effect that $\Brr(x) \subset \Bt(y)$. Recalling that $\sup_{D} \r \lesssim \inf_D \r$, this implies the following stronger form
of~\eqref{annealed-1.2}
\begin{equation*}
\int_D \expecL{\Big(\fint_{\Br(y)}|\nabla u_{D,1}|^2 \mu_\xi\Big)^\frac{q_1}{2}}\dd y
\lesssim \,  \expecL{ \int_D \Big(\fint_{\Bt(y)} |g|^2 \Big)^\frac{q_1}{2}\dd y}  .
\end{equation*}
We now turn to the control of $u_{D,2}$, and claim that 
\begin{equation}\label{annealed-1.2b}
\Big(\fint_{\frac18D} \expecL{\Big(\fint_{\Br(y)}|\nabla u_{D,2}|^2 \mu_\xi\Big)^\frac{q_1}{2}}^\frac{m_1}{q_1}\dd y \Big)^\frac1{m_1} \,\lesssim\, \Big(\fint_{D} \expecL{\Big(\fint_{\Br(y)} |\nabla u_{D,2}|^2 \mu_\xi \Big)^\frac{q_1}{2}}\dd y\Big)^\frac{1}{q_1}.
\end{equation}
The starting point is the Minkowski inequality: Since $\frac{m_1}{q_1}\ge 1$,
\begin{equation}\label{annealed-1.3}
\Big(\fint_{\frac18D} \expecL{\Big(\fint_{\Br(y)}|\nabla u_{D,2}|^2 \mu_\xi\Big)^\frac{q_1}{2}}^\frac{m_1}{q_1}\dd y \Big)^\frac1{m_1} \,\le\, \expecL{\Big(\fint_{\frac18D}  \Big(\fint_{\Br(y)}|\nabla u_{D,2}|^2 \mu_\xi\Big)^\frac{m_1}{2}\dd y\Big)^\frac{q_1}{m_1}}^\frac{1}{q_1}.
\end{equation}
We then appeal to the local Meyers estimate~\eqref{unweightedmeyerslocal}
to bound the right-hand side
\begin{equation*}
\fint_{\frac18D}\Big(\fint_{\Br(y)} |\nabla u_{D,2}|^2\mu_{\xi}\Big)^{\frac{m_1}{2}}\dd y \, \lesssim_{K}\, \Big(\fint_{\frac14D}\Big(\fint_{\Br(y)} |\nabla u_{D,2}|^2\mu_{\xi}\Big) \,\dd y\Big)^{\frac{m_1}{2}}
+\fint_{\frac14D}\Big(\fint_{\Br(y)}| g|^2 (1-\mathds{1}_{\Dr})\Big)^{\frac{m_1}{2}}\dd y.
\end{equation*}
Since for all $y \in \frac14D$, one has $\r(y)\le \r(x)+\frac{1}{16} \frac14 r_D \le \frac34 (r_D \vee (2\r(x)))$, $\Br(y) \subset \Dr$ and the second right hand side term vanishes identically.
Combined with \eqref{annealed-1.3} and Jensen's inequality in space (using that $\frac{q_1}2\ge 1$), this entails
\begin{eqnarray*}
\Big(\fint_{\frac18D} \expecL{\Big(\fint_{\Br(y)}|\nabla u_{D,2}|^2 \mu_\xi\Big)^\frac{q_1}{2}}^\frac{m_1}{q_1}\dd y \Big)^\frac1{m_1} &\lesssim& \expecL{\Big(\fint_{\frac14D}\Big(\fint_{\Br(y)} |\nabla u_{D,2}|^2\mu_{\xi}\Big) \,\dd y\Big)^\frac{q_1}{2}}^\frac{1}{q_1} \nonumber 
\\
&\le &\expecL{\fint_{\frac14D}\Big(\fint_{\Br(y)} |\nabla u_{D,2}|^2\mu_{\xi}\Big)^\frac{q_1}{2} \,\dd y}^\frac{1}{q_1},
\end{eqnarray*}
from which \eqref{annealed-1.2b} follows.

\noindent We are in the position to conclude. 
Setting $F:x \mapsto \expecL{\Big(\fint_{\Br(x)} |\nabla u|^2\mu_\xi\Big)^\frac{q_1}{2}}^\frac1{q_1}$, $G:x\mapsto \expecL{\Big(\fint_{\Bt(x)} |g|^2\Big)^\frac{q_1}2}^\frac{1}{q_1}$,
$F_{D,1}:x \mapsto \expecL{\Big(\fint_{\Br(x)}|\nabla u_{D,1}|^2 \mu_\xi \Big)^\frac{q_1}{2}}^\frac1{q_1}$, and $F_{D,2}:x \mapsto  \expecL{\Big(\fint_{\Br(x)}|\nabla u_{D,2}|^2 \mu_\xi \Big)^\frac{q_1}{2}}^\frac1{q_1}$, the assumptions of  Lemma~\ref{lem:Shen} are satisfied, and 
the claimed estimate \eqref{annealed-1.1} follows.

\medskip

\step2 Reformulation of \eqref{annealed-1.1}.

\noindent Since both $\r$ and $5\r$ satisfy stretched exponential moment bounds, Lemma~\ref{lem:postproc} allows us to reformulate~\eqref{annealed-1.1}
as: For all $2 \le q < m < \bar m$ and $0<r \le \frac 12$,
\begin{equation}
\int_{Q_L}\expecL{\Big(\fint_{B(x)}\vert\nabla u\vert^2\mu_{\xi}\Big)^{\frac{q}{2}}}^\frac m{q} \dd x\lesssim_{K}  r^{-\frac{m-2}{2m}} |\log r |^\frac{2(m-q)}{qm}\int_{Q_L}\expecL{\Big(\fint_{B(x)}\vert g \vert^2\Big)^{\frac{q+r}{2}}}^\frac m{q+r} \dd x.
\label{annealed-2.1}
\end{equation}

\medskip

\step3 Proof of \eqref{annealedmeyers}.

\noindent First,  we show that 
for all $\bar m' <m< q \le 2$ and $0<r \ll 1$,
\begin{equation}
\int_{Q_L}\expecL{\Big(\fint_{B(x)}\vert\nabla u\vert^2\mu_{\xi}\Big)^{\frac{q}{2}}}^\frac m{q} \dd x\lesssim_{K}
 r^{-\frac{2-m}{2m}} |\log r |^\frac{2(q-m)}{qm}  \int_{Q_L}\expecL{\Big(\fint_{B(x)}\vert g \vert^2\Big)^{\frac{q-r}{2}}}^\frac m{q-r} \dd x.
\label{annealed-3.1}
\end{equation}
Indeed, by duality we have
\begin{equation*}
\Big(\int_{Q_L}\expecL{\Big(\fint_{B(x)}\vert\nabla u\vert^2\mu_{\xi}\Big)^{\frac{q}{2}}}^\frac m{q} \dd x\Big)^\frac1m\,
\,=\, \sup_{h}\Big\{\expecL{\int_{Q_L}\nabla u \cdot h \sqrt{\mu_\xi}} \Big\},
\end{equation*}
where the supremum runs over maps $h \in C^\infty_\per(Q_L,L^\infty(d\mathbb P_L))^d$
such that $\int_{Q_L}\expecL{\Big(\fint_{B(x)}|h|^2\Big)^{\frac{q'}{2}}}^\frac {m'}{q'} \dd x  =1$.
For such $h$, denote by $v_h$ the unique $Q_L$-periodic solution of 
$
-\nabla \cdot a_\xi^* \nabla v_h =\nabla \cdot (h\sqrt{\mu_\xi}).
$
Testing this equation with $u$ and the defining equation \eqref{LSMequationu} for $u$ by $v_h$, we obtain (using periodicity in the last equality)
$$
\int_{Q_L}\nabla u \cdot h \sqrt{\mu_\xi}=\int_{Q_L}\nabla v_h \cdot g \sqrt{\mu_\xi}
\,=\, \int_{Q_L} \Big(\fint_{B(x)} \nabla v_h \cdot g \sqrt{\mu_\xi} \Big)\dd x.
$$
By Cauchy-Schwarz' inequality on $B(x)$, followed by H\"older's inequality with exponents $(q-r,\frac{q-r}{q-r-1})$ on $Q_L$ and with exponent $(m,m')$ in probability, this yields 
\begin{equation*}
{\Big|\expecL{\int_{Q_L}\nabla v_h \cdot g \sqrt{\mu_\xi}}\Big|}
\,\le \,\Big(\int_{Q_L}\expecL{\Big(\fint_{B(x)}\vert g \vert^2\Big)^{\frac{q-r}{2}}}^\frac m{q-r} \dd x\Big)^\frac1m\Big(\int_{Q_L}\expecL{\Big(\fint_{B(x)}\vert \nabla v_h \vert^2 \mu_\xi\Big)^{\frac{(q-r)'}{2}}}^\frac {m'}{(q-r)'} \dd x\Big)^\frac1{m'}.
\end{equation*}
Since $(q-r)'-q'=\frac{r}{(q-1)(q-1-r)}$, we may apply \eqref{annealed-2.1} to $\nabla v_h$
to the effect that
\begin{eqnarray*}
\lefteqn{\Big|\expecL{\int_{Q_L}\nabla v_h \cdot g \sqrt{\mu_\xi}}\Big|}
\\
&\lesssim& r^{-\frac{m'-2}{2m'}} |\log r |^\frac{2(m'-q')}{q'm'} \Big(\int_{Q_L}\expecL{\Big(\fint_{B(x)}\vert g \vert^2\Big)^{\frac{q-r}{2}}}^\frac m{q-r} \dd x\Big)^\frac1m\Big(\int_{Q_L}\expecL{\Big(\fint_{B(x)}|h|^2\Big)^{\frac{q'}{2}}}^\frac {m'}{q'} \dd x\Big)^\frac1{m'},
\end{eqnarray*}
from which \eqref{annealed-3.1} follows by the arbitrariness of $h$ and the identities $\frac{m'-2}{2m'}=\frac{2-m}{2m}$
and $\frac{2(m'-q')}{q'm'} =\frac{2(q-m)}{qm}$.

\medskip

Replacing $r$ by $qr$ in \eqref{annealed-2.1} and \eqref{annealed-3.1}, and using the bounds
$\frac{m-2}{2m} \le \frac14$ and $\frac{2(m-q)}{qm} \le \frac 12$ for $2\le q \le m \le 3$ and  $\frac{2-m}{2m} \le \frac14$ and $\frac{2(q-m)}{qm} \le \frac 12$ for $\frac 32 \le m \le q \le 2$, we have thus proved that \eqref{annealedmeyers} holds for all $2 \le q < m < \bar m \wedge 3$
and for all $\bar m' \vee \frac32<m< q \le 2$.
By choosing $\kappa=\frac{(\bar m-2)\wedge 1}8$,  the validity of  \eqref{annealedmeyers} in the full range of exponents 
$2-\kappa \le m,q\le 2+\kappa$ then follows by real interpolation.
\end{proof}
We conclude this subsection by the annealed version of the maximal regularity for the Laplacian.
\begin{theorem}\label{th:annealed-lap}
Let $L\ge 1$. For all $Q_L$-periodic random fields $g$ and $u$ related via
\begin{equation*} 
-\triangle u=\nabla \cdot g,
\end{equation*}
we have for all exponents $1<m,q<\infty$,
\begin{equation}
\Big(\int_{Q_L}\expecL{\Big(\fint_{B(x)}\vert\nabla u\vert^2 \Big)^{\frac{q}{2}}}^\frac mq \dd x \Big)^\frac1m \lesssim_{m,q} \, \Big( \int_{Q_L}\expecL{\Big(\fint_{B(x)}\vert g \vert^2\Big)^{\frac{q }{2}}}^\frac m{q} \dd x\Big)^\frac1m.
\label{annealedmeyers-lap}
\end{equation}
\end{theorem}
A proof of this result can be found in \cite[Section~7.1]{josien2020annealed}.
A simpler argument (based on CZ estimates for Hilbert-valued operators and interpolation) would show that the multiplicative constant in \eqref{annealedmeyers-lap} is of the order $m+m'+q+q'$ (this finer result will not be used here).

\section{Control of correctors: Proof of Theorem~\ref{th:corrNL}}\label{sec:NL}

The proof relies on the following upgrade of Proposition~\ref{weakNL}
based on Theorem~\ref{boundrNLprop} and on Theorem~\ref{th:annealedmeyers}.
\begin{corollary}\label{cor:average-per-NL}
Under Hypothesis~\ref{hypo}, there exists $\gamma>0$ such that for all $K\ge 1$, all $\xi \in \R^d$ with $|\xi|\le K$, all $L \ge 1$, and all $g \in L^2(\R^d)$ compactly supported in $Q_L$, the random field
$\calF:=\int_{Q_L}g(\nabla\phi_\xi,\nabla \sigma_\xi)$
satisfies for all $q \ge 1$
\begin{equation}\label{e.cor:average-per-NL}
\expecL{|\calF|^{2q}}^{\frac{1}{q}}\lesssim_{K} q^\gamma  \int_{Q_L}|g|^2.
\end{equation}
\end{corollary}
For future reference, we state the following consequence of local regularity and of the hole-filling estimate.
\begin{lemma}\label{lem:supnablaphi}
Under Hypothesis~\ref{hypo}, with $0<{\delta}\le d$ the nonlinear hole-filling exponent of Lemma~\ref{ctrlavNL}, we have
for all $\xi \in \R^d$ and $x\in \R^d$
\begin{equation}
\|\xi+\nabla\corNL\|_{\text{C}^{\alpha}(B(x))}\, \lesssim \, (1+|\xi|)  (\r(x))^{\frac{d-{\delta}}{p}}.
\label{deterboundproof1}
\end{equation}
\end{lemma}
\begin{proof}
By the deterministic regularity theory of Lemma~\ref{regestiNL} applied to the equation~\eqref{e.cor-eq} combined with the estimate \eqref{controlunitball}, we  indeed have
\begin{align*}
\|\xi+\nabla\corNL\|_{\text{C}^{\alpha}(B(x))}&\lesssim_{\|A\|_{C^{0,\alpha}(\mathbb{R}^d)}} \Big(\fint_{B_{2}(x)}\vert \xi+\nabla\corNL \vert^p\Big)^{\frac{1}{p}}\nonumber\\
&\stackrel{\eqref{controlunitball}}{\leq}_{\|A\|_{C^{0,\alpha}(\mathbb{R}^d)}}(1+|\xi|)  (\r(x))^{\frac{d-{\delta}}{p}}.
\end{align*}
\end{proof}
Before we turn to the proof of Corollary~\ref{cor:average-per-NL}, let us quickly argue that it yields Theorem~\ref{th:corrNL}.
\begin{proof}[Proof of Theorem~\ref{th:corrNL}]
By \eqref{deterboundproof1} and Theorem~\ref{boundrNLprop}, assumption~\eqref{convergenceofthecor-hyp} in Proposition~\ref{convergenceofperiodiccorrectors} is satisfied for $\nabla \phi_\xi$.
Let us  show that this also yields assumption~\eqref{convergenceofthecor-hyp} for $\nabla \sigma_\xi$. Indeed, by maximal regularity for the Laplacian applied to
equation \eqref{e.Laplace-sig} we have for all $q> 1$,
$
\Big(\int_{Q_L} |\nabla \sigma_\xi|^q\Big)^\frac1q \lesssim q \Big(\int_{Q_L} |\xi+\nabla \phi_\xi|^q\Big)^\frac1q,
$
so that assumption~\eqref{convergenceofthecor-hyp} for $\nabla \sigma_\xi$ follows from taking the expectation of the $q$-th power of this estimate and using the stationarity of the extended corrector gradient together with the moment bound on $\nabla \phi_\xi$.
By \eqref{convergenceofthecor}, we can then pass to the limit $L\uparrow +\infty$ in the moment bounds on the extended corrector gradient for the periodized ensemble, and obtain~\eqref{e.bdd-grad-corrNL}.
Likewise, the claimed estimate~\eqref{e:corr-NL-CLT} follows from Corollary~\ref{cor:average-per-NL} for $g$ compactly supported by passing to the limit $L\uparrow \infty$ using \eqref{convergenceofthecor}. The result for general $g \in L^2(\R^d)$ is then obtained by approximation.
The control \eqref{e.growth-nlc} of the growth of the extended corrector is a direct consequence of~\eqref{e:corr-NL-CLT} by ``integration'' (see for instance \cite[Proof of Theorem~4.2, Step~3]{DG-21} -- the argument is also displayed in the proof of Corollary~\ref{coro:corr-diff}). 
\end{proof}
It remains to prove Corollary~\ref{cor:average-per-NL}.
\begin{proof}[Proof of Corollary~\ref{cor:average-per-NL}]
We split the proof into two steps, first treat averages of $\nabla \phi_\xi$ and then turn to averages of $\nabla \sigma_\xi$.

\medskip

\step1 Averages of $\nabla \phi_\xi$.

\noindent In this step we set $\calF:=\int_{Q_L}g \cdot \nabla\phi_\xi$ for some $g \in L^2(\R^d)^d$ compactly supported in $Q_L$.
The starting point is the estimate~\eqref{e.sensi-estimNL-ant} in the proof of Proposition~\ref{weakNL}, which takes the form for all $q\ge 1$ of
\begin{equation*} 
\expecL{\vert\calF\vert^{2q}}^{\frac{1}{q}}\,\lesssim\,  q K^p\expecL{\Big(\int_{Q_L}\r(x)^{d-\delta}\Big(\fint_{B(x)}\vert\nabla u\vert^2\mu_{\xi}\Big) \dd x\Big)^q}^{\frac{1}{q}},
\end{equation*}
where $u$ is the unique weak $Q_L$-periodic solution (with zero average) of \eqref{equationdual}, that is, $-\nabla\cdot \aL^*\nabla u=\nabla\cdot g$.
By duality,  we may reformulate the right-hand side as
\begin{eqnarray*}
\lefteqn{\expecL{\Big(\int_{Q_L}\r(x)^{d-\delta}\Big(\fint_{B(x)}\vert\nabla u\vert^2\mu_{\xi}\Big) \dd x\Big)^q}^{\frac{1}{q}}}
\\
&=&\sup_{\mathbb E_L[|Y|^{q'}]=1} \expecL{Y \int_{Q_L}\r(x)^{d-\delta}\Big(\fint_{B(x)}\vert\nabla  u\vert^2\mu_{\xi}\Big) \dd x}
\\
&=&\sup_{Y\ge 0,\mathbb E_L[Y^{q'}]=1} \expecL{Y \int_{Q_L}\r(x)^{d-\delta}\Big(\fint_{B(x)}\vert\nabla  u\vert^2\mu_{\xi}\Big) \dd x}
\\
&=& \sup_{\mathbb E_L[|X|^{2q'}]=1} \expecL{\int_{Q_L}\r(x)^{d-\delta}\Big(\fint_{B(x)}\vert\nabla X u\vert^2\mu_{\xi}\Big) \dd x},
\end{eqnarray*}
where the supremum runs over random variables $X \in  L^{2q'}(d\mathbb P_L)$ which are independent of the space variable (which allows us to put $X$ inside the gradient).
Let $0<\eta<1$ be some exponent (to be fixed later) small enough so that $\frac{q'}{1+\eta}>1$.
We then appeal to H\"older's inequality with exponents $(\frac{q'}{q'-1-\eta},\frac{q'}{1+\eta})$ and to the stationarity of $\r$  to the effect that 
\begin{equation*}
\expecL{\int_{Q_L}\r(x)^{d-\delta}\Big(\fint_{B(x)}\vert\nabla X u\vert^2\mu_{\xi}\Big) \dd x}
\, \le \,  \expec{\r^{\frac{q'}{q'-1-\eta}(d-\delta)}}^\frac{q'-1-\eta}{q'}   \int_{Q_L} \expecL{\Big(\fint_{B(x)}\vert\nabla X u\vert^2\mu_{\xi}\Big)^{\frac{q'}{1+\eta}}}^\frac{1+\eta}{q'} \dd x.
\end{equation*}
Provided $2q'\le 2+\kappa$, we may appeal to Theorem~\ref{th:annealedmeyers} on the second right-hand side factor, which yields (recall that $X$ does not depend on the space variable, that $\mathbb E_L[|X|^{2q'}]=1$, and that $\mu_\xi \ge 1$) 
\begin{multline*}
\int_{Q_L} \expecL{\Big(\fint_{B(x)}\vert\nabla X u\vert^2\mu_{\xi}\Big)^{\frac{q'}{1+\eta}}}^\frac{1+\eta}{q'} \dd x
\\
\lesssim \, \eta^{-\frac14}|\log(\eta)|^\frac12 \int_{Q_L} \expecL{|X|^{2q'}\Big(\fint_{B(x)}\vert g \vert^2\tfrac{1}{\mu_{\xi}}\Big)^{q'}}^\frac{1}{q'} \dd x
\,
\le  \,  \eta^{-\frac14}|\log(\eta)|^\frac12 \int |g|^2.
\end{multline*}
The choice $\eta=\frac12(q'-1)=\frac1{2(q-1)}$ is legitimate provided $q\gg 1$, in which case
the above combined with the moment bound on $\r$ of Theorem~\ref{boundrNLprop}
yields 
\begin{equation*} 
\expecL{\vert\calF\vert^{2q}}^{\frac{1}{q}}\,\lesssim\,  q^\nu K^p \int |g|^2.
\end{equation*}
for some exponent $\nu>0$ independent of $q$. 
This entails \eqref{e.cor:average-per-NL} for $\nabla \phi_\xi$ for a suitable exponent $\gamma>0$ (depending only on $\nu$).

\medskip

\step2 Averages of $\nabla \sigma_\xi$.

\noindent Fix $1\le i,j \le d$. We proceed as for $\nabla \phi_\xi$: We first derive a representation formula for the sensitivity of $\calF:= \int_{Q_L}g \cdot \nabla \sigma_{\xi,ij}$
with respect to changes of the coefficient $A$, and then use the annealed estimates of
Theorems~\ref{th:annealedmeyers} and \ref{th:annealed-lap}, and the moment bounds on $\r$ to conclude.

\medskip

\substep{2.1} Sensitivity calculus.

\noindent Recall the defining equation for $\sigma_{\xi,ij}$
\begin{equation*}
-\triangle \sigma_{\xi,ij} \,=\, \partial_i (a(\cdot,\xi+\nabla\phi_{\xi})\cdot e_j)-\partial_j(a(\cdot,\xi+\nabla\phi_{\xi})\cdot e_i).
\end{equation*}
As in Step~1 of the proof of Proposition~\ref{weakNL}, we proceed by duality. This time we introduce two auxiliary functions $u_1$ and $u_2$ as $Q_L$-periodic solutions of 
$$
-\triangle u_1 = \nabla \cdot g, \quad -\nabla \cdot a_\xi^* \nabla u_2 = \nabla \cdot a_\xi^*(\partial_i u_1 e_j-\partial_j u_1 e_i),
$$
and claim that 
\begin{equation}
\delta_x \calF = \int_{B(x)}\vert \aa(\xi+\nabla\phi_{\xi})\otimes (\nabla  u_2+\partial_i u_1 e_j-\partial_j u_1 e_i)\vert.
\label{e.sens-sigma}
\end{equation}
Let us quickly argue in favor of \eqref{e.sens-sigma}. 
With the notation of Step~1 of the proof of Proposition~\ref{weakNL}, and $\delta A$ an increment of $A$ localized in $B(x)$, we have
by the defining equations for $\sigma_{\xi,ij}$ and $u_1$ 
$$
\delta^h \calF := \frac{\calF(A+h\delta A)-\calF(A)}{h}
\,=\, \int (\partial_i u_1 e_j-\partial_j u_1 e_i) \cdot \delta^h \Big(a(\xi+\nabla \phi_\xi)\Big),
$$
where
\begin{eqnarray*}
\delta^h \big(a(\xi+\nabla \phi_\xi)\big)&=& \frac{(A+h\delta A)\aa(\xi+\nabla \phi_\xi(A+h\delta A))-A\aa(\xi+\nabla \phi_\xi)}{h}
\\
&=&\delta A \aa(\xi+\nabla \phi_\xi(A+h\delta A))+ a^h_\xi \nabla \delta^h \phi_\xi.
\end{eqnarray*}
Passing to the limit $h\downarrow 0$, and testing  the equation for $u_2$ with $\delta \phi_\xi$ and equation \eqref{sensiNLequation1+} with $u_2$,
we obtain 
\begin{eqnarray*}
\delta \calF \,=\,\lim_{h\downarrow 0} \delta^h \calF 
&=& \int  (\partial_i u_1 e_j-\partial_j u_1 e_i) \cdot \Big(\delta A \aa(\xi+\nabla \phi_\xi)+a_\xi \nabla \delta \phi_\xi\Big)
\\
&=&\int  (\nabla u_2+\partial_i u_1 e_j-\partial_j u_1 e_i) \cdot  \delta A \aa(\xi+\nabla \phi_\xi),
\end{eqnarray*}
and the claim follows by taking the supremum over $\delta A$.

\medskip

\substep{2.2} Proof of \eqref{e.cor:average-per-NL}.

\noindent Combining \eqref{e.sens-sigma} with the logarithmic-Sobolev inequality,  we obtain for all $q\ge 1$
\begin{equation*}
\expecL{|\calF|^{2q}}^{\frac{1}{q}}\,\lesssim\,  q \expecL{\Big(\int_{Q_L} \Big(\fint_{B(x)}|\aa(\xi+\nabla \phi_\xi)|\vert\nabla  u_2+\partial_i u_1 e_j-\partial_j u_1 e_i\vert \Big)^2 \dd x\Big)^q}^{\frac{1}{q}}.
\end{equation*}
We treat differently the terms involving $u_1$ and $u_2$. For $u_2$ we proceed as in Step~3 of the proof of Proposition~\ref{weakNL} (using the definition~\eqref{e.def-mu}
 of $\mu_\xi$ and \eqref{controlunitball}), whereas for $u_1$ we directly use \eqref{deterboundproof1}.
This yields
\begin{multline*}
\expecL{|\calF|^{2q}}^{\frac{1}{q}}\,\lesssim_{K}\,  q \expecL{\Big(\int_{Q_L}\r(x)^{d-\delta}\Big(\int_{B(x)}\vert\nabla  u_2\vert^2\mu_{\xi}\Big) \dd x\Big)^q}^{\frac{1}{q}}
\\
+q \expecL{\Big(\int_{Q_L}\r(x)^{\frac{2(p-1)}p(d-\delta)}\Big(\int_{B(x)}\vert\nabla u_1\vert^2\Big) \dd x\Big)^q}^{\frac{1}{q}}.
\end{multline*}
As in Step~1, this entails
\begin{multline*}
\expecL{|\calF|^{2q}}^{\frac{1}{q}}\,\lesssim_{K}\,q  \sup_{\mathbb E_L[|X|^{2q'}]=1} \expecL{\int_{Q_L}\r(x)^{d-\delta}\Big(\int_{B(x)}\vert \nabla X  u_2\vert^2\mu_{\xi}\Big) \dd x}
\\
+q  \sup_{\mathbb E_L[|X|^{2q'}]=1}  \expecL{ \int_{Q_L}\r(x)^{\frac{2(p-1)}p(d-\delta)}\Big(\int_{B(x)}\vert\nabla Xu_1\vert^2\Big) \dd x }.
\end{multline*}
For the second right-hand side term, we proceed as in Step~1 (using Theorem~\ref{th:annealed-lap} in place of Theorem~\ref{th:annealedmeyers}),
and it remains to treat the first right-hand side term. 
We   use H\"older's inequality with exponents  $(\frac{q'}{q'-(1+\eta)^2},\frac{q'}{(1+\eta)^2})$ for some $0<\eta<1$ (that satisfies $q'>(1+\eta)^2$) to be chosen below to the effect that
\begin{multline*}
\expecL{\int_{Q_L}\r(x)^{d-\delta}\Big(\int_{B(x)}\vert\nabla X  u_2\vert^2\mu_{\xi}\Big) \dd x}
\\
\le \,  \expec{\r^{\frac{q'}{q'-(1+\eta)^2}(d-\delta)}}^\frac{q'-(1+\eta)^2}{q'}   \int_{Q_L} \expecL{\Big(\fint_{B(x)}\vert\nabla X  u_2\vert^2\mu_{\xi}\Big)^{\frac{q'}{(1+\eta)^2}}}^\frac{(1+\eta)^2}{q'} \dd x.
\end{multline*}
We then appeal to  the annealed Meyers estimate of Theorem~\ref{th:annealedmeyers} under the additional condition 
that $2\le \frac{2q'}{(1+\eta)^2}\le 2+\kappa$, and obtain 
\begin{equation*}
\int_{Q_L} \expecL{\Big(\fint_{B(x)}\vert\nabla X  u_2\vert^2\mu_{\xi}\Big)^{\frac{q'}{(1+\eta)^2}}}^\frac{(1+\eta)^2}{q'} \dd x
\,
\lesssim \,\eta^\frac14 |\log \eta|^\frac12 \int_{Q_L} \expecL{\Big(\fint_{B(x)}\vert \mu_\xi \nabla Xu_1  \vert^2\tfrac{1}{\mu_{\xi}}\Big)^{\frac{q'}{1+\eta}}}^\frac{1+\eta}{q'}
\dd x
\end{equation*}
since under the assumption $0<\eta<\frac12$, we have $(1+\eta)^2-1\lesssim \eta$.
Bounding $\mu_\xi$ by $\r^{\frac{p-2}p(d-\delta)}$ (cf.~Lemma~\ref{ctrlavNL}) and using H\"older's inequality
with exponents $(\frac{1+\eta}{\eta},1+\eta)$, the integral in the right-hand side is controlled by
$$
\int_{Q_L} \expecL{\Big(\fint_{B(x)}\vert \mu_\xi \nabla Xu_1  \vert^2\tfrac{1}{\mu_{\xi}}\Big)^{\frac{q'}{1+\eta}}}^\frac{1+\eta}{q'}
\,\lesssim \, \expecL{\r^{\frac{q'}{\eta} \frac{p-2}p(d-\delta) }}^\frac{\eta}{q'}
\int_{Q_L} \expecL{\Big(\fint_{B(x)}\vert   \nabla Xu_1  \vert^2 \Big)^{ q' }}^\frac{1}{q'}.
$$
We finally estimate the integral term by Theorem~\ref{th:annealed-lap}, which yields (since there is no loss in the stochastic exponent, $g$ is deterministic, and $1\le q'\le 2$)
$$
\int_{Q_L} \expecL{\Big(\fint_{B(x)}\vert   \nabla Xu_1  \vert^2 \Big)^{ q' }}^\frac{1}{q'}
\,\lesssim\, \expecL{|X|^{2q'}} \int_{Q_L}|g|^2 = \int_{Q_L}|g|^2.
$$
The conclusion follows by choosing $\eta=\frac14(q'-1)$ and $q\gg 1$,
and using  the moment bound on $\r$ of Theorem~\ref{boundrNLprop}.
\end{proof}

\section{Control of corrector differences: Proof of Theorem~\ref{th:corrL}}\label{sec:corr-diff}

\subsection{Reduction argument}

As for nonlinear correctors, by Proposition~\ref{convergenceofperiodiccorrectors}  it is enough to prove estimates for $L$-periodic ensembles
that are uniform with respect to $L$.
We split the version of Corollary~\ref{cor:average-per-NL} for the linearized corrector into two statements: Proposition~\ref{prop:average-per-L} below shows that averages of the gradient of the extended linearized corrector decay at the CLT scaling provided we have good control of moments of $\nabla \tilde \phi_{\xi,e}$, whereas Proposition~\ref{prop:average-per-L+} provides the latter.
\begin{proposition}\label{prop:average-per-L}
Under Hypothesis~\ref{hypo},  for all $K\ge 1$ and all $0<\theta<1$, there exists $\gamma>0$   (depending on $K$ and $\theta$) such that for all $L \ge 1$, all $\xi \in \R^d$ with $|\xi|\le K$, all $g \in L^2(\R^d)$ compactly supported in $Q_L$, and all 
unit vectors $e \in \R^d$, the random field
$\calF:=\int_{Q_L}g\cdot (\nabla\tilde \phi_{\xi,e},\nabla \tilde \sigma_{\xi,e})$
satisfies for all $q \ge 1$ such that  $2q'\le 2+\kappa$ (where $\kappa>0$ is as in Theorem~\ref{th:annealedmeyers})
\begin{equation}\label{e.prop:average-per-L}
\expecL{|\calF|^{2q}}^{\frac{1}{q}}\lesssim_{K,\theta} q^\gamma \expecL{\Big(\sup_B |\nabla\tilde\phi_{\xi,e}+e|^2 \mu_\xi\Big)^{q(1+\theta)}}^\frac1{q(1+\theta)} \Big(\int_{Q_L}|g|^2\Big)  
\end{equation}
\end{proposition}
The proof of Proposition~\ref{prop:average-per-L} relies on a sensitivity estimate by duality combined with the annealed Meyers estimate of Theorem~\ref{th:annealedmeyers}.
\begin{proposition}[Control of moments]\label{prop:average-per-L+}
Under Hypothesis~\ref{hypo},  for all $K\ge 1$ there exists $\gamma>0$ (depending on $K$) such that for all $L \ge 1$, all $\xi \in \R^d$ with $|\xi|\le K$,  and all 
unit vectors $e \in \R^d$ we have
\begin{equation}\label{e.prop:moment-per-L+}
\expecL{\Big(\sup_B |\nabla\tilde \phi_{\xi,e}+e|^2 \mu_\xi\Big)^{q}}^\frac1{q} \,\lesssim \, q^\gamma.
\end{equation}
\end{proposition}
The proof of Proposition~\ref{prop:average-per-L+} is based on Proposition~\ref{prop:average-per-L} and a buckling argument. Because the linearized corrector equation has unbounded coefficients, we cannot use the elegant approach of \cite{Otto-Tlse} (see also \cite[Proposition~4.5]{DG-21}) to buckle on moments of $\nabla \tilde \phi_{\xi,e}$ themselves. Instead, as we did for $\rNL$, we have to go through the super levelsets of some minimal radius controlling the growth of averages of $|\nabla \tilde \phi_{\xi,e}|^2\mu_\xi$.

\medskip

Before we turn to the proofs, let us show how bounds on linearized correctors allow us to derive bounds on nonlinear corrector differences in the form of Corollary~\ref{coro:corr-diff}.
\begin{proof}[Proof of Corollary~\ref{coro:corr-diff}]
For simplicity, we only treat $\phi_\xi$.

\medskip

\step1 Statement for differences of corrector gradients.

\noindent By \eqref{convergenceofthecor} in Proposition~\ref{convergenceofperiodiccorrectors} in form of (note the difference of expectations)
$$
\expec{|\nabla (\phi_\xi-\phi_{\xi'})|^q}^\frac1q \,=\, \lim_{L\uparrow +\infty}\expecL{\vert \nabla (\phi_\xi-\phi_{\xi'})\vert^q}^\frac1q,
$$
it suffices to prove the statement for the periodized ensemble.
By  Lemma~\ref{lemmadiffcor}, $\mathbb P_L$-almost surely, $\xi \mapsto \nabla \phi_\xi$ is differentiable and we have by the fundamental theorem of calculus (with implicit sum on the repeated index $i$)
\begin{equation}\label{e.good-diff}
\nabla \phi_\xi-\nabla \phi_{\xi'}= \int_0^1 \nabla \tilde \phi_{\xi+t(\xi'-\xi),e_i}  (\xi-\xi')_i\dd t,
\end{equation}
so that by taking the $q$-th moment and using Proposition~\ref{prop:average-per-L+}, one obtains
\begin{equation}\label{e.approxL-per-cor+}
\expecL{|\nabla \phi_\xi-\nabla \phi_{\xi'}|^q}^\frac1q \,\le \, |\xi-\xi'| \sum_i \int_0^1 \expecL{|\nabla \phi_{\xi+t(\xi'-\xi),e_i}|^q}^\frac1q\dd t  \, \lesssim \, q^\gamma |\xi-\xi'|,
\end{equation}
which yields the claim by taking the limit $L\uparrow \infty$.

\medskip

\step2 Statement for corrector differences.

\noindent By \eqref{convergenceofthecor+}, since $\int_B \phi_\xi =0$, for all $x\in \R^d$ we have 
for all $q\ge 1$  
\begin{equation}\label{e.approxL-per-cor}
\expec{\Big(\int_{B(x)} |\phi_\xi-\phi_{\xi'}|^2\Big)^\frac{q}2}^\frac1q
=\lim_{L\uparrow \infty}\expecL{\Big(\int_{B(x)} \Big|\phi_\xi-\phi_{\xi'}-\fint_B \phi_\xi-\phi_{\xi'}\Big|^2\Big)^\frac{q}2}^\frac1q.
\end{equation}
To control the right-hand side, we shall  bound moments of periodic random fields $\zeta$ by moments of averages of their gradients $\nabla \zeta$.
Indeed, by Poincar\'e's inequality on $B(x)$ for $x \in Q_L$, we have for $c=\fint_{B} \zeta$
(recall that $B$ denotes the unit ball centered at the origin)
\begin{equation}\label{eq:phi-bnd0}
\expecL{ \Big( \int_{B(x)} (\zeta-c)^2 \Big)^\frac q2}^\frac1q \,\lesssim\, \expecL{|\nabla \zeta|^q}^\frac1q +\expecL{\Big| \fint_{B(x)} \zeta -c \Big|^{q}}^\frac1q,
\end{equation}
and it remains to estimate the second right-hand side term.
For that purpose, we write
\[\fint_{B(x)}\zeta-\fint_{B}\zeta=\int_{Q_L}\nabla\zeta\cdot\nabla h_{x},\]
where $h_x$ denotes the unique weak periodic solution in $Q_L$ of
$-\triangle h_x=\tfrac{1}{|B|}(\mathds1_{B(x)}-\mathds1_B)$.
We apply this to $\zeta=\phi_\xi-\phi_{\xi'}$ and 
rewrite the gradient as $\nabla \zeta = \int_0^1 \nabla \phi_{\xi+t(\xi'-\xi),e_i}(\xi-\xi')_i\dd t$, to the effect that 
$$
\fint_{B(x)}(\phi_\xi-\phi_{\xi'})-\fint_{B}(\phi_\xi-\phi_{\xi'})=(\xi-\xi')_i\int_0^1 \Big(\int_{Q_L}\nabla \phi_{\xi+t(\xi'-\xi),e_i}\cdot\nabla h_{x}\Big)\dd t.
$$
Using Propositions~\ref{prop:average-per-L} and~\ref{prop:average-per-L+}, this yields
$$
\expecL{\Big|\fint_{B(x)}(\phi_\xi-\phi_{\xi'})-\fint_{B}(\phi_\xi-\phi_{\xi'})\Big|^q}^\frac1q
\,\le_{K} \,|\xi-\xi'| q^\gamma \Big(\int_{Q_L} |\nabla h_x|^2\Big)^\frac12.
$$
A direct computation with Green's kernel gives $\|\nabla h_{x}\|_{L^2(Q_L)}\,\lesssim\,\mu_d(x)$, 
and thus
$$
\expecL{\Big|\fint_{B(x)}(\phi_\xi-\phi_{\xi'})-\fint_{B}(\phi_\xi-\phi_{\xi'})\Big|^q}^\frac1q
\,\le_{K} \,|\xi-\xi'| q^\gamma \mu_d(x).
$$
Combined with \eqref{eq:phi-bnd0}, \eqref{e.approxL-per-cor}, and \eqref{e.approxL-per-cor+}, this entails
$
\expec{\Big(\int_{B(x)} |\phi_\xi-\phi_{\xi'}|^2\Big)^\frac{q}2}^\frac1q \, \lesssim\, |\xi-\xi'|q^\gamma \mu_d(x),
$
from which the claim follows using local regularity in form of Lemma~\ref{regestiNL} and
\eqref{e.approxL-per-cor+} in the limit $L\uparrow +\infty$. 

\medskip

\step3 Regularity of $\xi \mapsto \bar a(\xi)$.

\noindent The starting point is the definition $\bar a(\xi):=\expec{a(\xi+\nabla \phi_\xi)}=\expec{A(0)\aa(\xi+\nabla \phi_\xi(0))}$ and of its approximation by periodization 
$\bar a_L(\xi):=\expecL{A(0)\aa(\xi+\nabla \phi_\xi(0))}$ for all $L\ge 1$.
Since $\bar a_L(\xi) \to \bar a(\xi)$ as $L\uparrow +\infty$, it is enough to prove that $D\bar a_L$ is Lipschitz-continuous uniformly with respect to $L$ and given for all $\xi,e \in \R^d$ by
$$
D\bar a_L(\xi) e:= \bar a_{L,\xi} e=\expecL{A(0)D\aa(\xi+\nabla \phi_\xi(0))(e+\nabla \tilde \phi_{\xi,e}(0))}.
$$
The differentiability of $\xi \mapsto \bar a_L(\xi)$ and the formula for $D \bar a_L(\xi)$ follow from \eqref{e.good-diff}, the continuity of $\xi \mapsto \nabla \corL$, and the moment bounds on $\nabla \corL$. It remains to argue that $\xi \mapsto D\bar a_L(\xi)$ is Lipschitz-continuous. Since $\xi \mapsto \nabla \phi_\xi$ is continuously differentiable with stretched exponential moment bounds, it is enough to prove that $\xi \mapsto \nabla \tilde \phi_{\xi,e}$ is itself Lipschitz-continuous in $L^2(d\mathbb P_L)$. This is a direct consequence of the defining equation~\eqref{e.Lcorr} in the form for all $\xi,\xi' \in \R^d$ of
\begin{equation*}
-\nabla \cdot D a(\cdot,\xi+\nabla \phi_\xi)\nabla (\tilde \phi_{\xi,e}-\tilde \phi_{\xi',e}) \,=\,
\nabla \cdot (D a(\cdot,\xi+\nabla \phi_\xi)-D a(\cdot,\xi'+\nabla \phi_{\xi'})) (e+\nabla \tilde \phi_{\xi',e})
\end{equation*}
combined with the differentiability of $\xi \mapsto \nabla \phi_\xi$, uniform moment bounds on $\nabla\tilde \phi_{\xi',e}$ and $\nabla \phi_{\xi}$, and an energy estimate.
\end{proof}

\subsection{CLT-scaling: Proof of Proposition~\ref{prop:average-per-L}}

In this paragraph, we fix $e$ and $\xi$, and use the short-hand notation $\r$ for $\rNL$, $\phi$ for $\phi_\xi$, $\mu$ for $\mu_\xi$,  $\tilde \phi$ for $\tilde \phi_{\xi,e}$, $\tilde \sigma$ for $\tilde \sigma_{\xi,e}$.
We split the proof into three steps. In the first two steps, we derive sensitivity estimates for averages of $\nabla \tilde \phi$ and of $\nabla \tilde \sigma$, respectively, and then conclude in the third step using Theorems~\ref{th:annealedmeyers} and~\ref{th:annealed-lap}.

\medskip

\step1 Sensivity formula for $\nabla \tilde \phi$: The random 
variable $\calF_1:= \int_{Q_L} g \cdot \nabla \tilde \phi$ (where $g$ abusively denotes $g e'$ for some unit vector $e' \in \R^d$) satisfies for all $x\in Q_L$
\begin{equation}
\delta_x \calF_1\,=\, \int_{B(x)}\vert D\aa(\xi+\nabla\phi)(e+\nabla\phiL)\otimes \nabla u_1+\aa(\xi+\nabla\phi)\otimes \nabla u_2\vert,
\label{functioderivL}
\end{equation}
where we recall that $\aa: \xi\in\mathbb{R}^d\mapsto (1+\vert\xi\vert^{p-2})\xi$, and where $u_1,u_2\in H^1_\per(Q_L)$ are the unique weak solutions of 
\begin{equation}
-\nabla\cdot\TaL\nabla u_1=\nabla\cdot g 
\label{equationdualL1}
\end{equation}
and  (with an implicit sum over the repeated index $k$)
\begin{equation}
-\nabla\cdot\TaL\nabla u_2=\partial_k(D^2a(\xi+\nabla\phi)(e+\nabla\phiL)e_k\cdot \nabla u_1).
\label{equationdual2}
\end{equation}
(These equations are well-posed since the $Q_L$-periodic maps $\nabla \phi$ and $\TaL$ are bounded almost surely.)

\medskip

Let us give the quick argument. Following Step~1 of the proof of 
Proposition~\ref{weakNL}, we let $\delta A$ be an increment of $A$ localized in $B(x)$
and consider for $h$ small enough
\begin{align*}
\delta^h\mathcal{F}_{1}&:=\frac{\mathcal{F}_{1}(A+h\delta A)-\mathcal{F}_{1}(A)}{h}=\int_{Q_L}g\cdot \nabla\delta^h\phiL,\quad
\delta^h\phiL:=\frac{\phiL(A+h\delta A)-\phiL(A)}{h},\\
b^h_\xi&:=A\int_{0}^1D^2 \aa(\xi+t\nabla\phi(A+h\delta A)+(1-t)\nabla\phi(A))  \dd t ,
\end{align*}
and recall the notation \eqref{functionderivNLnota1} and \eqref{functionderivNLnota2}. 
By the defining equation~\eqref{e.Lcorr} for the linearized corrector, we obtain
\begin{align*}
-\nabla\cdot\aL\nabla\delta^h\phiL=&\nabla\cdot \delta A D\aa(\xi+\nabla \phi(A+h\delta A))(e+\nabla\phiL(A+h\delta A))\nonumber\\
&+\frac{1}{h}\nabla\cdot A(D\aa(\xi+\nabla\phi(A+h\delta A))-D\aa(\xi+\nabla\phi(A))(e+\nabla\phiL(A+h\delta A)), 
\end{align*}
which we rewrite, by the fundamental theorem of calculus and the definition of $b^h_{\xi}$, as
\begin{equation*}
-\nabla\cdot \aL\nabla\delta^h\phiL=\nabla\cdot \delta A D\aa(\xi+\nabla \phi(A+h\delta A))(e+\nabla\phiL(A+h\delta A))+\nabla\cdot b^h_\xi \nabla\delta^h\phi(e+\nabla\phiL(A+h\delta A)).
\end{equation*}
As in Step~1 of the proof of 
Proposition~\ref{weakNL}, we can pass to the limit as $h\downarrow 0$ and obtain 
that $\delta^h\phiL$ converges in $C^{1,\alpha}(Q_L)$ to the solution $\delta\phiL \in H^1_\per(Q_L)$ of
\begin{equation}
-\nabla\cdot \aL\nabla\delta\phiL=\nabla\cdot \delta A D\aa(\xi+\nabla \phi )(e+\nabla\phiL )+\nabla\cdot b_\xi \nabla\delta\phi(e+\nabla\phiL),
\label{sensiLequation1+}
\end{equation}
with $b_\xi:=D^2 a(\xi+ \nabla\phi)$.  

We now proceed by duality.
First, we test \eqref{sensiLequation1+} with $u_1$ and \eqref{equationdualL1} with $\delta\phiL$ to obtain
\begin{equation}
\delta\calF_{1}=\lim_{h\downarrow 0} \delta^h \calF_1=\int_{Q_L}\nabla u_1\cdot \delta A D\aa(\xi+\nabla \phi)(e+\nabla\phiL)+\int_{Q_L}\nabla u_1\cdot b_\xi \nabla\delta \phi(e+\nabla\phiL ).
\label{sensiLequation3}
\end{equation}
Second, we test \eqref{sensiNLequation1+} with $u_2$ and \eqref{equationdual2} with $\delta\phi$ to get
\begin{equation}
\int_{Q_L}\nabla u_1\cdot b_\xi \nabla\delta \phi(e+\nabla\phiL )=\int_{Q_L}\nabla u_2\cdot \delta A \aa(\xi+\nabla\phi ).
\label{sensiLequation4}
\end{equation} 
The combination of \eqref{sensiLequation3} and \eqref{sensiLequation4} then entails
the claim \eqref{functioderivL} by taking the supremum over $\delta A$.

\medskip

\step2 Sensitivity formula for $\nabla \tilde \sigma_{ij}$ (for $i,j$ fixed): The random 
variable $\calF_2:= \int_{Q_L} g \cdot \nabla \tilde \sigma_{ij}$ 
satisfies for all $x\in Q_L$
\begin{equation}
\delta_x \calF_2\,=\, \int_{B(x)}\vert  D\aa(\xi+\nabla \phi)(e+\nabla  \phiL) \otimes  (\nabla w_1+\partial_i v e_j-\partial_j v e_i)+\aa(\xi+\nabla \phi)\otimes \nabla w_2  \vert ,
\label{functioderiv-sigL}
\end{equation}
where the functions $v,w_1,w_2 \in H^1_\per(Q_L)$ solve (with an implicit sum over the repeated index $k$)
\begin{eqnarray}
-\triangle v&=&\nabla \cdot g,\label{e.sens-ant-v}
\\
-\nabla \cdot a_\xi^* \nabla w_1&=& \nabla \cdot  a_\xi^* (\partial_i  v e_j-\partial_j v e_i),\label{e.sens-ant-w1}
\\
-\nabla \cdot a_\xi^* \nabla w_2&=&\partial_k\big(D^2a(\xi+\nabla\phi)(e+\nabla\phiL)e_k\cdot (\nabla w_1+\partial_i v e_j-\partial_j v e_j)\big).\label{e.sens-ant-w2}
\end{eqnarray}
We only display the algebra of the argument (passing already to the limit $h\downarrow 0$,
which entails that $\delta=\lim_{h\downarrow 0}\delta^h$ satisfies the Leibniz rule).
Recall the defining equation for $\tilde \sigma_{ij}$ with the notation 
$a_\xi=D a(\xi+\nabla \phi)$
\begin{equation*}
-\triangle \tilde \sigma_{ij} \,=\, \partial_i (a_\xi(e+\nabla  \phiL)\cdot e_j)-\partial_j(a_\xi(e+\nabla \phiL)\cdot e_i).
\end{equation*}
First, by \eqref{e.sens-ant-v},
$$
\delta \calF_2 = \int  (\partial_i v e_j-\partial_j v e_i)  \cdot \delta \Big( D a(\xi+\nabla \phi)(e+\nabla  \phiL)\Big).
$$
Since $\delta$ satisfies the Leibniz rule, we have
\begin{equation*}
\delta \big( D a(\xi+\nabla \phi)(e+\nabla  \phiL)\big)
\,=\, \delta A D\aa(\xi+\nabla \phi)(e+\nabla  \phiL)+ D^2 a(\xi+\nabla \phi) \nabla \delta \phi (e+\nabla  \phiL)+a_\xi \nabla \delta \phiL. 
\end{equation*}
The first right-hand term directly gives the right-hand side contribution of \eqref{functioderiv-sigL} involving $\nabla v$.
For the second term, we introduce the solutions $w_{2,1}$ and $w_{2,2}$ of 
$-\nabla \cdot a_\xi^* \nabla w_{2,1}=\partial_k\big(D^2a(\xi+\nabla\phi)(e+\nabla\phiL)e_k\cdot (\partial_i v e_j-\partial_j v e_j)\big)$
and $-\nabla \cdot a_\xi^* \nabla w_{2,2}=\partial_k\big(D^2a(\xi+\nabla\phi)(e+\nabla\phiL)e_k\cdot \nabla w_1 \big)$
to the effect that $w_2=w_{2,1}+w_{2,2}$.
By using \eqref{sensiNLequation1+}, we obtain
$$
 \int  (\partial_i u e_j-\partial_j u e_i)  \cdot D^2 a(\xi+\nabla \phi) \nabla \delta \phi (e+\nabla  \phiL)\,=\, \int \nabla w_{2,1} \cdot \delta A \aa(\xi+\nabla \phi).
$$
This yields part of the right-hand side contribution of \eqref{functioderiv-sigL} involving $\nabla w_2$.
We conclude with the third term.
Using first \eqref{e.sens-ant-w1} we obtain 
$$
\int  (\partial_i u e_j-\partial_j u e_i)  \cdot a_\xi \nabla \delta \phiL = -\int  \nabla w_1\cdot a_\xi \nabla \delta \phiL,
$$
and therefore using \eqref{sensiLequation1+}
$$
\int  (\partial_i u e_j-\partial_j u e_i)  \cdot a_\xi \nabla \delta \phiL = \int  \nabla w_1\cdot 
\Big(\delta A D\aa(\xi+\nabla \phi )(e+\nabla\phiL )+ D^2a(\xi+\nabla \phi) \nabla\delta\phi(e+\nabla\phiL)\Big).
$$
The first right-hand side term yields the right-hand side contribution of \eqref{functioderiv-sigL} involving $\nabla w_1$. For the second term, we use $w_{2,2}$, and conclude using \eqref{sensiNLequation1+}
that
$$
 \int  \nabla w_1\cdot D^2a(\xi+\nabla \phi) \nabla\delta\phi(e+\nabla\phiL)\,=\,
  \int  \nabla w_{2,2} \cdot \delta A \aa(\xi+\nabla \phi).
$$
This gives the second part of the right-hand side contribution of \eqref{functioderiv-sigL} involving $\nabla w_2$, recalling that $\nabla w_2=\nabla w_{2,1}+\nabla w_{2,2}$.

\medskip

\step3 Proof of \eqref{e.prop:average-per-L}.

\noindent From the logarithmic-Sobolev inequality, and Steps~1 and~2, we deduce
by the triangle inequality that for all $q\ge 1$,
\begin{multline*}
\expecL{ |\calF|^{2q}}^\frac1q \,\lesssim \, \underbrace{q \expecL{\Big(\int_{Q_L} \Big(\int_{B(x)}| D\aa(\xi+\nabla\phi)||e+\nabla\phiL|(|\nabla u_1|+|\nabla v|+|\nabla w_1|)\Big)^2\dd x\Big)^q}^\frac1q}_{\displaystyle =:I_1}
\\
+\underbrace{q \expecL{\Big(\int_{Q_L} \Big(\int_{B(x)}|\aa(\xi+\nabla\phi)|(|\nabla u_2|+|\nabla w_2|)\Big)^2\dd x\Big)^q}^\frac1q}_{\displaystyle =:I_2}.
\end{multline*}
To control these terms we proceed as in the proof of Corollary~\ref{cor:average-per-NL}: using duality in probability and Theorems~\ref{th:annealedmeyers} and~\ref{th:annealed-lap}.
We treat the two right-hand sides separately.
(In what follows, $\gamma$ denotes finite positive exponents independent of $q$, the precise value of which we are not interested in.) 

\medskip

\substep{3.1} Proof of 
\begin{equation}\label{e.sens-ant-3.1}
I_1 \,\lesssim \, q^\gamma \expecL{\Big(\int_{B(x)} |e+\nabla\phiL|^2 \mu_\xi\Big)^{q(1+\theta)}}^\frac{1}{q(1+\theta)} \int_{Q_L} |g|^2.
\end{equation}
The most technical term to treat is the one involving $w_1$ (which is defined by solving two equations successively, whereas $u_1$ and $v$ are defined by solving one equation only). 
By Cauchy-Schwarz' inequality, and the definitions of $\aa$ and $\mu_\xi$,
\begin{multline*}
\expecL{\Big(\int_{Q_L} \Big(\int_{B(x)}| D\aa(\xi+\nabla\phi)||e+\nabla\phiL||\nabla w_1|\Big)^2\dd x\Big)^q}^\frac1q\\
\lesssim \, 
\expecL{\Big(\int_{Q_L} \Big(\int_{B(x)} |e+\nabla\phiL|^2 \mu_\xi\Big)\Big(\int_{B(x)} |\nabla w_1|^2 \mu_\xi\Big)\dd x\Big)^q}^\frac1q.
\end{multline*}
By duality (in probability), this entails
\begin{eqnarray*}
\lefteqn{\expecL{\Big(\int_{Q_L} \Big(\int_{B(x)}| D\aa(\xi+\nabla\phi)||e+\nabla\phiL||\nabla w_1|\Big)^2\dd x\Big)^q}^\frac1q}\\
&\lesssim & \sup_{X} \expecL{\int_{Q_L} \Big(\int_{B(x)} |e+\nabla\phiL|^2 \mu_\xi\Big)\Big(\int_{B(x)} |\nabla X w_1|^2 \mu_\xi\Big)\dd x},
\end{eqnarray*}
where the supremum runs over random variables $X$ (independent of the space variable) such that $\expec{|X|^{2q'}}=1$.
To obtain the claimed dependence on the moments of $\int_{B(x)} |e+\nabla\phiL|^2 \mu_\xi$, we set $\eta_\circ:=\frac{\theta}{(1+\theta)(q-1)}$, to the effect that $q'>1+\eta_\circ$ and 
$\frac{q'}{q'-(1+\eta_\circ)}=q(1+\theta)$, and use H\"older's inequality with exponents
$(\frac{q'}{q'-(1+\eta_\circ)},\frac{q'}{1+\eta_\circ})$, so that the above turns into
\begin{eqnarray*}
\lefteqn{\expecL{\Big(\int_{Q_L} \Big(\int_{B(x)}| D\aa(\xi+\nabla\phi)||e+\nabla\phiL||\nabla w_1|\Big)^2\dd x\Big)^q}^\frac1q}\\
&\lesssim & \expecL{\Big(\int_{B(x)} |e+\nabla\phiL|^2 \mu_\xi\Big)^{q(1+\theta)}}^\frac{1}{q(1+\theta)} \sup_{X} \int_{Q_L}\expecL{\Big(\int_{B(x)} |\nabla X w_1|^2 \mu_\xi\Big)^\frac{q'}{1+\eta_\circ}}^\frac{1+\eta_\circ}{q'}\dd x.
\end{eqnarray*}
For convenience, we rewrite $1+\eta_\circ$ as $(1+\eta)^2$,
and apply Theorem~\ref{th:annealedmeyers} to \eqref{e.sens-ant-w1}, which yields provided $
2q' \le 2+ \kappa$, 
$$
\int_{Q_L}\expecL{\Big(\int_{B(x)} |\nabla X w_1|^2 \mu_\xi\Big)^\frac{q'}{(1+\eta)^2}}^\frac{(1+\eta)^2}{q'}\dd x\,\lesssim \, \zeta(\eta_\circ) \int_{Q_L}\expecL{\Big(\int_{B(x)}   |\nabla X v|^2 \mu_\xi\Big)^\frac{q'}{1+\eta}}^\frac{1+\eta}{q'}\dd x, 
$$
where $\zeta:t \mapsto t^{-\frac14}|\log t|^\frac12$ (since for $0<\eta_\circ<\frac12$, $\zeta(\eta)=\zeta(\sqrt{1+\eta_\circ}-1) \lesssim \zeta(\eta_\circ)$).
By the bound $\mu_\xi\lesssim \r^{(d-\delta)\frac{p-2}{p}}$ and H\"older's inequality with exponents $(\frac{1+\eta}{\eta},1+\eta)$, followed by Theorem~\ref{th:annealed-lap} applied to \eqref{e.sens-ant-v} (with exponent $q' \lesssim 1$)
we further have  
\begin{eqnarray*}
{\int_{Q_L}\expecL{\Big(\int_{B(x)}   |\nabla X v|^2 \mu_\xi\Big)^\frac{q'}{1+\eta}}^\frac{1+\eta}{q'}\dd x}
&\lesssim & \expecL{ \r^{\frac{q'}{\eta}(d-\delta)\frac{p-2}{p}}}^\frac{\eta}{q'}
\int_{Q_L}\expecL{\Big(\int_{B(x)} |\nabla X v|^2  \Big)^{q'}}^\frac{1}{q'}\dd x
\\
&\lesssim&  \expecL{ \r^{ \frac{q'}{\eta} (d-\delta)\frac{p-2}{p}}}^\frac{\eta}{q'}
\expecL{|X|^{2q'}}^\frac{1}{q'}\int_{Q_L} |g|^2
\\
&=&   \expecL{ \r^{ \frac{q'}{\eta} (d-\delta)\frac{p-2}{p}}}^\frac{\eta}{q'}
 \int_{Q_L} |g|^2,
\end{eqnarray*}
where we used that $g$ is deterministic and $\expec{|X|^{2q'}}=1$.
We have thus proved  that 
\begin{multline*}
\expecL{\Big(\int_{Q_L} \Big(\int_{B(x)}| D\aa(\xi+\nabla\phi)||e+\nabla\phiL||\nabla w_1|\Big)^2\dd x\Big)^q}^\frac1q\\
\lesssim \, 
 \expecL{\Big(\int_{B(x)} |e+\nabla\phiL|^2 \mu_\xi\Big)^{q(1+\theta)}}^\frac{1}{q(1+\theta)}
 \zeta(\eta_\circ)  \expecL{ \r^{\frac{q'}{\sqrt{1+\eta_\circ}-1}\frac{p-2}{p}(d-\delta)}}^\frac{\sqrt{1+\eta_\circ}-1}{q'}
 \int_{Q_L} |g|^2.
\end{multline*}
Since $\eta_\circ=\frac{\theta}{(1+\theta)(q-1)}$, by definition of $\zeta$ and by the moment bounds on $\r$ of Theorem~\ref{boundrNLprop}, 
$$
 q\zeta(\eta_\circ)  \expecL{ \r^{ \frac{q'}{\sqrt{1+\eta_\circ}-1}\frac{p-2}{p}(d-\delta)}}^\frac{\sqrt{1+\eta_\circ}-1}{q'} \,\lesssim \, q^\gamma
$$
for some exponent $\gamma>0$ independent of $q$.
This entails the claimed estimate \eqref{e.sens-ant-3.1}.

\medskip

\substep{3.2} Proof of 
\begin{equation}\label{e.sens-ant-3.2}
I_2\,
\lesssim\, q^\gamma  \expecL{\sup_{B}\{ |e+\nabla\phiL|^2 \mu_\xi\}^{q(1+\theta)}}^\frac{1}{q(1+\theta)} 
\int_{Q_L}|g|^2.
\end{equation}
We only display the argument for the term involving $\nabla w_2$, which is defined by solving three equations successively (which will compel us to appeal to Theorem~\ref{th:annealedmeyers} twice in a row, and then to Theorem~\ref{th:annealed-lap}).
By Cauchy-Schwarz' inequality, and the definition of $\aa$ and $\mu_\xi$,
$$
\expecL{\Big(\int_{Q_L} \Big(\int_{B(x)}|\aa(\xi+\nabla\phi)||\nabla w_2|\Big)^2\dd x\Big)^q}^\frac1q
\,\lesssim \, \expecL{\Big(\int_{Q_L} \Big(\int_{B(x)} \mu_\xi\Big)  \Big( \int_{B(x)}\mu_\xi |\nabla w_2|^2\Big)\dd x\Big)^q}^\frac1q.
$$
By duality and the bound $\mu_\xi\lesssim \r^{(d-\delta)\frac{p-2}{p}}$, we have
\begin{equation*}
{\expecL{\Big(\int_{Q_L} \Big(\int_{B(x)}|\aa(\xi+\nabla\phi)||\nabla w_2|\Big)^2\dd x\Big)^q}^\frac1q}
\,\lesssim\,\sup_X\expecL{\int_{Q_L}  \r^{(d-\delta)\frac{p-2}{p}} \Big( \int_{B(x)}\mu_\xi |\nabla Xw_2|^2\Big)  \dd x},
\end{equation*}
where the supremum runs over random variables $X$ (thus independent of the space variable) such that $\expec{|X|^{2q'}}=1$. 
We now introduce exponents: $\eta_2:=\frac{1}{q-1} \frac{\theta}{8(1+\theta)}$
and $\eta_1:=\frac{1}{(q-1)(1+\eta_2)^2(1+\theta)}$ which are chosen so that
$\frac{q'}{(1+\eta_2)^2\eta_1}=q(1+\theta)$ and $\frac{q'}{(1+\eta_2)^3(1+\eta_1)}>1$.
Let us quickly check the second property:
$$
(1+\eta_2)^3(1+\eta_1)=(1+\eta_2)^3+\frac{1+\eta_2}{(q-1)(1+\theta)}
\le 1+(7+\frac1{(q-1)(1+\theta)})\eta_2+\frac1{(q-1)(1+\theta)} <1+\frac{1}{q-1}=q'.
$$
With these exponents at hands, we first use H\"older's inequality with exponents 
$(\frac{q'}{q'-(1+\eta_2)^3(1+\eta_1)},\frac{q'}{(1+\eta_2)^3(1+\eta_1)})$ together with the stationarity of $\r$, and obtain
\begin{multline*}
\expecL{\int_{Q_L}  \r^{(d-\delta) \frac{p-2)}{p}} \Big( \int_{B(x)}\mu_\xi |\nabla Xw_2|^2\Big)  \dd x}
\\
\lesssim\, \expecL{\r^{\frac{q'}{q'-(1+\eta_2)^3(1+\eta_1)}{(d-\delta)\frac{p-2}{p}}}}^\frac{q'-(1+\eta_2)^3(1+\eta_1)}{q'}
 \int_{Q_L}   \expec{\Big( \fint_{B(x)}\mu_\xi |\nabla Xw_2|^2\Big)^{\frac{q'}{(1+\eta_2)^3(1+\eta_1)}}}^\frac{(1+\eta_2)^3(1+\eta_1)}{q'} \dd x.
\end{multline*}
Provided $2q'\le 2+\kappa$, Theorem~\ref{th:annealedmeyers}  applied to~\eqref{e.sens-ant-w2} yields  
\begin{multline*}
 \int_{Q_L}   \expec{\Big( \fint_{B(x)}\mu_\xi |\nabla Xw_2|^2\Big)^{\frac{q'}{(1+\eta_2)^3(1+\eta_1)}}}^\frac{(1+\eta_2)^3(1+\eta_1)}{q'} \dd x
\\
\lesssim \, \zeta(\eta_2) \int_{Q_L}   \expecL{\Big( \fint_{B(x)}\mu_\xi^{-1}|D^2a(\xi+\nabla\phi)|^2|e+\nabla\phiL|^2(|\nabla Xw_1|^2+|\nabla Xv|^2)\Big)^{\frac{q'}{(1+\eta_2)^2(1+\eta_1)}}}^\frac{(1+\eta_2)^2(1+\eta_1)}{q'} \dd x.
\end{multline*}
Since $\mu_\xi\ge 1$ and $|D^2a(\xi+\nabla\phi)|\le \mu_\xi$, this yields
\begin{multline*}
\int_{Q_L}   \expec{\Big( \fint_{B(x)}\mu_\xi |\nabla Xw_2|^2\Big)^{\frac{q'}{(1+\eta_2)^3(1+\eta_1)}}}^\frac{(1+\eta_2)^3(1+\eta_1)}{q'} \dd x
\\
\lesssim \, \zeta(\eta_2) \int_{Q_L}   \expecL{\sup_{B(x)}\{ |e+\nabla\phiL|^2 \mu_\xi\}^{\frac{q'}{(1+\eta_2)^2(1+\eta_1)}}\Big( \fint_{B(x)}(|\nabla Xw_1|^2+|\nabla Xv|^2)\Big)^{\frac{q'}{(1+\eta_2)^2(1+\eta_1)}}}^\frac{(1+\eta_2)^2(1+\eta_1)}{q'} \dd x.
\end{multline*}
We only treat the term involving $w_1$, which is the most subtle of the two.
We then apply H\"older's inequality with exponents $(\frac{1+\eta_1}{\eta_1},1+\eta_1)$, and use the stationarity of $x\mapsto \sup_{B(x)}\{ |e+\nabla\phiL|^2 \mu_\xi\}$ and the definition of 
$\eta_1$ and $\eta_2$ to the effect that 
\begin{multline*}
 \int_{Q_L} \expecL{\sup_{B(x)}\{ |e+\nabla\phiL|^2 \mu_\xi\}^{\frac{q'}{(1+\eta_2)^2(1+\eta_1)}}\Big( \fint_{B(x)}(|\nabla Xw_1|^2+|\nabla Xv|^2)\Big)^{\frac{q'}{(1+\eta_2)^2(1+\eta_1)}}}^\frac{(1+\eta_2)^2(1+\eta_1)}{q'}\dd x
\\
\le \, \expecL{\sup_{B}\{ |e+\nabla\phiL|^2 \mu_\xi\}^{q(1+\theta)}}^\frac{1}{q(1+\theta)} \int_{Q_L}   \expecL{ \Big( \fint_{B(x)}|\nabla Xw_1|^2\Big)^{\frac{q'}{(1+\eta_2)^2}}}^\frac{(1+\eta_2)^2}{q'} \dd x.
\end{multline*}
In view of equation~\eqref{e.sens-ant-w1}, one may appeal to Theorem~\ref{th:annealedmeyers}, and obtain  
$$
\int_{Q_L}   \expecL{ \Big( \fint_{B(x)}|\nabla Xw_1|^2\Big)^{\frac{q'}{(1+\eta_2)^2}}}^\frac{(1+\eta_2)^2}{q'} \dd x\,\lesssim \, \zeta(\eta_2) \int_{Q_L}   \expecL{ \Big( \fint_{B(x)}\mu_\xi |\nabla X v|^2\Big)^{\frac{q'}{1+\eta_2}}}^\frac{1+\eta_2}{q'} \dd x.
$$
We finally bound $\mu_\xi$ using $\r$, use H\"older's inequality with exponents $(\frac{1+\eta_2}{\eta_2},1+\eta_2)$ and we apply Theorem~\ref{th:annealed-lap} to equation~\eqref{e.sens-ant-v}
\begin{eqnarray*}
\lefteqn{ \int_{Q_L}   \expecL{ \Big( \fint_{B(x)}\mu_\xi |\nabla X v|^2\Big)^{\frac{q'}{1+\eta_2}}}^\frac{1+\eta_2}{q'} \dd x}
\\
&\le&\expecL{\r^{\frac{q'}{\eta_2}{(d-\delta)\frac{p-2}{p}}}}^\frac{\eta}{q'}
\int_{Q_L} \expecL{ \Big( \fint_{B(x)}  |\nabla X v|^2\Big)^{q'}}^\frac1{q'} \dd x
\\
&\lesssim &\expecL{\r^{\frac{q'}{\eta}{(d-\delta)\frac{p-2}{p}}}}^\frac{\eta_2}{q'}
\expecL{|X'|^{2q'}} \int_{Q_L} |g|^2 = \expecL{\r^{\frac{q'}{\eta_2}{(d-\delta)\frac{p-2}{p}}}}^\frac{\eta_2}{q'}\int_{Q_L} |g|^2.
\end{eqnarray*}
As in Substep~3.1, the above estimates combine to \eqref{e.sens-ant-3.2}
using Theorem~\ref{boundrNLprop} and our choice of $\eta_2$.

\subsection{Control of level sets: Proof of Proposition~\ref{prop:average-per-L+}}

As mentioned above, we do not buckle on moments of $\nabla \tilde \phi_{\xi,e}$ but rather
on a minimal scale that controls the growth of $R\mapsto \fint_{B_R}|\nabla \tilde \phi_{\xi,e}|^2 \mu_\xi$ by the growth of $\fint_{B_{2R}} \mu_\xi$.
\begin{definition}[Linear minimal scale]\label{defminimalscaleL}
Let $\xi\in\mathbb{R}^d$, $L\geq 1$, $\vert e\vert=1$ and $C>0$. 
For all $x\in Q_L$, we define the linear minimal scale $\rL(x,C)$ via
\begin{equation}
\rL(x,C): = \inf_{r=2^N, N\in\mathbb{N}}\left\{\forall R\geq r\,:\,\fint_{B_R}|\nabla \tilde \phi_{\xi,e}|^2 \mu_\xi \le C \fint_{B_{2R}} \mu_\xi \right\}.
\label{defr*L}
\end{equation}
\end{definition}
As for the Meyers minimal radius, $\rL(\cdot,C)$ is bounded by $L$ as soon as $C$ is large enough, due to periodicity and to the plain energy estimates  for  $\tilde \phi_{\xi,e}$  in form of 
$
\int_{Q_L} |\nabla \tilde \phi_{\xi,e}|^2 \mu_\xi  \,\lesssim \, \int_{Q_L} \mu_\xi.
$
In what follows we fix such a constant $C$, fix $e$ and $\xi$, and use the short-hand notation $\rr$ for $\rL(\cdot,C)$, $\r$ for $\rNL$, $\phi$ for $\phi_\xi$, and $\tilde \phi$ for $\tilde \phi_{\xi,e}$.
The upcoming lemma uses local regularity and hole-filling to control $\sup_B |\nabla\tilde \phi +e|^2 \mu_\xi$ by $\rr$ and $\r$.
\begin{lemma}[Quenched bounds on the linearized correctors]\label{smallscalereg}
For all $K\ge 1$ there exist two exponents  $0<\beta\le d$ (the linear hole-filling exponent of Lemma~\ref{ctrlavNL}) and $\gamma>0$, and a non-negative stationary random field $\chi$ (depending on $\r$, $\|A\|_{C^{0,\alpha}(\mathbb{R}^d)}$ and $K$)  with the following properties: For all $\xi \in \R^d$ with $|\xi|\le K$ and all $x\in\mathbb{R}^d$ 
\begin{equation}
\sup_{B(x)}  |e+\nabla\tilde \phi|^2 \mu_\xi \le \chi(x) (\rr(x))^{d-\beta},
\label{deterboundL}
\end{equation}
and all $q \ge 1$
\begin{equation}
\mathbb E_L[\chi^q]^{\frac{1}{q}}\, \lesssim_{K}\, q^{\gamma}.
\label{momentchireg}
\end{equation}
\end{lemma}
\begin{proof}
We split the proof into two steps. In the first step, we control the $C^{\alpha}$-norm of $a_\xi$ that we use in the second step to control the linearized corrector via classical Schauder theory for elliptic systems. W.l.o.g we may assume that $x=0$.\newline
\newline
\step1 Proof that
\begin{equation}
\|\aL\|_{C^{\alpha}(B)}\leq C \r^{(d-\delta)\frac{p-2}{p}},
\label{deterboundNL}
\end{equation}
for some constant $C>0$ depending on $d$, $p$, $\|A\|_{C^{0,\alpha}(\mathbb{R}^d)}$, and $|\xi|$, where $0<\delta\le d$ is the nonlinear hole-filling exponent of Lemma~\ref{ctrlavNL}. (We recall that $\|X\|_{C^\alpha(B)}=\|X\|_{L^\infty(B)}+\|X\|_{C^{0,\alpha}(B)}$.)

On the one hand, by Lemma~\ref{regestiNL} applied to the equation \eqref{e.cor-eq} combined with the estimate \eqref{controlunitball}, we have
\begin{equation}
\|\xi+\nabla\phi\|_{C^{\alpha}(B)}\,\lesssim_{\|A\|_{C^{0,\alpha}(\mathbb{R}^d)}} \,\Big(\fint_{B_{2}}\vert \xi+\nabla\phi \vert^p \Big)^{\frac{1}{p}}\,\stackrel{\eqref{controlunitball}}{\leq}_{\|A\|_{C^{0,\alpha}(\mathbb{R}^d)}}(1+|\xi|) \r^{\frac{d-{\delta}}{p}}.
\label{deterboundproof1+}
\end{equation}
On the other hand, recall that  $\aL=AD\aa(\xi+\nabla\phi)$ with $\aa : \zeta\in\mathbb{R}^d\mapsto (1+\vert\zeta\vert^{p-2})\zeta$,  and thus for all $\zeta\in\mathbb{R}^d$
\begin{equation}
\vert D\aa(\zeta)\vert\lesssim 1+\vert\zeta\vert^{p-2} \text{ and } \vert D^2\aa(\zeta)\vert\lesssim 1+\vert\zeta\vert^{p-3}.
\label{deterboundproof8}
\end{equation}
Therefore, by \eqref{deterboundproof1+} and \eqref{deterboundproof8},
\begin{align}
\|D\aa(\xi+\nabla\phi)\|_{C^{\alpha}(B)}&\leq \|D\aa(\xi+\nabla\phi)\|_{L^\infty(B)}+\|D^2\aa(\xi+\nabla\phi)\|_{L^{\infty}(B)}\|\xi+\nabla\phi\|_{\text{C}^{0,\alpha}(B)}\nonumber\\
&\stackrel{\eqref{deterboundproof8}}{\lesssim}
1+ \|\xi+\nabla\phi\|^{p-2}_{C^\alpha(B )} \,\stackrel{\eqref{deterboundproof1}}{\lesssim}_{\|A\|_{C^{0,\alpha}(\mathbb{R}^d)} } \, 
(1+|\xi|)^{p-2} \r^{(d-{\delta})\frac{p-2}{p}},
\label{deterboundproof3}
\end{align}
from which the claim  \eqref{deterboundNL}  follows since  
$\|\aL\|_{C^{\alpha}(B)}\,\leq\, \|A\|_{C^\alpha(B)}\|D\aa(\xi+\nabla\phi)\|_{C^{\alpha}(B)}.
$%

\medskip

\step2 Proof of \eqref{deterboundL}. 

\noindent We first argue that 
\begin{equation}
\int_{B}\vert e+\nabla\phiL \vert^2\mu_\xi \,\lesssim\, \rr^{d-\beta} \r^{\frac{p-2}{p}(d-\delta)+\beta}.\label{deterboundproof5}
\end{equation}
If $\rr < \r$, the claim follows from  the defining property~\eqref{defr*L} in form of 
$$
\int_{B}\vert e+\nabla\phiL \vert^2\mu_\xi \,\lesssim \, 2\rr^{d}\fint_{B_{\rr}}(1+|\nabla\phiL|^2)\mu_\xi \,  {\le} \,2(C+1) \rr^{d} \fint_{B_{2\rr}} \mu_\xi \, \lesssim \, \rr^d \r^{(d-\delta)\frac{p-2}{p}} \,\lesssim \,\rr^{d-\beta}  \r^{(d-\delta)\frac{p-2}{p}+\beta}.
$$
If  $\rr \geq  \r$, we appeal to the hole filling estimate \eqref{Lholefillingesti}, to  the defining property~\eqref{defr*L}, and use \eqref{defr*NL2} \&~\eqref{encadrementrNL}, to the effect that
\begin{equation*}
\int_{B}\vert e+\nabla\phiL \vert^2\mu_\xi \,\lesssim \,\r^{d}\fint_{B_{\r}}\vert e+\nabla\phiL\vert^2\mu_\xi 
\,\stackrel{\eqref{Lholefillingesti}}{\lesssim}\, \rr^{d-\beta}\r^{\beta}\fint_{B_{\rr}}\vert e+\nabla\phiL\vert^2\mu  \, \lesssim \, \rr^{d-\beta}\r^{\beta } \fint_{B_{\rr}} \mu_\xi 
\,\lesssim\,
\rr^{d-\beta} \r^{\beta }.
\end{equation*}
We now argue that \eqref{deterboundproof5} entails~\eqref{deterboundL}. By the Schauder estimate \cite[Theorem~5.19]{giaquinta2013introduction} applied to~\eqref{e.Lcorr} (for which the constant depends algebraically on the ellipticity ratio and the $C^{0,\alpha}$-seminorm of the coefficients, which we may encapsulate in the $C^{\alpha}$-norm since $\mu_\xi \ge 1$), and the bound~\eqref{deterboundNL} on the coefficient and  \eqref{deterboundproof5}, there is some $\gamma>0$ (depending on $\alpha$ and $d$) such that
\begin{equation*}
\|e+\nabla\phiL\|_{L^{\infty}(B)}\,\lesssim\,\|\aL\|^{\gamma}_{C^{\alpha}(B)}\Big(\fint_{B_2}\vert e+\nabla\phiL\vert^2 \Big)^{\frac{1}{2}}\,\stackrel{\eqref{deterboundNL},\eqref{deterboundproof5}}{\lesssim} \, \r^{\gamma(d-{\delta})\frac{p-2}{p}}\rr^{\frac{1}2 (d-\beta)}\r^{\frac12((d-\delta)\frac{p-2}{p} +\beta)},
\end{equation*}
which yields~\eqref{deterboundL} for $\chi:=C \r^{2 (\gamma+1)(d-\delta)\frac{p-2}p+\beta}$ (for some constant $C >0$ depending on $d$, $p$, $\vert\xi\vert$ and $\|A\|_{C^{0,\alpha}(\mathbb{R}^d)}$). The claimed moment
bounds on $\chi$  follow from Theorem~\ref{boundrNLprop} (for a suitable $\gamma>0$).
\end{proof}
The main result of this section is the following control of $\rr$, 
which implies Proposition~\ref{prop:average-per-L+} in combination with Lemma~\ref{smallscalereg}.
\begin{proposition}\label{boundrLprop} 
For all $K\ge 1$ there exists an exponent $\gamma>0$ such that for all $\xi \in \R^d$ with $|\xi|\le K$ and all $q\ge 1$, 
$
\mathbb E_L[\rr^q]^{\frac{1}{q}}\, \lesssim_K\, q^{\gamma}
$
(where $\rr$ is associated with $\xi$).
\end{proposition}
\begin{proof}
We split the proof into three steps. In the first step, we control the level set $\{\rr=R\}$ for all dyadic $R$ using averages of the corrector gradient. 
In Step~2, we use Proposition~\ref{prop:average-per-L} to reformulate the right-hand side using moments of $\rr$ itself, and then buckle in Step~3 by exploiting the gain of integrability provided by the hole-filling exponent $\beta>0$.

\medskip

\step1 Control of level sets of $\rr$.

\noindent We claim that there exists a constant $c>0$ (depending on $\xi$, $d$, $p$)
such that for all dyadic $R \in [1,L]$, and all  $0<\kappa,\e<1$ and $q\ge 1$
\begin{equation}
\mathbb P_L[\rr=R]\, \le \, c^qR^{-(d-\beta+ 2(1-\kappa)-\e)q}\expecL{ \calC_{\star,R}^q \rr^{(d-\beta)q}}+c^qR^{\e q}\expecL{ \calC_{\star,R}^q\Big(\fint_{B_{R}}\big\vert\fint_{B_{R^{\kappa}}(x)}\nabla\phiL \big\vert^2\dd x\Big)^q} ,
\label{estimomentboundrL}
\end{equation}
where $\calC_{\star,R}:= R^{-\e}\|\mu_\xi\|_{L^{\infty}(B_{4R})}^2$.
By the defining property \eqref{defr*L} of $\rr$ (with a constant $C$ to be chosen below), we have
\begin{eqnarray}
\fint_{B_{2R}}\vert\nabla\phiL\vert^2\mu_\xi &\le & C \fint_{B_{4R}} \mu_\xi,\label{e.rL-ant1}
\\
\fint_{B_{R/2}}\vert\nabla\phiL\vert^2\mu_\xi &\ge & C \fint_{B_{R}} \mu_\xi.\label{e.rL-ant2}
\end{eqnarray}
By the Caccioppoli inequality of Lemma~\ref{cacciopounbounded}, \eqref{e.rL-ant2} yields
$$
\inf_{c\in\mathbb{R}^d}\frac{1}{R^2}\fint_{B_{R}}\vert \phiL-c\vert^2\mu_\xi+ \fint_{B_R} \mu_\xi \,\gtrsim\, C \fint_{B_{R}} \mu_\xi,
$$
so that, provided $C$ is chosen large enough in \eqref{defr*L}, we have
$$
\inf_{c\in\mathbb{R}^d}\frac{1}{R^2}\fint_{B_{R}}\vert \phiL-c\vert^2\mu_\xi \,\gtrsim\, \fint_{B_R} \mu_\xi \,\gtrsim\, 1.
$$
Set $c_R:=\fint_{B_{R}}\fint_{B_{R^{\kappa}}(x)}\phiL(y)\dd y\, \dd x$.
By the triangle inequality,  Poincar\'e's inequality in $L^2(B_R)$, and the definition of $\calC_{\star,R}$, the above turns into
\begin{eqnarray*}
1\, \lesssim \, \inf_{c\in\mathbb{R}^d}\frac{1}{R^2}\fint_{B_{R}}\vert \phiL-c\vert^2\mu_\xi 
&\leq &\sqrt{\calC_{\star,R}}R^{\frac\e2} \frac{1}{R^2}\fint_{B_{R}}\vert\phiL -c_R\vert^2 \\
&\lesssim& \sqrt{\calC_{\star,R}}R^{\frac\e2}\Big(\frac{1}{R^2}\fint_{B_{R}}\big\vert\phiL(x)-\fint_{B_{R^{\kappa}}(x)}\phiL\big \vert^2\dd x+\frac{1}{R^2}\fint_{B_{R}}\big\vert\fint_{B_{R^{\kappa}}(x)}\phiL-c_R\big\vert^2\dd x\Big)\\
&\lesssim & \sqrt{\calC_{\star,R}}R^{\frac\e2}\Big(R^{2(\kappa-1)}\fint_{B_{2R}}\vert\nabla\phiL\vert^2+\fint_{B_{R}}\big\vert\fint_{B_{R^{\kappa}}(x)}\nabla\phiL\big\vert^2\dd x\Big)\\
&\stackrel{\eqref{e.rL-ant1}}{\lesssim}& \sqrt{\calC_{\star,R}}R^{\frac\e2}\left(R^{2(\kappa-1)} \sqrt{\calC_{\star,R}}R^{\frac\e2}+\fint_{B_{R}}\big\vert\fint_{B_{R^{\kappa}}(x)}\nabla\phiL\big\vert^2\dd x\right)
\\
&\stackrel{\rr=R}=&   {\calC_{\star,R}}R^{\e}\left(R^{2(\kappa-1)}R^{-d+\beta} \rr^{d-\beta}+\fint_{B_{R}}\big\vert\fint_{B_{R^{\kappa}}(x)}\nabla\phiL\big\vert^2\dd x\right).
\end{eqnarray*}
The claim now follows from Markov' inequality.

\medskip

\step2 Control of the right-hand side of \eqref{estimomentboundrL}: For all $0<\e,\kappa,\theta < 1$, and all dyadic $R$ and exponents $q\ge 1$
\begin{equation}
\mathbb P_L[\rr=R]\, \le \, c^qq^\gamma (R^{-(d-\beta+ 2(1-\kappa)-\e)q}+R^{-(d\kappa-\e)q}) \expecL{ \rr^{(d-\beta) (1+\theta)^3 q}}^\frac1{(1+\theta)^3} ,
\label{estimomentboundrL+}
\end{equation}
for some constant $c>0$ depending on $|\xi|$, $p$, $d$, $\e$, $\kappa$, $\theta$, but not on $R$ and $q$.

Since $\rr \le L$, it suffices to establish the statement for dyadic $R \le L$.
By Lemma~\ref{unifproba} and Theorem~\ref{boundrNLprop}, there exists $\gamma>0$ such that for all $q\ge \frac d\e$ and $R\ge 1$, we have
\begin{equation}\label{e.bdC*-ant}
\mathbb E_L[\calC_{\star,R}^q]^\frac1q \lesssim  q^\gamma,
\end{equation}
where the multiplicative constant does not depend on $R$ and $\e$.
By H\"older's inequality with exponents $(\frac{1+\theta}{\theta},1+\theta)$,
we then get for the first right hand side term of  \eqref{estimomentboundrL}
\begin{equation}\label{e.bdC*-ant2}
\expecL{ \calC_{\star,R}^q \rr^{(d-\beta)q}}^\frac1q \, \le\, \expecL{ \calC_{\star,R}^{q\frac{1+\theta}{\theta}}}^\frac{\theta}{1+\theta} \expecL{ \rr^{(d-\beta)(1+\theta)q}}^\frac1{q(1+\theta)}
\,\lesssim_\theta \, q^\gamma  \expecL{ \rr^{(d-\beta)(1+\theta)q}}^\frac1{q(1+\theta)}.
\end{equation}
We turn to the second right hand side term of  \eqref{estimomentboundrL}.
By H\"older's inequality with exponents $(\frac{1+\theta}{\theta},1+\theta)$,
stationarity of $\nabla \phiL$, and \eqref{e.bdC*-ant}, we first have 
\begin{eqnarray*}
\expecL{ \calC_{\star,R}^q\Big(\fint_{B_{R}}\big\vert\fint_{B_{R^{\kappa}}(x)}\nabla\phiL\big\vert^2\dd x\Big)^q}^\frac1q
&\le& \expecL{ \calC_{\star,R}^{q\frac{1+\theta}\theta}}^\frac{\theta}{q(1+\theta)}
\expecL{\Big(\fint_{B_{R}}\big|\fint_{B_{R^{\kappa}}(x)}\nabla\phiL\big|^2\dd x\Big)^{q(1+\theta)}}^\frac1{q(1+\theta)}
\\
&\lesssim _\theta &q^\gamma \expecL{\big|\fint_{B_{R^{\kappa}}}\nabla\phiL\big|^{2q(1+\theta)}}^\frac1{q(1+\theta)}.
\end{eqnarray*}
Then, by Proposition~\ref{prop:average-per-L} applied to $g=|B_{R^\kappa}|^{-1}\mathds 1_{B_{R^\kappa}}$, followed by Lemma~\ref{smallscalereg},  by H\"older's inequality with exponent $(\frac{1+\theta}{\theta},1+\theta)$, and \eqref{momentchireg}, we have 
\begin{eqnarray*} 
\expecL{\big|\fint_{B_{R^{\kappa}}}\nabla\phiL\big|^{2q(1+\theta)}}^\frac1{q(1+\theta)} &\lesssim_{|\xi|,\theta}& q^\gamma \expecL{\Big(\sup_B |\nabla\tilde\phi_{\xi,e}+e|^2 \mu_\xi\Big)^{q(1+\theta)^2}}^\frac1{q(1+\theta)^2} \Big(\int_{Q_L}|g|^2\Big)  
\\
&\lesssim &  q^\gamma  \expecL{\chi^{q(1+\theta)^2} \rr^{(d-\beta) q(1+\theta)^2}}^\frac1{q(1+\theta)^2} R^{-d\kappa}
\\
&\lesssim&q^\gamma  R^{-d\kappa} \expecL{ \rr^{(d-\beta) q(1+\theta)^3}}^\frac1{q(1+\theta)^3} 
\end{eqnarray*}
(where we changed the value of $\gamma$ from one line to the other).
Combined with \eqref{e.bdC*-ant2}, this entails \eqref{estimomentboundrL+} by redefining $\gamma$ once more.

\medskip

\step3 Buckling argument.

\noindent Recall that all the quantities we consider are finite since $\rr \le L$.
We now express moments of $\rr$ using its level sets and obtain by \eqref{estimomentboundrL+} for some $K>1$ to be fixed below and all $q\ge 1$
\begin{eqnarray*}
\expecL{\rr^{q(d-\frac{\beta}{K})}}&\le & 1+\sum_{n=1}^\infty 2^{nq(d-\frac{\beta}{K})} 
\mathbb P_L[\rr=2^n]
\\
&\stackrel{\eqref{estimomentboundrL+}}\le & 1+\sum_{n=1}^\infty 2^{nq(d-\frac{\beta}{K})} 
c^qq^\gamma (2^{-nq(d-\beta+ 2(1-\kappa)-\e)}+2^{-nq(d\kappa-\e)})
 \expecL{ \rr^{q(d-\beta)(1+\theta)^3}}^\frac1{(1+\theta)^3}
\\
&\le& 1+\expecL{ \rr^{q(d-\beta)(1+\theta)^3}}^\frac1{(1+\theta)^3}c^qq^\gamma   \sum_{n=1}^\infty (2^{nq(-\frac\beta K+\beta-2(1-\kappa)+\e)}+2^{nq(d(1-\kappa)+\e-\frac\beta K)}).
\end{eqnarray*}
We now choose the exponents. We first fix $0\le \kappa < 1$ so   that $d(1-\kappa)=\frac\beta 2$, and then set $\e:=\frac{\beta}{5d}$ and $\frac1K:=1-\frac1{5d}$, 
to the effect that 
$$
 \frac12(2^{nq(-\frac\beta K+\beta-2(1-\kappa)+\e)}+2^{nq(d(1-\kappa)+\e-\frac\beta K)}) \,\le \,2^{-nq \frac \beta{5d}}.
$$
With this choice, the series is summable and the above turns into 
\begin{eqnarray*}
\expec{\rr^{q(d-\frac{\beta}{K})}}&\le & 1+c^qq^\gamma\expecL{ \rr^{q(d-\beta)(1+\theta)^3}}^\frac1{(1+\theta)^3} 
\end{eqnarray*}
for some redefined constant $c$.
We may then absorb part of the right-hand side into the left-hand side by Young's inequality upon choosing $0<\theta <1$  so small that $(d-\beta)(1+\theta)^3 < d-\frac{\beta}{K}$ (which is possible since $K>1$), and the claimed moment bound follows for some suitable choice of $\gamma>0$.
\end{proof}

\section{Quantitative two-scale expansion: Proof of Theorem~\ref{th:2s}}\label{sec:2s}

We assume $\delta \le 1$, and split the proof into four steps. In the first step, we show that the two-scale expansion error satisfies a nonlinear PDE in conservative form (crucially using the flux corrector).
In the second step we give a bound for the $H^{-1}(\R^d)$-norm of the right-hand side, the moments of which we control in the third step.
We then conclude in the fourth step by using the monotonicity of the heterogeneous operator $a_\e$.
In the following, we use the short-hand notation $\xi_k:=(\nabla\bar{u})_{k,\delta}$.

\medskip

\step1 Equation for the two-scale expansion error: 
\begin{equation}
-\nabla\cdot (a(\tfrac x\e,\nabla\bar{u}^{2s}_{\e,\delta}(x))-a(\tfrac x\e,\nabla u_{\e}(x)))=\nabla\cdot R_{\e,\delta}(x),
\label{2sc:Eq1}
\end{equation}
where
\begin{eqnarray*}
R_{\e,\delta}(x)
&=&\Big(\sum_{k\in\delta\mathbb{Z}^d}\eta_k(x)(\bar{a}(\xi_k)-\bar{a}(\nabla\bar{u}(x))\Big)- \Big(\sum_{k\in\delta\mathbb{Z}^d}\e\sigma_{\xi_k}(\tfrac x\e)\nabla\eta_k(x)\Big)\nonumber\\
&&+ \Big(\sum_{k\in\delta\mathbb{Z}^d}\eta_k(x)(a(\tfrac x\e,\nabla\bar{u}(x)+\nabla\phi_{\xi_k}(\tfrac x\e))-a(\tfrac x\e,\xi_k+\nabla\phi_{\xi_k}(\tfrac x\e)))\Big)\nonumber\\
&&+ \Big(a\Big(\tfrac x\e,\nabla\bar{u}(x)+\sum_{k\in\delta\mathbb{Z}^d}\nabla\phi_{\xi_k}(\tfrac x\e)\eta_k(x)\Big)-\sum_{k\in\delta\mathbb{Z}^d}\eta_k(x) a(\tfrac x\e,\nabla\bar{u}(x)+\nabla\phi_{\xi_k}(\tfrac x\e))\Big)\nonumber\\
&&+ \Big(a\Big(\tfrac x\e,\nabla\bar{u}(x)+\sum_{k\in\delta\mathbb{Z}^d}\nabla\phi_{\xi_k}(\tfrac x\e)\eta_k(x)+\e \phi_{\xi_k}(\tfrac x\e)\nabla\eta_k(x)\Big)-a\Big(\tfrac x\e,\nabla\bar{u}(x)+\sum_{k\in\delta\mathbb{Z}^d}\nabla\phi_{\xi_k}(\tfrac x\e)\eta_k(x)\Big)\Big).
\end{eqnarray*} 
To start with, we expand $\nabla \us$ as
$$
\nabla\cdot a(\tfrac x\e,\nabla \us)=\nabla\cdot a\Big(\tfrac x\e,\nabla\bar{u}(x)+ \sum_{k\in \delta\mathbb{Z}^d} \e \phi_{\xi_k}(\tfrac x\e)\nabla\eta_k(x)+\nabla\phi_{\xi_k}(\tfrac x\e)\eta_k(x)\Big),
$$
which we rewrite in the form of the telescopic sum (using that $\sum_{k\in\delta\mathbb{Z}^d}\eta_k\equiv 1$)
\begin{eqnarray*}
\lefteqn{\nabla\cdot a(\tfrac x\e,\nabla \us(x))-\nabla\cdot \bar{a}(\nabla\bar{u}(x))}
\\
&=&\nabla\cdot \Big(\sum_{k\in\delta\mathbb{Z}^d}\eta_k(x)(\bar{a}(\xi_k)-\bar{a}(\nabla\bar{u}(x))\Big)+\nabla\cdot\Big(\sum_{k\in\delta\mathbb{Z}^d}\eta_k(x)(a(\tfrac x\e,\xi_k+\nabla\phi_{\xi_k}(\tfrac x\e))-\bar{a}(\xi_k))\Big)\nonumber\\
&&+\nabla\cdot \Big(\sum_{k\in\delta\mathbb{Z}^d}\eta_k(x)(a(\tfrac x\e,\nabla\bar{u}(x)+\nabla\phi_{\xi_k}(\tfrac x\e))-a(\tfrac x\e,\xi_k+\nabla\phi_{\xi_k}(\tfrac x\e)))\Big)\nonumber\\
&&+\nabla\cdot\Big(a\Big(\tfrac x\e,\nabla\bar{u}(x)+\sum_{k\in\delta\mathbb{Z}^d}\nabla\phi_{\xi_k}(\tfrac x\e)\eta_k(x)\Big)-\sum_{k\in\delta\mathbb{Z}^d}\eta_k(x) a(\tfrac x\e,\nabla\bar{u}(x)+\nabla\phi_{\xi_k}(\tfrac x\e))\Big)\nonumber\\
&&+\nabla\cdot\Big(a\Big(\tfrac x\e,\nabla\bar{u}(x)+\sum_{k\in\delta\mathbb{Z}^d}\nabla\phi_{\xi_k}(\tfrac x\e)\eta_k(x)+\e \phi_{\xi_k}(\tfrac x\e)\nabla\eta_k(x)\Big)-a\Big(\tfrac x\e,\nabla\bar{u}(x)+\sum_{k\in\delta\mathbb{Z}^d}\nabla\phi_{\xi_k}(\tfrac x\e)\eta_k(x)\Big)\Big).
\end{eqnarray*} 
First, using \eqref{e.eps-eq} and \eqref{e.hom-eq} we may replace $-\nabla\cdot \bar{a}(\nabla\bar{u}(x))$ by $-\nabla\cdot a(\tfrac x\e,\nabla u_\e(x))$ in the left-hand side.
In the right-hand side, all the terms obviously converge strongly to zero in $H^{-1}(\R^d)$ (and are present in the definition of $R_{\e,\delta}$) except the second term, which we need to reformulate.
More precisely, using the flux corrector $\sigma$ (see Definition \ref{defsigmaNL}) in form of the property \eqref{e.div-sig}, we have for all $k\in\delta\mathbb{Z}^d$ (implicitly summing on the repeated indices $ij$) 
\begin{align}
\nabla\cdot \eta_k(x)(a(\tfrac x\e,\xi_k+\nabla\phi_{\xi_k}(\tfrac x\e))-\bar{a}(\xi_k))&=\nabla\cdot(\eta_k(x)\nabla\cdot \sigma_{\xi_k}(\tfrac x\e))=\partial_j(\eta_k(x)\partial_{i}\sigma_{\xi_k,ji})\nonumber\\
&=\e\partial_i(\partial_j\eta_k(x)\sigma_{\xi_k,ji}(\tfrac x\e))-\partial_i\partial_j(\eta_k(x))\sigma_{\xi_k,ji}(\tfrac x\e),\label{2sc:Eq4}
\end{align}
where the last term vanishes thanks to the skew-symmetry of $(\sigma_{\xi_k,ji})_{j,i}$ and the symmetry of $(\partial_i\partial_j\eta_k)_{j,i}$. 
By the skew-symmetry of $\sigma_\xi$, one has $\e\partial_i(\partial_j\eta_k(x)\sigma_{\xi_k,ji}(\tfrac x\e))=-\e\nabla\cdot(\sigma_{\xi_k}(\tfrac x\e)\nabla\eta_k(x))$, and we thus deduce
$$
\nabla\cdot \sum_{k\in\delta\mathbb{Z}^d}\eta_k(x)(a(\tfrac x\e,\xi_k+\nabla\phi_{\xi_k}(\tfrac x\e))-\bar{a}(\xi_k))=-\e\nabla\cdot\Big(\sum_{k\in\delta\mathbb{Z}^d}\sigma_{\xi_k}(\tfrac x\e)\nabla\eta_k(x)\Big).
$$
This yields \eqref{2sc:Eq1}.

\medskip

\step2 Control by continuity of the operators: The remainder $R_{\e,\delta}$ satisfies
\begin{eqnarray}
\lefteqn{\int_{\R^d} |R_{\e,\delta}|^2 \,\lesssim\,\sum_{k} \int_{\R^d} \eta_k |\xi_k-\nabla\bar{u}|^2(1+|\xi_k|+|\nabla \bar u|+|\nabla\phi_{\xi_k}(\tfrac \cdot\e)|)^{2(p-2)}}
\nonumber\\
&&+ \int_{\R^d} \sum_{k} \eta_{k} \Big|\sum_{k'}\e (\phi_{\xi_k'}-\phi_{\xi_k},\sigma_{\xi_k'}-\sigma_{\xi_k})(\tfrac \cdot\e)\nabla\eta_{k'}\Big|^2 \Big(1+|\nabla \bar u|+\Big|\sum_{k''}\nabla\phi_{\xi_{k''}}(\tfrac \cdot\e)\eta_{k''}\Big|\Big)^{2(p-2)}
\nonumber\\
&&+\int_{\R^d} \sum_{k} \eta_{k}\Big|\sum_{k'}\e (\phi_{\xi_k'}-\phi_{\xi_k})(\tfrac \cdot\e)\nabla\eta_{k'}\Big|^{2(p-1)}
+\sum_{k }  \int_{\R^d} \eta_k \Big|\sum_{k' }\nabla (\phi_{\xi_k}-\phi_{\xi_{k'}})(\tfrac \cdot \e)  \eta_{k'} \Big|^2 (1+|\nabla \bar u|)^{2(p-2)}
\nonumber \\
&&+\sum_{k\in\delta\mathbb{Z}^d}  \int_{\R^d} \eta_k \Big| \sum_{k'\in\delta\mathbb{Z}^d}\nabla ( \phi_{\xi_k}-\phi_{\xi_{k'}})(\tfrac \cdot \e) \eta_{k'} \Big|^{2(p-1)}.\label{2sc:Eq2}
\end{eqnarray}
This estimate directly follows from the definition of $R_{\e,\delta}$ together with the 
 continuity of the operator in form of 
$$
|\tilde a (\xi_1)-\tilde a(\xi_2)|\lesssim C|\xi_1-\xi_2|(1+|\xi_1|+|\xi_1-\xi_2|)^{p-2}
$$
for $\tilde a=a_\e$ and $\tilde a = \bar a$, and with the observation that $\sum_{k'} \nabla \eta_{k'}=0$ so that for all maps $(\zeta_k)_k$ one has
$$
\sum_{k'}\zeta_{k'} \nabla\eta_{k'}\,=\,\sum_{k'}(\zeta_{k'}-\zeta_{k})\nabla\eta_{k'},
$$
which we applied to $\zeta_k=\e (\phi_{\xi_k},\sigma_{\xi_k}) (\tfrac \cdot \e)$.

\medskip

\step3 Control of moments of $\int_{\R^d} |R_{\e,\delta}|^2$: For all $q\ge 1$,
\begin{equation}\label{e.momentbd-remain}
\expec{\Big(\int_{\R^d} |R_{\e,\delta}|^2\Big)^\frac q2}^\frac1q \,\leq C \, q^\gamma
(\e+\delta)\mu_d(\tfrac1\e)\|\mu_d\nabla^2\bar{u}\|_{L^2(\mathbb{R}^d)},
\end{equation}
for some constant $C$ and an exponent $\gamma>0$ depending on $\|\nabla\overline{u}\|_{L^{\infty}(\mathbb{R}^d)}$.  We treat the second right-hand side term of \eqref{2sc:Eq2} (that we denote by $\tilde{R}_{\e,\delta}$) -- the other terms are easier and can be treated similarly.  Since for all $k'$, $\vert\nabla\eta_{k'}\vert\lesssim \delta^{-1}\mathds{1}_{Q_{\delta}(k')}$, we have
\begin{align*}
\sum_{k} \eta_{k} \Big|\sum_{k'}\e (\phi_{\xi_k'}-\phi_{\xi_k},\sigma_{\xi_k'}-\sigma_{\xi_k})(\tfrac \cdot\e)\nabla\eta_{k'}\Big|^2\lesssim (\tfrac{\e}{\delta})^2\sum_{k}\eta_k\sum_{k'}\mathds{1}_{Q_{\delta}(k')} | (\phi_{\xi_k'}-\phi_{\xi_k},\sigma_{\xi_k'}-\sigma_{\xi_k})(\tfrac \cdot\e)|^2.
\end{align*}
Inserting this estimate in $\tilde{R}_{\e,\delta}$, and using the assumption $\nabla \bar {u}\in L^{\infty}(\mathbb{R}^d)$, we obtain for all $q\ge 1$
by Cauchy-Schwarz' inequality followed by Minkowski's inequality in probability,  the support condition $Q_{\delta}(k)\cap Q_{\delta}(k')\neq \emptyset \Rightarrow \vert k-k'\vert<2\delta$,
and the stationarity of $\nabla \phi_{k''}$,
\begin{multline*}
\mathbb{E}\Big[\Big(\int_{\mathbb{R}^d}\vert\tilde{R}_{\e,\delta}\vert^2\Big)^{\frac{q}{2}}\Big]^{\frac{1}{q}}
\\
\,\lesssim \, \frac{\e}{\delta}\Big(\sum_k\sum_{k'\in Q_{2\delta}(k)}\int_{Q_{\delta}(k)}\mathbb{E}[|(\phi_{\xi_k'}-\phi_{\xi_k},\sigma_{\xi_k'}-\sigma_{\xi_k})(\tfrac \cdot\e)|^{2q}]^{\frac{1}{q}}\Big(1+\|\nabla\bar {u}\|^{2(p-2)}_{L^{\infty}(\mathbb{R}^d)}+ \sum_{k''\in Q_{2\delta}(k)}\mathbb{E}[\vert\nabla\phi_{\xi_{k''}}\vert^{2q(p-2)}]^{\frac{1}{q}}\Big)\Big)^{\frac{1}{2}}.
\end{multline*}
By Theorem~\ref{th:corrNL} and Corollary~\ref{coro:corr-diff}, and using that $\mu_d$ satisfies $\mu_d(t_1t_2)\lesssim \mu_d(t_1)\mu_d(t_2)$ and $\sup_{Q_{4\delta}(k)} \mu_d \lesssim \inf_{Q_{4\delta}(k)} \mu_d$, this turns into
\begin{equation}\label{2sc:Eq14}
\mathbb{E}\Big[\Big(\int_{\mathbb{R}^d}\vert\tilde{R}_{\e,\delta}\vert^2\Big)^{\frac{q}{2}}\Big]^{\frac{1}{q}}\leq Cq^{\gamma}(\tfrac{\e}{\delta})\mu_d(\tfrac1\e)\Big(\sum_k(\inf_{Q_{4\delta}(k)} \mu_d)\sum_{k'\in Q_{2\delta}(k)}\vert\xi_{k'}-\xi_k\vert^2\vert Q_{\delta}\vert\Big)^{\frac{1}{2}},
\end{equation}
for some constant $C$ and an exponent $\gamma>0$ depending on $\|\nabla\overline{u}\|_{L^{\infty}(\mathbb{R}^d)}$. It remains to reformulate the right-hand side sum.
By Poincar\'e's inequality on $Q_{4\delta }(k)$, we have
\begin{equation*}
\sum_{k'\in Q_{2\delta}(k)}\vert\xi_k-\xi_{{k'}}\vert^2\vert Q_{\delta}\vert\lesssim \delta^2\int_{Q_{4\delta}(k)}\vert\nabla^2 \bar{u}\vert^2,
\end{equation*}
so that~\eqref{e.momentbd-remain} follows from~\eqref{2sc:Eq14}.

\medskip

\step4 Conclusion by monotonicity.

\noindent We test \eqref{2sc:Eq1} with $u_\e-\us$, and deduce by monotonicity of $a_\e$ that
$$
\int_{\R^d} |\nabla (u_\e-\us)|^2+ |\nabla (u_\e-\us)|^p \, \lesssim \, \int_{\R^d} R_{\e,\delta} \cdot \nabla (u_\e-\us).
$$
By Young's inequality, we may absorb part of the right-hand side into the left-hand side, and obtain after taking the $q$-th moment of this inequality
$$
\expec{\Big(\int_{\R^d} |\nabla (u_\e-\us)|^2+ |\nabla (u_\e-\us)|^p \Big)^q}^\frac1q\, \lesssim \, \expec{\Big(\int_{\R^d} |R_{\e,\delta}|^2\Big)^q}^\frac1q.
$$
This entails the claim in combination with \eqref{e.momentbd-remain} and the choice $\delta=\e$.

\appendix 

\section{Deterministic PDE estimates and consequences}\label{append:standard-ineq}

In this appendix, we recall mostly standard inequalities for nonlinear operators $-\nabla\cdot a(\cdot,\nabla)$ and linear operators $-\nabla\cdot a\nabla$ (with unbounded coefficients) needed in the proofs of the paper. 
We assume the conditions \eqref{*cont}, \eqref{*coer+-}, and \eqref{*coer+}.
Based on these results we also prove the qualitative differentiability of correctors (with respect to $\xi$) when the equation is posed on a bounded domain, and we prove part of Theorem~\ref{th:isotropic} for statistically isotropic operators.

\subsection{Nonlinear systems: Caccioppoli, hole-filling, and Schauder}

We start with Caccioppoli's inequality for nonlinear elliptic systems.
Recall the notation $t^{2\& p}=t^2+t^p$ for all $t\ge 0$.
\begin{lemma}[Caccioppoli's inequality]\label{cacciopoNL} Let $r>0, c_2>0$, $x\in\mathbb{R}^d$ and $u\in W^{1,p}_{\loc}(\mathbb{R}^d)$ be a weak solution of 
\begin{equation}
-\nabla\cdot a(\cdot,\nabla u)=0 \text{ in $B_{c_2 r}(x)$.}
\label{equationcacciopoNL}
\end{equation}
Then for all $0<c_1<c_2$,
\begin{equation} 
\fint_{B_{c_1 r}(x)}\vert\nabla u\vert^{2\& p} \lesssim_{c_1,c_2} \inf_{c\in\mathbb{R}^d}\fint_{B_{c_2 r}(x)\backslash B_{c_1 r}(x)}\Big(\frac{\vert u-c\vert}{r}\Big)^{2\& p}.
\label{cacciopoNLesti}
\end{equation}
\end{lemma}
\begin{proof}
Without loss of generality, we may assume that $x=0$. Let $\eta\in C^{\infty}_c(\mathbb{R}^d)$ be a standard cut-off for $B_{c_1 r}$ in $B_{c_2 r}$ and set $\zeta^2=\eta^p$
to the effect that $2 \zeta \nabla \zeta=p \eta^{p-1}\nabla \eta$. By testing the equation \eqref{equationcacciopoNL} with $\zeta^2 (u-c)$ and by making use of the monotonicity \eqref{*coer+-} of $a$ and the property $a(\cdot,0)=0$
in form of $\int \zeta^2 \nabla u \cdot a(\cdot,\nabla u) \gtrsim \int \zeta^2 |\nabla u|^2(1+|\nabla u|^{p-2})$, we have
$$\int \zeta^2\vert\nabla u\vert^2(1+\vert\nabla u\vert^{p-2})\lesssim \int \vert\zeta\vert\vert\nabla\zeta\vert\vert u-c \vert \vert \nabla u\vert+
\int \vert\zeta\vert\vert\nabla\zeta\vert\vert u-c \vert \vert \nabla u\vert^{p-1}.
$$
This implies the desired estimate \eqref{cacciopoNLesti} by Young's inequality, with exponents $(2,2)$ and $(p,\frac{p}{p-1})$ for the first and second right-hand side terms, respectively,  together with the identity $2 \zeta \nabla \zeta=p \eta^{p-1}\nabla \eta$, and absorbing part of the right-hand side into the left-hand side.
\end{proof}
The Widman hole-filling estimate for nonlinear systems  follows from Lemma~\ref{cacciopoNL} by simple iteration (see e.g.~\cite[Section 4.4]{giaquinta2013introduction}).
\begin{lemma}[Hole-filling estimate]\label{HolefillingNL'}
There exists $0< \delta\leq d$ such that if $u\in W^{1,p}_{\loc}(B_R)$ is a weak solution of $-\nabla\cdot a(\cdot,\nabla u)=0$ in the ball $B_{R}$ for some $R>0$, 
then for all $0<r\le R$ we have
\begin{equation}
\fint_{B_{r}}\vert \nabla u\vert^{2\& p} \lesssim\Big(\frac{R}{r}\Big)^{d-\delta}\fint_{B_R}\vert \nabla u\vert^{2 \& p}.
\label{HolefillingNL}
\end{equation}
\end{lemma}
We finally state regularity results for nonlinear equations, which are direct consequences of \cite{Uhlenbeck} and \cite[Theorem 4]{kuusi2014nonlinear} (for the uniform bound on the gradient).
\begin{lemma}\label{regestiNL}
Let $a$ be a monotone operator which has the form \eqref{e.def-a} and assume that $A\in C^{\alpha}(B_{4R})$, for some $R>0$ and $\alpha\in (0,1)$.  Let $u\in W^{1,p}(B_{4R})$ be a distributional solution of 
$$-\nabla\cdot a(\cdot,\nabla u)=0.$$
Then, $u\in C^{1,\alpha}(B_R)$ and there exists a constant $c$ depending on $R$ and $\|A\|_{C^{\alpha}(B_{4R})}$ such that 
$$\|\nabla u\|_{C^{\alpha}(B_R )}\leq c\Big(\fint_{B_{4R}}\vert\nabla u\vert^p\Big)^{\frac{1}{p}},$$
where we recall that $\|X\|_{C^{\alpha}}=\|X\|_{L^{\infty}}+\|X\|_{C^{0,\alpha}}$.
\end{lemma}

\subsection{Linear elliptic systems: Caccioppoli and Lemma~\ref{lemmabella}}

We state and prove Caccioppoli's inequality for linear elliptic systems with unbounded coefficients, from which we deduce  Lemma~\ref{lemmabella} by optimizing the cut-off.
\begin{lemma}[Caccioppoli's inequality for linear elliptic systems with unbounded coefficients]\label{cacciopounbounded}Let $R>0$, $a : B_R\mapsto \mathbb{R}^{d\times d}$,  and  $\mu\in L^1(B_R)$ be such that there exists a constant $\kappa>0$ for which
we have for all $x\in B_R$ and all $h\in\R^d$
\begin{equation}
h \cdot a(x) h \le |h|^2 \mu(x) \le \kappa h\cdot a(x)h.
\label{mucacciopounbounded}
\end{equation}
For all functions $g$ and $u$ related (in the weak sense) in $B_R$ via
\begin{equation}
-\nabla\cdot a\nabla u=\nabla\cdot (g\sqrt{\mu}) ,
\label{equationcacciopounbounded}
\end{equation}
we have for all $0<\rho <\sigma \le R$,
\begin{equation}
\int_{B_{\rho}}\vert\nabla u\vert^2\mu\lesssim_{\kappa}\mathcal{J}(\rho,\sigma,\mu,u)+\int_{B_{\sigma}}\vert g\vert^2,
\label{esticaccioppounbounded}
\end{equation}
where
\begin{equation}
\mathcal{J}(\rho,\sigma , \mu,u,g):=\inf\Big\{\int_{B_{\sigma}}\mu\Big\vert u-\fint_{B_{\sigma}}u \Big\vert^2\vert\nabla\eta\vert^2\,\Big|\,\eta\in C^{1}_c(B_{\sigma}),\, 
0\leq \eta\leq 1,\, 
\eta\equiv 1 \text{ in $B_{\rho}$}\Big\}.
\label{defJ}
\end{equation}
\end{lemma}
\begin{proof}
Without loss of generality,  we may assume that $\fint_{B_{\sigma}}u=0$. Let $\eta\in C^1_c(B_{\sigma})$ be such that $\eta\equiv 1$ in $B_{\rho}$ and $0\leq \eta\leq 1$. 
Testing the equation \eqref{equationcacciopounbounded} with $\eta^2 u$ and using the condition \eqref{mucacciopounbounded} yield
\begin{align}
&\int \eta^2\vert\nabla u\vert^2\mu\lesssim_{\kappa} \int \eta^2\vert\nabla u\vert\sqrt{\mu}\vert g\vert+\int |\eta| \vert\nabla\eta\vert\vert u\vert\vert\nabla u\vert\mu+\int |\eta| \vert\nabla \eta\vert\vert u\vert\sqrt{\mu}\vert g\vert.
\label{proofcacciopounbounded1}
\end{align}
By Young's inequality with exponents $(2,2)$, after absorbing part of the right-hand side into the left-hand side, and using in addition the support condition on $\eta$, \eqref{proofcacciopounbounded1} turns into 
$$\int_{B_{\rho}}\vert\nabla u\vert^2\mu\lesssim_{\kappa}\int_{B_{\sigma}}\mu\vert u\vert^2\vert\nabla\eta\vert^2+\int_{B_{\sigma}}\vert g\vert^2,$$
which yields \eqref{esticaccioppounbounded} by optimizing over $\eta$.
\end{proof}
We then turn to the proof of Lemma~\ref{lemmabella}.
\begin{proof}[Proof of Lemma~\ref{lemmabella}]
We split the proof into two steps.

\medskip

\step1 Proof that for all $\gamma>0$
\begin{equation}
\mathcal{J}(\rho,\sigma,\mu,v)\leq (\sigma-\rho)^{-1-\frac{1}{\gamma}}\left(\int_{\rho}^{\sigma}\left(\int_{S_r}\mu\vert v\vert^2\right)^{\gamma}\dd r\right)^{\frac{1}{\gamma}},
\label{lemmabellastep11}
\end{equation}
where $S_r:=\partial B_r$. By scaling we may assume without loss of generality that $\rho=1$ and $\sigma=2$.
Estimate \eqref{lemmabellastep11} essentially follows by minimizing among radially symmetric cut-off functions. 
Indeed, for all $\e>0$ we have
\begin{equation*}
\mathcal{J}(1,2,\mu,v)\leq \inf\Big\{\int_{1}^{2}\eta'(r)^2\Big(\int_{S_r}\mu\vert v\vert^2+\e\Big)\dd r\Big\vert \eta\in C^1(1,2),\, 0\le \eta \le 1, \eta(1)=1,\, \eta(2)=0\Big\}.
\end{equation*}
This one-dimensional minimization problem can be solved explicitly. 
Set $f(r):= \int_{S_r}\mu\vert v\vert^2+\varepsilon$.
Using the competitor $\eta(r):=1-\frac{\int_1^r f^{-1}}{\int_1^2 f^{-1}}$ yields a control of this minimum by the harmonic average of $f$, 
\begin{equation*}
  \mathcal{J}(1,2,\mu,v)\leq \Big(\int_{1}^{2} \Big(\int_{S_r}\mu\vert v\vert^2+\varepsilon\Big)^{-1} \dd r\Big)^{-1}. 
\end{equation*}
By standard relations between quasi-arithmetic means,  since $\gamma>-1$,
\begin{equation*}
 \Big(\int_{1}^{2} \Big(\int_{S_r}\mu\vert v\vert^2+\varepsilon\Big)^{-1} \dd r\Big)^{-1}\leq \Big(\int_{1}^{2} \Big(\int_{S_r}\mu\vert v\vert^2+\varepsilon\Big)^{\gamma} \dd r\Big)^{\frac1\gamma},
\end{equation*}
and the claim \eqref{lemmabellastep11} follows by letting $\e \downarrow 0$.

\medskip

\step2 Proof of~\eqref{estilemmabella}. 

\noindent Let us first assume $d\geq 3$ and note that $q>\frac{d-1}{2}$ implies $q_*\in [1,2)$. Recall that  for all $s\in [1,d-1)$, $r>0$ and $\phi\in W^{1,s}(S_r)$, Poincar\'e-Sobolev' inequality yields for  $s_*=\frac{(d-1)s}{d-1-s}$
\begin{equation}
\Big(\fint_{S_r}\vert\phi\vert^{s_*}\Big)^{\frac{1}{s_*}}\lesssim r\Big(\fint_{S_r}\vert\nabla\phi\vert^s\Big)^{\frac{1}{s}}+\Big(\fint_{S_r}\vert\phi\vert^s\Big)^{\frac{1}{s}}.
\label{sobolevineglemmabella}
\end{equation}
By H\"older's inequality with exponents $(q,\frac{q}{q-1})$, followed by \eqref{sobolevineglemmabella} with $s=q_*$ and $s_*=\frac{2q}{q-1}$, 
\eqref{lemmabellastep11} turns into
\begin{align*}
\mathcal{J}(\rho,\sigma,\mu,v)&\leq\frac{1}{(\sigma-\rho)^{1+\frac{1}{\gamma}}}\Big(\int_{\rho}^{\sigma}\Big(\int_{S_r}\mu^q\Big)^{\frac{\gamma}{q}}\Big(\int_{S_r}\vert v\vert^{\frac{2q}{q-1}}\Big)^{\frac{(q-1)\gamma}{q}}\dd r\Big)^{\frac{1}{\gamma}}\\
&\stackrel{\eqref{sobolevineglemmabella}}{\lesssim_q} \frac{1}{(\sigma-\rho)^{1+\frac{1}{\gamma}}}\Big(\int_{\rho}^{\sigma}\Big(\int_{S_r}\mu^q\Big)^{\frac{\gamma}{q}}\Big(\Big(\int_{S_r}\vert\nabla v\vert^{q_*}\Big)^{\frac{2\gamma}{q_*}}+r^{-2\gamma}\Big(\int_{S_r}\vert v\vert^{q_*}\Big)^{\frac{2\gamma}{q_*}}\Big)\dd r\Big)^{\frac{1}{\gamma}}.
\end{align*}
We then choose $\gamma=\frac{d-1}{d+1}$ to the effect that $\frac{\gamma}{q}+\frac{2\gamma}{q_*}=1$, so that by H\"older's inequality with exponents $(\frac{\gamma}{q},\frac{2\gamma}{q_*})$ we obtain
$$
\mathcal{J}(\rho,\sigma,\mu,v)\lesssim_q \frac{1}{(\sigma-\rho)^{\frac{2d}{d-1}}}\Big(\int_{B_{\sigma}\backslash B_{\rho}}\mu^q\Big)^{\frac{1}{q}}\Big(\Big(\int_{B_{\sigma}\backslash B_{\rho}}\vert\nabla v\vert^{q_*}\Big)^{\frac{1}{q_*}}+\frac{1}{\rho^2}\Big(\int_{B_{\sigma}\backslash B_{\rho}}\vert v\vert^{q_*}\Big)^{\frac{2}{q_*}}\Big),
$$
which is the desired estimate \eqref{estilemmabella}. For $d=2$, in which case $q_*=1$,  we use the one-dimensional Sobolev inequality $\|\phi\|_{L^{\infty}(S_r)}\lesssim r\fint_{S_r}\vert\nabla\phi\vert+\fint_{S_r}\vert\phi\vert$ instead  of \eqref{sobolevineglemmabella}, and  \eqref{estilemmabella} follows from \eqref{lemmabellastep11}. The case $d=1$ is similar.
\end{proof}

\subsection{Qualitative differentiability of correctors on bounded domains}

In this section, we consider the approximation of correctors on bounded domains, both with Dirichlet and periodic boundary conditions.
More precisely, let $D$ be a smooth bounded domain of $\R^d$ (resp. a cube $Q_L$, $L>0$),  let $A: D \to \Md(\lambda)$ be of class $C^\alpha$ (resp. $C^\alpha_\per(Q_L)$),
and set $a:D\times \R^d \to \R^d, (x,\xi) \mapsto A(x)(1+|\xi|^{p-2})\xi$ for some $p\ge 2$.
We show that the corrector gradients $\xi\mapsto (\nabla\corNL,\nabla\sigma_{\xi})$, where $(\corNL,\sigma_\xi)$ are solutions of \eqref{e.cor-eq} and 
\eqref{e.Laplace-sig} on $D$ with homogeneous boundary conditions (resp. $Q_L$-periodic with vanishing average),  are Fr\'echet-differentiable and that their derivatives are given by the linearized corrector gradients.  More precisely: 
\begin{lemma}[Differentiability of correctors]\label{lemmadiffcor}
The corrector gradients $\xi\mapsto (\nabla\sigma_{\xi},\nabla\corNL)$ are Fr\'echet-differentiable in $C^0(D)$ (resp. in $C^0_{\per}(Q_L)$) and for all directions $e\in\mathbb{R}^d$
we have
$$(\partial_{\xi}\nabla\phi_{\xi}\cdot e,\partial_{\xi} \nabla\sigma_{\xi}\cdot e)=(\nabla\corL,\nabla\sigL_{\xi,e}),$$
where $(\corL,\sigL_{\xi,e})$ solve \eqref{e.Lcorr} and \eqref{e:eq-sigmaL}
with homogeneous Dirichlet boundary conditions (resp. $Q_L$-periodic with vanishing average).
\end{lemma}
\begin{proof}
We only give the arguments for  $\xi\mapsto \nabla\corNL$. The differentiability of $\xi\mapsto\nabla\sigma_\xi$ can be proved similarly. Let $\xi,e\in\mathbb{R}^d$, $h\subset (0,1)$ be a sequence that goes to $0$, and set $\delta^h\corNL:=\frac{\phi_{\xi+he}-\corNL}{h}$.  
We first show that we can extract a converging subsequence of  $(\nabla\delta^h\corNL)$ by  local regularity and Arzela-Ascoli's theorem.  Then, we show that this limit coincides with $\nabla\corL$. 
The starting point is the corrector equation \eqref{e.cor-eq} in the form
\begin{equation}
-\nabla\cdot(a(x,\xi+\nabla\phi_{\xi+he})-a(x,\xi+\nabla\phi_{\xi}))=\nabla\cdot(a(x,\xi+\nabla\phi_{\xi+he})-a(x,\xi+he+\nabla\phi_{\xi+he})),
\label{Diff:Eq1}
\end{equation}
that we rewrite, using the smoothness of $\xi\mapsto a(\cdot,\xi)$,  as
\begin{equation}
-\nabla\cdot a^{(1)}_h\nabla\delta^h\corNL=\nabla\cdot a^{(2)}_h e,
\label{Diff:Eq6}
\end{equation}
where
$$a^{(1)}_h:=\int_0^1 Da(\cdot, \xi+\nabla\corNL+t(\nabla\phi_{\xi+he}-\nabla\corNL))\dd t \text{ and } a^{(2)}_h:=\int_0^1 Da(\cdot, \xi+\nabla\phi_{\xi+he}+(1-t)he)\dd t.$$
Next, for all $\tilde{\xi}\in\mathbb{R}^d$,  using local regularity (this time up to the boundary) in form of Lemma \ref{regestiNL} (applied to the equation \eqref{e.cor-eq}) and an energy estimate on $D$, we have 
\begin{equation}
\|\tilde{\xi}+\nabla\phi_{\tilde{\xi}}\|_{C^{\alpha}(D)}\lesssim_{D,\|A\|_{C^{\alpha}(D)}}\Big(\int_{D}\vert \tilde{\xi}+\nabla\phi_{\tilde{\xi}}\vert^p\Big)^{\frac{1}{p}}\lesssim_D 1+\vert\tilde{\xi}\vert^p.
\label{Diff:Eq2}
\end{equation}
Therefore, $a^{(1)}_h\in C^{\alpha}(D)$ and by using \eqref{Diff:Eq2} for both $\tilde{\xi}=\xi$ and $\tilde{\xi}=\xi+he$ and arguing as in \eqref{deterboundproof3}, there exists a finite constant $c_1$ depending on $\vert\xi\vert$, $D$ , $\|A\|_{C^{\alpha}(D)}$ such that
\begin{equation}
\|a^{(1)}_{h}\|_{C^{\alpha}(D)}+ \|a^{(2)}_{h}\|_{C^{\alpha}(D)}\leq c_1.
\label{Diff:Eq3}
\end{equation}
On the one hand, by testing the equation \eqref{Diff:Eq6} with $\delta^h\corNL$ and using that $a^{(1)}_h$ is uniformly elliptic, we deduce
\begin{equation}
\int_{D}\vert \nabla\delta^h\corNL\vert^2\lesssim \int_{D}\vert a^{(2)}_h \vert^2\stackrel{\eqref{Diff:Eq3}}{\lesssim} c_1.
\label{Diff:Eq4}
\end{equation}
On the other hand, by the Schauder estimate \cite[Theorem 5.19]{giaquinta2013introduction} applied to \eqref{Diff:Eq6}, and the bounds \eqref{Diff:Eq3} and \eqref{Diff:Eq4}, there exists some $\gamma>0$ (depending on $\alpha$ and $d$) such that
\begin{equation}
\|\nabla\delta^h\corNL\|_{C^{\alpha}(D)}\lesssim\|a^{(1)}_h\|^{\gamma}_{C^{\alpha}(D)}\Big(\Big(\int_{D}\vert \nabla\delta^h\corNL\vert^2\Big)^{\frac{1}{2}}+\|a^{(2)}_h\|_{C^{\alpha}(D)}\Big)\stackrel{\eqref{Diff:Eq3},\eqref{Diff:Eq4}}\leq c_2,
\label{Diff:Eq5}
\end{equation}
for a finite constant $c_2$ depending on $c_1$, $d$ and $D$.  
By~\eqref{Diff:Eq5} and Arzela-Ascoli's theorem, there exists $\tilde{\psi}\in C^{1,\alpha}_0(D)$ (resp.~$C^{1,\alpha}_\per(Q_L)$) such that (up to a subsequence that we do not relabel)
\begin{equation}
\nabla\delta^h\corNL\stackrel{h\downarrow 0}{\rightarrow} \nabla\tilde{\psi} \text{ in $C^0(D)$}.
\label{Diff:Eq9}
\end{equation}
It remains to show that $\nabla\tilde{\psi}=\nabla\corL$, which directly follows from the weak formulation of \eqref{Diff:Eq6}. 
For all $w\in H^1_0(D)$ (resp. $H^1_\per(Q_L)$)
\begin{equation}
\int_{D}\nabla w\cdot a^{(1)}_h\nabla\delta^h\corNL=-\int_{D}\nabla w\cdot a^{(2)}_he.
\label{Diff:Eq10}
\end{equation}
Using the convergence of the gradient \eqref{Diff:Eq9}, we can pass to the limit as $h\downarrow 0$ in \eqref{Diff:Eq10}, which implies that $\tilde{\psi}$ solves \eqref{e.Lcorr}.  By uniqueness, $\tilde{\psi}=\corL$, and \eqref{Diff:Eq9} holds without extracting a subsequence.
\end{proof}

\subsection{Periodic setting: Proof of Theorem~\ref{th:isotropic-per}}\label{app:isotropic-per}

In this paragraph, we show that if we have a good control of the 
critical set of the corrector of the leading order operator (an anisotropic $p$-Laplacian), then
the  homogenized operator $\bar a$ satisfies \eqref{*coer+} on top of \eqref{*coer+-}.
In what follows we set $b:(x,\xi)\mapsto A(x)|\xi|^{p-2} \xi$ and $c:(x,\xi)\mapsto A(x)\xi$.
We first introduce $\bar b:\R^d \to \R^d, \xi \mapsto \fint_Q b(x,\xi+\nabla \psi_\xi(x)) \dd x$,
where $\psi_\xi \in W^{1,p}_\per(Q)$ solves the corrector equation  
$$
-\nabla \cdot b(x,\xi+\nabla \psi_\xi(x))  = 0.
$$
By homogeneity, for all $t>0$ and all $\xi \in \R^d$ we have $\bar b(t\xi) = t^{p-1} \bar b(\xi)$.

\medskip

\step1 Proof of \eqref{*coer+}  for $|\xi_1|,|\xi_2|\gg 1$.

\medskip

\substep{1.1} Reformulation.

\noindent By Lemma~\ref{lemmadiffcor}, correctors are differentiable and we thus have
for all $\xi,e\in \R^d$
$$
e\cdot D\bar a(\xi) e\,=\, \fint_Q (e+\nabla \tilde \phi_{\xi,e}) \cdot a_\xi (e+\nabla \tilde \phi_{\xi,e}),
$$
where $\tilde \phi_{\xi,e} \in H^1_\per(Q)$ solves $-\nabla \cdot a_\xi (e+\nabla \tilde \phi_{\xi,e})=0$
and $a_\xi:x\mapsto Da(x,\xi+\nabla \phi_\xi(x))$. Hence, for all $\xi_1,\xi_2 \in \R^d$ we have
\begin{equation}
(\bar a(\xi_1)-\bar a(\xi_2))\cdot (\xi_1-\xi_2)\,=\, \int_0^1 (\xi_1-\xi_2) \cdot D\bar a(\xi_1+t(\xi_2-\xi_1)) \cdot (\xi_1-\xi_2) \dd t.
\end{equation}
Since $D \bar a$ is non-negative, the claim \eqref{*coer+} follows for $|\xi_1|,|\xi_2|\gg 1$ provided we prove that $e \cdot D \bar a(\xi) e \ge c |\xi|^{p-2} $ for all $|\xi|\gg1$ and all $e\in \R^d$ with $|e|=1$.

\medskip

Fix such a direction $e$.
For all $s >0$ let $\xi_s \in \R^d$ be such that $|\xi_s|=1$ and 
$
e \cdot D \bar a(s \xi_s) e = \inf_{ |\xi|=1} e \cdot D \bar a(s\xi) e,
$
which exists since $\xi \mapsto D\bar a(\xi)$ is continuous.
In the following two substeps we prove the needed estimate in form of  
\begin{equation}\label{e.desired-strong-ell}
\liminf_{s \uparrow \infty}  e \cdot  \tfrac{1}{s^{p-2}} D\bar a (s\xi_s) e\, > 0.
\end{equation}

\medskip

\substep{1.2} Proof of
\begin{equation}\label{e.approx-asco-xis}
\lim_{s \uparrow \infty} \| \tfrac{1}{s^{p-2}}a_{s\xi_s} - b_{\xi_s} \|_{C^{\alpha}(Q)} \,= \, 0,
\end{equation}
where $b_\xi:x\mapsto Db(x,\xi+\nabla \psi_\xi(x))$.

\noindent On the one hand, by the corrector equations for $\phi_{\xi}$ and $\psi_\xi$ we have
$$
-\nabla \cdot (a(\xi+\nabla \phi_\xi)-a(\xi+\nabla \psi_\xi)) = \nabla \cdot c(x,\xi+\nabla \psi_\xi),
$$
so that, testing with $\psi_\xi-\phi_\xi$ and using the monotonicity of $a$, we obtain
\begin{multline*}
\int_Q |\nabla (\psi_\xi-\phi_\xi)|^2 (1+|\xi+\nabla \psi_\xi|^{p-2}+|\xi+\nabla \phi_\xi|^{p-2})
\,\lesssim\, \int_Q |\nabla (\psi_\xi-\phi_\xi)| |\xi+\nabla \psi_\xi|
\\
\,\lesssim\, \int_Q |\nabla (\psi_\xi-\phi_\xi)| (1+|\xi+\nabla \psi_\xi|)^\frac{p-2}2  (1+|\xi+\nabla \psi_\xi|)^{2-\frac{p}2} 
\end{multline*}
and therefore 
$$
\int_Q |\nabla (\psi_\xi-\phi_\xi)|^2 (1+|\xi+\nabla \psi_\xi|^{p-2}+|\xi+\nabla \phi_\xi|^{p-2})
\,\lesssim \, \int_Q (1+ |\xi+\nabla \psi_\xi| )^{4-p}.
$$
Applied to $\xi=s\xi_s$ this yields using that $\psi_{s\xi_s}=s\psi_{\xi_s}$ 
\begin{equation}\label{e.convLp-xis}
\int_Q |\nabla (\tfrac{1}{s}\phi_{s\xi_s}-\psi_{\xi_s})|^p \lesssim s^{2(2-p)} \int_Q (1+|\xi_s+\nabla \psi_{\xi_s}|^p) \, \stackrel{|\xi_s|=1}\lesssim \,  s^{2(2-p)}.
\end{equation}
Next we argue that $\{\tfrac{1}{s}\phi_{s\xi_s}\}_{s\ge 1}$ is a bounded sequence in $C^\alpha(Q)$, in which case \eqref{e.approx-asco-xis} follows from  \eqref{e.convLp-xis}  by Arzela-Ascoli's theorem.
To this end we rewrite the corrector equation for $\chi_s:=\tfrac1s \phi_{s\xi_s}$ as 
$$
-\nabla \cdot  \tilde a_s (x,\xi_s+\nabla \chi_s) = 0
$$
with $\tilde a_s(x,\xi):=  b (x,\xi ) + \frac1{s^{p-1}} c(x,s \xi)$. Since by assumption, $\tilde a_s$ satisfies \eqref{darkside} for some $\omega$ independent of $s\ge 1$,  $\{\chi_s\}_{s\ge 1}$ is indeed bounded in $C^{1,\alpha}(Q)$ by \cite[Theorem 13]{kuusi2014guide}.
The conclusion then follows using that $\tfrac{1}{s^{p-2}}a_{s\xi_s}= D\tilde a_s(x,\xi_s+\nabla \chi_s)$ and that $\xi \mapsto D \tilde a_s(x,\xi)$ is continuous (uniformly wrt $s,x$).

\medskip

\substep{1.3} Proof of \eqref{e.desired-strong-ell}.

\noindent
We assume without loss of generality that $\xi_s \to \xi_\infty$.
Since $\xi \mapsto \nabla \psi_\xi$ is Lipschitz from $\R^d$ to $C^\alpha(Q)$,  \eqref{e.approx-asco-xis} can be upgraded to (along the subsequence giving the liminf)
\begin{equation}\label{e.approx-asco-xis+}
\lim_{s \uparrow +\infty} \| \tfrac{1}{s^{p-2}}a_{s\xi_s} - b_{\xi_\infty} \|_{C^{\alpha}(Q)} \,= \, 0.
\end{equation}
By definition we have 
\begin{eqnarray*}
e\cdot \tfrac{1}{s^{p-2}} D\bar a(s\xi_s) e&=& \fint_Q (e+ \nabla \tilde \phi_{s\xi_s,e}) \cdot  \tfrac{1}{s^{p-2}}a_{s\xi_s}^\sym (e+\nabla \tilde \phi_{s\xi_s,e}),
\end{eqnarray*}
where $a_{s\xi_s}^\sym$ is the symmetric part of $a_{s\xi_s}$.
By assumption, there exists $r>0$ such that $\R^d \setminus \calT_r(\xi_\infty)$ is connected (where $\calT_r(\xi_\infty):=\{ x+B_r \,|\,x  \in \R^d, |\xi_\infty+\nabla \psi_{\xi_\infty}(x)|=0\}$).
Since $[0,1]^d \cap \R^d \setminus \calT_r(\xi_\infty)$ is closed and $\nabla \psi_{\xi_\infty}$ is continuous, there exists $\kappa>0$ such that the symmetric part $b_{\xi_\infty}^{\sym}$ of $b_{\xi_\infty}$ satisfies $b_{\xi_\infty}^\sym|_{Q \setminus \calT_r(\xi_\infty)} \ge \kappa \id$. Hence, by \eqref{e.approx-asco-xis+}, there exists $s_\star<\infty$ such that, for all $s\ge s_\star$, $a_{s\xi_s}^\sym|_{Q \setminus \calT_r(\xi_\infty)}  \ge \frac12 \kappa \id$. 
For all $s\ge s_\star$ we thus have 
\begin{equation*}
e\cdot \tfrac{1}{s^{p-2}} D\bar a(s\xi_s) e\,\ge \,\frac \kappa 2 \int_{Q \setminus \calT_r(\xi_\infty)} |e+ \nabla \tilde \phi_{s\xi_s,e} |^2
\,
\ge \, \frac \kappa 2  \inf_{\tilde \phi\in H^1_\per(Q)} \int_{Q \setminus \calT_r(\xi_\infty)} |e+ \nabla \tilde \phi |^2,
\end{equation*}
which is positive since $\R^d \setminus \calT_r(\xi_\infty)$ is connected in $\R^d$. This proves \eqref{e.desired-strong-ell}.

\medskip

\step2 Proof of \eqref{*coer+} in the remaining range: $|\xi_1|,|\xi_2|\lesssim 1$
and $|\xi_2|\gg 1,|\xi_1| \lesssim 1$.

\noindent On the one hand, since $c(x,\xi)\xi \ge \frac1C |\xi|^2$, we have 
for all $\xi_1,\xi_2 \in \R^d$
\begin{equation*} 
(\bar a(\xi_1)-\bar a(\xi_2),\xi_1-\xi_2)\,\ge \,c |\xi_1-\xi_2|^2 ,
\end{equation*}
from which \eqref{*coer+} follows for $|\xi_1|,|\xi_2|\lesssim 1$.
On the other hand, \eqref{*coer} implies \eqref{*coer+} for $|\xi_2| \gg 1, |\xi_1| \lesssim 1$ using 
$|\xi_1-\xi_2|^p \gtrsim |\xi_1-\xi_2|^2(|\xi_2|^{p-2}-|\xi_1|^{p-2}) \gtrsim |\xi_1-\xi_2|^2(1+|\xi_2|^{p-2}+|\xi_1|^{p-2})$.

\begin{remark}
The strong assumption on the critical set of $\psi_\xi$ is solely used to ensure that 
$$
\liminf_{s\uparrow +\infty}  \int_Q (e+ \nabla \tilde \phi_{s\xi_s,e}) \cdot  \tfrac{1}{s^{p-2}}a_{s\xi_s}^\sym (e+\nabla \tilde \phi_{s\xi_s,e}) \ge \inf_{\tilde \phi \in H^1_\per(Q)} \int_Q (e+ \nabla \tilde \phi ) \cdot b_{\xi_\infty}^\sym (e+\nabla \tilde \phi),
$$
which might not hold true in general due to the Lavrentieff phenomenon -- see e.g.~\cite{Zhikov-01} in a similar context.
\end{remark}

\subsection{Statistically isotropic random setting: Proof of  Theorem~\ref{th:isotropic}}\label{app:closed}

In this subsection, we exhibit a class of random monotone operators $a$ whose homogenized operator $\bar a$ satisfies \eqref{*coer+} next to \eqref{*coer+-}.
If $a$ is a $p$-Laplacian then $\bar a$ is homogeneous of degree $p-1$. The upcoming
result relies on a perturbation of this property.
To define this class we make both structural assumptions on $(x,\xi) \mapsto a(x,\xi)$ and on the probability law $\mathbb P$.
We emphasize that our arguments are purely qualitative and do not require Hypothesis~\ref{hypo}.
We start with the structural assumption on the operator (which quantifies what we mean by perturbation of a $p$-Laplacian)
\begin{definition}
We define a class $\calA$ of nonlinear maps $\hat a : [\lambda,1] \times\mathbb{R}^d\rightarrow \mathbb{R}^d$ with quasi-diagonal structure, that is, such that for all $(\alpha,\xi)\in [\lambda,1]\times\mathbb{R}^d$
\begin{equation}
\hat a(\alpha,\xi)=\rho(\alpha ,\vert\xi\vert)\xi,
\label{formMbar}
\end{equation}
where $\rho :[\lambda,1]\times\mathbb{R}_+\rightarrow \R_+$ is continuously differentiable, and such that   
the map $\xi\mapsto \hat a(\alpha,\xi)$ is asymptotically of $p$-Laplacian type.
More precisely, we assume that $\inf_\alpha \rho(\alpha,t) \ge \lambda (1+t^{p-2})$ and that there exist two  differentiable functions $\rho_1 : [\lambda, 1] \rightarrow \mathbb{R}_+$ and $\rho_{2}: [\lambda, 1] \times\mathbb{R}_+\rightarrow \mathbb{R}$ such that for all $(\alpha,t)\in [\lambda,1]\times \mathbb{R}_+$
$$\rho(\alpha,t)=\rho_1(\alpha )t^{p-2}+\rho_2(\alpha,t),$$
and that there exist a constant $C$ and an exponent $0 \le \beta<p-2$ such that
$$|\rho_{2}(\alpha,t)|+t\vert\partial_{t}\rho_2(\alpha,t)\vert\leq C (1+t^{\beta}).$$
As a consequence, $\hat a$  is variational in the sense that
$\hat a(\alpha, \xi)=D\hat W(\alpha,\xi)$, where $\hat W$ is given by $\hat W(\alpha,\xi):=\int_0^{|\xi|} s \rho(\alpha,s)\dd s \ge \lambda (\frac12 |\xi|^2+\frac1p |\xi|^p)$.
\end{definition}
As the following examples show, $\calA$ is not empty.
\begin{example}The  following nonlinear maps belong to $\calA$:
\begin{enumerate}
\item The non-degenerate $p$-Laplacian operator $\hat a:   [\lambda, 1]\times\mathbb{R}^d\ni (\alpha,\xi)\mapsto \alpha(1+\vert\xi\vert^{p-2})\xi,$
with $\rho_1(\alpha)=\alpha$ and $\rho_2(\alpha,t)=\alpha$ (and therefore $C=\frac12$ and $\beta=0$).
\item The operator
$\hat a : [\lambda,1]\times\mathbb{R}^d\ni (\alpha,\xi)\mapsto \Big(\frac{1+\vert \xi\vert^{p+q-2}}{1+\alpha\vert\xi\vert^q}+1\Big)\xi,$
for any $q\ge 0$, and with  $\rho_1(\alpha):=\frac{1}{\alpha}$, $\rho_2(\alpha,t):=\frac{\alpha-t^{p-2}}{\alpha(1+\alpha t^q)}+1$, $C=\frac{1}{\alpha}\max\{1,q,\frac{p-2}{\alpha}\}$ and $\beta=(p-2-q)\vee 0$.
\end{enumerate}
\end{example}
We are in position to state the main result of this section.
\begin{theorem}\label{th:isotropic-erg}
Let $\hat a \in \calA$ and let $A:\R^d \to [\lambda, 1]$ be a stationary and ergodic random field (locally $C^\alpha$ -- this is convenient but not necessary) which is statistically isotropic in the sense that for all rotations $R \in SO(d)$,
$A$ and $A(R \cdot)$ have the same (joint) distribution. Consider the random monotone operator $a:\R^d \times \R^d \ni (x,\xi)\mapsto \hat a(A(x),\xi)$.
Then, the associated homogenized map $\bar a$ 
satisfies \eqref{*cont}, \eqref{*coer+-}, and \eqref{*coer+}.
\end{theorem}
Note that the example of Theorem~\ref{th:isotropic} satisfies the assumptions of Theorem~\ref{th:isotropic-erg}.
The proof of Theorem~\ref{th:isotropic-erg} relies on an approximation argument, an ODE argument, and the following property.
\begin{lemma}\label{criterion}
Let $\aa : \mathbb{R}^d\rightarrow \mathbb{R}^d$ be such that for all $\xi\in\mathbb{R}^d$
\begin{equation}
\aa(\xi)=\rho(\vert \xi\vert)\xi,
\label{E1}
\end{equation}
where $\rho:\R_+\rightarrow \mathbb{R}_+$ is a differentiable function that satisfies for some constant $c>0$ and for all $t>0$
\begin{equation}
\frac{\dd}{\dd t}(t \rho(t))\geq c(1+t^p)^{\frac{p-2}{p}}.
\label{E2}
\end{equation}
Then, there exists a constant $\tilde{c}>0$ depending on $c$ and $p$ such that for all $\xi_1,\xi_2 \in\mathbb{R}^d$
\begin{equation}\label{*coer++}
(a(\xi_1)-a(\xi_2))\cdot(\xi_1-\xi_2)\geq \tilde{c}(1+\vert\xi_1\vert^{p-2}+\vert\xi_2\vert^{p-2})\vert\xi_1-\xi_2\vert^2.
\end{equation}
\end{lemma}
\begin{proof}
W.l.o.g we may assume that $\vert\xi_1\vert>\vert\xi_2\vert>0$. We fix $s>0$ and we define $f : [s,+\infty)\rightarrow \mathbb{R}$ as  
$$f:t\mapsto t \rho(t)-s \rho(s)-\tilde{c}(1+t^p+s^p)^{\frac{p-2}{p}}(t-s),$$
where $\tilde{c}$ will be fixed later. Differentiating $f$ and using the assumption~\eqref{E2}, we obtain
\begin{align*}
f'(t)&=\frac{\dd}{\dd t}(t \rho(t))-\tilde{c}((p-2)(1+t^p+s^p)^{-\frac{2}{p}}t^{p-1}(t-s)+(1+t^p+s^p)^{\frac{p-2}{p}})\\
&\geq c(1+t^p)^{\frac{p-2}{p}}-\tilde{c}((p-2)(1+t^p+s^p)^{-\frac{2}{p}}t^{p-1}(t-s)+(1+t^p+s^p)^{\frac{p-2}{p}}).
\end{align*}
Since $(t-s)t^{p-1} \le t^p$ and $1+t^p+s^p \le 2(1+t^p)$, this yields
$f'(t)\geq (c-2 \tilde c (p-1))(1+t^p)^{\frac{p-2}{p}}$.
With the choice $\tilde c = \frac{c}{2(p-1)}$, this entails $f'(t) \ge 0$, and thus $f(t)\ge f(s)$,
which takes the form
\begin{equation}
t \rho(t)-s \rho(s)\geq \tilde{c}(1+t^p+s^p)^{\frac{p-2}{p}}(t-s).
\label{E3}
\end{equation}
Used with $s=0$, this implies the lower bound for all $t_1 >0$
\begin{equation}
\rho(t_1) \ge \tilde c (1+t_1^p )^{\frac{p-2}{p}} .
\label{E4}
\end{equation}
By the definition \eqref{E1} of $\aa$, we have with the notation $t_1=|\xi_1|>t_2=|\xi_2|$,
\begin{align*}
(\aa(\xi_1)-\aa(\xi_2))\cdot(\xi_1-\xi_2)&=t_1^2\rho(t_1)+t_2^2 \rho(t_2)-(\xi_1\cdot\xi_2)(\rho(t_1)+\rho(t_2))\\
&= (t_1 \rho(t_1)-t_2\rho(t_2))(t_1-t_2)+(t_1t_2-\xi_1\cdot\xi_2)(\rho(t_1)+\rho(t_2))
\\
&\stackrel{\eqref{E3},\eqref{E4}}\ge \tilde c (1+t_1^p )^\frac{p-2}{p} (t_1-t_2)^2 
+\tilde c (t_1t_2-\xi_1\cdot\xi_2)(1+t_1^p )^\frac{p-2}{p} 
\\
&\ge \frac{\tilde c}2 (1+t_1^p )^\frac{p-2}{p}|\xi_1-\xi_2|^2,
\end{align*}
and the claim follows by redefining $\tilde c$.
\end{proof}
We now prove Theorem~\ref{th:isotropic-erg}.
\begin{proof}[Proof of Theorem~\ref{th:isotropic-erg}]
We focus on the proof of \eqref{*coer++} for $\bar a$ (which implies both \eqref{*coer+} and \eqref{*coer+-}).
We split the proof into two steps.
In the first step we argue that it suffices to prove a version of~\eqref{*coer++} obtained by an approximation of the corrector on bounded domains, which we prove in the second step.

\medskip

\step1 Approximation. 

\noindent By assumption there exists $W:\R^d\times \R^d\to \R$ such that $a(x,\xi)=DW(x,\xi)$.
Likewise, there exists $\bar W:\R^d \to \R$ such that $\bar a(\xi)=D\bar W(\xi)$.
For all $L \ge 1$, we denote by $\phi_\xi^L$ the unique weak solution in $W^{1,p}_0(B_L)$
of \eqref{e.cor-eq}, and define
$$
a^L(\xi) := \fint_{B_L} a(x,\xi+\nabla \phi^L_\xi(x))\dd x, \quad W^L(\xi) := \fint_{B_L} W(x,\xi+\nabla \phi^L_\xi(x))\dd x,
$$
which satisfy $a^L(\xi)=D W^L(\xi)$.
As a direct consequence of the homogenization result we have almost-surely 
$$
\lim_{L\uparrow +\infty} (a^L(\xi),W^L(\xi))\,=\, (\bar a(\xi),\bar W(\xi)).
$$
Set $\bar a^L(\xi):=\mathbb{E}[a^L(\xi)]$ and $\bar W^L(\xi):=\mathbb{E}[W^L(\xi)]$. 
Since we have the uniform almost sure bound $ |a^L(\xi)|+|W^L(\xi)| \le C(1+|\xi|^p)$, 
we obtain on the one hand that $\bar a^L(\xi)=D \bar W^L(\xi)$ and on the other hand (by dominated convergence) that  
$$
\lim_{L\uparrow +\infty} (\bar a^L(\xi),\bar W^L(\xi))\,=\, (\bar a(\xi),\bar W(\xi)).
$$
Hence, \eqref{*coer++} follows if we prove that there exists a constant $C>0$ independent of $L$ such that for all $\xi_1,\xi_2 \in \R^d$ we have 
\begin{equation}
(\bar{a}^L(\xi_1)-\bar{a}^L(\xi_2))\cdot(\xi_1-\xi_2)
\geq \frac1C(1+\vert\xi_1\vert+\vert\xi_2\vert)^{p-2}\vert\xi_1-\xi_2\vert^2.
\label{*coer+L}
\end{equation}
The advantage of \eqref{*coer+L} over \eqref{*coer++} is that it involves correctors on a bounded domain rather than on the whole space, for which differentiability with respect to $\xi$ can be easily established (cf.~Lemma~\ref{lemmadiffcor}). 
The advantage of $\bar W^L$ over $W^L$ (which motivates the choice of the ball $B_L$ for the domain) is that, as $\bar W$, $\xi \mapsto \bar W^L$ is isotropic in the sense that there exists $\zeta^L : \R_+ \to \R$ such that $\bar W^L (\xi) = \zeta^L(|\xi|)$.
In particular, we necessarily have $\bar a^L (\xi) = (\zeta^L)'(|\xi|) \frac{\xi}{|\xi|}$,
so that, by Lemma~\ref{criterion}, \eqref{*coer+L} will follow provided we show that $t \mapsto (\zeta^L)'(t)$ is differentiable and satisfies for some $c>0$ independent of $L$
for all $t>0$
\begin{equation}
(\zeta^L)''(t)\geq c(1+t^p)^{\frac{p-2}{p}}.
\label{E8}
\end{equation}
Fix a unit vector $e \in \R^d$. By definition, we have
$$
(\zeta^L)'(t)\,=\, \bar a^L(te) \cdot e = \expec{\fint_{B_L} a(x,te+\nabla \phi^L_{te}(x)) \cdot e\dd x}.
$$
By Lemma~\ref{lemmadiffcor}, correctors are differentiable, and differentiating the above yields using~\eqref{e.Lcorr} (posed on $B_L$ with Dirichlet boundary conditions)
\begin{eqnarray*}
(\zeta^L)''(t)&=& \expec{\fint_{B_L} Da(x,te+\nabla \phi^L_{te}(x))(e+\nabla \tilde \phi_{te,e}^L(x)) \cdot e  \dd x}
\\
&=&  \expec{\fint_{B_L} Da(x,te+\nabla \phi^L_{te}(x))(e+\nabla \tilde \phi_{te,e}^L(x)) \cdot (e+\nabla \tilde \phi_{te,e}^L(x)) \dd x}.
\end{eqnarray*}
Since $a(x,\xi) = \hat a(A(x),\xi)$, $\xi \mapsto a(x,\xi)$ is differentiable and satisfies 
for some $c>0$ and for all $\xi,h \in \R^d$
\begin{equation*}
Da(x,\xi) : h \otimes h \ge c(1+|\xi|^{p-2}) |h|^2.
\end{equation*}
Hence (using that $\int_{B_L} \nabla \tilde \phi_{te,e}^L=0$), 
$
(\zeta^L)''(t) \, \ge \,  c \expec{\fint_{B_L}|e+\nabla \tilde \phi_{te,e}^L|^2}  \ge c,
$
which yields \eqref{E8} provided $t \lesssim 1$. It remains to treat the case $t\gg 1$.

\medskip

\step2 Proof of \eqref{E8}.

\noindent If $\bar a^L$ were the $p$-Laplacian, $\zeta^L$ would satisfy $(\zeta^L)'(t)=\frac pt \zeta^L(t)$. The idea is to derive a similar ODE in our setting based on the identity $\bar a^L(\xi) = D\bar W^L(\xi)$, from which we shall prove \eqref{E8} by differentiating and using our structural assumptions on $a$.

\medskip

\substep{2.1} Formula for $\zeta^L$ via an ODE argument.

\noindent On the one hand, by the weak formulation of \eqref{e.cor-eq} tested with $\phi_{te}^L$, we have 
$$
(\zeta^L)'(t) =\frac1t \expec{\fint_{B_L} a(x,te+\nabla \phi^L_{te}(x))\cdot (te+\nabla  \phi_{te}^L(x))   \dd x},
$$
which, in combination with the form of $a$, yields
\begin{multline}\label{E16}
t(\zeta^L)'(t)\, =\, \expec{\fint_{B_L} \rho(A(x),|te+\nabla \phi^L_{te}(x)|) |te+\nabla  \phi_{te}^L(x)|^2  \dd x}
\\
=\,  \expec{\fint_{B_L} \rho_1(A(x))|te+\nabla \phi^L_{te}(x)|^p \dd x}
+  \expec{\fint_{B_L} \rho_2(A(x),|te+\nabla \phi^L_{te}(x)|) |te+\nabla  \phi_{te}^L(x)|^2  \dd x}.
\end{multline}
On the other hand, using that $W(x,\xi)= \int_0^{|\xi|} s \rho(A(x),s)\dd s$
and the decomposition of $\rho$, we also have
\begin{multline}
\zeta^L(t)=\bar W^L(te)= \expec{\fint_{B_L} \int_0^{|te+\nabla \phi^L_{te}(x)|} s \rho(A(x),s)\dd s \dd x}
\\
=\, \frac1p\expec{\fint_{B_L} \rho_1(A(x))|te+\nabla \phi^L_{te}(x)|^p \dd x}+\expec{\fint_{B_L} \int_0^{|te+\nabla \phi^L_{te}(x)|} s \rho_2(A(x),s)\dd s \dd x}
.
\label{E17}
\end{multline}
From \eqref{E16} and \eqref{E17} we infer that $\zeta^L$ satisfies the differential relation
$$
t(\zeta^L)'(t)-p\zeta^L(t)\,=\,\expec{\fint_{B_L} \rho_2(A(x),|te+\nabla \phi^L_{te}(x)|) |te+\nabla  \phi_{te}^L(x)|^2  \dd x-p\fint_{B_L} \int_0^{|te+\nabla \phi^L_{te}(x)|} s \rho_2(A(x),s)\dd s \dd x},
$$
which we rewrite as
$(\zeta^L)'(t)-\frac pt \zeta^L(t)\,=\,\frac1t h(t)$
with 
$$
h(t):=\, \expec{\fint_{B_L} \rho_2(A(x),|te+\nabla \phi^L_{te}(x)|) |te+\nabla  \phi_{te}^L(x)|^2  \dd x-p\fint_{B_L} \int_0^{|te+\nabla \phi^L_{te}(x)|} s \rho_2(A(x),s)\dd s \dd x}.
$$
Let $t_\star \ge 1$ to be fixed later.
The solution of this ODE is explicitly given for all $t\ge t_\star$ by
%
$\zeta^L(t)=\frac{\zeta^L(t_\star)}{t_{\star}^p} t^p + +t^p\int_{t_\star}^{t}s^{-1-p}h(s) \dd s$.
Differentiating twice (which we can since $t \mapsto \nabla \phi_{te}^L(x)$ is differentiable for all $x\in B_L$), this yields
\begin{multline}
(\zeta^L)''(t)
\\
=p(p-1) \underbrace{\Big(\frac{\zeta^L(t_\star)}{t_\star^p} + \int_{t_\star}^{t}s^{-1-p}h(s) \dd s\Big)}_{\displaystyle =: \gamma^L(t)} t^{p-2}+\underbrace{\frac{\dd^2}{\dd t^2} \Big(t^p\int_{t_\star}^{t}s^{-1-p}h(s) \dd s\Big)-p(p-1)t^{p-2}\int_{t_\star}^{t}s^{-1-p}h(s) \dd s\Big) }_{\displaystyle =:R^L(t)}.
\label{E10}
\end{multline}
In the following two substeps we provide a bound from below 
for $\gamma^L(t)$ and a bound from above for $R^L(t)$.

\medskip

\substep{2.2} Choice of $t_\star^0$ and lower bound on  $\gamma^L(t)$.

\noindent On the one hand, recall that $\zeta^L(t)=\bar W^L(t e)$ 
 and that $W(x,\xi) =\int_0^{|\xi|} s \rho(A(x),s)\dd s \ge \lambda (\frac12 |\xi|^2 + \frac1p |\xi|^p)$, 
 so that by Jensen's inequality and $\int_{B_L}\nabla\phi^L_\xi=0$
$$
\zeta^L(t)\,=\, \bar W^L(t e) \,=\,\expec{\fint_{B_L} W(x,t e +\nabla\phi^L_{te}) \dd x}
\ge \frac{\lambda}p \expec{\fint_{B_L} |te+\nabla\phi^L_{te}|^p} = \frac{\lambda}p t^p.
$$
On the other hand, the assumption on $\rho_2$ implies that there exists a constant $c>0$ such that for all $t\geq 1$ and $\alpha \in [\lambda, 1]$,
\begin{equation}
\vert \rho_2(\alpha,t)\vert+t\vert\partial_{t}\rho_2(\alpha,t)\vert\leq c(1+t^{\beta}).
\label{E18}
\end{equation}
(The constant $c$ will change from line to line, but remains independent of $t$ and $L$.)
Hence, by the (deterministic) energy estimate $\fint_{B_L} |\nabla \phi_\xi^L|^p \le c(1+|\xi|^p)$,  \eqref{E18} yields for all $t\geq 1$
\begin{equation}
|h(t)|\,\le\, c t^{\beta+2},
\label{E12}
\end{equation}
so that 
$$
\Big| \int_{t}^{\infty}s^{-1-p}h(s) \dd s\Big| \, \le \, c t^{\beta-(p-2)}.
$$
Since $\beta<p-2$, we deduce that for $t_\star^0 \sim 1$ large enough (and independent of $L$) we have for all $t\ge t_\star \ge t_\star^0$
\begin{equation}\label{E12+}
 \gamma^L(t) \ge \frac{\lambda}{2p} .
\end{equation}

\medskip

\substep{2.3} Choice of $t_\star^1$ and upper bound on $R^L(t)$.

\noindent 
Since $R^L(t)=(p-1)t^{-2}h(t)+t^{-1}h'(t)$, it remains to estimate $h'(t)$. 
By differentiating $h$, using \eqref{E18}, and rearranging the terms, Cauchy-Schwarz' inequality yields
\begin{eqnarray*}
|h'(t)|& \le &c \expec{\fint_{B_L} (1+|te +\nabla \phi_{te}^L|^\beta)|te+\nabla \phi^L_{te}||e+\nabla \tilde \phi^L_{te,e}| }
\\
&\le &  c\expec{\fint_{B_L} (1+|te+\nabla \phi^L_{te}|)^{\beta+1-\frac{p-2}2}(1+|te +\nabla \phi_{te}^L|^{p-2})^\frac12|e+\nabla \tilde \phi^L_{te,e}| }
\\
&\le &  c\expec{\fint_{B_L} (1+|te +\nabla \phi_{te}^L|^{p-2})|e+\nabla \tilde \phi^L_{te,e}|^2 }^\frac12 \expec{\fint_{B_L}(1+|te+\nabla \phi^L_{te}|)^{p-2(p-2-\beta)} }^\frac12.
\end{eqnarray*}
For the second right-hand side factor, we use the (deterministic) energy estimate on $\nabla \phi^L_\xi$, which yields for $t\ge 1$
$$
 \expecL{\fint_{B_L}(1+|te+\nabla \phi^L_{te}|)^{p-2(p-2-\beta)} }^\frac12 \, \le \, c  t^{\frac p2-(p-2-\beta)},
$$
whereas for the first right-hand side factor we  use the (deterministic) energy estimate on $\nabla \phi^L_{\xi,e}$ (in favor of which we shall argue below), which yields for $t\ge 1$
\begin{equation}\label{E12K}
\expec{\fint_{B_L} (1+|te +\nabla \phi_{te}^L|^{p-2})|e+\nabla \tilde \phi^L_{te,e}|^2 }^\frac12
\,\le\, ct^{\frac{p-2}2}.
\end{equation}
These last three estimates then combine to 
$|h'(t)| \,\le \, c   t^{p-1-(p-2-\beta)}$.
With \eqref{E12} and the formula for $R^L(t)$, this entails for all $t\ge 1$
the control $|R^L(t)| \, \le \, c t^\beta$.
In particular, since $\beta<p-2$, there exists $t_\star^1 \ge 1$ such that for all $t  \ge t_\star^1$, we have
$|R^L(t)| \le \frac{\lambda(p-1)}{4}t^{p-2}$.

We conclude with the argument in favor of \eqref{E12K}.
By testing \eqref{e.Lcorr} with $\tilde \phi_{te,e}^L\in H^1_0(B_L)$ we obtain 
$$
\int_{B_L} \nabla \tilde \phi_{te,e}^L \cdot Da(x,te+\nabla \phi_{te}^L) \nabla \tilde \phi_{te,e}^L\dd x
= \int_{B_L} \nabla \tilde \phi_{te,e}^L \cdot Da(x,te+\nabla \phi_{te}^L)e\dd x,
$$
which, by our assumptions on $a$, entails
$$
\int_{B_L} |\nabla \tilde \phi_{te,e}^L|^2 (1+|te+\nabla \phi_{te}^L|^{p-2})
\,\lesssim \, \int_{B_L} |\nabla \tilde \phi_{te,e}^L|  (1+|te+\nabla \phi_{te}^L|^{p-2}),
$$
and therefore  \eqref{E12K} by Cauchy-Schwarz' inequality and  the energy estimate on $\nabla \phi_{te}^L$.

\medskip

\substep{2.4} Definition of $t_\star$ and proof of \eqref{E8}.

\noindent Set $t_\star = t_\star^0 \vee t_\star^1$, which is independent of $L$.
Step~1 yields \eqref{E8} in the regime $t\le t_\star$, whereas in the regime $t\ge t_\star$, \eqref{E8} follows from \eqref{E10}  in combination with
Substeps~2.2 and~2.3.
\end{proof}


\section{Periodization in law and functional inequalities}\label{append:per}

\subsection{Periodization in law of $a$}
We give the definition of the periodized ensemble $\mathbb{P}_L$. We generate the Gaussian field $G$ via a model $m\in L^1(\mathbb{R}^d)$ and a centred stationary white noise $W$ in form of
\begin{equation}\label{ModelC}
G=m\star W,
\end{equation}
where we assume that $m$ satisfies for some $\alpha>0$ and for all $x\in\mathbb{R}^d$
\begin{equation}\label{DecayModel}
\vert m(x)\vert\lesssim (1+\vert x\vert)^{-d-\alpha}.
\end{equation}
According to \eqref{ModelC}, the covariance function reads $\mathcal{C}=m\star m$. In the following, we denote by $\mathbb{P}$ the law of $W$ and $\mathbb{E}$ the expectation with respect to $\mathbb{P}$.
\begin{definition}\label{defi:PL}
Let $L\geq 1$ and $Q_L:=[\frac{L}{2},\frac{L}{2})^d$. The probability $\mathbb{P}_L$ (with expectation $\mathbb{E}_L$) is the stationary and centered Gaussian ensemble of scalar fields $G_L$ defined by the covariance function 
\begin{equation}\label{convCorrector:Eq13}
\mathcal{C}_L : x\in\mathbb{R}^d\mapsto \sum_{k\in\mathbb{Z}^d} \mathcal{C}(x+Lk).
\end{equation}
Equivalently, we have 
\begin{equation}\label{DefPeriodicWithModel}
G_L=(\mathds{1}_{Q_L} m_L)\star W,
\end{equation}
where 
\begin{equation}\label{PeriodizedModel}
m_L: x\in\mathbb{R}^d\mapsto \sum_{k\in\mathbb{Z}^d}m(x+kL).
\end{equation}
Clearly, the covariance function $\mathcal{C}_L$ and thus the realizations $G_L$ are $Q_L$-periodic. We identify $\mathbb{P}_L$ with its push forward under the map
$G\mapsto A:=(x\mapsto B(G(x))),$
where $B$ is defined in Hypothesis~\ref{hypo}.
\end{definition}
The following lemma shows the qualitative convergence as $L\uparrow+\infty$ of the coefficient field generated by $G_L$ towards the coefficient field generated by $G$. 
\begin{lemma}\label{approxcoef}
The coefficient fields $A:= \chi \star B(G)$ and $A_L:=\chi \star B(G_L)$ satisfy for all $L\ge 1$, $q\ge 1$, and $x\in Q_L$
%
%
\begin{equation}\label{convergenceofthecoef}
\mathbb{E} \Big[ \vert (A-A_L)(x)\vert^q\Big]^\frac1q
\,\lesssim_q\, (1+\dist(x,\partial Q_L))^{-\frac d2-\alpha}.
\end{equation}
\end{lemma}
\begin{proof}
Since $B$ is bounded and globally Lipschitz and $\chi$ is compactly supported (where w.l.o.g we assume $\text{supp }\chi\subset B$), we have for all $q\ge 1$ and $x\in Q_L$
$$
\vert (A-A_L)(x)\vert^q \lesssim_{q,B,\chi} \Big(\fint_{B(x)} |G_L-G|\Big)^q.
$$
By Gaussianity, cf.~\cite{bogachev1998gaussian},
$$\mathbb{E} \big[\vert G(z)-G_L(z)\vert^q\big]^\frac1q\lesssim_q \|m(z-\cdot)-\mathds{1}_{Q_L}m_L(z-\cdot)\|_{L^2(\R^d)}.$$
A direct computation yields 
\begin{eqnarray*}
\lefteqn{\int_{\R^d} (m(z-y)-\mathds{1}_{Q_L}m_L(z-y))^2dy}
\\
&\stackrel{\eqref{PeriodizedModel}}{=}&\int_{\mathbb{R}^d\backslash Q_L}m(z-y)^2dy+\int_{Q_L}\Big(\sum_{k\neq 0} m(z-y-kL)\Big)^2 dy
\\
&\stackrel{\eqref{DecayModel}}{\lesssim} &\int_{\mathbb{R}^d\backslash Q_L}(1+\vert z-y\vert)^{-2 (d+\alpha)}dy+\int_{Q_L} \Big(\sum_{k\neq 0}(1+\vert z-y-kL\vert)^{-(d+\alpha)} \Big)^2dy
\\
&\lesssim  &(1+\dist(x,\partial Q_L))^{-d-2\alpha},
\end{eqnarray*}
and the claim \eqref{convergenceofthecoef}
 follows.
\end{proof}

\subsection{Functional calculus}\label{sec:FC}

The ensemble $\mathbb{P}_L$ satisfies the following logarithmic-Sobolev  inequality (see for instance \cite{DG2,COJX-20}).
\begin{proposition}
There exists $\rho>0$ such that for all functional $F$ of $A$ with $\mathbb{E}_L[\vert F\vert^2]<+\infty$
\begin{equation}
\mathbb{E}_L[ F^2\log(F)]-\mathbb{E}_L[ F^2]\mathbb{E}_L[ \log(F)]\,\leq\, \frac{1}{\rho}\mathbb{E}_L\Big[ \int_{Q_L}\vert\partial_x F\vert^2\dd x\Big],
\label{LSIperiodic}
\end{equation}
where for all $x\in Q_L$
\begin{equation}
\partial_xF(A):=\sup_{\delta A}\limsup_{h\downarrow 0}\frac{F(A+h\delta A)-F(A)}{h},
\label{deffunctioderivper}
\end{equation}
and where the supremum runs over coefficient fields $\delta A$ that are supported in $B(x)$ and bounded by $1$ in $C^{\alpha}(B(x))$.\footnote{The fact one can assume $\|\delta A\|_{C^{\alpha}(B(x))}\leq 1$ comes from the convolution in \eqref{e.defAGauss}. }
\end{proposition}
The logarithmic Sobolev inequality \eqref{LSIperiodic} yields control of moments (see e.g.~\cite{DG1}).
\begin{lemma}\label{SGp} For all $q\ge 1$ and all random variables $F$
we have
\begin{equation}
\expecL{\vert F-\mathbb{E}[F]\vert^q}^{\frac{1}{q}}\,\lesssim  \, \sqrt{q}\,\expecL{\Big(\int_{Q_L}\vert \partial_{x} F\vert^2\dd x\Big)^{\frac{q}{2}}}^{\frac{1}{q}}.
\label{SGinegp1}
\end{equation}
\end{lemma}
The following standard result gives the link between algebraic moments and stretched exponential moments for non-negative random variables.
\begin{lemma}\label{momentexp}Let $X$ be a non-negative random variable. The following two statements are equivalent: 
\begin{itemize}
\item[(i)]There exists $C_1>0$ such that 
$$\mathbb{E}\big[\exp(\tfrac{1}{C_1}X)\big]\leq 2.$$
\item[(ii)] There exists $C_2>0$ such that
$$\mathbb{E}[X^p]^{\frac{1}{p}}\leq q\,C_2\quad\text{ for any $q\geq 1$}.$$
\end{itemize}
\end{lemma}
The last result of this subsection allows us to exchange supremum and expectation.
\begin{lemma}\label{unifproba}
Let  $X$ be a stationary random field. 
If there exists an exponent $\gamma>0$ such that for all $q\ge 1$
\begin{equation}
\expecL{\|X\|^q_{L^{\infty}(B)}}^{\frac{1}{q}}\leq q^{\gamma},
\label{probatoolsassume}
\end{equation}
then we have for all $\varepsilon>0$, $R\ge 1$, and $q\ge 1$
\begin{equation}
\expecL{(R^{-\varepsilon}\|X\|_{L^{\infty}(B_R )})^q}^{\frac{1}{q}}\lesssim_{d,\varepsilon} q^{\gamma}.
\label{probatoolsesti1}
\end{equation}
\end{lemma}
\begin{proof}
Let $R\ge 1$ and $\varepsilon>0$. Consider $N(R,d)\lesssim R^d$ points $(x_i)_{i\in\llbracket 1,N\rrbracket}\subset B_R$ such that $B_R\subset \bigcup_{i=1}^{N} B(x_i)$. We then
have 
\begin{equation}
\|X\|_{L^{\infty}(B_R)}\leq \max_{i\in\llbracket 1,N\rrbracket}\|X\|_{L^{\infty}(B(x_i))}.
\label{probatoolsesti3}
\end{equation}
By the discrete $\ell^{q}-\ell^{\infty}$ estimate for all $q\ge 1$,  \eqref{probatoolsesti3} turns into
$$\|X\|_{L^{\infty}(B_R(x))}\leq \Big(\sum_{i=1}^N\|X\|^q_{L^{\infty}(B(x_i))}\Big)^{\frac{1}{q}}.$$
Therefore, by taking the $q$-th moment, the stationarity of $X$ and the assumption \eqref{probatoolsassume}, we get
$$
\expecL{\|X\|^q_{L^{\infty}(B_R)}}^{\frac{1}{q}}\leq \max_{i\in\llbracket 1,N\rrbracket}\expecL{\|X\|^q_{L^{\infty}(B(x_i))}}^{\frac{1}{q}}N^{\frac{1}{q}}\stackrel{\eqref{probatoolsassume}}{\lesssim} q^{\gamma}R^{\frac{d}{q}},
$$
which yields the desired estimate \eqref{probatoolsesti1} provided $q\geq \frac{d}{\varepsilon}$ (and therefore in the whole range of exponents by H\"older's inequality).
\end{proof}

\subsection{Convergence of the periodization in law of the correctors}

In this subsection (and here only), we denote by $(\nabla\corNL,\nabla\sigma_{\xi})$ (resp. $(\nabla\corL,\nabla\tilde{\sigma}_{\xi,e})$)
the nonlinear corrector gradients (resp. linearized corrector gradients) associated with the coefficient field $A:=\chi\star B(G)$, and 
by $(\nabla\corNL^L,\nabla\sigma_{\xi}^L)$ (resp. $(\nabla\corL^L,\nabla\tilde{\sigma}^L_{\xi,e})$) the nonlinear corrector gradients (resp. linearized corrector gradients) associated with the periodized coefficient field $A_L:=\chi\star B(G_L)$. Note that under the coupling \eqref{DefPeriodicWithModel} it holds
$$(\nabla\corNL,\nabla\sigma_{\xi})_\#\mathbb{P}_L=(\nabla\corNL^L,\nabla\sigma^L_{\xi})_\#\mathbb{P} .$$
\begin{proposition}\label{convergenceofperiodiccorrectors}
For all $L\geq 1$, 
if $(\nabla\corNL,\nabla\sigma_{\xi})$ satisfies for all $q\ge 1$
\begin{equation}\label{convergenceofthecor-hyp}
\expecL{|(\nabla\corNL,\nabla\sigma_{\xi})|^q}^\frac1q \lesssim_{\xi,q} 1
\end{equation}
(where the multiplicative constant does not depend on $L$),
then for all $R\geq 1$ and $q\ge 1$, we have 
\begin{equation}
 \mathbb{E} \Big[ \int_{Q_R} \vert(\nabla\corNL,\nabla\sigma_{\xi})-(\nabla\corNL^L,\nabla\sigma^L_{\xi})\vert^q\Big]^\frac1q \underset{L\uparrow +\infty}{\longrightarrow}0.
\label{convergenceofthecor}
\end{equation}
As a consequence, 
for all $x \in \R^d$, we  have 
\begin{equation}
\mathbb{E} \Big[ \Big(\int_{B(x)} \Big|(\corNL,\sigma_{\xi})-(\corNL^L,\sigma^L_{\xi})+ \fint_{B}(\corNL^L,\sigma^L_{\xi})\Big|^2\Big)^\frac q2\Big]^\frac1q \underset{L\uparrow +\infty}{\longrightarrow}0.
\label{convergenceofthecor+}
\end{equation}
In addition, for all unit vectors $e \in \R^d$, the linearized correctors $(\nabla\corL,\nabla\tilde{\sigma}_{\xi,e})$ are well-defined, and if for all $q\ge 1$
\begin{equation}\label{convergenceofthecor-hyp+}
\expecL{|(\nabla\corL,\nabla\tilde{\sigma}_{\xi,e})|^q}^\frac1q \lesssim_q 1
\end{equation}
then we have for all $R\geq 1$ and $q\ge 1$,
\begin{equation}
 \mathbb{E} \Big[ \int_{Q_R}\vert(\nabla\corL,\nabla\tilde\sigma_{\xi,e})-(\nabla\corL^L,\nabla\tilde\sigma_{\xi,e}^L)\vert^q\Big]^\frac1q \underset{L\uparrow +\infty}{\longrightarrow}0.
\label{convergenceofthecorL}
\end{equation}
\end{proposition}

\begin{proof}
%
We  split the proof into three steps.

\medskip

\step1 Proof of \eqref{convergenceofthecor}.

\noindent By the assumption \eqref{convergenceofthecor-hyp} it is enough to prove the following (purely qualitative) statement: For all $R\ge 1$, 
\begin{equation}\label{e.approx-++0}
\lim_{L\uparrow +\infty} \expec{\int_{Q_R} |(\nabla \phi_\xi^L,\nabla \sigma_\xi^L)-(\nabla \phi_\xi,\nabla \sigma_\xi)|^p} =0.
\end{equation}
We start with the argument for $\nabla\corNL^L$.
For all $R,L\ge 1$ we have the a priori estimate $\expec{\int_{Q_R} |\nabla \corNL^L|^p}\lesssim R^d |\xi|^{2\& p}$, so that $\nabla \corNL^L$ is bounded in $L^p(d\Pp,L^p(Q_R))$.
As such it converges along an extraction to some $\nabla \corNL^\infty$ weakly in $L^p(d\Pp,L^p(Q_R))$ for all $R\ge 1$. In addition, since $\expec{|\nabla \corNL^L|^{p+\e}}^{\frac p{p+\e}}\lesssim |\xi|^{2\& p}$ (for some $\e>0$ by Meyers' estimate), we obtain by a standard equi-integrability argument (allowing to discard boundary effects) that $\nabla \corNL^\infty$ is a stationary random field. Hence we need to argue that $\nabla \corNL^\infty=\nabla \corNL$ and that the convergence is strong. 
Note that similarly, for all $R\ge 1$, the flux $a_\xi^L:=a_L(\cdot,\xi+\nabla \corNL^L)$ converges (say, along the same extraction) weakly in $L^{\frac{p+\e}{p-1}}(d\Pp,L^{\frac{p+\e}{p-1}}(Q_R))$ towards some stationary random field $a_\xi^\infty$. 

\medskip

We first recall the following consequence of qualitative homogenization (cf.~Theorem~\ref{th:qual-hom}): Almost surely we have
\begin{equation}\label{e.approx-++1}
\bar a(\xi)\,=\, \lim_{L\uparrow +\infty} \fint_{Q_L} a(x,\xi+\nabla \zeta^L_\xi(x))\dd x \,=\,
\lim_{L\uparrow +\infty}\expec{\fint_{Q_L} a(x,\xi+\nabla \zeta^L_\xi(x))\dd x},
\end{equation}
where $\zeta^L_\xi$ is the unique weak solution in $W^{1,p}_\per(Q_L)$ of 
the corrector equation $-\nabla \cdot a(\cdot,\xi+\nabla\zeta^L_\xi)=0$ on the cube $Q_L$ (note that we use $a$ and not $a_L$, so that $\zeta^L_\xi\ne \corNL^L$).

\medskip

Then we argue that 
\begin{equation}\label{e.approx-++2}
\lim_{L\uparrow +\infty} \expec{\fint_{Q_L} \big(a(x,\xi+\nabla \zeta^L_\xi(x))-a_L(x,\xi+\nabla \corNL^L(x))\big)\dd x}\,=\,0.
\end{equation}
We write $a_L(x,\xi+\nabla \corNL^L(x))=a(x,\xi+\nabla \corNL^L(x))+a_L(x,\xi+\nabla \corNL^L(x))-a(x,\xi+\nabla \corNL^L(x))$, and control each contribution to \eqref{e.approx-++2} separately.
Testing the equation for $\corNL^L$ (resp.~$\zeta^L_\xi$) with $\zeta^L_\xi$ (resp.~$\corNL^L$), we obtain
$$
\int_{Q_L} \nabla (\corNL^L-\zeta^L_\xi) \cdot \big(a(\xi+\nabla \corNL^L)-a(\xi+\nabla \zeta^L_\xi)\big) \,=\,\int_{Q_L} \nabla (\corNL^L-\zeta^L_\xi) \cdot \big(a(\xi+\nabla \corNL^L)-a_L(\xi+\nabla\corNL^L)\big).
$$
Taking  the expectation of this identity, using that for all $\xi'$ we have  $(a-a_L)(\xi') = (A-A_L) (1+|\xi'|^{p-2})\xi'$, using monotonicity on the left-hand side, and H\"older's inequality on the right-hand side, we obtain 
\begin{multline*}
\expec{\fint_{Q_L} |\nabla (\corNL^L-\zeta^L_\xi) |^p}
\,\lesssim \, \expec{\fint_{Q_L} |\nabla (\corNL^L-\zeta^L_\xi) |^p}^\frac1p
\expec{\fint_{Q_L} |A-A_L|^{q' \frac p{p-1}}}^\frac{p-1}{q'p} \expec{\fint_{Q_L} 1+|\xi+\nabla \corNL^L|^{qp}}^\frac{p-1}{qp},
\end{multline*} 
for some $q>1$ chosen such that $qp<p+\e$ (the Meyers exponent).
By our choice of $q$, the last right-hand side factor is bounded uniformly wrt $L$ whereas the second right-hand side factor goes to zero as $L\uparrow +\infty$ by Lemma~\ref{approxcoef}. This entails by continuity of $a$ that
$$
\Big|\expec{\fint_{Q_L} \big(a(x,\xi+\nabla \zeta^L_\xi(x))-a(x,\xi+\nabla \corNL^L(x))\big)\dd x}\Big|\, \lesssim_{|\xi|}\, \expec{\fint_{Q_L}|\nabla \zeta^L_\xi-\nabla \corNL^L|^p}^\frac1p
\stackrel{L\uparrow +\infty}{\longrightarrow}\,0.
$$
We now treat the second contribution to \eqref{e.approx-++2}.
By definition of $a-a_L$, we have 
$$
\expec{\fint_{Q_L} |a(\xi+\nabla \corNL^L)-a_L(\xi+\nabla \corNL^L)|}
\,\lesssim \, \expec{\fint_{Q_L} |A-A_L|^p}^\frac1p  \expec{\fint_{Q_L}1+ |\xi+\nabla \corNL^L|^p}^\frac1p ,
$$
so that also the second contribution to \eqref{e.approx-++2} vanishes, and \eqref{e.approx-++2} is proved.

\medskip

We are in the position to conclude.
The combination of \eqref{e.approx-++1}~\&~\eqref{e.approx-++2} 
with the periodization in law and Lemma~\ref{approxcoef} in form for all $x\in Q_L$ of 
\begin{eqnarray*}
{\xi \cdot \expec{\fint_{Q_L} a_L(\xi+\nabla\corNL^L)}}
&\stackrel{\eqref{e.cor-eq}}{=}&\expec{\fint_{Q_L} (\xi+\nabla\corNL^L) \cdot a_L(\xi+\nabla\corNL^L)}
\\
&=&\expec{ (\xi+\nabla\corNL^L)\cdot a_L(\xi+\nabla\corNL^L)(x)}
\\
&=&\expec{ (\xi+\nabla\corNL^L)\cdot  a(\xi+\nabla\corNL^L)(x)}+O_{|\xi|}((1+\dist(x,\partial Q_L))^{-\frac d2-\alpha})
\end{eqnarray*}
and
\begin{eqnarray*}
{ \expec{\xi \cdot\fint_{Q_L} a_L(\xi+\nabla\corNL^L)}}
&= &\expec{ \xi \cdot  a(\xi+\nabla\corNL^L)(x)}+O_{|\xi|}((1+\dist(x,\partial Q_L))^{-\frac d2-\alpha})
\end{eqnarray*}
 yields
\begin{equation}\label{e.approx-++3}
\lim_{L\uparrow +\infty} \expec{\nabla \corNL^L\cdot a(\xi+\nabla \corNL^L)(x)}\,=\,\expec{\nabla \phi_\xi \cdot a(\xi+\nabla \phi_\xi)(x)}= 0.
\end{equation}
Provided we establish that \eqref{e.approx-++3} can be post-processed into 
\begin{equation}\label{e.approx-++4}
\lim_{L\uparrow +\infty} \expec{(\nabla \phi_\xi-\nabla \corNL^L)\cdot \big(a(\xi+\nabla \phi_\xi)-a(\xi+\nabla \corNL^L)\big)(x)}=0,
\end{equation}
the desired statement~\eqref{e.approx-++0} for $\nabla \corNL^L$ will follow from monotonicity and integration over $Q_R$.
We now prove  \eqref{e.approx-++4}. Since
\begin{multline*}
(\nabla \phi_\xi-\nabla \corNL^L)\cdot \big(a(\xi+\nabla \phi_\xi)-a(\xi+\nabla \corNL^L)\big)
\,=\, \nabla \phi_\xi  \cdot a(\xi+\nabla \phi_\xi)+\nabla \corNL^L\cdot a(\xi+\nabla \corNL^L)
\\
-\nabla \corNL^L\cdot a(\xi+\nabla \phi_\xi)-\nabla \phi_\xi\cdot a(\xi+\nabla \corNL^L),
\end{multline*}
it remains to control the last two right-hand side terms using weak convergence.
On the one hand,
$$
\lim_{L\uparrow +\infty} \expec{\nabla \corNL^L\cdot a(\xi+\nabla \phi_\xi)(x)}=\expec{\nabla \corNL^\infty \cdot a(\xi+\nabla \phi_\xi)}=0
$$
by the weak formulation of the corrector equation and the stationarity of $\nabla \corNL^\infty$. On the other hand,
$$
\lim_{L\uparrow +\infty} \expec{\nabla \phi_\xi\cdot a(\xi+\nabla \corNL^L)(x)}=
\expec{\nabla \phi_\xi\cdot a_\xi^\infty}.
$$
The latter also vanishes since for all compactly supported test functions $\chi$ and $p$-integrable random variables $X$ we have for all $L\ge 1$
$$
\expec{X \int_{\R^d} \nabla \chi \cdot  a_\xi^L}=0,
$$
which entails at the limit $\expec{X \int_{\R^d} \nabla \chi \cdot  a_\xi^\infty}=0$, 
and, by stationarity of $a_\xi^\infty$, extends to stationary potential fields such as $\nabla \phi_\xi$ in form of $\expec{\nabla \phi_\xi\cdot a_\xi^\infty}=0$. Estimate~\eqref{e.approx-++4} is proved.

\medskip

We proceed the same way to prove the convergence \eqref{convergenceofthecor} for $(\nabla\sigma^L_{\xi})$,  combining the equations \eqref{e.Laplace-sig} and \eqref{e.div-sig} with the strong convergence
$$ \mathbb{E} \big[\vert a(\cdot,\xi+\nabla\corNL^L)(x)-a(\cdot,\xi+\nabla\corNL)(x)\vert^{\frac{p(1+\e)}{p-1}}\big]\underset{L\uparrow +\infty}{\rightarrow} 0,$$
which follows from \eqref{convergenceofthecor} for $(\nabla\phi^L_\xi)$.

\medskip

\step2 Proof of \eqref{convergenceofthecor+}.

\noindent First we claim that for all $R \ge 1$ and all functions $\zeta$ on $Q_R$ we have
the Poincar\'e inequality
\begin{equation}\label{e.poin-weak}
\int_{Q_R}\Big| \zeta-\fint_{B}\zeta\Big|^2 \,\lesssim \, \left\{
\begin{array}{rcl}
d=1&:&R^2\\
d=2&:& R^2\log^2(R+1)\\
d>2&:&R^d
\end{array}
\right\}
 \int_{Q_R} |\nabla \zeta|^2. 
\end{equation}
This simply follows by summation over dyadic scales of the following standard estimates for all $0\le i \le \log R$
$$
\int_{Q_{2^{i+1}}} \Big| \zeta-\fint_{Q_{2^i}}\zeta\Big|^2 \,\lesssim\, (2^i)^2 \int_{Q_{2^{i+1}}}|\nabla \zeta|^2, \quad \Big| \fint_{Q_{2^i}}\zeta -\fint_{Q_{2^{i-1}}}\zeta\Big|^2 \,\lesssim\,(2^i)^2 \fint_{Q_{2^i}} |\nabla \zeta|^2.
$$
Applying \eqref{e.poin-weak} to $\zeta=(\corNL,\sigma_{\xi})-(\corNL^L,\sigma^L_{\xi})$
and using the anchoring condition $\fint_{B}(\corNL,\sigma_{\xi})=0$,
 for all $x\in \R^d$ we have with $R=2(|x|+1)$
$$
\int_{B(x)} \Big|(\corNL,\sigma_{\xi})-(\corNL^L,\sigma^L_{\xi})+ \fint_{B}(\corNL^L,\sigma^L_{\xi})\Big|^2\,\lesssim\, \left\{
\begin{array}{rcl}
d=1&:& {R}^2\\
d=2&:& R^2\log^2(R+1)\\
d>2&:&R^d
\end{array}
\right\}
 \int_{Q_R}  |\nabla (\corNL,\sigma_{\xi})-\nabla (\corNL^L,\sigma^L_{\xi})|^2,
$$
so that \eqref{convergenceofthecor+} follows from  \eqref{convergenceofthecor}.

\medskip

\step3 Proof of \eqref{convergenceofthecorL}.

\noindent The argument is essentially the same as for  \eqref{convergenceofthecor}, and we only argue for $\nabla \corL$.
First, by \eqref{convergenceofthecor} and Lemma~\ref{lem:def-lincorr}, the linearized extended corrector $(\corL,\tilde \sigma_{\xi,e})$ is well-defined.
By the Schauder theory in form of \cite[Theorem 5.19]{giaquinta2013introduction} applied to the equation \eqref{e.Lcorr}  (the H\"older norm of the coefficients is controlled by Lemma~\ref{regestiNL} and local regularity theory for $\corNL^L$) and assumption~\eqref{convergenceofthecor-hyp+}, $\sup_{L\geq 1}\mathbb{E}[\|e+\nabla\corL^L\|^q_{C^{\alpha}(\overline{Q}_R)}]\lesssim_{R,\vert\xi\vert} 1$ for all $q\ge 1$.  We then extract a converging subsequence as before and identify the limit using the weak formulation of \eqref{e.Lcorr} together with the strong convergence $Da_L(\cdot,\xi+\nabla\corNL^L)\underset{L\uparrow +\infty}{\rightarrow}Da(\cdot,\xi+\nabla\corNL)$ (which follows from~\eqref{convergenceofthecor} and Lemma~\ref{approxcoef}).
\end{proof}

\section{Large-scale averages}

We prove in this section estimates used to control large-scale averages, which are variations around \cite{josien2020annealed}. We fix $L\geq 1$ and $\xi\in \mathbb{R}^d$, and use
the short-hand notation $\r$ for $\rNL$.
\begin{lemma}\label{reverse}
Let $m\in (0,1)$, $r\leq 3 \r(0)$ and $f: \mathbb{R}^d\rightarrow \mathbb{R}_+$ a measurable function. We have 
\begin{equation}
\fint_{B_r}\Big(\fint_{\Br(x)}f\Big)\, \dd x\lesssim \Big(\fint_{B_{\frac 32r}}\Big(\fint_{\Br(x)} f\Big)^m \dd x\Big)^{\frac{1}{m}}.
\label{average1}
\end{equation}
\end{lemma}

\begin{proof}
The estimate \eqref{average1} follows from
\begin{equation}
\sup_{x_0\in B_{r}}\fint_{\Br(x_0)} f \lesssim  \Big(\fint_{B_{\frac32r}}\Big(\fint_{\Br(x)} f\Big)^m \dd x\Big)^{\frac{1}{m}}.
\label{average2}
\end{equation}
Let $x_0\in B_{r}$ be fixed. Since $\rNL$ is $\frac 1{16}$-Lipschitz, we have 
for all $x \in B_{\frac r 8}$
\begin{equation}
\Br(x_0)\subset B_{\r(x_0)+\frac{9}{8}r}(x).
\label{inclusionctrlav1}
\end{equation}
Indeed, if $y\in \Br(x_0)$ and $x\in B_{\frac{r}{8}}$, we have 
$$\vert y-x\vert\leq \vert y-x_0\vert+\vert x_0-x\vert\leq \rNL(x_0)+r+\frac{r}{8}= \r(x_0)+\frac{9}{8}r.$$
By the $\frac 1{16}$-Lipschitz property of $\r$, we have
$\r(0)-\frac1{16}r \le \r(x_0) \le \r(0)+\frac1{16}r$. Together with
the assumption $r\le 2\r(0)$, this entails
$ \frac{13}{16}\r(0) \le \r(x_0)$ and $\r(x_0)+\frac{9}{8}r \le \frac{35}8 \r(0)$
so that  for all $x \in B_{\frac r 8}$, 
$$
|\Br(x_0)| \sim |B_{\r(x_0)+\frac{9}{8}r}(x)|.
$$
Combined with \eqref{inclusionctrlav1} this yields
$$\Big(\fint_{\Br(x_0)} f\Big)^{m}\lesssim \fint_{B_{\frac{r}{8}}}\Big(\fint_{B_{\r(x_0)+\frac{9}{8}r}(x)}f\Big)^m \dd x.$$
Now let $N\in\mathbb{N}$ depending only on $d$ and $(x_i)_{i\in \llbracket 1,N\rrbracket}\subset B_{\frac{11}{8}r}$ be such that 
\begin{equation}
B_{\frac{11}{8}r}\subset \bigcup_{i=1}^N B_{\frac{1}{16}r}(x_i).
\label{decompboule1}
\end{equation}
We claim that
\begin{equation}
B_{\r(x_0)+\frac{9}{8}r}\subset \bigcup_{i=1}^{N}B_{\r(x_0)-\frac{3}{16}r}(x_i).
\label{decompboule2}
\end{equation}
Indeed, let $y\in B_{\r(x_0)+\frac{9}{8}r}$ and set $z=\frac{\frac{11}{8}r}{\r(x_0)+\frac{9}{8}r}y$. Since $\vert z\vert\leq \frac{11}{8}r$, there exists $i\in \llbracket 1,N\rrbracket$ such that $\vert z-x_i\vert\leq \frac{1}{16}r$. Thus by the triangle inequality
$$\vert y-x_i\vert\leq \vert y-z\vert+\vert z-x_i\vert\leq \vert y\vert\left\vert 1-\frac{\frac{11}{8}r}{\r(x_0)+\frac{9}{8}r}\right\vert+\frac{1}{16}r.$$
Since $y\in B_{\r(x_0)+\frac{9}{8}r}$ and since the $\frac{1}{16}$-Lipschitz property of $\r$ entails
$\r(x_0)\geq \r(0)-\frac{1}{16}\vert x_0\vert\geq \frac r3-\frac r{16}= \frac{13}{48}r\geq \frac{1}{4}r$,
we have 
$$\vert y\vert\left\vert 1-\frac{\frac{11}{8}r}{\rNL(x_0)+\frac{9}{8}r}\right\vert\leq \r(x_0)-\frac{1}{4}r,$$
and consequently
$$\vert y-x_i\vert\leq \frac{1}{16}r+\r(x_0)-\frac{1}{4}r= \r(x_0)-\frac{3}{16}r,$$
which concludes the proof of \eqref{decompboule2}. We deduce from the sub-additive property of $x\in \mathbb{R}^+\mapsto x^{m}$ and the fact that for all $i\in \llbracket 1,N\rrbracket$, $x_i+B_{\frac{1}{8}r}\subset B_{\frac 32 r}$
\begin{equation*}
\fint_{B_{\frac{r}{8}}}\Big(\fint_{B_{\r(x_0)+\frac{9}{8}r}(x)}f\Big)^m \dd x\,\lesssim \,\sum_{i=1}^N\fint_{B_{\frac{r}{8}}}\Big(\fint_{B_{\r(x_0)-\frac{3}{16}r}(x+x_i)}f\Big)^m \dd x
\,\lesssim\, \fint_{B_{\frac 32 r}}\Big(\fint_{B_{\r(x_0)-\frac{3}{16}r}(x)}f\Big)^m \dd x,
\end{equation*} 
which concludes the proof of \eqref{average2} since the $\frac{1}{16}$-Lipschitz property of $\r$ entails for all $x_0 \in B_r$ and $x\in B_{\frac 32 r}$
\begin{equation*}
\r(x_0)-\frac{3}{16}r\,\leq\, \r(x)+\frac1{16}|x-x_0|-\frac{3}{16}r
\le \r(x)+ \frac1{16}(\frac32r+r)-\frac{3}{16}r
\, \le\, \r(x).
\end{equation*}
\end{proof}

\begin{lemma}\label{addfint*}Let $x_0\in\mathbb{R}^d$ and $r\geq 3\r(x_0)$, $f: \mathbb{R}^d\rightarrow \mathbb{R}_+$ a measurable function. We have 
\begin{equation}
\int_{B_{r}(x_0)}\Big(\fint_{\Br(x)}f\Big)\dd x\lesssim \int_{B_{\frac{67}{48}r}(x_0)} f ,
\label{fintout}
\end{equation}
and
\begin{equation}
\int_{B_{\frac{17}{12}r}(x_0)} f \lesssim \int_{B_{2r}(x_0)}\Big(\fint_{\Br(x)}f\Big)\, \dd x.
\label{fintint}
\end{equation}
By reverting from balls to cubes, this implies that there exists $C\ge 1$ depending only on $d$ such that for all $r\ge C\r(x_0)$
\begin{equation}
\int_{Q_{r}(x_0)}\Big(\fint_{\Br(x)}f\Big)\,\dd x\lesssim \int_{Q_{2r}(x_0)} f ,
\label{fintoutC}
\end{equation}
and
\begin{equation}
\int_{Q_r(x_0)} f \lesssim \int_{Q_{2r}(x_0)}\Big(\fint_{\Br(x)}f\Big)\, \dd x.
\label{fintintC}
\end{equation}
This constant $C$ (which only depends on $d$ and our upper bound  $\frac1{16}$
on the Lipschitz constant of $\r$) will be used to define the best Lipschitz constant for $\r$ (cf. Definition~\ref{minimalscaleNL}).
\end{lemma}
\begin{proof}
Without loss of generality, we may assume that $x_0=0$. 
Since  $\r$ is $\frac1{16}$-Lipschitz, we have for all integrable non-negative functions $g$
\begin{equation}\label{eq-integrals}
\int_{\R^d}\Big( \fint_{\Br(x)} g\Big) \dd x \sim \int_{\R^d} g ,
\end{equation} 
cf.~\cite[(140)]{GNO-reg} (which relies on the construction of a Calder\'on-Zygmund partition of $\R^d$ based on $\r$).
We start with the proof of \eqref{fintout}. 
Since $\r$ is $\frac1{16}$-Lipschitz and $3\r(0)\le r$, for all $x \in B_r$, $\r(x) \le \r(0)+\frac r{16} \le  \frac{19}{48}r$ so that $\Br(x) \subset B_{\frac{67}{48} r}$.
This yields \eqref{fintout} in form of
\begin{equation*}
\int_{B_r}\Big(\fint_{\Br(x)}f\Big)\,\dd x\,=\,\int_{B_r}\Big(\fint_{\Br(x)}f  \mathds{1}_{B_{\frac{67}{48} r}}\Big) \dd x 
\,\le\, \int_{\R^d}\Big(\fint_{\Br(x)}f  \mathds{1}_{B_{\frac{67}{48} r}}\Big) \dd x  \stackrel{\eqref{eq-integrals}}{\sim}  \int_{\R^d}f  \mathds{1}_{B_{\frac{67}{48} r}} .
\end{equation*}
We now turn to \eqref{fintint}.
By \eqref{eq-integrals},
$$
\int_{B_{\frac{17}{12}r}} f  \sim \int_{\R^d} \Big(\fint_{\Br(x)} f \mathds{1}_{B_{\frac{17}{12}r}}\Big) \dd x.
$$
Since $\r$ is $\frac{1}{16}$-Lipschitz and $3\r(0)\le r$, if $|x|\ge 2r$, 
$
\r(x) \le \r(0)+\frac{2}{16} r \le \frac{11}{24} r \le \frac{11}{48} |x| ,
$
so that $\Br(x) \subset \R^d \setminus B_{\frac{85}{48}r} \subset \R^d \setminus B_{\frac{17}{12}r}$.
Hence, exploiting the indicator function $\mathds{1}_{B_{\frac{17}{12}r}}$, the above turns into  
$$
\int_{B_{\frac{17}{12}r}} f  \sim \int_{B_{2r}} \Big(\fint_{\Br(x)} f \mathds{1}_{B_{\frac{17}{12}r}}\Big)\dd x \le \int_{B_{2r}} \Big(\fint_{\Br(x)} f\Big) \dd x,
$$
that is, \eqref{fintint}.
\end{proof}

\section*{Acknowledgements, Funding and Competing interests}

The authors warmly thank Mitia Duerinckx for discussions on annealed estimates,
and Mathias Sch\"affner for pointing out that the conditions of \cite{Bella_2020} apply to $\bar a$ in the setting of Theorem~\ref{th:2s} and for discussions on 
 regularity theory for operators with non-standard growth conditions.
The authors received financial support from the European Research Council (ERC) under the European Union's Horizon 2020 research and innovation programme (Grant Agreement n$^\circ$~864066).

\medskip

\noindent Data sharing not applicable to this article as no datasets were generated or analysed during the current study.

\medskip

\noindent The authors have no competing interests to declare that are relevant to the content of this article.

\bibliographystyle{plain}

\begin{thebibliography}{10}

\bibitem{AS-01}
G.~Alessandrini and M.~Sigalotti.
\newblock Geometric properties of solutions to the anisotropic {$p$}-{L}aplace
  equation in dimension two.
\newblock {\em Ann. Acad. Sci. Fenn. Math.}, 26(1):249--266, 2001.

\bibitem{AD-18}
S.~Armstrong and P.~Dario.
\newblock Elliptic regularity and quantitative homogenization on percolation
  clusters.
\newblock {\em Comm. Pure Appl. Math.}, 71(9):1717--1849, 2018.

\bibitem{AFK-20}
S.~Armstrong, S.~J. Ferguson, and T.~Kuusi.
\newblock Higher-order linearization and regularity in nonlinear
  homogenization.
\newblock {\em Arch. Ration. Mech. Anal.}, 237(2):631--741, 2020.

\bibitem{AFK-+}
S.~Armstrong, S.~J. Ferguson, and T.~Kuusi.
\newblock Homogenization, linearization, and large-scale regularity for
  nonlinear elliptic equations.
\newblock {\em Comm. Pure Appl. Math.}, 74(2):286--365, 2021.

\bibitem{AGK-16}
S.~Armstrong, A.~Gloria, and T.~Kuusi.
\newblock Bounded correctors in almost periodic homogenization.
\newblock {\em Arch. Ration. Mech. Anal.}, 222(1):393--426, 2016.

\bibitem{AKM-book}
S.~Armstrong, T.~Kuusi, and J.-C. Mourrat.
\newblock {\em Quantitative stochastic homogenization and large-scale
  regularity}, volume 352 of {\em Grundlehren der Mathematischen Wissenschaften
  [Fundamental Principles of Mathematical Sciences]}.
\newblock Springer, Cham, 2019.

\bibitem{AKM1}
S.~N. {Armstrong}, T.~{Kuusi}, and J.-C. {Mourrat}.
\newblock Mesoscopic higher regularity and subadditivity in elliptic
  homogenization.
\newblock {\em Comm. Math. Phys.}, 347(2):315--361, 2016.

\bibitem{AKM2}
S.~N. {Armstrong}, T.~{Kuusi}, and J.-C. {Mourrat}.
\newblock The additive structure of elliptic homogenization.
\newblock {\em Invent. Math.}, 208:999--1154, 2017.

\bibitem{AS}
S.~N. Armstrong and C.~K. Smart.
\newblock Quantitative stochastic homogenization of convex integral
  functionals.
\newblock {\em Ann. Sci. \'Ec. Norm. Sup\'er. (4)}, 49(2):423--481, 2016.

\bibitem{Avellaneda-Lin-87}
M.~Avellaneda and F.-H. Lin.
\newblock Compactness methods in the theory of homogenization.
\newblock {\em Comm. Pure Appl. Math.}, 40(6):803--847, 1987.

\bibitem{Avellaneda-Lin-91}
M.~Avellaneda and F.-H. Lin.
\newblock {$L^p$} bounds on singular integrals in homogenization.
\newblock {\em Comm. Pure Appl. Math.}, 44(8-9):897--910, 1991.

\bibitem{Barchiesi-Gloria-10}
M.~Barchiesi and A.~Gloria.
\newblock New counterexamples to the cell formula in nonconvex homogenization.
\newblock {\em Arch. Ration. Mech. Anal.}, 195(3):991--1024, 2010.

\bibitem{bella2018liouville}
P.~Bella, B.~Fehrman, and F.~Otto.
\newblock A {L}iouville theorem for elliptic systems with degenerate ergodic
  coefficients.
\newblock {\em Ann. Appl. Probab.}, 28(3):1379--1422, 2018.

\bibitem{Bella_2020}
P.~Bella and M.~Sch\"affner.
\newblock On the regularity of minimizers for scalar integral functionals with
  (p,q)-growth.
\newblock {\em Analysis \& PDE}, 13(7):2241--2257, 2020.

\bibitem{bella2019local}
P.~Bella and M.~Sch{\"a}ffner.
\newblock Local boundedness and {Harnack} inequality for solutions of linear
  nonuniformly elliptic equations.
\newblock {\em Comm. Pure Appl. Math.}, 74(3):453--477, 2021.

\bibitem{bogachev1998gaussian}
V.~I. Bogachev.
\newblock {\em Gaussian measures}, volume~62 of {\em Mathematical Surveys and
  Monographs}.
\newblock American Mathematical Society, Providence, RI, 1998.

\bibitem{Braides-85}
A.~Braides.
\newblock Homogenization of some almost periodic coercive functional.
\newblock {\em Rend. Accad. Naz. Sci. XL Mem. Mat. (5)}, 9(1):313--321, 1985.

\bibitem{Braides98}
A.~Braides and A.~Defranceschi.
\newblock {\em Homogenization of multiple integrals}, volume~12 of {\em Oxford
  Lecture Series in Mathematics and its Applications}.
\newblock The Clarendon Press Oxford University Press, New York, 1998.

\bibitem{BF-15}
M.~Briane and G.~A. Francfort.
\newblock Loss of ellipticity through homogenization in linear elasticity.
\newblock {\em Math. Models Methods Appl. Sci.}, 25(5):905--928, 2015.

\bibitem{CP-98}
L.~A. Caffarelli and I.~Peral.
\newblock On {$W^{1,p}$} estimates for elliptic equations in divergence form.
\newblock {\em Comm. Pure Appl. Math.}, 51(1):1--21, 1998.

\bibitem{CNV-15}
J.~Cheeger, A.~Naber, and D.~Valtorta.
\newblock Critical sets of elliptic equations.
\newblock {\em Comm. Pure Appl. Math.}, 68(2):173--209, 2015.

\bibitem{CS-04}
K.~D. Cherednichenko and V.~P. Smyshlyaev.
\newblock On full two-scale expansion of the solutions of nonlinear periodic
  rapidly oscillating problems and higher-order homogenised variational
  problems.
\newblock {\em Arch. Ration. Mech. Anal.}, 174(3):385--442, 2004.

\bibitem{DalMaso-Gconv}
V.~Chiad\`o~Piat, G.~Dal~Maso, and A.~Defranceschi.
\newblock {$G$}-convergence of monotone operators.
\newblock {\em Ann. Inst. H. Poincar\'{e} Anal. Non Lin\'{e}aire},
  7(3):123--160, 1990.

\bibitem{CD-16}
A.~Chiarini and J.-D. Deuschel.
\newblock Invariance principle for symmetric diffusions in a degenerate and
  unbounded stationary and ergodic random medium.
\newblock {\em Ann. Inst. Henri Poincar\'{e} Probab. Stat.}, 52(4):1535--1563,
  2016.

\bibitem{Clozeau-20}
N.~Clozeau.
\newblock Optimal decay of the parabolic semigroup in stochastic homogenization
  for correlated coefficient fields.
\newblock {\em Stoch. Partial Differ. Equ. Anal. Comput.}
\newblock In press.

\bibitem{COJX-20}
N.~Clozeau, M.~Josien, F.~Otto, and Q.~Xu.
\newblock Bias in the representative volume element method: periodize the
  ensemble instead of its realizations.
\newblock {\em Found. Comput. Math.}
\newblock In press.

\bibitem{DalMaso-93}
G.~Dal~Maso.
\newblock {\em An introduction to {$\Gamma$}-convergence}.
\newblock Progress in Nonlinear Differential Equations and their Applications,
  8. Birkh\"auser Boston Inc., Boston, MA, 1993.

\bibitem{DalMaso-corr}
G.~Dal~Maso and A.~Defranceschi.
\newblock Correctors for the homogenization of monotone operators.
\newblock {\em Differential Integral Equations}, 3(6):1151--1166, 1990.

\bibitem{DalMaso86}
G.~Dal~Maso and L.~Modica.
\newblock Nonlinear stochastic homogenization.
\newblock {\em Ann. Mat. Pura Appl. (4)}, 144:347--389, 1986.

\bibitem{DalMaso-Modica-86b}
G.~Dal~Maso and L.~Modica.
\newblock Nonlinear stochastic homogenization and ergodic theory.
\newblock {\em J. Reine Angew. Math.}, 368:28--42, 1986.

\bibitem{Dario}
P.~Dario.
\newblock Optimal corrector estimates on percolation cluster.
\newblock {\em Ann. Appl. Probab.}, 31(1):377--431, 2021.

\bibitem{DeGiorgi-Franzoni-75}
E.~De~Giorgi and T.~Franzoni.
\newblock Su un tipo di convergenza variazionale.
\newblock {\em Atti Accad. Naz. Lincei Rend. Cl. Sci. Fis. Mat. Natur. (8)},
  58(6):842--850, 1975.

\bibitem{duerinckx2019scaling}
M.~Duerinckx, J.~Fischer, and A.~Gloria.
\newblock Scaling limit of the homogenization commutator for {G}aussian
  coefficient fields.
\newblock {\em Ann. Appl. Probab.}, 32(2):1179--1209, 2022.

\bibitem{DG-16b}
M.~Duerinckx and A.~Gloria.
\newblock Stochastic homogenization of nonconvex unbounded integral functionals
  with convex growth.
\newblock {\em Arch. Ration. Mech. Anal.}, 221(3):1511--1584, 2016.

\bibitem{DG2}
M.~Duerinckx and A.~Gloria.
\newblock Multiscale functional inequalities in probability: concentration
  properties.
\newblock {\em ALEA Lat. Am. J. Probab. Math. Stat.}, 17(1):133--157, 2020.

\bibitem{DG1}
M.~Duerinckx and A.~Gloria.
\newblock Multiscale functional inequalities in probability: constructive
  approach.
\newblock {\em Ann. H. Lebesgue}, 3:825--872, 2020.

\bibitem{DG-21}
M.~Duerinckx and A.~Gloria.
\newblock Quantitative homogenization theory for random suspensions in steady
  {S}tokes flow.
\newblock {\em J. \'{E}c. polytech. Math.}, 9:1183--1244, 2022.

\bibitem{DGO2}
M.~Duerinckx, A.~Gloria, and F.~Otto.
\newblock Robustness of the pathwise structure of fluctuations in stochastic
  homogenization.
\newblock {\em Probab. Theory Related Fields}, 178(1-2):531--566, 2020.

\bibitem{DGO1}
M.~Duerinckx, A.~Gloria, and F.~Otto.
\newblock The structure of fluctuations in stochastic homogenization.
\newblock {\em Comm. Math. Phys.}, 377(1):259--306, 2020.

\bibitem{DO-20}
M.~Duerinckx and F.~Otto.
\newblock Higher-order pathwise theory of fluctuations in stochastic
  homogenization.
\newblock {\em Stoch. Partial Differ. Equ. Anal. Comput.}, 8(3):625--692, 2020.

\bibitem{fischer2019optimal}
J.~Fischer and S.~Neukamm.
\newblock Optimal homogenization rates in stochastic homogenization of
  nonlinear uniformly elliptic equations and systems.
\newblock {\em Arch. Ration. Mech. Anal.}, 242(1):343--452, 2021.

\bibitem{FG-16}
G.~A. Francfort and A.~Gloria.
\newblock Isotropy prohibits the loss of strong ellipticity through
  homogenization in linear elasticity.
\newblock {\em C. R. Math. Acad. Sci. Paris}, 354(11):1139--1144, 2016.

\bibitem{FMT}
G.~A. Francfort, F.~Murat, and L.~Tartar.
\newblock Homogenization of monotone operators in divergence form with
  {$x$}-dependent multivalued graphs.
\newblock {\em Ann. Mat. Pura Appl. (4)}, 188(4):631--652, 2009.

\bibitem{FJM-02}
G.~Friesecke, R.~D. James, and S.~M\"{u}ller.
\newblock A theorem on geometric rigidity and the derivation of nonlinear plate
  theory from three-dimensional elasticity.
\newblock {\em Comm. Pure Appl. Math.}, 55(11):1461--1506, 2002.

\bibitem{GMT-93}
G.~Geymonat, S.~M\"{u}ller, and N.~Triantafyllidis.
\newblock Homogenization of non-linearly elastic materials, microscopic
  bifurcation and macroscopic loss of rank-one convexity.
\newblock {\em Arch. Rational Mech. Anal.}, 122(3):231--290, 1993.

\bibitem{giaquinta2013introduction}
M.~Giaquinta and L.~Martinazzi.
\newblock {\em An introduction to the regularity theory for elliptic systems,
  harmonic maps and minimal graphs}, volume~11 of {\em Appunti. Scuola Normale
  Superiore di Pisa (Nuova Serie) [Lecture Notes. Scuola Normale Superiore di
  Pisa (New Series)]}.
\newblock Edizioni della Normale, Pisa, second edition, 2012.

\bibitem{GN-11}
A.~Gloria and S.~Neukamm.
\newblock Commutability of homogenization and linearization at identity in
  finite elasticity and applications.
\newblock {\em Ann. Inst. H. Poincar\'{e} Anal. Non Lin\'{e}aire},
  28(6):941--964, 2011.

\bibitem{GNO2}
A.~Gloria, S.~Neukamm, and F.~Otto.
\newblock An optimal quantitative two-scale expansion in stochastic
  homogenization of discrete elliptic equations.
\newblock {\em M2AN Math. Model. Numer. Anal.}, 48(2):325--346, 2014.

\bibitem{GNO1}
A.~Gloria, S.~Neukamm, and F.~Otto.
\newblock Quantification of ergodicity in stochastic homogenization: optimal
  bounds via spectral gap on {G}lauber dynamics.
\newblock {\em Invent. Math.}, 199(2):455--515, 2015.

\bibitem{GNO-reg}
A.~Gloria, S.~Neukamm, and F.~Otto.
\newblock A regularity theory for random elliptic operators.
\newblock {\em Milan J. Math.}, 88(1):99--170, 2020.

\bibitem{GNO-quant}
A.~Gloria, S.~Neukamm, and F.~Otto.
\newblock Quantitative estimates in stochastic homogenization for correlated
  coefficient fields.
\newblock {\em Anal. PDE}, 14(8):2497--2537, 2021.

\bibitem{GO4}
A.~Gloria and F.~Otto.
\newblock The corrector in stochastic homogenization: optimal rates, stochastic
  integrability, and fluctuations.
\newblock Preprint, arXiv:1510.08290.

\bibitem{GO1}
A.~Gloria and F.~Otto.
\newblock An optimal variance estimate in stochastic homogenization of discrete
  elliptic equations.
\newblock {\em Ann. Probab.}, 39(3):779--856, 2011.

\bibitem{GO2}
A.~Gloria and F.~Otto.
\newblock An optimal error estimate in stochastic homogenization of discrete
  elliptic equations.
\newblock {\em Ann. Appl. Probab.}, 22(1):1--28, 2012.

\bibitem{Gloria-Otto-10b}
A.~Gloria and F.~Otto.
\newblock Quantitative results on the corrector equation in stochastic
  homogenization.
\newblock {\em Journ. Europ. Math. Soc. (JEMS)}, 19:3489--3548, 2017.

\bibitem{GR-19}
A.~Gloria and M.~Ruf.
\newblock Loss of strong ellipticity through homogenization in 2{D} linear
  elasticity: a phase diagram.
\newblock {\em Arch. Ration. Mech. Anal.}, 231(2):845--886, 2019.

\bibitem{G99}
S.~Guti\'{e}rrez.
\newblock Laminations in linearized elasticity: the isotropic non-very strongly
  elliptic case.
\newblock {\em J. Elasticity}, 53(3):215--256, 1998/99.

\bibitem{JKO94}
V.~V. Jikov, S.~M. Kozlov, and O.~A. Ole{\u\i}nik.
\newblock {\em Homogenization of differential operators and integral
  functionals}.
\newblock Springer-Verlag, Berlin, 1994.
\newblock Translated from Russian by G. A. Iosif{\cprime}yan.

\bibitem{josien2020annealed}
M.~Josien and F.~Otto.
\newblock The annealed {C}alder\'{o}n-{Z}ygmund estimate as convenient tool in
  quantitative stochastic homogenization.
\newblock {\em J. Funct. Anal.}, 283(7):Paper No. 109594, 74, 2022.

\bibitem{kuusi2014nonlinear}
T.~Kuusi and G.~Mingione.
\newblock {A nonlinear Stein theorem}.
\newblock {\em Calc. Var. Partial Differential Equations}, 51(1-2):45--86,
  2014.

\bibitem{kuusi2014guide}
T.~Kuusi and G.~Mingione.
\newblock Guide to nonlinear potential estimates.
\newblock {\em Bull. Math. Sci.}, 4(1):1--82, 2014.

\bibitem{LiNiren-03}
Y.~Li and L.~Nirenberg.
\newblock Estimates for elliptic systems from composite material.
\newblock {\em Comm. Pure Appl. Math.}, 56(7):892--925, 2003.
\newblock Dedicated to the memory of J\"{u}rgen K. Moser.

\bibitem{LiVog-00}
Y.~Y. Li and M.~Vogelius.
\newblock Gradient estimates for solutions to divergence form elliptic
  equations with discontinuous coefficients.
\newblock {\em Arch. Ration. Mech. Anal.}, 153(2):91--151, 2000.

\bibitem{Lindvist-06}
P.~Lindqvist.
\newblock {\em Notes on the {$p$}-{L}aplace equation}, volume 102 of {\em
  Report. University of Jyv\"{a}skyl\"{a} Department of Mathematics and
  Statistics}.
\newblock University of Jyv\"{a}skyl\"{a}, Jyv\"{a}skyl\"{a}, 2006.

\bibitem{MaO}
D.~Marahrens and F.~Otto.
\newblock Annealed estimates on the {G}reen's function.
\newblock {\em Probab. Theory Related Fields}, 163(3-4):527--573, 2015.

\bibitem{Marcellini-78}
P.~Marcellini.
\newblock Periodic solutions and homogenization of nonlinear variational
  problems.
\newblock {\em Ann. Mat. Pura Appl. (4)}, 117:139--152, 1978.

\bibitem{Marcellini-91}
P.~Marcellini.
\newblock Regularity and existence of solutions of elliptic equations with
  {$p,q$}-growth conditions.
\newblock {\em J. Differential Equations}, 90(1):1--30, 1991.

\bibitem{Messaoudi-Michaille}
K.~Messaoudi and G.~Michaille.
\newblock Stochastic homogenization of nonconvex integral functionals.
\newblock {\em RAIRO Mod\'{e}l. Math. Anal. Num\'{e}r.}, 28(3):329--356, 1994.

\bibitem{DarkSide}
G.~Mingione.
\newblock Regularity of minima: an invitation to the dark side of the calculus
  of variations.
\newblock {\em Appl. Math.}, 51(4):355--426, 2006.

\bibitem{Muller-87}
S.~M\"{u}ller.
\newblock Homogenization of nonconvex integral functionals and cellular elastic
  materials.
\newblock {\em Arch. Rational Mech. Anal.}, 99(3):189--212, 1987.

\bibitem{MN-11}
S.~M\"{u}ller and S.~Neukamm.
\newblock On the commutability of homogenization and linearization in finite
  elasticity.
\newblock {\em Arch. Ration. Mech. Anal.}, 201(2):465--500, 2011.

\bibitem{NS}
A.~Naddaf and T.~Spencer.
\newblock Estimates on the variance of some homogenization problems.
\newblock Preprint, 1998.

\bibitem{NS-18}
S.~Neukamm and M.~Sch\"{a}ffner.
\newblock Quantitative homogenization in nonlinear elasticity for small loads.
\newblock {\em Arch. Ration. Mech. Anal.}, 230(1):343--396, 2018.

\bibitem{NeukSch-19}
S.~Neukamm and M.~Sch\"{a}ffner.
\newblock Lipschitz estimates and existence of correctors for nonlinearly
  elastic, periodic composites subject to small strains.
\newblock {\em Calc. Var. Partial Differential Equations}, 58(2):Paper No. 46,
  51, 2019.

\bibitem{Otto-Tlse}
F.~Otto.
\newblock Introduction to stochastic homogenization.
\newblock Lecture notes for a course taught during the Winterschool ``Calculus
  of Variations and Probability'' at Universit\'e Paul Sabatier, Toulouse,
  2019.

\bibitem{Pankov}
A.~Pankov.
\newblock {\em {$G$}-convergence and homogenization of nonlinear partial
  differential operators}, volume 422 of {\em Mathematics and its
  Applications}.
\newblock Kluwer Academic Publishers, Dordrecht, 1997.

\bibitem{Shen-07}
Z.~Shen.
\newblock The {$L^p$} boundary value problems on {L}ipschitz domains.
\newblock {\em Adv. Math.}, 216(1):212--254, 2007.

\bibitem{Torquato-02}
S.~Torquato.
\newblock {\em Random heterogeneous materials}, volume~16 of {\em
  Interdisciplinary Applied Mathematics}.
\newblock Springer-Verlag, New York, 2002.
\newblock Microstructure and macroscopic properties.

\bibitem{Uhlenbeck}
K.~Uhlenbeck.
\newblock Regularity for a class of non-linear elliptic systems.
\newblock {\em Acta Math.}, 138(3-4):219--240, 1977.

\bibitem{Yurinskii-86}
V.~V. Yurinski\u{\i}.
\newblock Averaging of symmetric diffusion in a random medium.
\newblock {\em Sibirsk. Mat. Zh.}, 27(4):167--180, 215, 1986.

\bibitem{Zhikov-01}
V.~V. Zhikov.
\newblock On the problem of the passage to the limit of nonuniformly elliptic
  equations in divergence form.
\newblock {\em Funktsional. Anal. i Prilozhen.}, 35(1):23--39, 96, 2001.

\end{thebibliography}
\def\cprime{$'$}

\end{document}